\documentclass{amsart}

\usepackage{mathrsfs}
\usepackage{amscd}
\usepackage{amsmath}
\usepackage{amssymb}
\usepackage{amsthm}
\usepackage{epsf}
\usepackage{latexsym}

\usepackage{verbatim}\usepackage[all, cmtip]{xy}
\usepackage{tikz}
\usetikzlibrary{math}
\usetikzlibrary{positioning}
\usetikzlibrary{matrix}
\usepackage{float}
\usepackage[hidelinks]{hyperref}
\usepackage{comment}
\usepackage{enumitem}

\usepackage{ytableau}

\usepackage[OT2,T1]{fontenc}
\DeclareSymbolFont{cyrletters}{OT2}{wncyr}{m}{n}
\DeclareMathSymbol{\Sha}{\mathalpha}{cyrletters}{"58}

\tikzstyle{bsq}=[rectangle, draw, thick, minimum width=.5cm, minimum height=.5cm]
\tikzstyle{bver}=[rectangle, draw, thick, minimum width=1cm, minimum height=2cm]
\tikzstyle{bhor}=[rectangle, draw, thick, minimum width=2cm, minimum height=1cm]

\usepackage[left=3.2cm, right=3.2cm]{geometry}

\newtheorem{theorem}{Theorem}[section]
\newtheorem{lemma}[theorem]{Lemma}

\newtheorem{corollary}[theorem]{Corollary}
\newtheorem{proposition}[theorem]{Proposition}

\newtheorem{varexample}[theorem]{Example}

\theoremstyle{definition}
\newtheorem{remark}[theorem]{Remark}
\newtheorem{definition}[theorem]{Definition}

\newcommand{\Spec}{\mathrm{Spec}\,}

\DeclareSymbolFont{extraup}{U}{zavm}{m}{n}
\DeclareMathSymbol{\varheart}{\mathalpha}{extraup}{86}
\DeclareMathSymbol{\vardiamond}{\mathalpha}{extraup}{87}

\def\PP{{\textbf P}}
\def\OO{\mathcal{O}}
\def\cN{\mathcal{N}}
\def\F{\mathcal{F}}
\def\P{\mathcal{P}}

\def\E{\mathcal{E}}
\def\G{\mathcal{G}}
\def\cS{\mathcal{S}}
\def\cT{\mathcal{T}}

\def\L{\mathcal{L}}

\def\cM{\mathcal{M}}

\def\mm{\overline{\mathcal{M}}}

\def\pm{\widetilde{\mathcal{M}}}

\def\ttem{\overline{\mathcal{M}}^{\circ}}

\newcommand{\RR}{\mathbb{R}}
\newcommand{\fM}{\mathfrak{M}}
\newcommand{\M}{\overline{M}}
\newcommand{\MM}{\overline{\mathfrak{M}}}
\newcommand{\Grd}{\widetilde{\mathfrak{G}}^r_d}
\newcommand{\fU}{\mathfrak{U}}

\newcommand{\cD}{\mathcal{D}}

\newcommand{\cO}{\mathcal{O}}

\newcommand{\cA}{\mathcal{A}}
\newcommand{\cB}{\mathcal{B}}

\newcommand{\ord}{\operatorname{ord}}
\newcommand{\Trop}{\operatorname{Trop}}
\newcommand{\trop}{\operatorname{trop}}
\newcommand{\ddiv}{\operatorname{div}}
\newcommand{\Div}{\operatorname{Div}}
\newcommand{\PL}{\operatorname{PL}}

\newcommand{\val}{\operatorname{val}}
\newcommand{\Pic}{\operatorname{Pic}}
\newcommand{\Sym}{\operatorname{Sym}}

\newcommand{\vir}{\mathrm{virt}}

\newcommand{\wt}{\mathrm{wt}}

\newcommand{\an}{\mathrm{an}}

\newenvironment{example}{\begin{varexample}
\begin{normalfont}}{\end{normalfont}
\end{varexample}}

\begin{document}
\title{The Kodaira dimensions of $\mm_{22}$ and $\mm_{23}$}
\author[G. Farkas]{Gavril Farkas}
\address{Gavril Farkas: Institut f\"ur Mathematik,   Humboldt-Universit\"at zu Berlin \hfill \newline\texttt{}
\indent Unter den Linden 6,
10099 Berlin, Germany}
\email{{\tt farkas@math.hu-berlin.de}}
\author[D. Jensen]{David Jensen}
\address{David Jensen: Department of Mathematics,  University of Kentucky \hfill \newline\texttt{}
\indent 733 Patterson Office Tower,
Lexington, KY 40506--0027, USA}
\email{{\tt dave.jensen@uky.edu}}
\author[S. Payne]{Sam Payne}
\address{Sam Payne: Department of Mathematics,  University of Texas at Austin \hfill \newline\texttt{}
\indent 2515 Speedway, RLM 8.100,
Austin, TX 78712, USA}
\email{{\tt sampayne@utexas.edu}}
\date{\today}
\bibliographystyle{alpha}

\vspace{-10 pt}

\begin{abstract}
We prove that the moduli spaces of curves of genus 22 and 23 are of general type.
To do this, we calculate certain virtual divisor classes of small slope associated to linear series of rank $6$ with quadric relations.  We then develop new tropical methods for studying linear series and independence of quadrics and show that these virtual classes are represented by effective divisors.
\end{abstract}

\maketitle

\setcounter{tocdepth}{1}
\tableofcontents

\vspace{-15 pt}

\section{Introduction}

Many of the familiar moduli spaces in algebraic geometry, such as those parametrizing curves, abelian varieties, or $K3$ surfaces, have infinitely many irreducible components, of which all but finitely many are of general type.  The remaining few components are typically uniruled.  Understanding which components are uniruled and which are of general type is often difficult.  Indeed, aside from the moduli of spin curves \cite{FarkasVerra14}, all of the standard moduli spaces include notorious open cases, that is, components whose Kodaira dimensions are unknown.  The present paper aims to resolve two long-standing cases for the moduli space $\mm_g$ of curves of genus $g$.

\begin{theorem}
\label{thm:generaltype}
The moduli spaces $\mm_{22}$ and $\mm_{23}$ are of general type.
\end{theorem}

\noindent
This extends earlier results of Harris, Mumford and Eisenbud, who showed that $\mm_g$ is of general type for $g\geq 24$, in the landmark papers \cite{HarrisMumford82, Harris84, EisenbudHarris87b}, and improves on the thesis result of the first author, who showed that $\mm_{23}$ has Kodaira dimension at least 2 \cite{Farkas00}. These general type statements contrast with the classical result of Severi \cite{Severi15} that $\mm_g$ is unirational for $g\leq 10$ (see \cite{ArbarelloCornalba81} for a modern treatment) and with the more recent results of many authors \cite{Sernesi81, ChangRan84, ChangRan86, Verra05, BrunoVerra05, Schreyer15}, which taken together show that $\mm_g$ is unirational for $g\leq 14$ and that $\mm_{15}$ is rationally connected.  Chang and Ran also argued that $\mm_{16}$ is  uniruled \cite{ChangRan91}, but Tseng recently found a fatal computational error in this argument \cite{Tseng19}, and this case is again open.  The Kodaira dimension of $\mm_g$ in unknown for $16 \leq g \leq 21$.

\subsection{Divisors of small slope} \label{sec:intro-smallslope}
As in the earlier proofs for $g \geq 24$, we show that $\mm_{22}$ and $\mm_{23}$ are of general type by producing effective divisors of slope less than $\frac{13}{2}$, which is the slope of the canonical divisor $K_{\mm_g}$.  The Slope Conjecture of Harris and Morrison \cite{HarrisMorrison90} predicted that all effective divisors on $\mm_g$ have slope at least $6+\frac{12}{g+1}$.  This led people to believe that $\mm_{g}$ would be uniruled for $g < 23$. The earliest known counterexample to the Slope Conjecture is the closure in $\mm_{10}$ of the locus of smooth curves lying on a $K3$ surface, which is equal to the divisorial component of the locus of curves $[X]\in \cM_{10}$ with a degree $12$ map to $\PP^4$ whose image is contained in a quadric  \cite{FarkasPopa05}.  This is the first in an infinite sequence of counterexamples in genus $2s^2+s$, for $s \geq 2$; the second is the closure in $\mm_{21}$ of the divisorial component of the locus of curves with a degree $24$ map to $\PP^6$ whose image is contained in a quadric \cite{Farkas09b, Khosla07}.

\vskip 5pt

Our divisors of small slope on $\mm_{22}$ and $\mm_{23}$ are natural generalizations of this second example on $\mm_{21}$.  Roughly speaking, the divisors $\widetilde{\mathfrak{D}}_{22}$ and $\widetilde{\mathfrak{D}}_{23}$ are the closures of the loci of smooth curves with a map to $\PP^6$ of degree $25$ and $26$, respectively, whose image is contained in a quadric. We note, however, that this example on $\mm_{21}$ is the push forward of a codimension 1 locus in a space of linear series that is generically finite over the moduli of curves.  A major new difficulty in the present construction is that we push forward a higher codimension locus in a space of linear series that maps onto $\mm_g$ with positive dimensional fibers. This makes both carrying out the intersection theory calculations and checking the needed transversality assumptions incomparably more challenging.  We now sketch the construction; for the precise details, see \S\ref{virtualis_divizorok}.

For a general curve $X$ of genus $g=22$ or $23$, the variety $W^6_{g+3}(X)$ is irreducible of dimension equal to $g-21$. Moreover, each line bundle $L\in W^6_{g+3}(X)$ is very ample, with $h^0(X,L) = 7$ and $h^1(X, L)=3$. Consider the multiplication map
\[
\phi_L \colon \mathrm{Sym}^2 H^0(X, L)\rightarrow H^0(X, L^{\otimes 2}).
\]
Note that $\mbox{dim } \mbox{Sym}^2 H^0(X,L)=28$ and, by Riemann-Roch, $h^0(X, L^{\otimes 2})=g+7$.  Therefore, the locus where $\phi_L$ is non-injective has expected codimension $h^0(X,L^{\otimes 2})-28+1=g-20$ in the space of such pairs $[X,L]$.  Since this expected codimension is one more than the dimension of $W^6_{g+3}(X)$, one expects its image in $\mm_g$, which is the locus of curves with a map of degree $g + 3$ to $\PP^6$ with image contained in a quadric, to have codimension 1. To use the closure of this locus to prove that $\mm_g$ is of general type for $g = 22$ and $23$, there are three significant challenges: (i) computing the expected slope, by pushing forward the virtual class of the degeneracy locus for a natural map of vector bundles whose fiber over $[X,L]$ is $\phi_L$, (ii) showing the closure of this locus is not all of $\mm_g$, and (iii) showing that the push forward of the virtual class is effective.  The next theorem concerns the computation of the expected slope.

We work over an open substack $\widetilde{\mathfrak{M}}_g$ of the moduli stack of stable curves $\MM_g$, whose rational divisor class group is freely generated by the Hodge class $\lambda$ and the boundary classes $\delta_0$ and $\delta_1$.  We then consider a stack of limit linear series $\sigma\colon \Grd \to \widetilde{\mathfrak{M}}_g$, where $r = 6$ and $d = g+3$, and a map of vector bundles over $\Grd$ that restricts to $\phi_L$ over $[X,L]$.  The locus $\mathfrak{U}$ of pairs $[X,L]$ where $\phi_L$ is not injective inherits a closed determinantal substack structure, as a degeneracy locus for this map of vector bundles, and hence it carries a virtual class of expected codimension $g-20$.  Let $[\widetilde{\mathfrak D}_g]^\vir$ be the push forward of this virtual class, which is a divisor class on $\widetilde{\mathfrak{M}}_g$.
\begin{theorem} \label{thm:slopes}
The virtual divisor classes $[\widetilde{\mathfrak D}_{22}]^\vir$ and $[\widetilde{\mathfrak D}_{23}]^\vir$ associated to the loci of curves of genus 22 and 23 with maps to $\mathbf{P}^6$ of degree 25 and 26 with image contained in a quadric are
\[
[\widetilde{\mathfrak{D}}_{22}]^{\mathrm{virt}}= \frac{2}{3}\binom{19}{8}\Bigl(17121 \lambda- 2636\ \delta_0-14511\ \delta_1\Bigr)\in CH^1(\widetilde{\mathfrak{M}}_{22})
\]
and respectively
\[
[\widetilde{\mathfrak{D}}_{23}]^{\mathrm{virt}}= \frac{4}{9}\binom{19}{8}\Bigl(470749 \lambda- 72725\ \delta_0-401951\ \delta_1\Bigr) \in CH^1(\widetilde{\mathfrak{M}}_{23}).
\]
\end{theorem}

For the precise definitions of $\widetilde{\mathfrak{M}}_g$, the stack $\Grd$ of limit linear series and the virtual classes $[\widetilde{\mathfrak{D}}_{g}]^{\mathrm{virt}}$, we refer the reader to \S\ref{virtualis_divizorok}.  Provided that these virtual classes are represented by effective divisors, pushing forward to the coarse space and then taking closures in $\mm_g$ produces divisor classes of slope $\frac{17121}{2636} =6.495\ldots$ and $\frac{470749}{72725}  =6.473\ldots$ in $\mm_{22}$ and $\mm_{23}$, respectively.  Most importantly for the proof of Theorem~\ref{thm:generaltype}, both of these slopes are strictly less than $\frac{13}{2}$.

This construction is inspired by results in \cite{FarkasPopa05}, where it is shown that any divisor on $\mm_g$ with slope less that $6+\frac{12}{g+1}$ must contain the locus $\mathcal{K}_g\subseteq \cM_g$ of curves lying on a $K3$ surface. Finding geometric divisors on $\mm_g$ which contain this locus has proven to be quite difficult, as curves on $K3$ surfaces behave generically with respect to many natural  geometric properties, such as Brill-Noether and Gieseker-Petri conditions.

\bigskip

\subsection{Strong Maximal Rank Conjecture}  The Maximal Rank Conjecture, now a theorem of Larson \cite{Larson17}, has classical origins in the work of M. Noether and Severi \cite{Severi15}. It was brought to modern attention by Harris \cite{Harris82}.  It says that if $X$ is a general curve of genus $g$ and $L\in W^r_d(X)$ is a general linear series, then the multiplication of global sections
\[
\phi_L^k\colon\mbox{Sym}^k H^0(X, L)\rightarrow H^0(X, L^{\otimes k})
\]
is of maximal rank for all $k$.  This determines the Hilbert function of the general embedding of the general curve for each degree and genus. The Maximal Rank Conjecture has been the focus of much activity over the decades, with many important cases, especially for small values of $k$, proved using embedded degenerations in projective space \cite{BallicoEllia89, BallicoFontanari10}, tropical geometry \cite{MRC, MRC2}, or limit linear series \cite{LOTZ21}.  These special cases have applications, including to the surjectivity of Wahl maps \cite{Voisin92} and the construction of counterexamples to the Slope Conjecture \cite{FarkasPopa05}.

The \emph{Strong Maximal Rank Conjecture} is a proposed refinement that takes into account \emph{every} linear series $L\in W^r_d(X)$ on a general curve, rather than just the \emph{general} one \cite[Conjecture 5.4]{AproduFarkas11}. The case $k = 2$ is of particular interest, because the failure of the map $\phi_L := \phi_L^2$ to be of maximal rank is equivalent to the existence of a rank $2$ vector bundle with a prescribed number of sections, and it is known due to work of Lazarsfeld and Mukai that this is a condition that distinguishes curves lying on $K3$ surfaces. It predicts that for a general curve $X$ of genus $g$, and for positive integers $r,d$ such that $0\leq \rho(g,r,d)\leq r-2$, the determinantal variety
\[
\Sigma^r_d(X):=\Bigl\{L\in W^r_d(X)| \phi_L \mbox{ is not of maximal rank} \Bigr\}
\]
has the expected dimension. In particular, the Strong Maximal Rank Conjecture predicts that $\phi_L$ is injective for \emph{every} line bundle $L\in W^r_d(X)$ when the following inequality is satisfied:
\begin{equation}\label{ineqpar}
\mathrm{expdim}\ \Sigma^r_d(X) := g-(r+1)(g-d+r)-(2d+1-g)+\frac{r(r+3)}{2} < 0.
\end{equation}

When $\mathrm{expdim}\ \Sigma_d^r(X)=-1$, the locus of curves for which $\Sigma^r_d(X)$ is not empty has expected codimension 1 in $\cM_g$, and contains the locus of curves on $K3$ surfaces.  So its divisorial part is a natural candidate for an effective divisor of small slope.  In the two cases $g=22$, $d=25$, $r=6$ and $g=23$, $d=26$ and $r=6$, the Strong Maximal Rank Conjecture amounts to the statement that the degeneracy locus $\fU$ discussed above does not dominate $\mm_g$, so its divisorial part is well-defined.  We prove the conjecture in these two cases.

\begin{theorem}
\label{Thm:MainThm}
Set $g=22$ or $23$. For a general curve $X$ of genus $g$, the multiplication map
\[
\phi_L \colon \Sym^2 H^0(X, L) \to H^0(X,L^{\otimes 2})
\]
is injective for \emph{all} line bundles $L\in W^6_{g+3}(X)$.
\end{theorem}

\noindent Theorem \ref{Thm:MainThm} shows that the determinantal locus $\fU$ does not map dominantly onto $\mm_g$.  It follows that $[\widetilde{\mathfrak{D}}_g]^\vir$ is a divisor, rather than just a divisor class.  In other words, the virtual class is a linear combination of the codimension 1 components of the image of $\fU$ in $\widetilde{\mathcal{M}}_g$.  The first proof of Theorem~\ref{Thm:MainThm} appeared in the preprint \cite{JensenPayne18}; that work was never submitted for publication and is incorporated into the present paper.  An alternative approach using limit linear series was put forward in \cite{LiuOssermanTeixidorZhang18}.

The main difficulty in the proof of Theorem \ref{Thm:MainThm}, in comparison with the corresponding cases of the Maximal Rank Conjecture, is that one must control \emph{all} linear series on the general curve $X$, rather than just a sufficiently general one.  For this purpose, the embedded degeneration methods initiated by Hartshorne, Hirschowitz, and much refined by Larson are unsuitable. Instead, we prove Theorem \ref{Thm:MainThm} by taking $X$ to be a curve over a nonarchimedean field whose skeleton is a chain of loops with specified edge lengths and applying tropical methods to study the linear series of degree $g+3$ and rank $6$.  Along the way, we develop new techniques for understanding the tropicalization of a linear series, based on the valuated matroids given by relations among collections of sections (see, e.g., Example~\ref{Ex:Interval}), and an effective criterion for verifying tropical independence (Theorem~\ref{thm:independencecriterion}).  Each of these represents a significant advance beyond the approach to maximal rank statements via tropical methods developed in \cite{MRC, MRC2}.

\subsection{Effectivity of the virtual divisor}

Together, Theorems~\ref{thm:slopes} and~\ref{Thm:MainThm} do not suffice to show that $[\widetilde{\mathfrak{D}}_g]^{\mathrm{virt}}$ is effective on $\widetilde{\mathcal{M}}_g$.
Theorem~\ref{Thm:MainThm} establishes that for $g=22$ or $23$, the image of the degeneracy locus $\fU$ has positive codimension.  Since the push forward of its virtual class is well-defined as a divisor class supported on its image, it follows that
\[
[\widetilde{\mathfrak{D}}_g]^\vir = a_1 \mathcal{Z}_1 + \cdots + a_s \mathcal{Z}_s
\]
is a linear combination of the codimension one components in the image of $\fU$.  A priori, the degeneracy locus $\fU$ could still have components of higher than expected dimension that map with positive dimensional fibers onto some of these codimension 1 components, in which case, some coefficient $a_i$ may be negative.  The following theorem rules out this possibility.

\begin{theorem} \label{thm:genfinite}
Let $\mathcal{Z} \subseteq \mm_g$ be the closure of a codimension one component of $\sigma(\fU)$.  Then the generic fiber of $\fU$ over $\mathcal{Z}$ is finite.
\end{theorem}

The proof of Theorem~\ref{thm:genfinite} has two main parts.  One part carried out in \S\ref{sec:genfinite} uses tropical methods, very similar to those used in the proof of Theorem~\ref{Thm:MainThm}, to show that if the generic fiber of $\fU$ over $\mathcal{Z}$ is infinite, then $\mathcal{Z}$ does not contain certain codimension $2$ strata in $\mm_g$, and to control its pull back under natural maps between moduli spaces.  More precisely, we show that the class $[\mathcal{Z}]$ pulls back to \emph{zero} under the map $\jmath_2\colon\mm_{2,1}\rightarrow \mm_g$ obtained by attaching a general pointed curve of genus $g-2$ to each pointed curve of genus $2$. Similarly, we show that $[\mathcal{Z}]$ pulls back to a nonnegative combination of the Weierstrass divisor and the hyperelliptic divisor in $\mm_{3,1}$ under the analogous attaching map $\jmath_3\colon\mm_{3,1}\rightarrow \mm_g$.  The other part, carried out in \S\ref{sec:vanishingcondition}, is a series of computations in $CH^1(\mm_g)$ showing that there is no nonzero effective divisor with these properties.

\bigskip

Theorem~\ref{thm:generaltype} follows in a straightforward manner from Theorems~\ref{thm:slopes}, \ref{Thm:MainThm}, and \ref{thm:genfinite}.  Indeed, Theorems~\ref{Thm:MainThm} and \ref{thm:genfinite} together imply that $[\widetilde{\mathfrak{D}}_g]$ is a well-defined effective divisor on $\widetilde{\mathcal{M}}_g$.  Taking closure in $\mm_g$, for $g = 22$ or $23$ gives an effective divisor whose slope is the ratio $\frac{a}{b_0}$ computed in Theorem~\ref{thm:slopes}.  It follows that $\mm_g$ is of general type, since this slope is less than $\frac{13}{2}$.

\subsection{Tropical independence}

Our proofs of Theorems~\ref{Thm:MainThm} and \ref{thm:genfinite} rely on tropical independence, as in \cite{MRC, MRC2}.  Roughly speaking, this is a method for proving that a set of sections $\{s_1, \ldots, s_n \}$ of a line bundle is linearly independent by extending the line bundle and its sections over a semistable degeneration such that the specialization map is diagonal, i.e., there are irreducible components $X_1, \ldots, X_n$ in the special fiber such that $s_i$ is nonzero on $X_j$ if and only if $i = j$. We now briefly summarize the foundations of the method, with the hope that this will be helpful for those accustomed to other degeneration techniques in the study of algebraic curves and their linear series.

\subsubsection{Tropicalization of rational functions} Let $X$ be a curve over a complete and algebraically closed valued field $K$, with valuation ring $R$.  Suppose $\mathcal X$ is a semistable model of $X$, that is, a flat and proper scheme over $\Spec R$ with generic fiber $\mathcal X_K \cong X$ and a reduced special fiber with only nodal singularities. Near each node in the special fiber, $\mathcal X$ is \'etale locally isomorphic to $xy = f$ for some $f$ in the maximal ideal of $R$.  The valuation $\val(f)$ is independent of the choice of coordinates and is called the \emph{thickness} of the node.

Let $\Gamma$ be the metric dual graph of this degeneration.  The underlying graph has one vertex for each irreducible component of the special fiber and one edge for each node.  Loops and multiple edges appear when irreducible components have self-intersections and when two components meet at multiple nodes.  The length of an edge is the thickness of the corresponding node.

Each point $v$ in $\Gamma$ is naturally identified with a valuation $\val_v$ on the function field $K(X)$.  Roughly speaking, the valuation at a vertex corresponds to the order of vanishing along the corresponding component of the special fiber.  More precisely, $\val_v(f)$ is equal to $-\val(a)$ for any $a \in K^*$ such that $af$ is regular and nonvanishing at the generic point of the corresponding component $X_v$.  The points in the interior of an edge correspond to monomial valuations in the local coordinates $x$ and $y$ at a node $xy = f$ that agree with the given valuation on the scalar subfield $K \subseteq K(X)$.

\begin{remark}
In the special case where $X$ is defined over a discretely valued subfield $K' \subseteq K$ and $\mathcal X$ is defined over the valuation ring $R' \subseteq K'$, the thickness of each node is an integer, and $\mathcal X$ has an $A_{n-1}$ singularity at a node of thickness $n$.  Recalling that the length of an edge equals the thickness of the corresponding node, the valuations given by the integer points on the corresponding edge of length $n$ are the vanishing orders along the $n-1$ exceptional components of the chain of rational curves in the minimal resolution of this singularity.
\end{remark}

Each valuation is naturally identified with a point in the Berkovich analytification $X^\an$, and the resulting map $\Gamma \to X^\an$ is a homeomorphism onto its image (and, in an appropriate sense, an isometry).  When no confusion seems possible, we identify $\Gamma$ with its image in $X^\an$.  There is a natural retraction $\Trop \colon X^\an \to \Gamma$, which is called \emph{tropicalization}, as is the induced map $\Div(X) \to \Div(\Gamma)$ taking a formal sum of $K$-points to the formal sum of their images in $\Gamma$.

The tropicalization of a nonzero rational function $f \in K(X)^*$ is defined as
\[
\trop(f)\colon \Gamma \to \RR; \ \ v \mapsto \val_v(f).
\]
This function is continuous, and piecewise linear on each edge, with integer slopes.  Moreover, $\trop(f)$ is determined up to an additive constant by $\Trop(\ddiv(f))$.

\subsubsection{Tropicalization of linear relations}

Suppose $\{f_1, ..., f_n\} \subseteq K(X)^*$  is a collection of non-zero rational functions satisfying a linear relation $a_1 f_1 + \cdots + a_n f_n = 0$, with $a_i \in K^*$.  Note that $\val_v (a_i f_i) = \trop(f_i)(v) + \val(a_i)$.  Therefore, at each point $v \in \Gamma$, the minimum in
\[
\theta(v) = \min_i \{ \trop(f_i)(v) + \val(a_i) \}
\]
must be achieved at least twice.  This is a strong restriction on the functions $\trop(f_i)$.  For instance, after subdividing $\Gamma$ so that each of these functions has constant slope on every edge, then on each edge there must be two functions with equal slope.

We say that a collection of  real-valued functions $\{ \psi_1 , \ldots, \psi_n \}$ on $\Gamma$ is \emph{tropically dependent} if there are real numbers $b_1, \ldots, b_n$ such that, for every point $v \in \Gamma$, the minimum in $\min_i \{ \psi_i(v) + b_i \}$ is achieved at least twice.  If there are no such real numbers, then we say that $\{\psi_1, \ldots, \psi_n \}$ is \emph{tropically indepedent}.  Tropical independence of $\{\trop(f_1), \ldots, \trop(f_n)\}$ is a sufficient condition for the linear independence of $\{f_1, \ldots, f_r\} \subseteq K(X)^*$.

Now suppose that $L = \mathcal O_X(D_X)$ is a line bundle in $W^r_d(X)$, as in \S\ref{sec:intro-smallslope} above.  We identify $H^0(X,L)$ and $H^0(X, L^{\otimes 2})$ with $K$-linear subspaces of $K(X)$, in the usual way.  We can then show that the map $\phi_L \colon \Sym^2 H^0(X,L) \to H^0(X, L^{\otimes 2})$ has rank at least $n$ by finding rational functions $f_1, \ldots, f_n$ in the image of $\phi_L$ such that $\{\trop(f_1), \ldots, \trop(f_n)\}$ is tropically independent.
In practice, we do not work directly with rational functions in the image of $\phi_L$.  Instead, we identify piecewise linear functions $\varphi_1, \ldots, \varphi_s$ in the image of $H^0(X,L)$ under tropicalization.  Then all pairwise sums $\varphi_{ij} = \varphi_i + \varphi_j$ are tropicalizations of functions in the image of $\phi_L$.  To prove that $\phi_L$ is injective, we look for a set of such pairwise sums, of size equal to $\dim \Sym^2 H^0(X,L)$, that is tropically independent.

\subsubsection{A characterization of tropical independence}  One of the foundational advances in this paper is a new necessary and sufficient condition for tropical independence. Given a finite set of PL functions $\{ \varphi_1, \ldots, \varphi_n \}$ on $\Gamma$, and real numbers $b_1, \ldots, b_n$, we consider the corresponding \emph{tropical linear combination}
\[
\theta = \min_i \{ \varphi_i + b_i \}.
\]
We say that $\varphi_i + b_i$ \emph{achieves the minimum} at $v$ if $\theta(v) = \varphi_i(v) + b_i$ and that it \emph{achieves the minimum uniquely} if, furthermore, $\theta(v) \neq \varphi_j(v) + b_j$ for $j \neq i$.

\begin{theorem}
\label{thm:independencecriterion}
A finite set $\{ \varphi_1, \ldots, \varphi_n \} \subset \PL(\Gamma)$ is tropically independent if and only if there are real numbers $b_1, \ldots, b_n$ such that each $\varphi_i + b_i$ achieves the minimum uniquely at some $v_i \in \Gamma$.
\end{theorem}

\noindent This is proved in \S\ref{sec:tropicalindependence}, using the Knaster-Kuratowski-Mazurkiewicz lemma, a set-covering variant of the Brouwer fixed-point theorem.

\subsubsection{From tropical independence to diagonal specialization}

We now return to the setup where $X$ is a curve over an algebraically closed valued field, $\mathcal{X}$ is a semistable model with metric dual graph $\Gamma$, $L = \mathcal{O}_X(D_X)$ is a line bundle, and $f_1, \ldots, f_n \in K(X)^*$ are sections of $L$.  Let $\varphi_i = \trop(f_i)$.  Let $\theta = \min_i \{ \varphi_i + b_i \}$ be a tropical linear combination in which each $\varphi_i + b_i$ achieves the minimum uniquely at some point $v_i \in \Gamma$.

We can then choose a toroidal modification $\mathcal{X'} \to \mathcal X$ so that each $v_i$ corresponds to an irreducible component $X_i$ of the special fiber of $\mathcal{X'}$, and each of the functions $\varphi_1, \ldots, \varphi_n$ and $\theta$ is linear on each edge of the metric dual graph $\Gamma'$, which is a subdivision of $\Gamma$.  Furthermore, we can extend $L$ to a line bundle $\mathcal L$ over $\mathcal X'$ so that $f \in H^0(X,L)$ is a regular section of $\mathcal L$ if and only if $\trop(f) \geq \theta$, and nonvanishing on the component $X_v$ corresponding to a vertex $v \in \Gamma'$ if and only if $\trop(f)(v) = \theta(v)$; see Proposition~\ref{prop:ag-interpretation}.  In particular, if $a_i \in K^*$ are scalars such that $\val(a_i) = b_i$, then $a_i f_i$ is a regular section of $\mathcal L$ and is nonvanishing on the irreducible component corresponding to $v_j$ if and only if $i = j$.  This diagonal specialization property ensures that $\{a_1 f_1, \ldots, a_n f_n \}$ is independent in the special fiber, and hence also independent in the general fiber.

\subsection{Chains of loops}

In our proofs of Theorems~\ref{Thm:MainThm} and \ref{thm:genfinite}, we apply the method of tropical independence (and Theorem~\ref{thm:independencecriterion} in particular) to linear series on curves $X$ whose skeletons are specific, carefully chosen graphs $\Gamma$.  As in \cite{MRC, MRC2}, the graphs $\Gamma$ are chains of loops with specified edge lengths.  The divisor classes on such graphs $\Gamma$ that can arise in tropicalizations of linear series of degree $d$ and dimension $r$ have been studied in \cite{tropicalBN, Pflueger17a, JensenRanganathan21, CookPowellJensen22}.  For those unfamiliar with such curves, we explain the geometry of their stable reductions.

Let $X$ be a curve of genus $g$ over $K$ whose skeleton is a chain of $g$ loops $\Gamma$.  Let $X_0$ be its stable reduction.   Then $[X_0] \in \overline{\mathcal{M}}_g$ is a $0$-stratum.  One can label its $2g-2$ rational components as $$X_1, Y_2, X_2, Y_3, \ldots, X_{g-1}, Y_g,$$ such that
\begin{enumerate}
\item The components $X_1$ and $Y_g$ each have one node, and the rest are smooth;
\item For $1 \leq i \leq g-1$, $X_i$ meets $Y_{i+1}$ at a single node;
\item For $2 \leq i \leq g-1$, $X_i$ meets $Y_i$ at two nodes.
\end{enumerate}
This curve $X_0$ is in the closure of the locus of hyperelliptic curves in $\mathcal{M}_g$.  Our arguments therefore cannot use the geometry of $X_0$ in any meaningful way.  Instead, we use the edge lengths of $\Gamma$.   Thinking of $X$ as the general fiber in a family over a germ of a curve, with central fiber $X_0$, the edge lengths specify the contact orders of this germ with the $3g-3$ branches of the boundary divisor that meet at $[X_0]$. Our proof of Theorem~\ref{Thm:MainThm} shows that the general member of such a family with certain specified contact orders satisfies the conclusion of the Strong Maximal Rank Conjecture.

\medskip

For those familiar with this method, we briefly describe the novel aspects of the constructions presented here.  Recall that the space of all divisor classes of degree $d$ on $\Gamma$ is a real torus of dimension $g$, and the subspace $W^r_d(\Gamma)$ parametrizing those that can come from linear series of degree $d$ and rank $r$ on an algebraic curve form a finite union of translates of subtori, called \emph{combinatorial types}. These combinatorial types are naturally indexed by certain tableaux.  When the edge lengths of $\Gamma$ are sufficiently general, an open dense subset of $W^r_d(\Gamma)$ consists of divisor classes that are \emph{vertex avoiding}, in the sense of  \cite{LiftingDivisors}.

Suppose $X$ is a curve whose tropicalization is $\Gamma$, $D_X$ is a divisor of degree $d$ in a linear series of dimension $r$, and assume the class of $D = \Trop(D_X)$ is vertex avoiding.  Then we have a canonical collection of PL functions $\{ \varphi_0, \ldots, \varphi_6 \}$ on  $\Gamma$ that is the tropicalization of a basis for $H^0(X, \mathcal{O}(D_X))$.  In \cite{MRC}, we fix one particular $D$, assume that there is a tropical linear combination $\theta = \min_{ij}  \{ \varphi_i + \varphi_j + b_{ij} \}$ such that the minimum is achieved at least twice at every point $v \in \Gamma$, compute the degree of the associated effective divisor $D + \ddiv(\theta)$, and derive a contradiction.  There are several difficulties in extending this approach to \emph{all} divisors of degree $d$ and rank $r$.  One is sheer combinatorial complexity.  The arguments in \cite{MRC} are specific to the combinatorial type of $D$.  When $\Gamma$ has genus 23, the number of combinatorial types in $W^6_{26} (\Gamma)$ is
\[
\frac{23!}{9!\cdot 8!\cdot 7!} = 350,574,510.
\]
This difficulty is overcome primarily through the new constructive method for proving tropical independence, given by Theorem~\ref{thm:independencecriterion}.

In the non vertex avoiding cases, we face the additional problem of understanding which functions in $R(D)$ are tropicalizations of functions in $H^0(X, \mathcal{O}(D_X))$, and finding a suitable substitute for the distinguished functions $\varphi_i$.  For an arbitrary divisor $D$, this seems to be an intractable problem.  However, when $\rho$ is small, we find that in most cases it is enough to understand the tropicalizations of certain pencils in $H^0(X, \mathcal{O}(D_X))$.  These, in turn, behave similarly to tropicalizations of pencils on $\PP^1$, which we analyze in Example~\ref{Ex:Interval}.  The possibilities for the tropicalizations of $H^0(X, \mathcal{O}(D_X))$ are then divided into cases, according to the combinatorial properties of these pencils.  We then construct a tropical independence case-by-case, in \S\S\ref{Sec:Construction}-\ref{Sec:Generic}, using a variant of the algorithm that works for vertex avoiding divisors.  Only one subcase, treated in \S\ref{sec:4d}, does not reduce to an analysis of pencils; the arguments in this subcase are nevertheless of a similar flavor, with a few more combinatorial possibilities to consider.

\subsection{Further constructions of virtual divisors of small slope}

The construction of the virtual divisor class $[\widetilde{\mathfrak{D}}_{22}]$ and the computation of its slope can be extended to an infinite family of cases, as follows.  Fix an integer $s\geq 2$, and set
\[
g := 2s^2+s+1, \ d := 2s^2+2s+1 \ \mbox{ and } \ r:=2s.
\]
Then a general curve $[X] \in \cM_{g}$ carries a $1$-dimensional family of line bundles
$L\in W^{r}_{d}(X)$. For each such $L$ we consider the multiplication map
\[
\phi_L \colon \mathrm{Sym}^2 H^0(X, L) \rightarrow H^0(X,L^{\otimes 2}).
\]
Observe that $h^0(X,L^{\otimes 2})-\dim\mathrm{Sym}^2 H^0(X,L)=1$.  Therefore, the locus where the map $\phi_L$ is non-injective has expected codimension \emph{two} in the parameter space of pairs $[X,L]$, where $L\in W^{r}_{d}(X)$.

Just as in the special case $s = 3$ and $g = 22$ discussed above, we work over a partial compactification $\widetilde{\mathfrak{M}}_g$ of an open substack of $\MM_g$ whose divisor class group is generated by $\lambda$, $\delta_0$, and $\delta_1$. We consider the stack
$\sigma\colon \Grd\rightarrow \widetilde{\mathfrak{M}}_g$ of limit linear series of type $\mathfrak g^r_d$.  On the resulting stack $\Grd$ we then construct a map of vector bundles that restricts to $\phi_L$ on the fiber over $[X,L]$, and we compute the push forward of the virtual class of the degeneracy locus $\mathfrak{U}$, where this map is not injective.

\begin{theorem}\label{rho1virtual}
Fix $s\geq 2$ and set $g:=2s^2+s+1$.  Let $\mathfrak{U} \subseteq \Grd$ be the degeneracy locus described above, and let $[\widetilde{\mathfrak{D}}_g]= \sigma_* [\mathfrak U]^\vir$.  Write  $[\widetilde{\mathfrak{D}}_g] = a \lambda -b_0 \delta_0 - b_1 \delta_1$.  Then
\begin{align*}
\frac{a}{b_0} &= \frac{3(48s^8-56s^7+92s^6-90s^5+86s^4+324s^3+317s^2+182s+48)}{24s^8-28s^7+22s^6-5s^5+43s^4+112s^3+100s^2+50s+12} \\
&= 6+\frac{12}{g+1}-\frac{3(120s^6-140s^5-162s^4+67s^3+153s^2+94s+24)}{(2s^2+s+1)(24s^8-28s^7+22s^6-5s^5+43s^4+112s^3+100s^2+50s+12)}.
\end{align*}
In particular, $\frac{a}{b_0} < 6+\frac{12}{g+1}$ for all $s\geq 3$.
\end{theorem}

\noindent Setting $s=3$, we obtain the virtual divisor $\widetilde{\mathfrak{D}}_{22}$ in $\widetilde{\mathcal{M}}_{22}$ that appears in Theorem \ref{thm:slopes}.  When $s=2$, $\sigma_* [\mathfrak U]^\vir$ is an effective divisor whose closure in $\mm_{11}$ has slope 7.  This interesting divisor has also appeared in \cite{FarkasOrtega12, BakkerFarkas18}, and can be seen as the closure of the locus of curves $[X]\in \cM_{11}$ possessing a semistable rank $2$ vector bundle $E$ whose Clifford index $\mbox{Cliff}(E)$ is strictly less than the Clifford index $\mbox{Cliff}(X)$, which
is, as usual, computed with respect to special line bundles on $X$.

For $g=22$, results from \cite{FarkasPopa05} show that the closure in $\mm_g$ of the image of any effective representative of $[\mathfrak{U}]^\vir$ has slope equal to $\frac{a}{b_0}$. It is possible to extend this statement for $s>3$, following closely the methods of \cite{Farkas06}.  However, in the interest of not increasing further the length of this paper we choose not to carry this out here.

\bigskip

We have written this paper for an audience that includes experts on moduli spaces as well as experts on tropical and nonarchimedean geometry.  While the class computations (\S\ref{Sec:23}-\ref{sect:rho1}) and the main tropical arguments (\S\ref{Sec:Construction}-\ref{sec:genfinite}) are necessarily technical, the presentation also includes detailed examples (such as Examples~\ref{Ex:Interval} and \ref{ex:randomtableau}) and complete arguments in special cases, such as the proof of injectivity of $\phi_L$ in the vertex avoiding case, in \S\ref{Sec:VertexAvoiding}. These are not logically necessary, but should clarify and motivate the essential steps in the proofs of the main theorems.

\bigskip

\noindent{\textbf{Acknowledgments.}}   We thank M. Fedorchuk for graciously pointing out that Theorems~\ref{thm:slopes} and \ref{Thm:MainThm} do not suffice to prove that $\mm_{22}$ and $\mm_{23}$ are of general type, highlighting the need for a stronger statement than the Strong Maximal Rank Conjecture, such as Theorem~\ref{thm:genfinite}.  We also thank O. Amini and D. Maclagan for helpful conversations related to finite generation of tropicalizations of linear series.  And we are especially grateful to the two referees for many insightful comments that led to significant improvements in the presentation.

Farkas was supported by the DFG grant \emph{Syzygien und Moduli} and by the ERC Advanced Grant \emph{Syzygy}.  Jensen is supported by NSF DMS--1601896.  Jensen would also like to thank Yale University for hosting him during the Summer and Fall of 2017, during which time much of the work on this paper was completed.  Payne was partially supported by NSF DMS--2001502,  NSF DMS--2053261 and a Simons Fellowship. Portions of this research were carried out while visiting MSRI.

\section{Preliminaries}

In this section we lay the groundwork for the main sections of the paper, establishing notation and recalling basic facts that will be used throughout, and proving two foundational results that may be more broadly useful.  In particular, we recall the notion of slopes of divisors on moduli spaces of stable curves, review the intersection numbers for curves and divisors on these moduli spaces, and prove a vanishing criterion for effective divisors (Proposition~\ref{prop:vanishing}).  We then discuss skeletons of curves over nonarchimedean fields and tropicalization of rational functions, before proving an effective criterion for tropical independence (Theorem~\ref{thm:independencecriterion}).

\subsection{Slopes of divisors}  \label{sec:slopesofdivisors}
We denote by $\MM_g$ the moduli stack of stable curves of genus $g\geq 2$ and by $\mm_g$ the associated coarse moduli space. All of the cycles and Chow groups that we consider are with rational coefficients.  The push forward $CH^*(\MM_g) \to CH^*(\mm_g)$ is an isomorphism, and we identify each cycle and cycle class on $\MM_g$ with its push forward to $\mm_g$.  In particular, if $\mathfrak{Y} \subseteq \MM_g$ is an irreducible closed substack with coarse moduli space $\mathcal{Y}$ and generic stabilizer $G$, then $[\mathfrak Y] = \frac{1}{|G|} [\mathcal{Y}]$.

All intersection theory calculations in this paper are carried out on the stack $\MM_g$, whereas the results about Kodaira dimension concern the coarse space $\mm_g$. We follow the standard convention that $CH^i$ denotes the Chow group of cycles of codimension $i$, modulo rational equivalence.

Recall that for $g\geq 3$ the group $CH^1(\mm_g)\cong CH^1(\MM_g)$  is freely generated by the Hodge class $\lambda$ and the classes of boundary divisors $\delta_i=[\Delta_i]$, for $i=0,  \ldots, \lfloor \frac{g}{2} \rfloor$. For $g=2$ one has the supplementary relation $\lambda=\frac{1}{10}\delta_0+\frac{1}{5}\delta_1$. The canonical class of $\mm_g$, computed in \cite{HarrisMumford82}, is
\[
K_{\mm_g} = 13 \lambda - 2 \delta_0 - 3 \delta_1 - 2 \delta_2 - \cdots - 2 \delta_{\lfloor \frac{g}{2} \rfloor}.
\]
The singularities of $\mm_g$ are mild enough that all sections of $nK_{\mm_g}$ extend to pluricanonical forms on a resolution of singularities \cite{HarrisMumford82}.  Therefore, $\mm_g$ is of general type if and only if there is an effective divisor class $[D] \sim a\lambda - \sum_{i=0}^{\lfloor \frac{g}{2} \rfloor} b_i\delta_i$ on $\mm_g$ that satisfies
\begin{equation} \label{eq:slope}
\frac{a}{b_i} < \frac{13}{2} \mbox{ for } i \neq 1 \mbox{ and } \frac{a}{b_1} < \frac{13}{3}.
\end{equation}

\noindent The \emph{slope} of an effective divisor $D$ on $\mm_g$ with class as above is $s(D) := \max_i \big\{ \frac{a}{b_i} \big\}$ \cite{HarrisMorrison90}.  For the purpose of studying the Kodaira dimension of $\mm_g$, one is most concerned with $\frac{a}{b_0}$.  Indeed, if $\frac{a}{b_0} < \frac{13}{2}$, then $\frac{a}{b_1} < \frac{13}{3}$, and if, moreover, $g \leq 23$, then $s(D) = \frac{a}{b_0}$ \cite[Theorem~1.1(c) and Corollary~1.2]{FarkasPopa05}. In particular, for $g \leq 23$, $\mm_g$ is of general type if and only if there is an effective divisor $D$ with $\frac{a}{b_0} < \frac{13}{2}$.

In their study of the Kodaira dimension of $\mm_g$ for $g \geq 24$, Harris, Mumford, and Eisenbud considered Brill-Noether divisors, defined as follows. For integers $r\geq 1$ and $d\leq g+r-2$ such that
\begin{equation} \label{eq:BN-1}
\rho(g,r,d):=g-(r+1)(g-d+r)=-1,
\end{equation}
one defines $\cM_{g,d}^r$ to be the (divisorial part of the) locus of curves $[X]\in \cM_g$ with a linear series $L \in W^r_d(X)$.
The fact that this locus is not all of $\cM_g$ is the essential content of the Brill-Noether Theorem \cite{GriffithsHarris80}.  Since the slope of the closure $\mm_{g,d}^r$ of $\cM_{g,d}^r$ inside $\mm_g$ is $6+\frac{12}{g+1}$ \cite{HarrisMumford82, EisenbudHarris87b}, it follows that $\mm_g$ is of general type when $g \geq 24$ and $g+1$ is composite.  (If $g + 1$ is prime, then the equation \eqref{eq:BN-1} has no solutions, and there is no Brill-Noether divisor on $\mm_g$.)

When $g$ is even and at least 28, one similarly obtains a virtual divisor satisfying \eqref{eq:slope} supported on the closure of the locus of  curves $[X]\in \cM_g$ with a line bundle $L\in W^r_d(X)$ where $\rho(g,r,d)=0$, such that the Petri map is not injective \cite{EisenbudHarris87b}.  The fact that this locus is not all of $\mm_g$, from which it follows that this virtual class is effective, is the essential content of the Gieseker-Petri Theorem \cite{Gieseker82}.  Note that both the Brill-Noether Theorem and the Gieseker-Petri Theorem  have more recent proofs by tropical arguments on chains of loops \cite{tropicalBN, tropicalGP}.  For $g<24$ the Brill-Noether and Gieseker-Petri divisors do not satisfy inequality \eqref{eq:slope}.

As explained in the introduction, for $g = 22$ and $23$, we construct a different virtual divisor with smaller slope.  Specifically, we consider the push forward of the virtual class of the locus of curves admitting a map of minimal degree to $\PP^6$  with image contained in a quadric.  Theorem~\ref{thm:slopes} says that these virtual divisors satisfy \eqref{eq:slope}, and Theorems~\ref{Thm:MainThm} and~\ref{thm:genfinite} combine to show that they are effective.

\subsection{Test curves and intersection numbers}

We introduce a few standard test curves in $\mm_g$ that will be used several times in the paper.
Choose a general pointed curve $[X, q]$ of genus $g-1$. Then construct the families of stable curves of genus $g$

\begin{equation}\label{f0}
F_0:=\Bigl\{[X_{yq}]:=[X/y\sim q] \mid y\in X \Bigr\}\subseteq \Delta_0 \subset \mm_{g} \ \ \mbox{ and }
\end{equation}
\begin{equation}\label{fell}
F_{\mathrm{ell}}:=\Bigl\{[X\cup_q E_t] \mid t\in \PP^1\Bigr\}\subseteq \Delta_1 \subset \mm_{g},
\end{equation}
where $\{[E_t, q]\}_{t\in \PP^1}$ denotes a pencil of plane cubics and $q$ is a fixed point of the pencil.
The intersection of these test curves with the generators of $CH^1(\mm_{g})$ is well-known, see  \cite{HarrisMumford82}:
\begin{align*}
F_0\cdot \lambda &= 0, &  F_0\cdot \delta_0 &= 2-2g, &  F_0 \cdot \delta_1 &= 1 \
\mbox{ and } & F_0 \cdot \delta_j &= 0 \ \mbox{ for } \ j=2, \ldots, \Big \lfloor \frac{g}{2}\Big \rfloor, \\
F_{\mathrm{ell}}\cdot \lambda &= 1, &  F_{\mathrm{ell}} \cdot \delta_0 &=12, & F_{\mathrm{ell}}\cdot \delta_1 &=-1 \ \mbox{ and } & F_{\mathrm{ell}} \cdot \delta_j &= 0 \ \mbox{ for } j=2, \ldots, \Big \lfloor \frac{g}{2} \Big \rfloor.
\end{align*}
For a fixed pointed curve $[C, p]\in \cM_{g-i,1}$ and a fixed curve $[D]\in \cM_i$, we consider the family
\begin{equation}\label{fi}
F_i:=\Bigl\{[C\cup_{p\sim y} D]\mid  y\in D\Bigr\}\subseteq \Delta_i\subset \mm_g.
\end{equation}
Then using again \cite{HarrisMumford82} we find
\[
F_i\cdot \lambda=0, \ F_i\cdot \delta_i=2-2i, \ \mbox{ and } F_i\cdot \delta_j=0\ \mbox{ for } j\neq i.
\]

\subsection{A vanishing condition for effective divisors}  \label{sec:vanishingcondition}

In order to prove Theorem~\ref{thm:genfinite}, we must show that there is no nonzero effective divisor in $\mm_g$ over which the fibers of the degeneracy locus $\fU$ inside $\Grd$ are generically infinite.  We will do so by applying certain sufficient vanishing conditions for effective divisors on $\mm_g$, which we explain next.
\begin{definition} \label{def:deltaij}
For $1\leq i, j<g$, let $\Delta_{i,j} \subset \mm_g$ be the codimension 2 boundary stratum parametrizing curves with two separating nodes, whose two tail components have genus $i$ and $j$. Such a curve has a third component, of genus $g - i - j$, meeting each tail component at a node.
\end{definition}

For $1\leq k< g$, let $\jmath_k \colon \mm_{k,1} \to \mm_g$ be the map obtained by attaching a fixed, general pointed curve $[X,p]\in \cM_{g-k,1}$ to an arbitrary pointed curve of genus $k$.  We let $\psi\in CH^1(\mm_{k,1})$ be the cotangent class, and we let $\delta_i=\delta_{i:1}\in CH^1(\mm_{k,1})$ denote the class of the closure of the locus of the union of two smooth curves of genera $i$ and $k-i$, with the marked point lying on the genus $i$ component, for $i=1, \ldots, k-1$.  We then have the following formulas:
{\begin{align}\label{pullback}
\jmath_k^*(\lambda) &= \lambda, & \jmath_k^*(\delta_0) &=\delta_0, &\jmath_k^*(\delta_k) &=-\psi+\delta_{2k-g},  & \jmath_k^*(\delta_{i}) &=\delta_{i+k-g}+\delta_{k-i} \mbox{ for } i\neq k.
\end{align}}
Here we make the convention $\delta_i:=0$, for $i<0$ or $i\geq g$.

The \emph{Weierstrass divisor} $\overline{\mathcal{W}}_k$ in $\mm_{k,1}$ is the closure of the locus of smooth pointed curves $[X, p]$, where $p\in X$ is a Weierstrass point. The \emph{hyperelliptic divisor} $\overline{\mathcal{H}}_3$ in  $\mm_{3,1}$ is the locus where the underlying curve is hyperelliptic.

The following result provides  sufficient intersection-theoretic conditions for an effective divisor on $\mm_g$ to be zero.
It relies on the existing detailed knowledge of the Picard group of $\mm_{g,1}$.

\begin{proposition} \label{prop:vanishing}
Let $g \geq 6$ and let $D$ be an effective divisor on $\mm_g$ with the following properties:
\begin{enumerate}
\item $D$ is the closure of a divisor in $\cM_g$;
\item $\jmath_2^*(D) = 0$;
\item $D$ does not contain any codimension $2$ stratum $\Delta_{2,j}$.
\item if $g$ is even then $\jmath_3^*(D)$ is a nonnegative combination of the classes $[\overline{\mathcal{W}}_3]$ and $[\overline{\mathcal{H}}_3]$ on $\mm_{3,1}$.
\end{enumerate}
Then $D = 0$.
\end{proposition}

\begin{proof}
Write the class of $D$ as
\[
[D] = a\lambda - b_0 \delta_0 - \cdots - b_{\lfloor  \frac{g}{2} \rfloor} \delta_{\lfloor \frac{g}{2} \rfloor}\in CH^1(\mm_g).
\]
Since $\mm_g$ is projective, to show that $D$ is zero, it suffices to show that $[D] = 0$.

First, note that $a \geq 0$, since $\lambda$ is nef and the complete curves disjoint from the boundary are dense in $\mathcal{M}_g$.  Next, we claim that $b_i \geq 0$ for all $i$.  For $i=2, \ldots, g-1$, the curve $F_i$ moves in a family that covers the boundary divisor $\Delta_i$.  Since $\Delta_i\nsubseteq \mathrm{supp}(D)$, it follows that $F_i\cdot D\geq 0$, hence $b_i\geq 0$.  For $i=0$, we similarly use the curve $F_0$ which moves in a family that covers $\Delta_0$, to deduce that
$(2g-2)b_0-b_1=F_0\cdot D\geq 0$, so $b_0\geq 0$.

By \cite[Theorem~2.1]{EisenbudHarris87b}, the condition $\jmath_2^*(D) = 0$ implies that $a = 10 b_0 = 5b_1$ and $b_2 = 0$.  Indeed, recall the relation $\lambda=\frac{1}{10}\delta_0+\frac{1}{5}\delta_1$ in genus $2$.  By (\ref{pullback}) we have $\jmath_2^*([D])=a\lambda-b_0\delta_0-b_1\delta_1+b_2\psi$, and the conclusion follows from the fact that the classes $\psi, \delta_0, \delta_1\in CH^1(\mm_{2,1})$ are independent.

Next we consider the test curve $F_{2, g-2-j}\subseteq \Delta_{2,g-j-2}$ obtained by gluing a fixed pointed curve of genus 2 to a moving point on a curve of genus $j$, which is itself glued at a fixed point to a curve of genus $g-j-2$.  If $j \neq 2, g-4$, we have
\begin{align*}
F_{2,g-2-j} \cdot \lambda &= 0, & F_{2, g-2-j} \cdot \delta_2 &= 1-2j, &  F_{2, g-2-j} \cdot \delta_j &= 1, \\
F_{2, g-2-j} \cdot \delta_{j+2} &= -1, & F_{2,g-2-j} \cdot \delta_i &= 0, \mbox{ for } i \neq 2,j,j+2.
\end{align*}
Similarly, for the test curves $F_{2,g-4}$ and $F_{2,2}$ we have
\begin{align*}
F_{2, g-4} \cdot \lambda &= 0, & F_{2, g-4} \cdot \delta_i &= 0, \mbox{ for } i \neq 2,4, &  F_{2,g-4} \cdot \delta_2 &= -2, & F_{2, g-4} \cdot \delta_4 &= -1. \\
F_{2, 2} \cdot \lambda &= 0, & F_{2, 2} \cdot \delta_i &= 0 \mbox{ for }  i \neq 2,4, & F_{2,2} \cdot \delta_2 &= -4, & F_{2,2} \cdot \delta_4 &= 1.
\end{align*}
Since $F_{2,g-2-j}$ covers the stratum $\Delta_{2, g-2-j}$, which is not contained in $D$, we have $F_{2, j} \cdot D \geq 0$.  Since $b_2 = 0$, it follows that $b_{j} \leq b_{j+2} \mbox{ for } 1 \leq j \leq g-3$,
where we adopt the usual convention that $b_k := b_{g-k}$ for $\lfloor \frac{g}{2} \rfloor < k < g$.  Replacing $j$ by $g - j - 2$, we see that
\begin{equation}\label{equalby2}
b_j = b_{j+2}, \mbox{  for } 1\leq j \leq g-3.
\end{equation}

When $g$ is odd,  combining (\ref{equalby2}) with the fact that $b_2 = 0$ shows that $b_j = 0$ for all $j \geq 1$.  Since $a = 10 b_0 = 5b_1$, it follows that $D = 0$.

\bigskip

It remains to consider the case when $g$ is even. So far, using (\ref{equalby2}) and the relation $a=10b_0=5b_1$, we have shown that
\[
[D]=c\bigl(10\lambda-\delta_0-2\delta_1-2\delta_3-\cdots-2\delta_{\frac{g}{2}}\bigr)
\]
for some $c\in \mathbb Q_{\geq 0}$. We aim to prove that $c=0$.

By assumption, we have a relation $[\jmath_3^*(D)]=\alpha[\overline{\mathcal{W}}_3]+\beta[\overline{\mathcal{H}}_3]$, for certain nonnegative rational constants $\alpha$ and $\beta$.  By \cite{Cukierman89} or \cite{EisenbudHarris87b}, the class of the Weierstrass divisor is
\[
[\overline{\mathcal{W}}_3]=-\lambda + \psi - 3 \delta_1 - 6 \delta_2 .
\]
By for instance \cite[Section~3.H]{HarrisMorrison98}, the class of the hyperelliptic divisor is
\[
[\overline{\mathcal{H}}_3]=9\lambda -\delta_0 - 3\delta_1 - 3\delta_2.
\]

Applying once more the formula (\ref{pullback}), we find that
\[
[\jmath_3^*(D)]=c\bigl(10\lambda-\delta_0-2\delta_2+2\psi)=\alpha\bigl(-\lambda-3\delta_1-6\delta_2 + \psi \bigr)+\beta\bigl(9\lambda-\delta_0-3\delta_1
-3\delta_2\bigr).
\]
Since the classes $\lambda, \delta_0, \delta_1, \delta_2, \psi$ freely generate $CH^1(\mm_{3,1})$, we immediately obtain from this relation that $c=\alpha=\beta=0$.  That is, $D=0$.
\end{proof}

\subsection{Tropical and nonarchimedean geometry of curves} \label{sec:tropicalprelim}

The techniques that we use to prove Theorems~\ref{Thm:MainThm} and \ref{thm:genfinite}, and thereby establish the transversality statements needed to produce effective divisors of small slope on $\mm_{22}$ and $\mm_{23}$, are based on tropical and nonarchimedean geometry.  Let $X$ be a curve of positive genus over an algebraically closed nonarchimedean field $K$ with valuation ring $R$ and residue field $\kappa$, of characteristic zero.  For simplicity, we assume that $K$ is spherically complete with value group $\RR$.  An example of such a field is $\mathbb{C}(\!(t^{\mathbb R})\!)$, the field of power series with real exponents and well-ordered support.  See \cite{Poonen93} for an exposition of nonarchimedean fields with such completeness properties.  Note that any two uncountable algebraically closed fields of the same cardinality and characteristic are isomorphic \cite[Proposition~2.2.5]{Marker02}.  In particular, we may choose $K$ to be isomorphic to $\mathbb{C}$, as an abstract field.  The additional nonarchimedean structure on $K$ gives us access to techniques from tropical geometry and Berkovich theory, just as the Euclidean norm on $\mathbb{C}$ gives access to techniques from Riemann surfaces and complex analytic geometry.

Since $K$ is spherically complete with value group $\mathbb{R}$, every point in the nonarchimedean analytification $X^\mathrm{an}$ has type 1 or 2.  Here, a type 1 point is simply a $K$-rational point, and a type 2 point $v$ corresponds to a valuation $\val_v$ on the function field $K(X)$ whose associated residue field is a transcendence degree 1 extension of the residue field $\kappa$ of $K$.  See, e.g.,  \cite[\S3.5]{BPR16}.

\subsubsection{Skeletons} The \emph{minimal skeleton} of $X^\an$ is the set of points with no neighborhood isomorphic to an analytic ball, and carries canonically the structure of a finite metric graph.  More generally, a \emph{skeleton}  for $X^\an$ is the underlying set of a finite connected subgraph of $X^\an$ that contains this minimal skeleton.  Any skeleton $\Gamma$ is contained in the set of type 2 points, and any decomposition of a skeleton $\Gamma$ into vertices and edges determines a semistable model of $X$ over $R$.  The vertices correspond to the irreducible components of the special fiber, and the irreducible component $X_v$ corresponding to $v$ has function field $\kappa(X_v)$, the transcendence degree 1 extension of $\kappa$ given by the completion of $K(X)$ with respect to $\val_v$.  The edges correspond to the nodes of the special fiber, with the length of each edge given by the thickness of the corresponding node.

\subsubsection{Tropicalizations and reductions of rational functions} \label{sec:slopesandreductions} Let $\Gamma$ be a skeleton for $X^\an$.   Since $\Gamma$ is contained in the set of type 2 points, for each nonzero rational function $f \in K(X)^*$ we get a real-valued function
\[
\trop(f):\Gamma \rightarrow \RR, \ \ v \mapsto \val_v(f).
\]
This function is piecewise linear with integer slopes, and its slope along an edge incident to $v$ is related to the reduction of $f$ at $v$.  This relation is known as the slope formula, a nonarchimedean analogue of the Poincar\'e-Lelong formula, which we now describe.

Given a nonzero rational function $f \in K(X)$ and a type 2 point $v \in X^\an$, choose $c \in K^*$ whose valuation is equal to $\val_v(f)$.  Then $f/c$ has valuation zero, and the \emph{reduction} of $f$ at $v$, denoted $f_v$, is defined to be its image in $\kappa (X_v)^*/\kappa^*$.  This does not depend on the choice of $c$, so $f_v$ is well-defined.  Divisors of rational functions are invariant under multiplication by nonzero scalars, and we denote the divisor on $X_v$ of any representative of $f_v$ in $\kappa(X_v)^*$ by $\ddiv (f_v)$.  Each germ of an edge of $\Gamma$ incident to $v$ corresponds to a point of $X_v$ (a node in the special fiber of a semistable model with skeleton $\Gamma$, in which $X_v$ appears as a component).  The slope formula then says that the outgoing slope of $\trop(f)$ along this germ of an edge is equal to the order of vanishing of $f_v$ at that point \cite[Theorem~5.15(3)]{BPR13}.

\subsubsection{Complete linear series on graphs} Let $\PL(\Gamma)$ be the set of continuous piecewise linear functions on $\Gamma$ with integer slopes.  Throughout, we will use both the additive group structure on $\PL(\Gamma)$, and the tropical module structure given by pointwise minimum and addition of real scalars.

A divisor $D$ on $\Gamma$ is an element of the free abelian group generated by the points of $\Gamma$, i.e. a finite formal linear combination of points of $\Gamma$ with integer coefficients.  The \emph{order} of $\varphi \in \PL (\Gamma)$ at a point $v \in \Gamma$, denoted $\ord_v (\varphi)$, is the sum of the incoming slopes of $\varphi$ at $v$.  The \emph{principal divisor} associated to $\varphi$ is then $\ddiv (\varphi) := \sum_{v \in \Gamma} \ord_v (\varphi) v$.  The \emph{complete linear series} of such a divisor $D$ is
\[
R(D) := \bigl\{ \varphi \in \PL (\Gamma): \ddiv (\varphi) + D \geq 0 \bigr\}.
\]
Note that $R(D) \subseteq \PL(\Gamma)$ is a tropical submodule, i.e.  it is closed under scalar addition and pointwise minimum.

By the Poincar\'{e}-Lelong formula, if $D_X$ is any divisor on $X$ tropicalizing to $D$ and $f$ is a section of $\mathcal{O}(D_X)$, then $\trop(f) \in R(D)$.  We refer the reader to \cite{BakerNorine07, Baker08} for further background on the divisor theory of graphs and metric graphs, and specialization from curves to graphs.

\subsection{Tropical independence}  \label{sec:tropicalindependence}

We now recall the notion of tropical independence, as defined in \cite{tropicalGP}, and prove Theorem~\ref{thm:independencecriterion}.
Let $\{ \psi_i  : \ i \in I\}$ be a finite collection of piecewise linear functions.  A \emph{tropical linear combination} is an expression
\[
\theta = \min \{ \psi_i + c_i \},
\]
for some choice of real coefficients $\{c_i\}$.  Note that different choices of coefficients may yield the same pointwise minimum, but we consider the coefficients $\{ c_i \}$ to be part of the data in a tropical linear combination, so the tropical linear combinations of  $\{ \psi_i \}$  are naturally identified with $\RR^I$.

Given a tropical linear combination $\theta = \min \{ \psi_i + c_i \}$, we say that $\psi_i + c_i$ \emph{achieves the minimum} at $v \in \Gamma$ if $\theta(v) = \psi_i(v) + c_i$, and \emph{achieves the minimum uniquely} if, moreover, $ \theta(v)\neq \psi_j(v) + c_j $ for $j \neq i$.  We say that \emph{the minimum is achieved at least twice} at $v \in \Gamma$ if there are at least two distinct indices $i \neq j$ such that $\theta(v)$ is equal to both $\psi_i(v) + c_i$ and $\psi_j(v) + c_j$.

A \emph{tropical dependence} is a tropical linear combination $\theta = \min \{ \psi_i + c_i \}$ such that the minimum of the functions $\psi_i + c_i$ is achieved at least twice at every point of $\Gamma$.   Equivalently, $\theta = \min \{ \psi_i + c_i \}$ is a tropical dependence if  $\theta = \min_{j \neq i} \{ \psi_j + c_j \}$, for all $i$.  If no such tropical linear combination exists, then $\{\psi_0 , \ldots \psi_r\}$ is  \emph{tropically independent}.

Most importantly for applications to Brill-Noether theory, if a set of nonzero functions $\{f_0, \ldots , f_r\}$ is linearly dependent over $K$, then the set of tropicalizations $\{ \trop(f_0), \ldots , \trop (f_r) \}$ is tropically dependent \cite[Lemma~3.2]{tropicalGP}.  Therefore, tropical independence of the tropicalizations is a sufficient condition for linear independence of rational functions.

\medskip

Our arguments in this paper use the following new characterization of tropical independence.

\begin{definition}
A tropical linear combination $\theta = \min \{ \psi_i + c_i \}$ is a \emph{certificate of independence} if each $\psi_i + c_i$ achieves the minimum uniquely at some point $v \in \Gamma$.
\end{definition}

\noindent Equivalently, $\theta = \min \{ \psi_i + c_i \}$ is a certificate of independence if $\theta \neq \min_{j \neq i} \{ \psi_j + c_j \}$, for all $i \in I$.

\begin{remark}
In linear algebra, a \emph{dependence} is a linear combination that shows a collection of vectors is linearly dependent.  Similarly, a tropical dependence is a tropical linear combination that shows a collection of PL functions is tropically dependent.  
By Theorem~\ref{thm:independencecriterion}, the existence of a certificate of independence for a collection of PL functions shows that these functions are tropically independent.  There is no analogous criterion for linear independence in linear algebra.
\end{remark}

Although not logically necessary for the proofs of our theorems, we include the following interpretation of tropical independence in the language of algebraic geometry.

\begin{proposition}  \label{prop:ag-interpretation}
Let $X$ be a curve over $K$ with skeleton $\Gamma$.  Let $f_0, \ldots, f_r$ be sections of $\mathcal{O}(D_X)$.  Then the following are equivalent:
\begin{enumerate}
\item  The collection $\{ \trop(f_0), \ldots, \trop(f_r) \}$ is tropically independent.
\item  There is a semistable model $\mathcal X$ of $X$, a line bundle $\mathcal{L}$ extending $\mathcal{O}(D_X)$, irreducible components $X_0, \ldots, X_r$ in the special fiber of $\mathcal X$, and scalars $a_0, \ldots, a_r \in K$ such that $a_i f_i$ extends to a regular section of $\mathcal{L}$ and vanishes on $X_j$ if and only if $i \neq j$.
\end{enumerate}
\end{proposition}

\begin{proof}
Suppose $\{ \trop(f_0), \ldots, \trop(f_r) \}$ is tropically independent.  Then there are real numbers $c_0, \ldots, c_r$ such that $\theta = \min \{ \trop(f_i) + c_i \}$ is independent.  Choose points $v_0, \ldots, v_r$ in $\Gamma$ such that $\trop(f_i) + c_i$ achieves the minimum uniquely at $v_i$.  Let $\mathcal X$ be the model corresponding to some semistable vertex set for $\Gamma$ that contains $\{v_0, \ldots, v_r\}$ and the tropicalization of every point in the support of $D_X$, and let $X_i$ be the irreducible component of the special fiber corresponding to $v_i$.

We define a subsheaf $\mathcal{L}$ of $K(X)$ on $\mathcal X$, extending $\mathcal{O}(D_X)$, as follows.  A rational function $f$ is a regular section of $\mathcal L$ at a point $x$ in the special fiber if and only if there is an affine open neighborhood $U$ of $x$ in the special fiber such that
\begin{enumerate}[label=(\roman*)]
\item $\ddiv(f) + D_X$ is effective on $\mathrm{sp}^{-1}(U)$, and
\item $\trop(f) \geq \theta$ on $\trop(\mathrm{sp}^{-1}(U))$.
\end{enumerate}
The choice of $\mathcal{X}$, which depends on both $D_X$ and $\theta$, guarantees that this sheaf is locally free of rank $1$.  Furthermore, by construction, a section $f$ of $\mathcal{O}(D_X)$ is regular on $X_i$ (resp. vanishes on $X_i$) if and only if $\trop(f) \geq \theta(v_i)$ (resp. $\trop(f) > \theta(v_i)$).  In particular, if we choose scalars $a_i \in K^*$ such that $\val(a_i) = c_i$, then the sections $a_i f_i$ of $\mathcal{O}(D_X)$ extend to regular sections of $\mathcal{L}$, and $a_i f_i$ vanishes on $X_j$ if and only if $i \neq j$, as required.

For the converse, given scalars $a_0, \ldots, a_r$, irreducible components $X_0, \ldots, X_r$, and an extension $\mathcal{L}$ of $\mathcal{O}(D_X)$ satisfying (2), set $c_i = \val(a_i)$.  By comparing $\trop(f_i)$ with the valuation of a local generator for $\mathcal L$ at the generic point of $X_i$, we conclude that $\trop(f_i) + c_i$ is strictly less than $\trop(f_j) + c_j$ at $v_i$, and hence $\theta = \min \{ \trop(f_i) + c_i \}$ is a certificate of independence.
\end{proof}

\begin{remark}
Proposition~\ref{prop:ag-interpretation} suggests some resemblance between our approach to proving linear independence of sections via tropical independences and the technique used to prove cases of the maximal rank conjecture via limit linear series and linked Grassmannians on chains of elliptic curves in \cite{LOTZ21}.  Osserman has also developed a notion of limit linear series for curves of pseudocompact type \cite{Osserman19a}, a class of curves that includes the semistable reduction of the curve $X$ we study here.  Relations to the Amini--Baker notion of limit linear series in tropical and nonarchimedean geometry are spelled out in \cite{Osserman19b}.
\end{remark}

Recall that Theorem~\ref{thm:independencecriterion} says a finite subset $\{ \psi_i \mid i \in I \} \subseteq \PL(\Gamma)$ is tropically independent if and only if there is a certificate of independence $\theta = \min \{ \psi_i + c_i \}$.

\begin{proof}[Proof of Theorem~\ref{thm:independencecriterion}]
First, we suppose that $\{ \psi_i \}$ is tropically dependent, and show that there is no such independence.  Choose real coefficients $c'_i$ such that the minimum of  $\{\psi_i + c'_i\}$ occurs at least twice at every point $v \in \Gamma$.  Now, consider an arbitrary tropical linear combination $\theta = \min \{ \psi_i + c_i \}$.  Choose $j \in I$ so that $c_j - c'_j$ is maximal.  At every $v \in \Gamma$, there is some $i \neq j$ such that $\psi_i(v) + c'_i \leq \psi_j(v) + c'_j$.  It follows that $\psi_i(v) + c_i \leq \psi_j(v) + c_j$, and hence $\theta = \min \{ \psi_i + c_i  \}$ is not a certificate of independence.

It remains to show that if there is no such independence, then $\{ \psi_i \}$ is tropically dependent.  Let $A_i$ be the set of vectors $c = (c_1, \ldots, c_n) \in \RR^{n}$ such that $\psi_i(v) + c_i \geq \min_{j \neq i} \{ \psi_j(v) + c_j \}$ for all $v \in \Gamma$.  Note that each $A_i$ is closed, and $c$ gives a certificate of independence if and only if it is contained in none of the $A_i$.  Similarly, $c$ gives a tropical dependence if and only if it is contained in all of the $A_i$. Hence, we must show that if the sets $A_i$ cover $\RR^{n}$ then their intersection is nonempty.

Suppose $A_1 \cup \cdots \cup A_n = \RR^{n}$. Choose $m$ sufficiently large so that $\psi_i(v) + m > \psi_j(v)$ for all $i,j$ and all $v \in \Gamma$.  Let $\Delta$ be the simplex spanned by $mn$ times the standard basis vectors in $\RR^{n}$.  If $c$ is in the proper face $\Delta_I \subset \Delta$ corresponding to $I \subset \{1, \ldots, n\}$, then $c_i > m$ for some $i \in I$ and $c_j = 0$ for $j \not \in I$.  Hence $\Delta_I$ is covered by $\{A_i \cap \Delta \}_{i \in I}$. The Knaster-Kuratowski-Mazurkiewicz lemma (that is, the set-covering variant of the Brouwer fixed-point theorem) then says that $A_1 \cap \cdots \cap A_n \cap \Delta$ is nonempty, as required.
\end{proof}

\section{Constructing the virtual divisors} \label{virtualis_divizorok}

In this section, we construct the virtual divisor class $[\widetilde {\mathfrak{D}}_{23}]^\vir$ as the push forward of the virtual class of a codimension $3$ determinantal locus.  This determinantal locus is contained inside a universal parameter space of limit linear series of type $\mathfrak g^6_{26}$ over an open substack $\widetilde{\mathfrak{M}}_{23}$  of $\overline \fM_{23}$ that differs from $\fM_{23}\cup \Delta_0\cup \Delta_1$ outside a subset of codimension $2$.  We follow a similar procedure in the case of the virtual divisor classes $[\widetilde{\mathfrak{D}}_g]^\vir$ on $\widetilde{\mathfrak{M}}_g$, for $g = 2s^2 + s + 1$, with $s\geq 2$. As long as the two constructions run parallel, we treat both simultaneously. Throughout, we work over an algebraically closed field $K$ of characteristic zero.

\vskip 4pt

We first recall the notation for vanishing and ramification sequences of limit linear series \cite{EisenbudHarris86}.

\begin{definition}
Let $X$ be a smooth curve of genus $g$, $q \in X$ a point, and $\ell = (L, V) \in G^r_d(X)$ a linear series on $X$.  The \emph{ramification sequence} of $\ell$ at $q$
\[
\alpha^{\ell}(q) : \alpha_0^{\ell}(q) \leq \cdots \leq \alpha_r^{\ell}(q)
\]
is obtained from the \emph{vanishing sequence}
\[
a^{\ell}(q) : a_0^{\ell}(q) < \cdots < a_r^{\ell}(q) \leq d
\]
by setting $\alpha^{\ell}_i(q) := a^{\ell}_i(q) - i$, for $i=0, \ldots, r$.  Sometimes, when $L$ is clear from the context, we write
$\alpha^{V}(q) = \alpha^{\ell}(q)$ and similarly $a^{V}(q) = a^{\ell}(q)$.  The \emph{ramification weight} of $q$ with respect to $\ell$ is $\mathrm{wt}^{\ell}(q) := \sum_{i=0}^r \alpha^{\ell}_i(q)$.  We denote by $\rho(\ell, q) := \rho(g,r,d) - \mathrm{wt}^{\ell}(q)$ the \emph{adjusted Brill-Noether number} of $\ell$ with respect to $q$.
\end{definition}

Recall from \cite[p. 364]{EisenbudHarris87b} that a \emph{generalized limit linear series} on a tree-like curve $X$ consists of a collection $\ell=\{(L_C, V_C) \mid C\mbox{ is a component of } X\}$, where $L_C$ is a rank $1$ torsion free sheaf of degree $d$ on $C$ and $V_C\subseteq H^0(C,L_C)$ is an $(r+1)$-dimensional space of sections satisfying the usual compatibility condition on the vanishing sequences at the nodes of $X$.
For such a tree-like curve $X$, we denote by $\overline{G}^r_d(X)$ the variety of generalized limit linear series of type $\mathfrak{g}^r_d$.

In what follows we fix positive integers $g$, $r$, and $d$ such that either
\begin{equation}\label{sit1}
g=23, \ r=6, \ d=26,  \ \mbox{ or }
\end{equation}
\begin{equation}\label{sit2}
g=2s^2+s+1, \ r=2s, \ d=2s^2+2s+1, \ \mbox{ where } \  s\geq 2.
\end{equation}

\noindent Note that $\rho(g,r,d)=2$ in case (\ref{sit1}) and $\rho(g,r,d)=1$ in case (\ref{sit2}).

\subsection{An open substack of $\overline \fM_g$}

We denote by $\cM_{g,d-1}^r$ the closed subvariety of $\cM_{g}$
parametrizing curves $X$  such that $W^r_{d-1}(X)\neq \emptyset$.  We claim that $\mathrm{codim}(\cM_{g, d-1}^r, \cM_{g})\geq 2$.  To see this, it suffices to observe that $\rho(g, r, d-1)$ is less than $-1$, and then apply \cite[Theorem 1.1]{EisenbudHarris89}.  For each curve $[X]\in \cM_g\setminus \cM_{g,d-1}^r$, every line bundle $L\in W^r_{d}(X)$ is base point free, with $H^1(X, L^{\otimes 2})=0$, since $d \geq g$. We denote by $\mm_{g,d-1}^r$ the closure of $\cM_{g,d-1}^r$ in $\mm_g$.

\vskip 4pt

Let $\Delta_1^{\circ}\subseteq \Delta_1\subseteq \mm_{g}$ be the locus of curves
$[X\cup_y E]$, where $X$ is a smooth curve of genus $g-1$ and $[E,y]\in \mm_{1,1}$ is an arbitrary elliptic curve.  The point of attachment $y \in X$ is chosen arbitrarily.  Furthermore, let  $\Delta_0^{\circ} \subseteq \Delta_0\subseteq \mm_{g}$ be the locus of curves $[X_{yq}:=X/y\sim q]\in \Delta_0$, where $[X, q]$  is a smooth curve of genus $g-1$ and $y \in X$ is an arbitrary point, together with their degenerations $[X \cup_q E_{\infty}]$, where
$E_{\infty}$ is a rational nodal curve (that is, $E_{\infty}$ is a nodal elliptic curve and $j(E_{\infty})=\infty$). Points of this form comprise the intersection
$\Delta_0^{\circ}\cap \Delta_1^{\circ}$. We define the following open subset of $\mm_g$:
$$\mm_g^{\circ}:=\cM_g\cup \Delta_0^{\circ}\cup \Delta_1^{\circ}.$$

\vskip 3pt

In order to define the open substack of $\mm_g^{\circ}$ over which Theorems \ref{thm:slopes} and \ref{rho1virtual} will be ultimately proved, we need further notation. Let $\mathcal{T}_0$ be the subvariety of $\Delta_0^{\circ}$ of curves $[X_{yq}:=X/y\sim q]$, where the curve $X$ satisfies $\overline{G}^{r+1}_d(X)\neq \emptyset$ or $\overline{G}^{r}_{d-2}(X)\neq \emptyset$.  Similarly, $\mathcal{T}_1\subseteq  \Delta_1^{\circ}$ denotes the subvariety of curves $[X\cup_y E]$, where $X$ is a smooth curve of genus $g-1$ with $G^{r+1}_d(X)\neq \emptyset$ or $G^r_{d-2}(X)\neq \emptyset$. Observe that both $\mathcal{T}_0$ and $\mathcal{T}_1$ are closed in $\mm_g^{\circ}$.

We introduce the following open subset of $\mm_g$:
\begin{equation}\label{def:tm}
\pm_{g}:=\mm_g^{\circ} \setminus \Bigl( \mm_{g,d-1}^r   \cup \mathcal{T}_0 \cup \mathcal{T}_1 \Bigr).
\end{equation}
We define $\widetilde{\Delta}_0:=\pm_g\cap \Delta_0\subseteq \Delta_0^{\circ}$ and $\widetilde{\Delta}_1:=\pm_g\cap \Delta_1\subseteq \Delta_1^{\circ}$, so
\[
\pm_g=\bigl(\cM_g\setminus \cM_{g,d-1}^r\bigr)\cup \widetilde{\Delta}_0 \cup \widetilde{\Delta}_1.
\]

Note that $\pm_g$ and $\cM_g\cup \Delta_0\cup \Delta_1$ differ outside a set of codimension $2$ and we use the identification $\Pic (\pm_{g}) \cong CH^1(\pm_{g}) = \mathbb Q \langle \lambda, \delta_0, \delta_1 \rangle$, where $\lambda$ is the Hodge class, $\delta_0:=[\widetilde{\Delta}_0]$ and $\delta_1:=[\widetilde{\Delta}_1]$.

\subsection{Stacks of limit linear series.}

Next we introduce the parameter spaces of limit linear series that we will use.
\begin{definition}\label{var_limitlinseries}
Let $\widetilde{\mathfrak{G}}^r_{d}$ be the stack of pairs $[X, \ell]$, where $[X]\in \widetilde{\mathfrak{M}}_{g}$ and $\ell$ is a (generalized) limit linear series on the tree-like curve $X$ in the sense of \cite{EisenbudHarris87b}. We consider the proper projection map
\[
\sigma\colon \widetilde{\mathfrak{G}}_{d}^r  \rightarrow  \widetilde{\mathfrak{M}}_g.
 \]
\end{definition}

We refer to \cite{EisenbudHarris86} and \cite{EisenbudHarris87b} for facts on limit linear series  and to \cite{Osserman06} and \cite{LieblichOsserman19} for details regarding the construction of $\widetilde{\mathfrak{G}}^r_{d}$. We discuss the fibers of $\sigma$. Over a curve $[X\cup_y E]\in \widetilde{\Delta}_1$, we identify $\sigma^{-1}([X\cup_y E])$ with the variety of limit linear series $\ell = (\ell_X, \ell_E) \in G^r_d(X) \times G^r_d(E)$ satisfying the compatibility conditions described in \cite{EisenbudHarris86}.  Over a point $[X\cup_y E_{\infty}] \in \widetilde{\Delta}_0 \cap \widetilde{\Delta}_1$, the fiber $\sigma^{-1}([X\cup_y E_{\infty}])$ is identified with the variety of generalized limit linear series $\overline{G}^r_d(X\cup _y E_{\infty})$.  In order to describe the fiber $\sigma^{-1}([X_{yq}])$ over  an irreducible curve $[X_{yq}] \in \widetilde{\Delta}_0$, we recall a few things about the variety $\overline{W}^r_d(X_{yq})$ of rank $1$ torsion free sheaves $L$ on $X_{yq}$ having $\mbox{deg}(L) = d$ and $h^0(X_{yq},L) \geq r+1$. We denote by $W^r_d(X_{yq})$ the open subvariety of $\overline{W}^r_d(X_{yq})$ consisting of line bundles.  If $\nu \colon X \rightarrow X_{yq}$ is the normalization map, and the curve $X$ satisfies $\overline{G}^{r+1}_d(X) = \emptyset$, we observe  that  $h^0(X_{yq}, L) = r+1$ for every sheaf $L\in \overline{W}_d^r(X_{yq})$.  In particular, we identify the fiber $\sigma^{-1}([X_{yq}])$ with $\overline{W}^r_d(X_{yq})$.  Moreover, the pull back map $\nu^* \colon W^r_d(X_{yq})\rightarrow \mbox{Pic}^d(X)$ is injective.

\vskip 4pt

For a pointed curve $[X, y,q]\in \cM_{g-1,2}$, by \cite[Proposition 12.1]{OdaSeshadri79}, there is a desingularization of the compactified Jacobian
\[
\widetilde{\mathrm{Pic}}^d(X_{yq}):=\PP \bigl(\P_{y}\oplus \P_{q}\bigr)\rightarrow \overline{\mathrm{Pic}}^d(X_{yq}).
\]
Here, $\P$ denotes a Poincar\'e bundle on $X\times \Pic^{d}(X)$, $\P_{y}$ denotes the restriction of $\P$ to $\{ y \} \times
\Pic^{d}(X)$, and $\P_q$ denotes the restriction of $\P$ to $\{q\}\times \Pic^{d}(X)$.  A point in $\widetilde{\Pic} ^d(X_{yq})$ can be thought of as a pair $(L,Q)$, where $L$ is a line bundle of degree $d$ on $X$ and $L_y\oplus L_q \twoheadrightarrow Q$ is a $1$-dimensional quotient.  The map $\widetilde{\Pic}^{d}(X_{yq}) \rightarrow \overline{\Pic}^{d}(X_{yq})$ assigns to a pair $(L,Q)$ the sheaf $L'$ on $X_{yq}$, defined by the exact sequence
\[
0\longrightarrow L'\longrightarrow \nu_*L\longrightarrow Q\longrightarrow 0.
\]

\begin{remark}\label{corresp2}
If the rank $1$ torsion free sheaf $L\in \overline{W}^r_d(X_{yq}) \setminus W^r_d(X_{yq})$ is not locally free, then this point corresponds to \emph{two} points in $\widetilde{\Pic}^d(X_{yq})$.  If $A \in W^r_{d-1}(X)$ is the unique line bundle such that $\nu_*(A)=L$, then these points are $\bigl(A(q)=A\otimes \OO_X(q),  A(q)_q\bigr)$ and $\bigl(A(y)=A\otimes \OO_X(y),  A(y)_y\bigr)$ respectively.
\end{remark}

\vskip 4pt

Let $\widetilde{\mathfrak{C}}_{g} \rightarrow \widetilde{\mathfrak{M}}_{g}$ be the universal curve, and let $p_2\colon \widetilde{\mathfrak{C}}_{g} \times_{\widetilde{\mathfrak{M}}_{g}} \widetilde{\mathfrak{G}}^r_{d} \rightarrow \widetilde{\mathfrak{G}}^r_{d}$ be the projection map.  We denote by $\mathfrak{Z}\subseteq \widetilde{\mathfrak{C}}_{g} \times_{\widetilde{\mathfrak{M}}_{g}} \widetilde{\mathfrak{G}}^r_{d}$ the codimension $2$ substack consisting
of pairs $[X_{yq}, L, z]$, where $[X_{yq}]\in \Delta_0^{\circ}$, the point $z$ is the node of $X_{yq}$  and $L\in \overline{W}^r_d(X_{yq})\setminus W^r_d(X_{yq})$ is a \emph{non-locally free} torsion free sheaf.  Let
\[
\epsilon:\widehat{\mathfrak{C}}_g:=\mathrm{Bl}_{\mathfrak{Z}}\Bigl(\widetilde{\mathfrak{C}}_{g} \times_{\widetilde{\mathfrak{M}}_{g}} \widetilde{\mathfrak{G}}^r_{d}\Bigr)\rightarrow \widetilde{\mathfrak{C}}_{g} \times_{\widetilde{\mathfrak{M}}_{g}} \widetilde{\mathfrak{G}}^r_{d}
\]
be the blow-up of this locus, and we denote the induced universal curve by
\[
\wp:=p_2\circ \epsilon\colon \widehat{\mathfrak{C}}_g\rightarrow \widetilde{\mathfrak{G}}^r_{d}.
\]

The fiber of $\wp$ over a point $[X_{yq}, L] \in \widetilde{\Delta}_0$, where $L\in \overline{W}^r_d(X_{yq})\setminus W^r_d(X_{yq})$, is the semistable curve $X\cup_{\{y,q\}} R$ of genus $g$, where $R$ is a smooth rational curve meeting $X$ transversally at $y$ and $q$.

\subsection{A degeneracy locus in the universal linear series}

We choose a Poincar\'e line bundle $\L$ over $\widehat{\mathfrak{C}}_g$ having the following properties:
\begin{enumerate}
\item For a curve $[X\cup_y E]\in \widetilde{\Delta}_1$ and a limit linear series $\ell = (\ell_X, \ell_E) \in G^r_{d}(X) \times G^r_{d}(E)$, we have that
$\L_{| [X\cup_y E, \ell]} \in \Pic^{d}(X) \times \Pic^0(E)$, where the restriction $\L_{|E}$ is obtained by twisting the underlying line bundle $L_E$ of the $E$-aspect $\ell_E$ by $\OO_E(-dy)$.

\item For a point $t=[X_{yq}, L]$, where $[X_{yq}] \in \widetilde{\Delta}_0$ and $L \in \overline{W}^r_d(X_{yq}) \setminus W^r_d(X_{yq})$, thus $L=\nu_*(A)$ for some $A\in W^r_{d-1}(X)$, we have $\mathcal{L}_{|X}\cong A$ and $\mathcal{L}_{|R}\cong \OO_R(1)$.  Here, as before, $\wp^{-1}(t) = X \cup R$.
\end{enumerate}

Next we introduce the sheaves
\begin{equation}\label{twosheaves}
\E:=\wp_*(\mathcal{L})  \ \mbox{ and }
\F:=\wp_*(\mathcal{L}^{\otimes 2})
\end{equation}
which play an essential role in the paper.  By Grauert's theorem, $\E$ is locally free and $\mbox{rank}(\E)=r+1$, and $\F_{\Grd\setminus \sigma^{-1}(\widetilde{\Delta}_1)}$ is also locally free and $\mbox{rank}(\F)=2d+1-g$.  We will show in Proposition \ref{vectorbundles} that in fact $\F$ is locally free over $\Grd$, and give a geometric interpretation of its fibers.

\vskip 4pt

There is a natural vector bundle morphism over $\widetilde{\mathfrak G}^r_{d}$ given by multiplication of sections,
\begin{equation}\label{morphism1}
\phi\colon \mbox{Sym}^{2}(\E)\rightarrow \F.
\end{equation}
We denote by  $\fU \subseteq  \widetilde{\mathfrak{G}}^r_{d}$ the first degeneracy locus of $\phi$, which carries a natural virtual class in the expected codimension, as the next definition explains.

\begin{definition}\label{def:virtclass}
We define the virtual divisor class  $[\widetilde{\mathfrak{D}}_{g}]^{\mathrm{virt}}:=\sigma_*([\fU]^\vir)$. Precisely, the classes $[\widetilde{\mathfrak{D}}_g]^{\mathrm{virt}}$ are virtual divisors in $\widetilde{\mathfrak{M}}_{g}$ given by
\[
[\widetilde{\mathfrak{D}}_{23}]^{\mathrm{virt}}:=\sigma_*\Bigl(c_3(\F-\mathrm{Sym}^2(\E))\Bigr)\in CH^1(\widetilde{\mathfrak{M}}_{23}),
\]
and, for $s\geq 2$ and $g = 2s^2 + s + 1$,
\[
[\widetilde{\mathfrak{D}}_g]^{\mathrm{virt}}:=\sigma_*\Bigl(c_2(\F-\mathrm{Sym}^2(\E))\Bigr)\in CH^1(\widetilde{\mathfrak{M}}_{g}).
\]
\end{definition}

In order to establish the local freeness of $\F$ and understand better the morphism $\phi$ in (\ref{morphism1}), we need further preparation.  For a pointed curve $[X,y]\in \cM_{g-1,1}$, we denote by $X_y'$ the genus $g$ curve obtained from $X$ by creating a cusp at $y$ and by $\nu\colon X\rightarrow X_{y}'$ the normalization map.  Recall that a \emph{pseudo-stable} curve is a connected curve having only nodes and cusps as singularities, such that its dualizing sheaf is ample and each smooth irreducible component of genus $1$ intersects the rest of the curve in at least two points.  Pseudo-stable curves of genus $g$ form a Delige-Mumford stack $\MM_g^{\mathrm{ps}}$.  One has a divisorial contraction $\pi\colon \MM_g\rightarrow  \MM_g^{\mathrm{ps}}$ replacing each elliptic tail of a stable curve with a cusp \cite{HassettHyeon09}.  Set-theoretically, $\pi\bigl([X \cup_y E]\bigr)=[X_y']$.

\begin{definition}\label{cuspcomp}
Let $\mathcal{W}\subseteq \sigma^{-1}(\widetilde{\Delta}_1)\subseteq \Grd$ be the divisor consisting of pairs
$[X\cup_y E, \ell]$, where $[X\cup_y E]\in \widetilde{\Delta}_1$ and $\ell=(\ell_X, \ell_E)$ is a limit linear series on $X\cup_y E$
such that $L_E=\OO_E(dy)$.
\end{definition}

If $[X\cup_y E, \ell] \in \sigma^{-1}(\widetilde{\Delta}_1)$, then the $X$-aspect $\ell_X$ of $\ell$ has a cusp at the point $y$. From the definition (\ref{def:tm}) of
$\pm_g$, it follows that $\ell_X$ must be complete. Arguing along the lines of \cite{HassettHyeon09} one sees that there is a divisorial contraction $\widetilde{\pi}\colon \Grd
\rightarrow \widetilde{\mathfrak{G}}_d^{r, \mathrm{ps}}$ of $\sigma^{-1}(\widetilde{\Delta}_1)$, where $\widetilde{\mathfrak{G}}_d^{r,\mathrm{ps}}$ denotes the stack of linear series of type $\mathfrak g^r_d$ over curves from the open substack $\pi(\widetilde{\mathfrak{M}}_g)$ of $\MM_g^{\mathrm{ps}}$. The morphism $\widetilde{\pi}$ replaces each curve $[X\cup_y E]\in \widetilde{\Delta}_1$ with the cuspidal curve $X_y'$ and a limit linear series $\bigl(\ell_X=|L_X|,\ell_E\bigr)$ on $X \cup_y E$ (where note that $L_X$ is locally free)
with the line bundle $L_X'\in W^r_d(X_y')$ such that $\nu^*(L_X')=L_X$. Observe that $h^0\bigl(X, L_X(-2y)\bigr)< h^0(X,L_X)$ since $[X\cup_y E]\notin \mathcal{T}_1$, therefore the line bundle $L_X'$ on $X_y'$ is uniquely determined by its pull back $L_X$ under the normalization map $\nu\colon X\rightarrow X_y'$.

\vskip 3pt

If we denote by $\Upsilon$ the divisor in $\wp^{-1}\bigl(\sigma^{-1}(\widetilde{\Delta}_1)\bigr)$ corresponding to marked points lying on the elliptic tail, then the morphism $\widetilde{\pi}:\Grd\rightarrow \widetilde{\mathfrak{G}}_d^{r, \mathrm{ps}}$ is induced by the linear series $\bigl|\wp_*\bigl(\omega_{\wp}(\Upsilon)\bigr)\bigr|$ (see \cite[Proposition 3.8]{HassettHyeon09} for a very similar claim).  We denote by $$\widetilde{\wp}:\widehat{\mathfrak{C}}_{g}^{\mathrm{ps}}\rightarrow \widetilde{\mathfrak{G}}_d^{r,\mathrm{ps}}$$ the universal curve and by $\mathcal{L}^{\mathrm{ps}}$ the Poincar\'e bundle on $\widehat{\mathfrak{C}}_{g}^{\mathrm{ps}}$.

\vskip 5pt

After this preparation, we now describe the morphism $\phi$ defined in (\ref{morphism1}) in more detail.

\begin{proposition}\label{vectorbundles}
Both sheaves $\E=\wp_*(\mathcal{L})$ and  $\F=\wp_*(\mathcal{L}^{\otimes 2})$ are locally free over $\Grd$.
\end{proposition}

\begin{proof}
We first show that for any $t\in \Grd$, one has $h^0\bigl(\wp^{-1}(t), \mathcal{L}_{| \wp^{-1}(t)}\bigr)=r+1$. Since the claim obviously holds for points in $\sigma^{-1}(\cM_g\setminus \cM_{g,d-1}^r)$, we assume first that $t = (X\cup_y E, \ell_X, \ell_E)\in \sigma^{-1}(\widetilde{\Delta}_1)$. Since $[X\cup_y E]\notin \mathcal{T}_1$, we have $h^0(X,L_X)=r+1$ and thus $\ell_X=|L_X|$ is a complete linear series. We have the exact sequence on $\wp^{-1}(t)$
\begin{equation}\label{exseq4}
0\longrightarrow H^0\bigl(\wp^{-1}(t), \mathcal{L}_{|\wp^{-1}(t)}\bigr)\longrightarrow H^0(X, L_X)\oplus H^0\bigl(E, L_E(-dy)\bigr)\stackrel{\mathrm{ev}_y}\longrightarrow \OO_y,
\end{equation}
where $\OO_y$ is the structure sheaf of the point $y$. We distinguish two cases. If $L_E\ncong \OO_E(dy)$, then $a_r^{\ell_E}(y)<d$, hence $a_0^{\ell_X}(y)>0$ and $\ell_X$ has a base point at $y$, in which case from (\ref{exseq4}) we get $H^0\bigl(\wp^{-1}(t),\mathcal{L}_{\wp^{-1}(t)}\bigr)\cong H^0(X,L_X)\cong H^0\bigl(X, L_X(-y)\bigr)$, which is $(r+1)$-dimensional.

If on the other hand $L_E\cong \OO_E(dy)$, then by restricting to the second factor, we see that the evaluation map $\mathrm{ev}_y$ is surjective, and again from (\ref{exseq4}) we obtain that $h^0\bigl(\wp^{-1}(t), \mathcal{L}_{|\wp^{-1}(t)}\bigr)=r+1$.

\vskip 4pt

Assume now that $t = [X_{yq}, L] \in \sigma^{-1}(\widetilde{\Delta}_0)$ . The case where $L$ is locally free is clear.  Assume instead that $L=\nu_*(A)$, with $A\in W^r_{d-1}(X)$. Recall that $\wp^{-1}(t)=X\cup_{\{y,q\}} R$, with $R$ being a smooth rational curve meeting $X$ at the points $y$ and $q$.  The Mayer-Vietoris sequence on $\wp^{-1}(t)$ then gives rise to exact sequences
\[
0\longrightarrow H^0\bigl(\wp^{-1}(t), \mathcal{L}_{|\wp^{-1}(t)}\bigr)\longrightarrow H^0(X, A)\oplus H^0(R, \OO_R(1))\longrightarrow
 \OO_y\oplus \OO_q,
\]
and
\[
0\longrightarrow H^0\bigl(\wp^{-1}(t), \mathcal{L}^{\otimes 2}_{|\wp^{-1}(t)}\bigr)\longrightarrow H^0(X, A^{\otimes 2})\oplus H^0(R, \OO_R(2))\longrightarrow \OO_y\oplus \OO_q.
\]
Since $[X_{yq}]\notin \mathcal{T}_0$, it follows that $h^0(X, A)=r+1$.  Again, by restricting to the second factor, we see that the righthand map is surjective.  Thus $h^0\bigl(\wp^{-1}(t), \mathcal{L}_{|\wp^{-1}(t)}\bigr)=r+1$ for every $t\in \Grd$, which shows that $\E$ is locally free.

\vskip 4pt

We now turn our attention to the sheaf $\F$ and first show that for $t \in \Grd\setminus \mathcal{W}$  we have that
\[h^0\bigl(\wp^{-1}(t), \mathcal{L}^{\otimes 2}_{|\wp^{-1}(t)}\bigr)=2d+1-g
\]
The case $t=[X_{yq}, \nu_*(A)]\in \sigma^{-1}(\widetilde{\Delta}_0)$ follows from the second exact sequence above.  Specifically, we have $h^1(X, A^{\otimes 2})=0$, so $h^0(X,A^{\otimes 2})=2d-g$, and by restricting to the second factor, we see that the righthand map is surjective.

If now $t = (X\cup_y E, \ell_X, \ell_E)\in \sigma^{-1}(\widetilde{\Delta}_1)$, we have an exact sequence
\[
0\longrightarrow H^0\bigl(\wp^{-1}(t), \mathcal{L}^{\otimes 2}_{|\wp^{-1}(t)}\bigr)\longrightarrow H^0(X, L_X^{\otimes 2})\oplus H^0\bigl(E, L_E^{\otimes 2}(-2dy)\bigr)\stackrel{\mathrm{ev}_y}\longrightarrow \OO_y .
\]
Since $h^1(X, L_X^{\otimes 2})=h^1(X,L_X^{\otimes 2}(-y))=0$, it follows that the map $\mathrm{ev}_y$ in the previous sequence is surjective.
If $L_E\ncong \OO_E(dy)$, we obtain $h^0\bigl(p^{-1}(t), \mathcal{L}^{\otimes 2}_{|p^{-1}(t)}\bigr)=2d+1-g$.

\vskip 4pt

If $t\in \mathcal{W}$, then $L_E=\OO_E(dy)$ and $h^0\bigl(\wp^{-1}(t), \mathcal{L}^{\otimes 2}_{|\wp^{-1}(t)}\bigr)=2d+2-g$ and this argument breaks down.  Instead, we recall that we introduced the divisorial contraction $\widetilde{\pi}\colon \Grd\rightarrow \widetilde{\mathfrak{G}}_d^{r,\mathrm{ps}}$ of $\sigma^{-1}(\widetilde{\Delta}_1)$. Then
\[
\wp_*(\mathcal{L}^{\otimes 2})=\widetilde{\pi}^*\Bigl(\widetilde{\wp}_*\bigl((\mathcal{L}^{\mathrm{ps}})^{\otimes 2}\bigr)\Bigr) .
\]
That is, for each $t\in \mathcal{W}$, the linear series $\wp_*(\mathcal{L}^{\otimes 2})_{|\wp^{-1}(t)}$ replaces the elliptic tail with a cusp.  Since $h^0\bigl(X_y', (L^{\mathrm{ps}})^{\otimes 2}\bigr)=2d+1-g$ for every cuspidal curve $X_y'$ and each $L^{\mathrm{ps}}\in W^r_d(X_y')$, applying Grauert's theorem over $\widetilde{\mathfrak{G}}_d^{r, \mathrm{ps}}$,  we conclude that
the sheaf $\F=\wp(\mathcal{L}^{\otimes 2})$ is locally free as well.
\end{proof}

\begin{remark} In situation (\ref{sit2}) the local freeness of $\F$ follows from general principles, without having to resort to the local analysis above. Indeed, applying \cite[Corollary 1.7]{Hartshorne80}, it follows that  $\wp_*(\mathcal{L}^{\otimes 2})$ is a reflexive sheaf, thus its singular locus is of codimension at least $3$ in
$\widetilde{\mathfrak{G}}^r_d$. Removing this locus, one can still define the virtual class $[\widetilde{\mathfrak{D}}_g]^{\mathrm{virt}}$ as in Definition \ref{def:virtclass}. This argument falls short in case (\ref{sit1}), however, where we cannot discard codimension $3$ loci in $\widetilde{\mathfrak{G}}^r_d$.
\end{remark}

The next corollary summarizes the fiberwise description of $\E$ and $\F$ implicitly obtained above. It follows from the application of Grauert's Theorem explained in the proof of Proposition \ref{vectorbundles}.

\begin{corollary}\label{cor:vectorbundles}
The vector bundle map  $\phi\colon \mathrm{Sym}^2(\E)\rightarrow \F$ has the following local description:

\noindent $(i)$  For $[X, L]\in \widetilde{\mathfrak{G}}^r_{d}$, with $[X]\in \cM_g\setminus \cM_{g, d-1}^r$ smooth, one has the following description of the fibers
\[
\E_{(X, L)}=H^0(X, L)\ \mbox{ and } \ \F_{(X, L)}=H^0(X, L^{\otimes 2})
\]
and $\phi_{(X,L)}\colon \mathrm{Sym}^2 H^0(X, L)\rightarrow H^0(X, L^{\otimes 2})$ is the usual multiplication map of global sections.

\vskip 5pt

\noindent $(ii)$
Suppose $t=(X\cup_y E, \ell_X, \ell_E)\in \sigma^{-1}(\widetilde{\Delta}_1)$,
where $X$ is a curve of genus $g-1$, $E$ is an elliptic curve  and $\ell_X=|L_X|$ is the $X$-aspect of the corresponding limit linear series with $L_X\in
W^r_{d}(X)$ such that $h^0(X,L_X(-2y))\geq r$. If $L_X$ has no base point at $y$, then
\[
\E_t=H^0(X, L_X)\cong H^0\bigl(X,L_X(-2y)\bigr)\oplus K \cdot u \ \mbox{ and
} \ \F_t=H^0\bigl(X, L_X^{\otimes 2}(-2y)\bigr)\oplus  K \cdot u^2,
\]
where $u\in H^0(X, L_X)$ is any section such that $\mathrm{ord}_y(u)=0$.

\vskip 5pt

If $L_X$ has a base point at $y$, then $\E_t=H^0(X, L_X)\cong H^0(X, L_X(-y))$ and the image of the
map $\F_t \rightarrow H^0(X, L_X^{\otimes 2})$ is the subspace $H^0\bigl(X, L_X^{\otimes 2}(-2y)\bigr)\subseteq H^0(X,L_X^{\otimes 2})$.

\vskip 4pt

\noindent $(iii)$  Let $t=[X_{yq}, L] \in \sigma^{-1}(\widetilde{\Delta}_0)$ be a point with $q,y \in X$ and let $L\in W^r_{d}(X_{yq})$ be a locally free sheaf of rank $1$, such that $h^0(X, \nu^*L(-y-q))\geq r$, where $\nu\colon X\rightarrow X_{yq}$ is the
normalization map. Then
\[
\E(t)=H^0(X, \nu^*L)\ \mbox{ and }\ \F(t)=H^0\bigl(X, \nu^*L^{\otimes
2}(-y-q)\bigr)\oplus  K \cdot u^2,
\]
where $u\in H^0(X, \nu^*L)$ is any section not vanishing at both points $y$ and $q$.

\vskip 5pt

\noindent $(iv)$ Let $t = [X_{yq}, \nu_*(A)]$, where $A\in W^r_{d-1}(X)$ and set again $X\cup_{\{y,q\}} R$ to be the fiber $\wp^{-1}(t)$.
Then $\E(t)=H^0(X\cup R, \mathcal{L}_{X\cup R})\cong H^0(X,A)$ and $\F(t)=H^0(X\cup R, \mathcal{L}^{\otimes 2}_{X\cup R})$. Furthermore,
$\phi(t)$ is the multiplication map on $X\cup R$.
\end{corollary}

\subsection{Pull back to test curves}

In preparation for the proofs of Theorems~\ref{thm:slopes} and~\ref{rho1virtual}, concerning the calculation of $[\widetilde{\mathfrak{D}}_{g}]^{\mathrm{virt}}$, we describe the restriction of the morphism $\phi$ along the pull backs of the three standard test curves $F_0$, $F_{\mathrm{ell}}$ and $F_1$ defined by (\ref{f0}), (\ref{fell}) and (\ref{fi}), respectively.    Recall that we fix a general pointed curve $[X, q]$ of genus $g-1$ and a pointed elliptic curve $[E, y]$.  We then have
\[
F_0:=\Bigl\{X_{yq}:=X/y\sim q \mid y\in X\Bigr\}\subset \Delta_0^{\circ}\subset \ttem_g \   \mbox{ and  } \  F_1:=\Bigl\{X\cup_y E \mid y \in X\Bigr\}\subset \Delta_1^{\circ} \subset \ttem_g.
\]

\begin{proposition}
One has that $F_0 \subset  \pm_g$ and $F_1\subset \pm_g.$
\end{proposition}

\begin{proof}
We only show that $F_1\subset \widetilde{\Delta}_1 \subset \pm_g$, the argument for $F_0$ being analogous. To that end, choose a point $[X\cup_y E]\in \Delta_1^{\circ}$, where $X$ is a general curve of genus $g-1$. Assuming $[X\cup_y E]\in \mm_{g,d-1}^r$, it follows that $\overline{G}^{r}_{d-1}(X\cup_y E)\neq \emptyset$. Denoting by $L_X\in \mbox{Pic}^{d-1}(X)$ the underlying line bundle of the $X$-aspect of $\ell$, we obtain $h^0(X, L_X(-2y))\geq r$, that is, $W_{d-3}^{r-1}(X)\neq \emptyset$. In both cases (\ref{sit1}) and (\ref{sit2}), we have $\rho(g-1, r-1, d-3)<0$, which contradicts the generality of $X$. The same consideration shows that $F_1$ is disjoint from both $\mathcal{T}_0$ and $\mathcal{T}_1$.
\end{proof}

We now turn our attention to the pull back $\sigma^*(F_0)\subset \widetilde{\mathfrak G}^r_{d}$. We consider the determinantal variety
\begin{equation}\label{defx}
Y:=\Bigl\{(y, L)\in X\times W^r_{d}(X) \mid h^0(X, L(-y-q))\geq r\Bigr\},
\end{equation}
together with the projection $\pi_1\colon Y \rightarrow X$.

\begin{proposition}\label{purediml}
The variety $Y$ is pure of dimension $\rho(g, r,d)+1$.  That is, $3$-dimensional in case (\ref{sit1}) and $2$-dimensional in case (\ref{sit2}).
\end{proposition}

\begin{proof}
We consider the projection $\pi_1\colon Y\rightarrow X$.  Its fiber over the point $q\in X$ is the variety of linear series $L\in W^r_d(X)$ having a cusp
at $q$, that is, $h^0(X, L(-2q))\geq r$.  By \cite[Theorem 1.1]{EisenbudHarris87b}, it follows that $\pi_1^{-1}(q)$ has the same dimension as the variety $W^r_d(X_{\mathrm{gen}})$ for a general curve $X_{\mathrm{gen}}$ of genus $g$, which is $\rho(g, r,d)$.  Furthermore, using a standard degeneration to a flag curve, it follows that for every point $y\in X$ we have $\mbox{dim } \pi_1^{-1}(y)\leq \rho(g,r,d)+1$.  Therefore each component of $Y$ has dimension $\rho(g,r,d)+1$.
\end{proof}

Inside $Y$ we introduce the following subvarieties of $Y$:
\begin{align*}
\Gamma_1 &:= \Bigl\{(y, A(y)) \mid y\in X, \ A\in W^r_{d-1}(X)\Bigr\} \ \mbox{  and } \\
\Gamma_2 &:= \Bigl\{(y, A(q)) \mid y\in X, \ A\in W^r_{d-1}(X)\Bigr\}.
\end{align*}
These are divisors intersecting transversally along the smooth locus
\[
\Gamma:=\Bigl\{(q, A(q)) \mid A\in W_{d-1}^r(X)\Bigr\}\cong W_{d-1}^r(X).
\]
We then consider the variety obtained from $Y$ by identifying for each $(y, A)\in X\times W^r_{d-1}(X)$, the points $(y, A(y))\in \Gamma_1$ and
$(y, A(q))\in \Gamma_2$, that is,
\[
\widetilde{Y} := Y/[ \Gamma_1 \cong \Gamma_2 ],
\]
and denote by $\vartheta\colon Y\rightarrow \widetilde{Y}$ the projection map.

\begin{proposition}\label{limitlin0}
With notation as above, there is a birational morphism
\[
f\colon \sigma^*(F_0)\rightarrow \widetilde{Y},
\]
which is an isomorphism outside $\vartheta (\pi_1^{-1}(q))$. The restriction of $f$ to $f^{-1}\bigl(\vartheta(\pi_1^{-1}(q))\bigr)$  forgets the aspect of each limit linear series on the elliptic curve $E_{\infty}$.  Furthermore, both $\E_{| \sigma^*(F_0)}$ and $\F_{| \sigma^*(F_0)}$ are pull backs under $f$ of vector bundles on $\widetilde{Y}$.
\end{proposition}

\begin{proof}
Let $y \in X \setminus \{ q \}$ and, as usual, let $\nu\colon X \rightarrow X_{yq}$ be the normalization.  Recall that we have identified $\sigma^{-1}([X_{yq}])$ with the variety $\overline{W}^r_{d}(X_{yq}) \subseteq \overline{\Pic}^d(X_{yq})$  of rank $1$ torsion-free sheaves on $X_{yq}$ with  $h^0(X_{yq}, L)\geq r+1$.  A locally free sheaf $L\in \overline{W}^r_{d} (X_{yq})$ is uniquely determined by its pull back $\nu^*(L)\in W^r_{d}(X)$, which has the property that $h^0(X, \nu^*L(-y-q))=r$.  Since $X$ is assumed to be Brill-Noether general $W^r_{d-2}(X)=\emptyset$, so there exists a section of $L$ that does not vanish simultaneously at both $y$ and $q$. In other words, the $1$-dimensional quotient $Q$ of $L_y\oplus L_q$ is uniquely determined as $\nu_*(\nu^*L)/L$.

Assume  $L\in \overline{W}_{d}^r(X_{yq})$ is not locally free, thus $L=\nu_*(A)$ for some line bundle $A \in W^r_{d-1}(X)$.  By Remark~\ref{corresp2}, this point corresponds to two points in $Y$, namely $(y, A(y))$ and $(y, A(q))$.  There is a birational morphism $\pi_1^{-1}(y)\rightarrow \overline{W}_{d}^r(X_{yq})$ which is an isomorphism over the locus $W^r_d(X_{yq})$ of locally free sheaves.  More precisely, $\overline{W}_{d}^r(X_{yq})$ is obtained from $\pi_1^{-1}(y)$ by identifying the disjoint divisors $\Gamma_1\cap \pi_1^{-1}(y)$ and $\Gamma_2\cap \pi_1^{-1}(y)$.

Finally, when $y=q$, then $X_{yq}$ degenerates to $X\cup _q E_{\infty}$, where $E_{\infty}$ is a rational nodal curve.  The fiber $\sigma^{-1} \bigl([ X \cup_q E_{\infty}] \bigr)$ is the variety of generalized limit linear series $\mathfrak g^r_d$ on $X \cup_q E_{\infty}$ and there is a map $\sigma^{-1} \bigl([ X \cup_q E_{\infty}] \bigr) \rightarrow \pi^{-1}(q)$ obtained by forgetting the $E_{\infty}$-aspect of each limit linear
series.  The statement about the restrictions
$\E_{| \sigma^*(F_0)}$ and $\F_{| \sigma^*(F_0)}$ follows from Corollary \ref{cor:vectorbundles} because both restrictions are defined by dropping the information coming from the elliptic tail.
\end{proof}

We now describe the pull back $\sigma^*(F_1)\subset \widetilde{\mathfrak{G}}^r_{d}$.  To that end, we define the locus
\begin{equation}\label{defz}
Z := \Bigl\{(y, L)\in X\times W^r_{d}(X) \mid h^0(X, L(-2y))\geq r\Bigr\} .
\end{equation}
By slight abuse of notation, we denote again by $\pi_1\colon Z\rightarrow X$ the first projection.  Arguing along the lines of Proposition \ref{purediml}, it follows that $Z$ is pure of dimension $\rho(g, r,d)+1$.

\begin{proposition}\label{limitlin1}
The variety $Z$ is an irreducible component of $\sigma^*(F_1)$.  Furthermore, we have
\begin{align*}
c_3\bigl(\F-\mathrm{Sym}^2(\E)\bigr)_{| \sigma^*(F_1)} &= c_3\bigl(\F-\mathrm{Sym}^2(\E)\bigr)_{| Z} \ \mbox{ in case (\ref{sit1}), and} \\
c_2\bigl(\F-\mathrm{Sym}^2(\E)\bigr)_{| \sigma^*(F_1)} &= c_2\bigl(\F-\mathrm{Sym}^2(\E)\bigr)_{| Z} \ \mbox{ in case (\ref{sit2}).}
\end{align*}
\end{proposition}

\begin{proof}
We deal with the case $(g,r,d)=(23,6,26)$, the case (\ref{sit2}) being analogous. By the additivity of the Brill-Noether number, if
$(\ell_X, \ell_E)\in \sigma^{-1}([X\cup_y E])$ is a limit linear series of type $\mathfrak{g}^6_{26}$, we
have that $2=\rho(23, 6, 26)\geq \rho(\ell_X, y)+\rho(\ell_E, y)$. Since
$\rho(\ell_E, y)\geq 0$, we obtain that $\rho(\ell_X, y)\leq 2$. If
$\rho(\ell_E, y)=0$, then $\ell_E=19y+|\OO_E(7y)|$.  This shows that $\ell_E$ is
uniquely determined, while the aspect $\ell_X\in G^6_{26}(X)$ is a complete
linear series with a cusp at $y\in X$.
This gives rise to an element from $Z$ and shows that $Z\times \{\ell_E\}\cong Z$ is a component of $\sigma^*(F_1)$.

The other components of $\sigma^*(F_1)$ are indexed by Schubert
indices
\[
\alpha:=\bigl(0\leq \alpha_0\leq \ldots \leq \alpha_6\leq 20=26-6\bigr),
\]
 such that lexicographically $\alpha>(0, 1, 1, 1, 1, 1,1)$, and $7\leq \sum_{j=0}^6 \alpha_j \leq 9$, for we must also have $-1\leq \rho(\ell_X,y)\leq 1$ for any point $y\in X$, see \cite[Theorem 0.1]{Farkas13}.   For such an index $\alpha$, we set $\alpha^c:=(20-\alpha_6,\ldots, 20-\alpha_0)$ to be the complementary Schubert index, and define
\[
Z_{\alpha}:=\bigl\{(y, \ell_X)\in X\times G^6_{26}(X) \mid \alpha^{\ell_X}(y)\geq \alpha\bigr\}\ \mbox{ and } W_{\alpha}:=\bigl\{\ell_E\in G^6_{26}(E) \mid \alpha^{\ell_E}(y)\geq \alpha^c\bigr\}.
\]
Then the following relation holds
\[
\sigma^*(F_1)=Z+\sum_{\alpha>(0,1,1,1,1,1,1)} m_{\alpha}\  Z_{\alpha}\times W_{\alpha},
\]
where the multiplicities $m_{\alpha}$ can be determined via Schubert calculus but play no role in our calculation. Our claim now follows for dimension reasons. Applying the Brill-Noether Theorem \cite[Theorem 1.1]{EisenbudHarris87b} in the pointed setting and using that $X$ is a general curve, we obtain  the estimate $\mbox{dim } Z_{\alpha}=1+\rho(22, 6, 26)-\sum_{j=0}^6 \alpha_j<3$, for every index $\alpha>(0,1,1, 1, 1, 1, 1)$. In the definition of the test curve $F_1$, the point of attachment $y\in E$ is fixed, therefore the restrictions of both $\E$ and $\F$ are pulled-back from $Z_{\alpha}$ and one obtains that $c_3\bigl(\F- \mbox{Sym}^2(\E)\bigr)_{| Z_{\alpha}\times W_{\alpha}}=0$ for dimension reasons.
\end{proof}

\section{The class of the virtual divisor on \texorpdfstring{$\widetilde{\mathcal{M}}_{23}$}{M23}}
\label{Sec:23}

In this section we compute the class of $[\widetilde{\cD}_{23}]^\vir$ and prove the $g = 23$ part of Theorem \ref{thm:slopes}.

\subsection{Chern numbers of tautological classes on Jacobians.}
We repeatedly use facts about intersection theory on Jacobians, and refer to \cite[Chapters VII--VIII]{ACGH} for background on this topic and to \cite{HarrisMumford82, Harris84, Farkas09b} for applications to divisor class calculations on $\mm_g$. Let $X$ be a Brill-Noether general curve of genus $g$.  Denote by $\P$ a Poincar\'e line bundle on $X\times \mbox{Pic}^d(X)$ and by
\[
\pi_1\colon X\times \mbox{Pic}^d(X)\longrightarrow X \mbox{ and } \pi_2\colon X\times \mbox{Pic}^d(X)\longrightarrow \mbox{Pic}^d(X)
\]
the two projections.  We introduce the class $\eta=\pi_1^*([x_0])\in H^2(X \times \mbox{Pic}^d(X), \mathbb Z)$, where $x_0\in X$ is an arbitrary point.  After picking a symplectic basis $\delta_1,\ldots, \delta_{2g}\in H^1(X, \mathbb Z)\cong H^1(\mbox{Pic}^d(X), \mathbb Z)$, we consider the class
\[
\gamma:=-\sum_{\alpha=1}^g \Bigl( \pi_1^*(\delta_{\alpha}) \pi_2^*(\delta_{g+\alpha}) - \pi_1^*(\delta_{g+\alpha})\pi_2^*(\delta_{\alpha}) \Bigr) \in H^2(X \times \mathrm{Pic}^d(X), \mathbb{Z} ).
\]
One has the formula $c_1(\P)=d\cdot\eta+\gamma$, which describes the K\"unneth decomposition of $c_1(\P)$, as well as the relations $\gamma^3=0$, $\gamma \eta=0$,  $\eta^2=0$, and $\gamma^2 = -2\eta \pi_2^*(\theta)$, see \cite[page 335]{ACGH}.  Assuming $W^{r+1}_d(X)=\emptyset$, that is, when the Brill-Noether number $\rho(g,r+1,d)$ is negative (which happens in both cases (\ref{sit1}) and (\ref{sit2})), the smooth variety $W^r_d(X)$ admits a rank $r+1$ vector bundle
\[
\mathcal{M}:=(\pi_2)_{*}\Bigl(\mathcal{P}_{| X\times W^r_d(X)}\Bigr)
\]
with fibers $\cM(L)\cong H^0(X,L)$, for $L\in W^r_d(X)$.  In order to compute the Chern numbers of $\mathcal{M}$, we repeatedly employ the Harris-Tu
formula \cite{HarrisTu84}, which we now explain.  We write
\[
\sum_{i=0}^r c_i(\mathcal{M}^{\vee})=(1+x_1)\cdots (1+x_{r+1}).
\]
Then, for each class $\zeta \in H^*\bigl(\mbox{Pic}^d(X), \mathbb{Z} \bigr)$, any Chern number $c_{j_1}(\mathcal{M}) \cdots c_{j_s}(\mathcal{M})\ \cdot \zeta \in H^{\mathrm{top}}(W^r_d(X), \mathbb{Z} )$ can be computed by using repeatedly the formal identities\footnote{Formula (\ref{harristu}) is to be interpreted as a \emph{formal recipe} for evaluating  the Chern numbers $c_{j_1}(\mathcal{M}) \cdots c_{j_s}(\mathcal{M})\ \cdot \zeta$. Precisely, $W^r_d(X)$ can be expressed as the degeneracy locus of a morphism of vector bundles
$\mathcal{V}_1\rightarrow \mathcal{V}_2$ over $\mbox{Pic}^d(X)$ and  $\mathcal{M}$ is the kernel bundle of the restriction of this map to $W^r_d(X)$. Passing to a flag variety $\alpha\colon \mathbb F:=F(\mathcal{V}_1)\rightarrow \mbox{Pic}^d(X)$ over which one has canonical choices for the Chern roots $x_1, \ldots, x_{r+1}$,  formula (\ref{harristu}) is then proven in \cite[Corollary 2.6]{HarrisTu84} at the level of $\mathbb F$.}:
\begin{equation}\label{harristu}
x_1^{i_1}\cdots x_{r+1}^{i_{r+1}}\cdot  \zeta=\mbox{det}\Bigl(\frac{\theta^{g+r-d+i_j-j+k}}{(g+r-d+i_j-j+k)!}\Bigr)_{1\leq
j, k\leq r+1}\ \zeta.
\end{equation}

Via the expression of the Vandermonde determinant, one has the identity, see also \cite[p. 320]{ACGH}:
$$\mbox{det}\Bigl(\frac{1}{(a_j+k-1)!}\Bigr)_{1\leq j, k\leq r+1}=\frac{\prod_{j>k}(a_k-a_j)}{\prod_{j=1}^{r+1}(a_j+r)!}.$$
Using then (\ref{harristu}), we obtain the following formula in $H^{\mathrm{top}}(W^r_d(X), \mathbb Z)$:

\begin{equation}\label{harristu2}
x_1^{i_1}\cdots x_{r+1}^{i_{r+1}}\cdot
\theta^{\rho(g,r,d)-i_1-\cdots-i_{r+1}} =g!\ \frac{\prod_{j>k} (i_{k}-i_j+j-k)}{\prod_{k=1}^{r+1} (g-d+2r+1+i_k-k)!}.
\end{equation}

Jet bundles are employed several times in this section, and we recall their definition. Denote by
\[
\mu, \nu\colon X \times X\times \mbox{Pic}^{d}(X) \longrightarrow X \times \mbox{Pic}^{d}(X)
\]
the two projections and by $\Delta \subset X \times X \times \mbox{Pic}^{d}(X)$ the diagonal. Then the \emph{jet bundle} of the Poincar\'e line bundle $\mathcal{P}$ on $X\times \mbox{Pic}^d(X)$ is defined as $J_1(\mathcal{P}):=\nu_*\bigl(\mu^*(\mathcal{P})\otimes \OO_{2\Delta}\bigr)$. Its fiber over a point $(y, L)\in X\times \mbox{Pic}^d(X)$ is naturally identified with $H^0(L\otimes \OO_{2y})$.

\subsection{Top intersection products in the Jacobian of a curve of genus $22$.}
We now specialize to the case of a general curve $X$ of genus $22$. By Riemann-Roch the duality $W^6_{26}(X)\cong W^1_{16}(X)$ holds.  Since $\rho(22,7,26)=-2<0$,  note that $W^7_{26}(X)=\emptyset$, so we can consider the rank $7$ tautological vector bundle $\cM$ on $W^6_{26}(X)$ with fibers $\cM_L\cong H^0(X, L)$. The vector bundle $\mathcal{N}:=(R^1\pi_2)_{*}\Bigl(\mathcal{P}_{| X\times W^6_{26}(X)}\Bigr)$ has rank $2$ and we explain how its two Chern classes determine all of the Chern classes of $\cM$.

\begin{proposition}
\label{chernosztalyok}
For a general curve $X$ of genus $22$ we set  $c_i:=c_i(\cM^{\vee})$, for $i=1, \ldots, 7$, and $y_i:=c_i(\cN)$, for $i=1,2$. Then the following relations hold in $H^*(W^6_{26}(X), \mathbb Z)$:
\begin{align*}
c_1 &= \theta-y_1 \mbox{ and} \\
c_{i+2} &= \frac{1}{i!}y_2 \theta^i-\frac{1}{(i+1)!}y_1 \theta^{i+1}+\frac{1}{(i+2)!}\theta^{i+2} \mbox{ for all } i\geq 0.
\end{align*}
\end{proposition}

\begin{proof}
Fix an effective divisor $D\in X_e$ of sufficiently large degree $e$.  There is an exact sequence
\[
0 \rightarrow \cM \rightarrow (\pi_{2})_*\Bigl( \P \otimes \OO(\pi^*D) \Bigr)\rightarrow (\pi_2)_*\Bigl( \P \otimes \OO(\pi_1^*D)_{| \pi_1^*D} \Bigr) \rightarrow R^1\pi_{2*} \Bigl( \P_{| X\times W^6_{26}(X)} \Bigr) \rightarrow 0.
\]
Recall that $\mathcal{N}$ is the vector bundle on the right in the exact sequence above.  By \cite[Chapter VII]{ACGH}, we have  $c_{\mathrm{tot}}\Bigl((\pi_2)_*(\P \otimes \OO(\pi_1^*D))\Bigr)=e^{-\theta}$, and the total Chern class of the vector bundle $(\pi_2)_* \Bigl( \P \otimes \OO(\pi_1^*D)_{| \pi_1^*D} \Bigr)$ is trivial.  We therefore obtain
\[
c_{\mathrm{tot}}(\mathcal{N}) \cdot e^{-\theta} = \sum_{i=0}^8 (-1)^i c_i.
\]
Hence $c_{i+2}=\frac{1}{i!}y_2 \theta^i-\frac{1}{(i+1)!}y_i \theta^{i+1}+\frac{1}{(i+2)!}\theta^{i+2}$ for all $i\geq 0$, as desired.
\end{proof}

Using Proposition \ref{chernosztalyok}, any Chern number on the smooth $8$-fold $W^6_{26}(X)$ can be expressed in terms of monomials in the classes $u_1$, $u_2$, and $\theta$, where $u_1$ and $u_2$ are the Chern roots of $\cN$, that is,
\[
y_1=c_1(\cN)=u_1+u_2 \ \mbox{ and } y_2=c_2(\cN)=u_1\cdot u_2.
\]
We record for further use the following formal identities on $H^{\mathrm{top}}(W^6_{26}(X),\mathbb Z)$, which are obtained by applying formula (\ref{harristu}) in the case $g=22$, $r=1$ and $d=16$, using the canonical isomorphism $H^1(X, L)\cong H^0\bigl(X, \omega_X\otimes L^{\vee}\bigr)^{\vee}$ provided by Serre duality.

\begin{align*}
u_1^3\theta^5 &= \frac{4\cdot 22!}{11!\cdot 7!}, & u_2^3\theta^5 &= -\frac{2\cdot 22!}{8!\cdot 10!}, & u_1^2\theta^6 &= \frac{3\cdot 22!}{10!\cdot 7!}, & u_2^2\theta^6 &= -\frac{22!}{8!\cdot 9!}, \\ u_1\theta^7 &= \frac{2\cdot 22!}{7!\cdot 9!}, &
u_2\theta^7 &= 0, & u_1u_2^4\theta^3 &= -\frac{2\cdot 22!}{9!\cdot 11!}, & u_1^4u_2\theta^3 &= \frac{4\cdot 22!}{8!\cdot 12!}, \\ u_1^2u_2\theta^5 &= \frac{2\cdot 22!}{8!\cdot 10!}, & u_1u_2^2\theta^5 &= 0, &
u_1^2 u_2^3\theta^3 &= 0, & u_1^3u_2^2\theta^3 &= \frac{2\cdot 22!}{9!\cdot 11!}, \\ u_1^2u_2^2\theta^4 &= \frac{22!}{9!\cdot 10!}, & u_1^4\theta^4 &= \frac{5\cdot 22!}{7!\cdot 12!}, & u_2^4\theta^4 &= -\frac{3\cdot 22!}{8!\cdot 11!}, &
\ u_1^3u_2\theta^4 &= \frac{3\cdot 22!}{8!\cdot 11!}, \\ u_1u_2^3\theta^4 &= -\frac{22!}{9!\cdot 10!}, & u_1u_2\theta^6 &= \frac{22!}{8!\cdot 9!}, & \theta^8 &= \frac{22!}{7!\cdot 8!}.
\end{align*}

To compute the corresponding Chern numbers on $W^6_{26}(X)$, one uses Proposition \ref{chernosztalyok} and the previous formulas.  Each Chern number corresponds to a degree $8$ polynomial in  $u_1$, $u_2$, and $\theta$, which is symmetric in $u_1$ and $u_2$.

We now compute the classes of the loci $Y$ and $Z$ appearing in Propositions \ref{limitlin0} and \ref{limitlin1}.

\begin{proposition}\label{prop:osztalyxy}
\label{xy}
Let $[X,q]$ be a general $1$-pointed curve of genus $22$, let $\cM$ denote the tautological rank $7$ vector bundle over $W^6_{26}(X)$, and let $c_i = c_i(\cM^{\vee}) \in H^{2i}(W^6_{26}(X), \mathbb{Z})$ as before.  Then the following hold:
\begin{enumerate}
\item  $[Z] = \pi_2^*(c_6) - 6\eta \theta \pi_2^*(c_4) + (94\eta + 2\gamma) \pi_2^*(c_5) \in H^{12} (X \times W^6_{26}(X), \mathbb{Z} )$.

\item  $[Y] = \pi_2^*(c_6) - 2\eta \theta \pi_2^*(c_4) + (25\eta + \gamma) \pi_2^*(c_5) \in H^{12}(X \times W^6_{26}(X), \mathbb{Z} )$.
\end{enumerate}
\end{proposition}

\begin{proof}
Recall that $W^6_{26}(X)$ is smooth of dimension $8$.  We realize the locus $Z$ defined by (\ref{defz})  as the degeneracy locus of a vector bundle morphism over $X \times W^6_{26}(X)$. Precisely, for each pair $(y,L) \in X \times W^6_{26}(X)$, there is a natural map
\[
H^0(X, L\otimes \OO_{2y})^{\vee} \longrightarrow H^0(X, L)^{\vee},
\]
which globalizes to a bundle morphism $\zeta \colon J_1(\mathcal{P})^{\vee} \rightarrow \pi_2^*(\cM)^{\vee}$ over $X\times W^6_{26}(X)$.  Then we have the identification $Z=Z_1(\zeta)$, that is, $Z$ is the first degeneracy locus of $\zeta$.  The Porteous formula yields $[Z] = c_6 \Bigl( \pi^*_2 (\cM)^{\vee} - J_1 (\mathcal{P})^{\vee} \Bigr)$.  To evaluate this class, we use the exact sequence over $X \times \mbox{Pic}^{26}(X)$ involving the jet bundle:
\[
0\longrightarrow \pi_1^*(\omega_X) \otimes \mathcal{P} \longrightarrow J_1(\mathcal{P}) \longrightarrow \mathcal{P} \longrightarrow 0 .
\]
We compute the total Chern class of the formal inverse of the jet bundle as follows:
\begin{align*}
c_{\mathrm{tot}} (J_1 (\mathcal{P})^{\vee})^{-1} &= \Bigl( \sum_{j\geq 0} (\mathrm{deg}(L) \eta + \gamma)^j \Bigr) \cdot \Bigl( \sum_{j\geq 0} \bigl( (2g(X)-2+\mathrm{deg}(L)) \eta + \gamma \bigr)^j \Bigr) \\
&= \bigl( 1+26\eta + \gamma + \gamma^2 + \cdots \bigr) \cdot \bigl( 1+68\eta + \gamma + \gamma^2 + \cdots \bigr) = 1 + 94\eta + 2\gamma - 6\eta\theta,
\end{align*}
leading to the desired formula for $[Z]$.

\vskip 4pt

To compute the class of the variety $Y$ defined in (\ref{defx}) we proceed in a similar way.  Recall that
\[
\mu, \nu\colon X \times X\times \mbox{Pic}^{26}(X) \rightarrow X \times \mbox{Pic}^{26}(X)
\]
denote the two projections and $\Delta \subseteq X \times X \times \mbox{Pic}^{26}(X)$ is the diagonal. Set $\Gamma_q:=\{ q \} \times \mbox{Pic}^{26}(X)$.
We introduce the rank $2$ vector bundle $\cB:=\mu_*\bigl( \nu^*(\mathcal{P}) \otimes \OO_{\Delta + \nu^*(\Gamma_q)} \bigr)$ over $X \times W^6_{26}(X)$.  Note that there is a bundle morphism $\chi: \cB^{\vee} \rightarrow (\pi_2)^*(\cM)^{\vee}$ such that $Y=Z_1(\chi)$.  Since we also have that
\begin{equation}\label{eq:classofB}
c_{\mathrm{tot}} ( \cB^{\vee} )^{-1}=\Bigl( 1 + (\mathrm{deg}(L) \eta + \gamma) + (\mathrm{deg}(L) \eta + \gamma)^2 + \cdots \Bigr) \cdot \bigl( 1 - \eta \bigr) = 1 + 25\eta + \gamma - 2\eta\theta,
\end{equation}
we immediately obtain the stated expression for $[Y]$.
\end{proof}

The following formulas are applications of Grothendieck-Riemann-Roch.

\begin{proposition}
\label{a121}
Let $X$ be a general curve of genus $22$, let $q \in X$ be a fixed point, and consider the vector bundles $\cA_2$ and $\cB_2$ on $X \times {\rm{Pic}}^{26}(X)$ having fibers
\[
\cA_2(y,L) = H^0 \bigl( X, L^{\otimes 2} (-2y) \bigr) \ \mbox{ and } \ \cB_2(y,L) = H^0 \bigl( X, L^{\otimes 2} (-y-q) \bigr),
\]
respectively.  One then has the following formulas:
\begin{align*}
c_1(\cA_2) &= -4 \theta - 4\gamma - 146 \eta, & c_1(\cB_2) &= -4 \theta - 2\gamma -51 \eta, \\
c_2(\cA_2) &= 8 \theta^2 + 560 \eta \theta + 16 \gamma \theta, & c_2(\cB_2) &= 8 \theta^2 + 196 \eta \theta + 8 \theta \gamma, \\
c_3(\cA_2) &= -\frac{32}{3} \theta^3 - 1072 \eta \theta^2 - 32 \theta^2 \gamma, & c_3(\cB_2) &= -\frac{32}{3} \theta^3 - 376 \eta \theta^2 - 16 \theta^2 \gamma.
\end{align*}
\end{proposition}

\begin{proof}
This is an immediate application of Grothendieck-Riemann-Roch with respect to the projection map $\nu\colon X\times X\times \mbox{Pic}^{26}(X)\rightarrow X\times \mbox{Pic}^{26}(X)$.  Since $H^1 (X, L^{\otimes 2}(-2y)) = 0$ for every $(y,L)\in X \times \mbox{Pic}^{26}(X)$, the vector bundle $\cA_2$ is realized as a push forward under the map $\nu$:
\[
\cA_2 = \nu_{!} \Bigl( \mu^*\bigl(\P^{\otimes 2}\bigr) \otimes \OO_{X \times X \times \mathrm{Pic}^{26}(X)}(-2\Delta)\Bigr) = \nu_* \Bigl( \mu^*\bigl(\P^{\otimes 2}\bigr) \otimes \OO_{X \times X \times \mathrm{Pic}^{26}(X)}(-2\Delta)\Bigr),
\]
and we apply Grothendieck-Riemann-Roch to $\nu$.  One finds  $\mbox{ch}_2(\cA_2) = 8 \eta \theta$ and $\mbox{ch}_n (\cA_2)=0$ for $n\geq 3$.  Furthermore, $\nu_* (c_1(\P)^2) = -2\theta$.  One then obtains  $c_1(\cA_2) = -4\theta - 4\gamma - (4d+2g-2) \eta$, which then yields the formula for $c_2(\cA_2)$.  Since $\mbox{ch}_3(\cA_2) = 0$, we find that $c_3(\cA_2) = c_1(\cA_2) c_2(\cA_2) - \frac{c_1^3(\cA_2)}{3}$, which by substitution leads to the claimed expression.

The calculation of $\cB_2$ is similar. We find that $c_1(\cB_2) = -4\theta -2\gamma - (2d-1)\eta$ and $\mbox{ch}_2(\cB_2) = 4\eta \theta$, whereas $\mbox{ch}_n (\cB_2) = 0$ for $n\geq 3$.
\end{proof}


\subsection{The slope computation}

In this section we complete the calculation of the virtual class $[\widetilde{\mathfrak{D}}_{23}]^{\mathrm{virt}}$.  We shall use repeatedly that if $\mathcal{V}$ is a vector bundle of rank $r+1$ on a stack, the Chern classes of its second symmetric product can be computed as follows:

\begin{enumerate}
\item $c_1(\mathrm{Sym}^2 (\mathcal{V})) = (r+2) c_1(\mathcal{V})$,
\item $c_2(\mathrm{Sym}^2 (\mathcal{V})) = \frac{r(r+3)}{2} c_1^2(\mathcal{V}) + (r+3)c_2(\mathcal{V})$,
\item $c_3(\mathrm{Sym}^2 (\mathcal{V})) = \frac{r(r+4)(r-1)}{6} c_1^3(\mathcal{V}) + (r+5)c_3(\mathcal{V}) + (r^2+4r-1)c_1(\mathcal{V}) c_2(\mathcal{V})$.
\end{enumerate}


We expand the virtual class
\[
[\widetilde{\mathfrak{D}}_{23}]^{\mathrm{virt}} = \sigma_* \Bigl( c_3(\F - \mbox{Sym}^2(\E)) \Bigr) = a\lambda - b_0\delta_0 - b_1\delta_1 \in CH^1(\pm_{23}).
\]
Our task is to determine the coefficients $a, b_0$ and $b_1$.  We begin with the coefficient of $\delta_1$.

\begin{theorem}
\label{d1}
Let $X$ be a general curve of genus $22$  and denote by $F_1 \subset \widetilde{\Delta}_1\subset \pm_{23}$ the associated test curve.  Then the coefficient of $\delta_1$ in the expansion of $[\widetilde{\mathfrak{D}}_{23}]^{\mathrm{virt}}$ is equal to
\[
b_1 = \frac{1}{2g(X)-2} \sigma^*(F_1)\cdot c_3 \bigr( \F- \mathrm{Sym}^2(\E) \bigr) = 13502337992 = \frac{4}{9}\binom{19}{8} 401951.
\]
\end{theorem}

\begin{proof} We intersect the degeneracy locus of the map $\phi\colon \mbox{Sym}^2(\E)\rightarrow \F$ with the $3$-fold $\sigma^*(F_1)$.
By Proposition \ref{limitlin1}, we have
\[
\sigma^*(F_1) \cdot c_3 \bigl( \F-\mbox{Sym}^2(\E) \bigr) = c_3 \bigl( \F - \mbox{Sym}^2(\E) \bigr)_{| Z} = c_3(\F_{|Z}) - c_3(\mbox{Sym}^2 \E_{| Z}) - c_1(\F_{| Z})c_2(\mbox{Sym}^2 \E_{| Z})
\]
\[
+ 2c_1(\mbox{Sym}^2\E_{| Z}) c_2(\mbox{Sym}^2\E_{| Z}) - c_1(\mbox{Sym}^2 \E_{| Z}) c_2(\F_{| Z}) + c_1^2(\mbox{Sym}^2 \E_{| Z}) c_1(\F_{| Z}) - c_1^3(\mbox{Sym}^2 \E_{| Z}).
\]
We now evaluate the terms that appear  in the righthand side of this expression.

\vskip 4pt

In the course of proving Proposition~\ref{xy}, we constructed a morphism  $\zeta\colon J_1(\P)^{\vee} \rightarrow \pi_2^*(\cM)^{\vee}$ of vector bundles on $Y$ globalizing the maps $H^0(\OO_{2y})^{\vee}\rightarrow H^0(X,L)^{\vee}$.  The kernel sheaf $\mbox{Ker}(\zeta)$ is locally free of rank $1$.  If $U$ is the line bundle on $Z$ with fiber
\[
U(y,L) = \frac{H^0(X, L)}{H^0(X, L(-2y))} \hookrightarrow H^0(X, L \otimes \OO_{2y})
\]
over a point $(y, L)\in Z$, then one has the following exact sequence over $Z$:
\[
0 \longrightarrow U \longrightarrow J_1(\P) \longrightarrow \bigl( \mbox{Ker}(\zeta) \bigr)^{\vee} \longrightarrow 0.
\]
In particular, since in the course of proving  Proposition~\ref{xy} we computed $c_1(J_1(\P))=94\eta+2\gamma$, we find that
\begin{equation}\label{Uosztaly}
c_1(U) = 2\gamma + 94\eta + c_1(\mathrm{Ker}(\zeta)).
\end{equation}  The products of the Chern class of $\mbox{Ker}(\zeta)$ with other classes coming from $X \times W^6_{26}(X)$ can be computed from the formula in \cite{HarrisTu84}:
\begin{equation}
\label{harristu1}
c_1 \bigl( \mathrm{Ker}(\zeta) \bigr) \cdot \xi_{| Z} = -c_7 \Bigl( \pi_2^*(\cM)^{\vee} - J_1(\P)^{\vee} \Bigr) \cdot \xi_{| X\times W^6_{26}(X)}
\end{equation}
\[
= -\Bigl( \pi_2^*(c_7) - 6\eta \theta \pi_2^*(c_5) + (94\eta +2 \gamma)\pi_2^*(c_6) \Bigr) \cdot \xi_{| X\times W^6_{26}(X)},
\]
where $\xi \in H^4( X \times W^6_{26}(X), \mathbb{Z})$.

If $\cA_3$ denotes the rank $31$ vector bundle on $Z$ having fibers
\[
\cA_3(y,L) = H^0(X, L^{\otimes 2})
\]
constructed as a push forward of a line bundle on $X \times X \times \mbox{Pic}^{26}(X)$, then $U^{\otimes 2}$ can be embedded in $\cA_3/\cA_2$.  We consider the quotient
\[
\G := \frac{\cA_3/\cA_2}{U^{\otimes 2}}.
\]
The morphism $U^{\otimes 2} \rightarrow \cA_3/\cA_2$ vanishes along the locus of pairs $(y,L)$ where $L$ has a base point.  This implies that $\G$ has torsion along the locus $\Gamma \subseteq Z$ consisting of pairs $(q,A(q))$, where $A \in W^6_{25}(X)$.  Furthermore, $\F_{|Z}$ is identified as a subsheaf of $\cA_3$ with the kernel of the map $\cA_3 \rightarrow \G$. Summarizing, there is an exact sequence of vector bundles on $Z$
\begin{equation}
\label{exseqZ}
0 \longrightarrow \cA_{2 |Z}\longrightarrow \F_{| Z} \longrightarrow U^{\otimes 2} \longrightarrow 0.
\end{equation}
Over a general point $(y,L)\in Z$, this sequence reflects the decomposition
\[
\F(y,L) = H^0(X, L^{\otimes 2}(-2y)) \oplus K \cdot u^2,
\]
where $u \in H^0(X, L)$ is a section such that $\mbox{ord}_y (u) = 1$.

Hence using the exact sequence (\ref{exseqZ}), one computes:
\begin{align*}
c_1(\F_{| Z}) &= c_1(\cA_{2 |Z})+2c_1(U), & c_2(\F_{| Z}) &= c_2(\cA_{2 |Z})+2c_1(\cA_{2 | Z}) c_1(U) \mbox{ and } \\
c_3(\F_{|Z}) &= c_3(\cA_2) + 2c_2(\cA_{2| Z}) c_1(U).
\end{align*}
Recalling that  $\E_{| Z} = \pi_2^*(\cM)_{|Z}$, we obtain that:
\begin{align*}
\sigma^*(F_1) &\cdot c_3 \bigl( \F - \mbox{Sym}^2 \E \bigr) = c_3(\cA_{2 |Z}) + c_2(\cA_{2 |Z}) c_1(U^{\otimes 2}) - c_3(\mathrm{Sym}^2 \pi_2^*\cM_{| Z}) \\
&- \Bigl( \frac{r(r+3)}{2} c_1(\pi_2^*\cM_{|Z}) + (r+3)c_2(\pi_2^*\cM_{|Z}) \Bigr) \cdot \Bigl( c_1(\cA_{2 |Z})+c_1(U^{\otimes 2}) - 2(r+2)c_1(\pi_2^*\cM_{|Z})\Bigr) \\
&- (r+2)c_1(\pi_2^*\cM_{|Z}) c_2(\cA_{2 |Z}) - (r+2)c_1(\pi_2^*\cM_{|Z}) c_1(\cA_{2 | Z}) c_1(U^{\otimes 2}) \\
&+ (r+2)^2 c_1^2(\pi_2^*\cM_{| Z}) c_1(\cA_{2 | Z}) + (r+2)^2 c_1^2(\pi_2^*\cM_{|Z}) c_1(U^{\otimes 2}) - (r+2)^3 c_1^3(\pi_2^*\cM_{| Z}).
\end{align*}
Here, $c_i(\pi_2^*\cM_{|Z}^{\vee}) = \pi_2^*(c_i)\in H^{2i}(Z, \mathbb{Z})$ and $r = \mbox{rk}(\cM)-1 = 6$.  The Chern classes of $\cA_{2|Z}$ are obtained by applying Proposition \ref{a121}.  Recall that in (\ref{Uosztaly}) we expressed $c_1(U)$ in terms of $c_1((\mbox{Ker}(\zeta))$ and the classes $\eta $ and $\gamma$. Substituting (\ref{Uosztaly}) for $c_1(U)$, when expanding
$\sigma^*(F_1)\cdot c_3 \bigl( \F-\mbox{Sym}^2( \E) \bigr)$, one distinguishes between terms that do and those that do not contain
the first Chern class of $\mathrm{Ker}(\zeta)$. The coefficient of $c_1 \bigl( \mathrm{Ker}(\zeta) \bigr)$ in  $\sigma^*(F_1)\cdot c_3 \bigl( \F-\mbox{Sym}^2( \E) \bigr)$
is evaluated using (\ref{harristu1}).   First we consider the part of this product that \emph{does not} contain $c_1 \bigl( \mathrm{Ker}(\zeta) \bigr)$, and we obtain
\begin{align*}
36 \pi_2^*(c_2) \theta - 148 \pi_2^*(c_1^2) \theta + 1554 \eta \pi_2^*(c_1^2) - 85 \pi_2^*(c_1c_2) - \frac{32}{3} \theta^3 + 304 \eta\theta^2 - 1280 \eta \theta \pi_2^*(c_1) \\
+ 130 \pi_2^*(c_1^3) - 378 \eta \pi_2^*(c_2) + 64 \theta^2\pi_2^*(c_1) + 11 \pi_2^*(c_3) \in H^6 \bigl( X \times W^6_{26}(X), \mathbb{Z} \bigr).
\end{align*}

This polynomial of degree $3$ gets multiplied by the class $[Z]$, expressed as the degree $6$ polynomial in $\theta$, $\eta$, and $\pi_2^*(c_i)$ obtained in Proposition~\ref{xy}.  Adding to it the contribution coming from $c_1 \bigl( \mbox{Ker}(\zeta) \bigr)$, one obtains a homogeneous polynomial of degree $9$ in $\eta$, $\theta$, and $\pi_2^*(c_i)$ for $i = 1, \ldots,7$.  The only nonzero monomials are those containing $\eta$.  After retaining only these monomials and dividing by $\eta$, the resulting degree $8$ polynomial in $\theta$, $c_i \in H^*(W^6_{26}(X), \mathbb{Z})$ can be brought to a manageable form using Proposition \ref{chernosztalyok}.  After lengthy but straightforward manipulations carried out using \emph{Maple}, one finds

\begin{align*}
\sigma^*(F_1) \cdot c_3 \bigl( \mathrm{Sym}^2(\E) - \F \bigr) &= \eta \pi_2^*\Bigl( -780 c_1^3 c_4 \theta + 12220 c_1^3 c_5 + 888 c_1^2 c_4 \theta^2 - 13468 c_1^2 c_5\theta - 5402 c_1^2 c_6 \\
&- 384 \theta^3 c_1 c_4 + 5632 \theta^2 c_1 c_5 + 510 \theta c_1 c_2 c_4 + 4480 c_1 c_6 \theta - 7990 c_1 c_2 c_5 \\
&+ 2336 c_1 c_7 - 216 c_2 c_4 \theta^2 + 3276 c_2 c_5 \theta - 66 c_3 c_4 \theta + 1034 c_3 c_5 + 1314 c_2 c_6 \\
&+ 64 c_4 \theta^4 - \frac{2720}{3} c_5 \theta^3 - 1072 c_6 \theta^2 - 1120 c_7 \theta \Bigr).
\end{align*}

We suppress $\eta$ and the remaining polynomial lives inside $H^{16}(W^6_{26}(X), \mathbb{Z})$.  Using (\ref{harristu}), we explicitly calculate all top Chern numbers on $W^6_{26}(X)$ and we eventually find that
\[
b_1 = \frac{1}{42} \sigma^*(F_1) \cdot c_3 \bigl( \F - \mbox{Sym}^2(\E) \bigr) = 13502337992,
\]
as required.
\end{proof}

\begin{theorem}
\label{d0}
Let $[X,q]$ be a general pointed curve of genus $22$ and let $F_0 \subset \widetilde{\Delta}_0\subset \pm_{23}$ be the associated test curve.  Then
\[
\sigma^*(F_0) \cdot c_3 \bigl( \F - \mathrm{Sym}^2(\E) \bigr) = 44b_0 - b_1 = 93988702808.
\]
It follows that $b_0 = \frac{4}{9}\binom{19}{8} 72725$.
\end{theorem}

\begin{proof}
By Proposition~\ref{limitlin0}, the vector bundles $\E_{| \sigma^*(F_0)}$ and $\F_{| \sigma^*(F_0)}$ are both pull backs of vector bundles on $\widetilde{Y} = Y / [\Gamma_1 \sim \Gamma_2 ]$. By abuse of notation we denote these vector bundles by the same symbols, that is, we have $\E_{| \sigma^*(F_0)} = f^*(\E_{| \widetilde{Y}})$ and $\F_{| \sigma^*(F_0)} = f^*(\F_{| \widetilde{Y}})$.  Following broadly the proof of Theorem~\ref{d1}, we evaluate the terms appearing in $\sigma^*(F_0) \cdot c_3(\F - \mbox{Sym}^2(\E))=c_3\bigl(\F_{|Y}-\mbox{Sym}^2(\E_{|Y})\bigr)$, where $\E_{|Y}=\vartheta^*(\E_{| \widetilde{Y}})$ and $\F_{|Y}=\vartheta^*(\F_{| \widetilde{Y}})$ respectively.

\vskip 4pt

Let $V$ be the line bundle on $Y$ with fiber
\[
V(y,L) = \frac{H^0(X, L)}{H^0(X, L(-y-q))} \hookrightarrow H^0(X, L \otimes \OO_{y+q})
\]
over a point $(y,L) \in Y$. There is an exact sequence of vector bundles over $Y$
\[
0 \longrightarrow V \longrightarrow \cB \longrightarrow \bigl( \mbox{Ker}(\chi) \bigr)^{\vee} \longrightarrow 0,
\]
where $\chi\colon \cB^{\vee} \rightarrow \pi_2^*(\cM)^{\vee}$ is the bundle morphism defined in the second part of the proof of Proposition~\ref{xy}.  In particular, $c_1(V) = 25\eta+ \gamma +c_1 \bigl( \mathrm{Ker}(\chi) \bigr)$, for the Chern class of $\cB$ has been computed in the proof of Proposition~\ref{xy}. By using again \cite{HarrisTu84}, we find the following formulas for the Chern numbers of $\mathrm{Ker}(\chi)$:
\[
c_1 (\mathrm{Ker}(\chi)) \cdot \xi_{|Y} = -c_7 \bigl( \pi_2^*(\cM)^{\vee} - \cB^{\vee} \bigr) \cdot \xi = -\Bigl( \pi_2^*(c_7) + \pi_2^*(c_6)(25\eta + \gamma) - 2\pi_2^*(c_5)\eta \theta \Bigr) \cdot \xi_{|X\times W^6_{26}(X)},
\]
for any class $\xi \in H^4(X\times W^6_{26}(X), \mathbb{Z})$.  We have previously defined the vector bundle $\cB_2$ over 
$X \times W^6_{26}(X)$ with fiber $\cB_2(y, L)=H^0(X, L^{\otimes 2}(-y-q))$.  We show that there is an exact sequence of bundles over $Y$
\begin{equation}
\label{exseq}
0 \longrightarrow \cB_{2 |Y} \longrightarrow \F_{| Y} \longrightarrow V^{\otimes 2} \longrightarrow 0.
\end{equation}
If $\cB_3$ is the vector bundle on $Y$ with fibers $\cB_3(y,L) = H^0(X, L^{\otimes 2})$, we have an injective morphism of sheaves $V^{\otimes 2} \hookrightarrow \cB_3/\cB_2$ locally given by
\[
v^{\otimes 2} \mapsto v^2 \mbox{ mod } H^0(X, L^{\otimes 2}(-y-q)),
\]
where $v \in H^0(X, L)$ is any section not vanishing at $q$ and $y$.  Then $\F_{|Y}$ is canonically identified with the kernel of the projection morphism
\[
\cB_3 \rightarrow \frac{\cB_3/\cB_2} {V^{\otimes 2}}
\]
and the exact sequence (\ref{exseq}) now becomes clear. Therefore
\begin{align*}
c_1(\F_{| Y}) &= c_1(\cB_{2 |Y}) + 2c_1(V), &  c_2(\F_{|Y}) &= c_2(\cB_{2 |Y}) + 2c_1(\cB_{2| Y}) c_1(V) \mbox{ and} \\
c_3(\F_{|Y}) &= c_3(\cB_{2| Y})+2c_2(\cB_{2 |Y}) c_1(V).
\end{align*}

The part of the intersection number $\sigma^*(F_0) \cdot c_3 (\F - \mbox{Sym}^2(\E))$ \emph{not} containing $c_1 \bigl( \mbox{Ker}(\chi) \bigr)$ equals
\begin{align*}
36 \pi_2^*(c_2) \theta - 148 \pi_2^*(c_1^2) \theta - 37 \eta \pi_2^*(c_1^2) - 85 \pi_2^*(c_1 c_2) - \frac{32}{3} \theta^3 -8 \eta \theta^2 +32 \eta \theta \pi_2^*(c_1) \\
+130 \pi_2^*(c_1^3) + 9 \eta \pi_2^*(c_2) + 64 \theta^2 \pi_2^*(c_1) + 11 \pi_2^*(c_3) \in H^6 \bigl( X \times W^6_{26}(X), \mathbb{Z} \bigr).
\end{align*}

We multiply this expression by the class $[Y]$ computed in Proposition~\ref{xy}.  The coefficient of $c_1 \bigl( \mbox{Ker}(\chi) \bigr)$ in $\sigma^*(F_0) \cdot c_3 \bigl( \F - \mathrm{Sym}^2(\E) \bigr)$ equals
\[
-2c_2(\cB_{2| Y}) - 2(r+2)^2 \pi_2^*(c_1^2) - 2(r+2) c_1(\cB_{2 |Y}) \pi_2^*(c_1) +r (r+3) \pi_2^*(c_1^2) + 2(r+3)\pi_2^*(c_2),
\]
where recall that $r=\mbox{rk}(\cM) -1 = 6$.  All in all, we find
\begin{align*}
44 b_0 - b_1 &= \sigma^*(F_0) \cdot c_3(\F - \mathrm{Sym}^2\E) = \eta \pi_2^*\Bigl( -260 c_1^3 c_4 \theta + 3250 c_1^3 c_5 + 296 c_1^2 c_4 \theta^2 - 3552 c_1^2 c_5 \theta \\
&- 1887 c_1^2 c_6 - 128 \theta^3 c_1 c_4 + 1472 \theta^2 c_1 c_5 + 170 \theta c_1 c_2 c_4 + 1568 c_1 c_6 \theta - 2125 c_1 c_2 c_5 \\
&+ 816 c_1 c_7 - 72 c_2 c_4 \theta^2 + 864 c_2 c_5 \theta - 22 c_3 c_4 \theta + 275 c_3 c_5 + 459 c_2 c_6 + \frac{64}{3} c_4 \theta^4 \\
&- \frac{704}{3} c_5 \theta^3 - 376 c_6 \theta^2 - 392 c_7 \theta \Bigr) \in H^{18} \bigl( X \times W^6_{26}(X), \mathbb{Z} \bigr).
\end{align*}
We evaluate each term in this expression by first deleting $\eta$ and then using (\ref{harristu}).
\end{proof}

The following result follows from the definition of the vector bundles $\E$ and $\F$ given in Proposition \ref{vectorbundles}.  It will provide the third relation between the coefficients of $[\widetilde{\mathfrak{D}}_{23}]^{\mathrm{virt}}$, and thus complete the calculation of its slope.

\begin{theorem}
\label{elltail}
Let $[X,q]$ be a general $1$-pointed curve of genus $22$ and $F_{\mathrm{ell}} \subset \pm_{23}$ be the pencil obtained by attaching at the fixed point $q \in X$ a pencil of plane cubics at one of the base points of the pencil.  Then one has the relation
\[
a - 12b_0 + b_1 = F_{\mathrm{ell}} \cdot \sigma_* c_3 \bigl( \F - \mathrm{Sym}^2(\E) \bigr) = 0.
\]
\end{theorem}
\begin{proof} Since the genus $g-1$ aspect of each curve in $F_{\mathrm{ell}}$ does not vary, it follows from Corollary~\ref{cor:vectorbundles} that the vector bundles $\mathcal{E}_{| \sigma^*(F_{\mathrm{ell}})}$ and $\mathcal{F}_{|\sigma^*(F_{\mathrm{ell}})}$ are both trivial, therefore $c_i\bigl(\mathcal{E}_{| \sigma^*(F_{\mathrm{ell}})}\bigr)=0$ and $c_i\bigl(\mathcal{F}_{| \sigma^*(F_{\mathrm{ell}})}\bigr)=0$ for $i\geq 1$, from which the conclusion follows.
\end{proof}

\begin{proof}[Proof of Theorem~\ref{thm:slopes} for ${[\widetilde{\mathfrak{D}}_{23}]}^{\mathrm{virt}}$]
By Theorems~\ref{d1} and~\ref{d0}, we have
\[
b_0 = \frac{4}{9} \binom{19}{8} 72725 \mbox{ and } b_1 = \frac{4}{9} \binom{19}{8} 401951.
\]
Combined with Theorem~\ref{elltail}, we obtain
\[
a = \frac{4}{9} \binom{19}{8} 470749,
\]
and the result follows.
\end{proof}

\section{The class of the virtual divisor on \texorpdfstring{$\widetilde{\mathcal{M}}_{2s^2+s+1}$}{M2s2+s+1}}\label{sect:rho1}

In this section we prove Theorem~\ref{rho1virtual}.  In particular, we determine the class $[\widetilde{\mathfrak{D}}_{22}]^{\mathrm{virt}}$ that will ultimately be used in the proof that $\mm_{22}$ is of general type.

\subsection{Top intersection products in the Jacobian of a curve of genus $2s^2+s$.}

We next turn our attention to the top intersection products on $W^{2s}_{2s^2+2s+1}(X)$, when $X$ is a general curve of genus  $2s^2+s$, for $s\geq 2$.
We apply (\ref{harristu}) systematically.  Our computations are analogous to those in \S\ref{Sec:23}, and in many cases we omit the details.  Observe that $\rho(2s^2+s, 2s,2s^2+2s) = 0$, so $W^{2s}_{2s^2+2s}(X)$ is reduced and 0-dimensional.  We denote its cardinality by
\begin{equation}
\label{nrlin0}
C_{2s+1} := \frac{(2s^2+s)!\ (2s)! \ (2s-1)! \cdots  2! \ 1!}{(3s)!\ (3s-1)! \cdots (s+1)! \ s!}\ = \# \Bigl( W^{2s}_{2s^2+2s}(X) \Bigr).
\end{equation}
Moreover $\rho(2s^2+s, 2s+1, 2s^2+2s+1) = -s < 0$, hence it follows that $W^{2s+1}_{2s^2+2s+1}(X) = \emptyset$ and we can consider the tautological rank $2s+1$ vector bundle $\cM$ over $W^{2s}_{2s^2+2s+1}(X)$.  We write
$\sum_{i=0}^{2s+1} c_i(\mathcal{M}^{\vee})=(1+x_1)\cdots (1+x_{2s+1})$. We collect the following formulas obtained by applying the Harris-Tu formula (\ref{harristu2}):

\begin{proposition}
\label{intersectionsit2}
Let $X$ be as above and set $c_i:=c_i ( \cM^{\vee}) \in H^{2i} \bigl( W^{2s}_{2s^2+2s+1}(X), \mathbb{Z} \bigr)$ to be the Chern classes of the dual of the tautological bundle on $W^{2s}_{2s^2+2s+1}(X)$.  The following hold:
\begin{align*}
c_{2s+1} &= x_1 x_2\cdots x_{2s+1}=C_{2s+1}, \\
c_{2s}\cdot c_1 &= x_1 x_2\cdots x_{2s+1}+x_1^2 x_2\cdots x_{2s}, \\
c_{2s-1}\cdot c_2 &= x_1x_2\cdots x_{2s+1}+x_1^2x_2\cdots x_{2s}+x_1^2 x_2^2x_3\cdots x_{2s-1}, \\
c_{2s-1}\cdot c_1^2 &= x_1x_2 \cdots x_{2s+1}+2x_1^2 x_2 \cdots x_{2s}+x_1^2 x_2^2 x_3\cdots x_{2s-1}+
x_1^3 x_2 x_3 \cdots x_{2s-1}, \\
c_{2s}\cdot \theta &= x_1x_2\cdots x_{2s}\cdot \theta=(2s+1)s \ C_{2s+1}, \\
c_{2s-1}\cdot c_1 \cdot \theta &= x_1x_2 \cdots x_{2s}\cdot \theta+x_1^2 x_2 \cdots x_{2s-1}\cdot \theta, \\
c_{2s-2}\cdot c_2\cdot \theta &= x_1 x_2 \cdots x_{2s}\cdot \theta + x_1^2 x_2 \cdots x_{2s-1}\cdot \theta+x_1^2 x_2^2 x_3 \cdots x_{2s-2}\cdot \theta, \\
c_{2s-2}\cdot c_1^2\cdot \theta &= x_1 x_2 \cdots x_{2s}\cdot \theta+2x_1^2x_2\ldots x_{2s-1}\cdot \theta +
x_1^2 x_2^2 x_3 \cdots x_{2s-2}\cdot \theta+x_1^3 x_2x_3 \cdots x_{2s-2}\cdot \theta, \\
c_{2s-1}\cdot \theta^2 &= x_1 x_2 \cdots x_{2s-1}\cdot \theta^2, \\
c_{2s-2}\cdot c_1\cdot \theta^2 &= x_1 x_2 \cdots x_{2s-1} \cdot \theta^2 +x_1^2 x_2 \cdots x_{2s-2}\cdot \theta^2.
\end{align*}
\end{proposition}

\begin{proof} This amounts to a repeated application of (\ref{harristu2}) and evaluating the corresponding determinants. The right hand side of each formula retains the non-zero terms that appear in the corresponding Chern number. To give an example, we evaluate $c_{2s}\cdot c_1$. Using (\ref{harristu2}), each monomial $x_1^{i_1} x_2^{i_2} \cdots x_{2s+1}^{i_{2s+1}}$ will vanish as long as there exists a pair $k<j$ such that $i_j-i_{k}=j-k$. We compute
$$c_{2s}\cdot c_1=(2s+1)x_1\cdots x_{2s+1}+x_2^2x_3\cdots x_{2s+1}+x_1x_3^2x_4\cdots x_{2s+1}+\cdots+x_1\cdots x_{2s-1}x_{2s+1}^2+x_1^2x_2\cdots x_{2s}.$$
All other terms vanish. Then by (\ref{harristu2}), evaluating each determinant we observe that
$$x_2^2x_3\cdots x_{2s+1}=\cdots=x_1x_2\cdots x_{2s-1} x_{2s+1}^2=x_1x_2\cdots x_{2s-2}x_{2s}^2 x_{2s+1}=
-x_1\cdots x_{2s+1},$$ which leads to the claimed formula for $c_{2s}\cdot c_1$.
The case of the remaining Chern numbers is analogous.
\end{proof}

Using Proposition \ref{intersectionsit2}, any top intersection product on the smooth $(2s+1)$-dimensional variety $W^{2s}_{2s^2+2s+1}(X)$ reduces to a sum of monomials in the variables $x_i$ and $\theta$.  Next we record the values of these monomials. All terms are essentially reduced to expressions involving $x_1 \cdots x_{2s+1}=C_{2s+1}$.

\begin{proposition}
\label{intersection2sit2}
Keep the notation from above.  The following hold in $H^{4s+2} \bigl( W_{2s^2+2s+1}^{2s}(X), \mathbb{Z} \bigr)$:
\begin{align*}
x_1 x_2\cdots x_{2s+1} &= C_{2s+1}, \\
x_1^2 x_2^2x_3\cdots x_{2s-1} &= \frac{(s-1)(s+1)^2(2s+1)^2}{3(3s+1)} C_{2s+1}, \\
x_1^2 x_2\cdots x_{2s} &= \frac{4s^2(s+1)}{3s+1} C_{2s+1}, \\
x_1^3 x_2 x_3 \cdots x_{2s-1} &= \frac{s^2(s+1)^2(2s-1)(2s+3)}{(3s+1)(3s+2)} C_{2s+1}, \\
x_1x_2\cdots x_{2s}\cdot \theta &= (2s+1)sC_{2s+1}, \\
x_1^2 x_2 \cdots x_{2s-1}\cdot \theta &= \frac{(s+1)^2(2s-1)}{3s+1}x_1x_2\cdots x_{2s}\cdot \theta, \\
x_1^2 x_2^2 x_3 \cdots x_{2s-2} &= \frac{(2s-3)(2s+1)(s+1)^2(s+2)}{9(3s+1)}x_1x_2\cdots x_{2s}\cdot \theta, \\
x_1^3 x_2x_3 \cdots x_{2s-2}\cdot \theta &= \frac{(s-1)(s+1)^2(s+2)(2s-1)(2s+3)}{3(3s+1)(3s+2)}x_1x_2\cdots x_{2s}\cdot \theta, \\
x_1 x_2 \cdots x_{2s-1}\cdot \theta^2 &= s^2(s+1)(2s+1)C_{2s+1}, \\
x_1^2 x_2 \cdots x_{2s-2}\cdot \theta^2 &= \frac{4(s+1)(s-1)(s+2)}{3(3s+1)}x_1 x_2 \cdots x_{2s-1}\cdot \theta^2 \\
x_1 x_2 \cdots x_{2s-2}\cdot \theta^3 &= \frac{(2s+1)(2s-1)(s+2)(s+1)s^2}{3}C_{2s+1}.
\end{align*}
\end{proposition}

We record the formulas for the classes of $Z$ and $Y$, the proofs being analogous to those of Proposition \ref{xy}.

\begin{proposition}
\label{xyrh1}
Let $[X,q] \in \cM_{2s^2+s, 1}$ be a general pointed curve. If $c_i := c_i(\cM^{\vee})$ are the Chern classes of the tautological vector bundle over $W^{2s}_{2s^2+2s+1}(X)$, then one has:
\begin{enumerate}
\item  $[Z] = \pi_2^* (c_{2s}) - 6\pi_2^* (c_{2s-2}) \eta \theta + 2\bigl( s(4s+3) \eta + \gamma \bigr) \pi_2^* (c_{2s-1}) \in H^{4s} \bigl( X \times W^{2s}_{2s^2+2s+1}(X), \mathbb{Z}\bigr)$.

\item  $[Y] = \pi_2^* (c_{2s}) - 2\pi_2^* (c_{2s-2}) \eta \theta + \bigl( 2s(s+1) \eta + \gamma \bigr) \pi_2^* (c_{2s-1}) \in H^{4s} \bigl( X \times W^{2s}_{2s^2+2s+1}(X), \mathbb{Z}\bigr)$.
\end{enumerate}
\end{proposition}

\begin{remark}
For future reference we also record the following formulas, where we recall that $J_1(\mathcal{P})$ denotes the jet bundle of the Poincar\'e bundle over $X\times \mbox{Pic}^{2s^2+2s+1}(X)$.
\begin{align}
\label{c3d1} c_{2s+1} \bigl( \pi_2^*(\cM)^{\vee} - J_1(\P)^{\vee} \bigr) &= \pi_2^*(c_{2s+1}) - 6\pi_2^*(c_{2s-1}) \eta \theta + 2\bigl(s(4s+3) \eta + \gamma \bigr) \pi_2^*(c_{2s}), \\
\label{c3d0}  c_{2s+1} \bigl( \pi_2^*(\cM)^{\vee} - \mathcal{B}^{\vee} \bigr) &= \pi_2^*(c_{2s+1}) - 2\pi_2^*(c_{2s-1}) \eta \theta + \bigl( 2s(s+1) \eta + \gamma \bigr) \pi_2^*(c_{2s}).
\end{align}
\end{remark}

The following formulas follow in an analogous way to Proposition~\ref{a121}.

\begin{proposition}
\label{a121Part2}
Let $X$ be a general curve of genus $2s^2+s$, let $q \in X$ be a fixed point, and consider the vector bundles $\cA_2$ and $\cB_2$ on $X \times {\rm{Pic}}^{2s^2+2s+1}(X)$ having fibers
\[
\cA_2(y,L) = H^0 \bigl( X, L^{\otimes 2}(-2y) \bigr) \ \mbox{ and } \ \cB_2(y,L) = H^0\bigl( X, L^{\otimes 2}(-y-q) \bigr),
\]
respectively.  One then has the following formulas:
\begin{align*}
c_1(\cA_2) &= -4 \theta - 4\gamma - 2(3s+1)(2s+1)\eta, & c_1(\cB_2) &= -4\theta - 2\gamma - (2s+1)^2\eta, \\
c_2(\cA_2) &= 8\theta^2 + 8(6s^2+5s-2)\eta \theta + 16 \gamma \theta, & c_2(\cB_2) &= 8\theta^2 + 4(4s^2+4s-1) \eta \theta + 8\theta \gamma.
\end{align*}
\end{proposition}

\subsection{The slope computation}

We are now in a position to  complete the proof of Theorem \ref{rho1virtual}.  Recall that $g = 2s^2+s+1$ and we express the virtual class
\[
[\widetilde{\mathfrak{D}}_g]^{\mathrm{virt}} = \sigma_* \bigl( c_2(\F - \mbox{Sym}^2(\E)) \bigr) = a\lambda - b_0 \delta_0 - b_1 \delta_1 \in CH^1(\widetilde{\mathcal{M}}_g).
\]
The determination of the coefficients $a$, $b_0$, and $b_1$ is similar to the computations for $g = 23$, and we shall highlight the differences. Recall that $C_{2s+1}$ denotes the number of linear series of type $\mathfrak{g}_{2s^2+2s}^{2s}$ on a general curve of genus $2s^2+s$.

\begin{theorem}
\label{d1sit2}
Let $X$ be a general curve of genus $2s^2+s$ and denote by $F_1\subset \widetilde{\Delta}_1\subset \pm_{2s^2+s+1}$ the associated test curve.  Then
the coefficient of $\delta_1$ in the expansion of $[\widetilde{\mathfrak{D}}_g]^{\mathrm{virt}}$ is equal to
\begin{align*}
b_1 &= \frac{1}{2g(X)-2} \sigma^*(F_1) \cdot c_2 \bigl( \F - \mathrm{Sym}^2(\E) \bigr) \\
&= C_{2s+1} \frac{2s(s-1)(2s+1)}{3(2s-1)(3s+1)(3s+2)} \Bigl( 24s^6 - 40s^5 + 18s^4 + 26s^3 + 30s^2 + 47s + 18 \Bigr).
\end{align*}
\end{theorem}

\begin{proof}
Recall that $W^{2s}_{2s^2+2s+1}(X)$ is a smooth variety of dimension $2s+1$.  We work on the product $X \times W^{2s}_{2s^2+2s+1}(X)$ and intersect the degeneracy locus of the map $\phi\colon \mathrm{Sym}^2(\E) \rightarrow \F$ with the surface $\sigma^*(F_1)$, containing $Z$ as an irreducible component.  It follows from Proposition~\ref{limitlin1} that $Z$ is the only component contributing to this intersection product, that is,
\begin{equation}
\label{contsit2}
\sigma^*(F_1) \cdot c_2 \bigl( \F - \mbox{Sym}^2(\E) \bigr) = c_2(\F_{|Z}) - c_2(\mbox{Sym}^2 \E_{| Z}) - c_1(\F_{|Z}) c_1(\mbox{Sym}^2 \E_{| Z}) + c_1^2(\mbox{Sym}^2\E_{|Z}).
\end{equation}

The kernel $\mbox{Ker}(\zeta)$ of the vector bundle morphism $\zeta\colon J_1(\P)^{\vee} \rightarrow \pi_2^*(\cM)^{\vee}$, defined in the proof of Proposition~\ref{xy}, is a line bundle on $Z$.  If $U$ is the line bundle on $Z$ with fiber
\[
U(y,L) = \frac{H^0(X, L)}{H^0(X, L(-2y))} \hookrightarrow H^0(X, L \otimes \OO_{2y})
\]
over a point $(y,L) \in Z$, then one has the following exact sequence over $Z$
\[
0 \longrightarrow U \longrightarrow J_1(\P) \longrightarrow \bigl( \mbox{Ker}(\zeta) \bigr)^{\vee} \longrightarrow 0.
\]
From this sequence, it follows that $c_1(U) = 2\gamma + 2(4s^2+3s)\eta + c_1(\mathrm{Ker}(\zeta))$, where the products of $c_1 \bigl( \mbox{Ker}(\zeta) \bigr)$ with arbitrary classes $\xi$ coming from $X \times W^{2s}_{2s^2+2s+1}(X)$ can be computed using once more the Harris-Tu formula \cite{HarrisTu84}:
\begin{equation}
\label{harristu1sit2}
c_1 \bigl( \mathrm{Ker}(\zeta) \bigr) \cdot \xi_{| Z} = -c_{2s+1} \bigl( \pi_2^*(\cM)^{\vee} - J_1(\P)^{\vee} \bigr) \cdot \xi_{|Z}.
\end{equation}

The Chern classes on the righthand side of (\ref{harristu1sit2}) have been evaluated in the formula (\ref{c3d1}).  Using a local analysis identical to the one in Theorem~\ref{d1}, we conclude that the restriction $\F_{|Z}$ appears in the following exact sequence of vector bundles
\[
0 \longrightarrow \cA_{2 |Z} \longrightarrow \F_{| Z} \longrightarrow U^{\otimes 2} \longrightarrow 0.
\]
We obtain the following intersection product on the surface $Z$:
\begin{align*}
c_2 \bigl( \F_{|Z} - \mathrm{Sym}^2(\E)_{|Z}\bigr) &= c_2 \bigl( \cA_{2|Z}) + 2c_1(\cA_{2|Z}) \cdot c_1 \bigl( J_1(\P) \bigr) - (2s+3)c_2(\pi_2^*\cM^{\vee}) \\
&+ \bigl( (2s+2)^2 - s(2s+3) \bigr) c_1^2 \bigl( \pi_2^*\cM^{\vee})+(2s+2)c_1\bigl( \cA_{2|Z}\bigr) \cdot c_1\bigl( \pi_2^*\cM^{\vee} \bigr) 
\\
&+c_1 \bigl( \mbox{Ker}(\zeta)) \cdot \Bigl( 2c_1(\cA_{2|Z}) 
+ 4(s+1)c_1(\pi_2^*\cM^{\vee})\Bigr)
\\
&+4(s+1)c_1\bigl(J_1(\mathcal{P})\bigr)\cdot c_1\bigl(\pi_2^*\mathcal{M}^{\vee}\bigr).
\end{align*}

This expression gets multiplied with the class $[Z]$ computed in Proposition~\ref{xyrh1}.  The Chern classes of $\cA_{2|Z}$ have been computed in Proposition~\ref{a121Part2}.  We obtain a homogeneous polynomial of degree $2s+2$ on $H^{\mathrm{top}} \bigl( X \times W^{2s}_{2s^2+2s+1}(X), \mathbb{Z} \bigr)$.  We first consider the terms that \emph{do not} involve $c_1 \bigl( \mbox{Ker}(\zeta) \bigr)$. Since $W^{2s}_{2s^2+2s+1}(X)$ is $(2s+1)$-dimensional, each non-zero term in this polynomial has to contain the class $\eta$. We collect these terms and obtain the following contribution:
\begin{align*}
2 \eta \pi_2^* \Bigr( -24\ c_{2s-2} \theta^3 - 8s(s+1)(4s+3)\ c_{2s-1} c_1 \theta - (6s^2+15s+12)\ c_{2s-2} c_1^2 \theta\\ + 8s(4s+3) c_{2s-1} \theta^2
+ 3(2s+3)\ c_{2s-2} c_2 \theta + s(4s+3)(2s^2+5s+4)\ c_{2s-1} c_1^2\\
- s(2s+3)(4s+3)\ c_{2s-1}c_2
+ 24(s+1)\ c_{2s-2} c_1 \theta^2 + 2(2s-1)(s+1)^2\ c_{2s} c_1 \\- (8s^2+4s-8)\ c_{2s} \theta \Bigr) \in H^{\mathrm{top}} \bigl( X \times W^{2s}_{2s^2+2s+1}(X), \mathbb{Z} \bigr).
\end{align*}

Finally, the contribution of the terms containing $c_1 \bigl( \mbox{Ker}(\zeta) \bigr)$ is evaluated using (\ref{c3d1}).  The coefficient of $c_1 \bigl( \mathrm{Ker}(\zeta) \bigr)$ is given by $2c_1(\cA_2)+ 4(s+1)c_1\bigl(\pi_2^*\mathcal{M}^{\vee}\bigr)$.  Substituting this for the class $\xi$ in formula (\ref{harristu1sit2}) and using the formula (\ref{c3d1}), we obtain the following contribution to the product $\sigma^*(F_1) \cdot c_2 \bigl( \F - \mbox{Sym}^2(\E) \bigr)$:
\begin{align*}
4\eta \pi_2^* \Bigl( (3s+1)(2s+1) c_{2s+1} - 12\ c_{2s-1} \theta^2 + 6(s+1)\ c_{2s-1} c_1 \theta \\
+ (16s^2 + 12s - 8) c_{2s} \theta - 2s(s+1)(4s+3) c_{2s} c_1 \Bigr) \in H^{\mathrm{top}} \bigl( X \times W^{2s}_{2s^2+2s+1}(X), \mathbb{Z} \bigr).
\end{align*}

The resulting intersection product is then evaluated with \emph{Maple}.  Dropping the class $\eta$ from the sum of the two displayed contributions, one is led to a sum of top Chern numbers on $W^{2s}_{2s^2+2s+1}(X)$, which can be evaluated individually using Propositions~\ref{intersectionsit2} and~\ref{intersection2sit2}.
\end{proof}

\begin{theorem}
\label{d0sit2}
Let $[X,q]$ be a general pointed curve of genus $2s^2+s$ and let $F_0 \subset \widetilde{\Delta}_0\subset \widetilde{\mathcal{M}}_{g}$ be the associated test curve.  Then the coefficient of $\delta_0$ in the expression of $[\widetilde{\mathfrak{D}}_g]^{\mathrm{virt}}$ is equal to
\begin{align*}
b_0 &= \frac{\sigma^*(F_0) \cdot c_2 \Bigl( \F - \mathrm{Sym}^2(\E) \Bigr) + b_1}{2s(2s+1)} \\
&= C_{2s+1} \frac{2(s-1)(24s^8-28s^7+22s^6-5s^5+43s^4+112s^3+100s^2+50s+12)}{9(2s-1)(3s+1)(3s+2)}.
\end{align*}
\end{theorem}

\begin{proof}
Using Proposition~\ref{limitlin0}, we observe that
\begin{equation}
\label{keplet1}
c_2 \bigl( \F - \mathrm{Sym}^2(\E) \bigr)_{| \sigma^*(F_0)} = c_2 \bigl( \F - \mathrm{Sym}^2(\E) \bigr)_{|Y}.
\end{equation}
To determine the Chern classes of $\F_{|Y}$, we introduce the line bundle $V$ on $Y$ with  fiber
\[
V(y,L) = \frac{H^0(X, L)}{H^0(X, L(-y-q))} \hookrightarrow H^0(X, L \otimes \OO_{y+q})
\]
over a point $(y,L) \in Y$.  There is an exact sequence of vector bundles over $Y$
\[
0 \longrightarrow V \longrightarrow \cB \longrightarrow  \bigl(\mbox{Ker}(\chi) ^{\vee} \bigr)\longrightarrow 0,
\]
where the morphism $\chi\colon \cB^{\vee} \rightarrow \pi_2^*(\cM)^{\vee}$ was defined in the second part of the proof of Proposition~\ref{prop:osztalyxy}. Precisely, $\cB$ is the vector bundle over $X\times W^{2s}_{2s^2+2s+1}(X)$ having fibres $\cB_{|(y, L)}=L_{y+q}$ and using (\ref{eq:classofB}) one has that $c_1(\cB)=(2s^2+2s)\eta+\gamma$.  Recalling the vector bundle $\cB_2$ defined in Proposition~\ref{a121Part2}, a local analysis similar to that in the proof of Theorem~\ref{d0} shows that one has an exact sequence on $Y$
\[
0 \longrightarrow \cB_{2|Y} \longrightarrow \F_{|Y} \longrightarrow V^{\otimes 2} \longrightarrow 0.
\]
This determines $c_i(\mathcal{F}_{|Y})$ in terms of $c_1\bigl(\mbox{Ker}(\chi)\bigr)$. Furthermore, by the Harris-Tu formula we have that $c_1 \bigl( \mbox{Ker}(\chi) \bigr) = -c_{2s+1} \bigl( \pi_2^*(\cM^{\vee}) - \cB^{\vee} \bigr)$, where the righthand side is estimated using (\ref{c3d0}).

We first collect terms that do not contain $c_1 \bigl( \mbox{Ker}(\chi) \bigr)$ in $c_2(\F-\mbox{Sym}^2(\E))$ and multiply the result by $[Y]$, which was computed in Proposition \ref{xyrh1}. This gives  the following contribution:
\begin{align*}
\eta \pi_2^*\Bigr( -16\ c_{2s-2} \theta^3 - 16s(s+1)^2\ c_{2s-1} c_1 \theta - (4s^2 + 10s + 8)\ c_{2s-2} c_1^2 \theta + 16s(s+1) c_{2s-1} \theta^2 \\
+ (4s+6)\ c_{2s-2} c_2 \theta + 2s(s+1)(2s^2+5s+4)\ c_{2s-1} c_1^2 - 2s(s+1)(2s+3)\ c_{2s-1} c_2 \\
+ 16(s+1)\ c_{2s-2} c_1 \theta^2 - 2(s+1)\ c_{2s} c_1 + 4\ c_{2s} \theta \Bigr) \in H^{\mathrm{top}} \bigl( X \times W^{2s}_{2s^2+2s+1}(X),\mathbb{Z} \bigr).
\end{align*}

We collect terms containing $c_1 \bigl( \mbox{Ker}(\chi) \bigr)$, and obtain an expression in $H^{\mathrm{top}} \bigl( X \times W^{2s}_{2s^2+2s+1}(X), \mathbb{Z} \bigr)$ that contributes towards $\sigma^*(F_0) \cdot c_2(\F - \mbox{Sym}^2(\E))$:
\[
\eta \pi_2^* \Bigl( -16\ c_{2s-1} \theta^2 + 8(s+1)\ c_{2s-1} c_1 \theta + 2(2s+1)^2 c_{2s+1} + (16s^2+16s-8)c_{2s} \theta - 8s(s+1)^2 c_{2s} c_1 \Bigr).
\]

We now substitute in (\ref{keplet1}), and as in the proof of Theorem~\ref{d1sit2}, after manipulations we obtain a polynomial of degree $2s+1$ on $W^{2s}_{2s^2+2s+1}(X)$, that we compute by applying (\ref{intersectionsit2}) and (\ref{intersection2sit2}).
\end{proof}

We can now complete the calculation of the slope of $[\widetilde{\mathfrak{D}}_g]^{\mathrm{virt}}$.

\begin{proof}[Proof of Theorem~\ref{rho1virtual}.]
We denote once more by $F_{\mathrm{ell}} \subset \pm_{g}$ the pencil obtained by attaching at the fixed point of a general curve $X$ of genus $2s^2+s$ a pencil of plane cubics at one of the base points of the pencil.  Then one has the relation
\[
a - 12b_0 + b_1 = F_{\mathrm{ell}} \cdot \sigma_* c_2 \bigl( \F - \mathrm{Sym}^2(\E) \bigr) = 0.
\]
We therefore find the following expression for the $\lambda$-coefficient
\[
a = C_{2s+1}\frac{2(s-1)(48s^8-56s^7+92s^6-90s^5+86s^4+324s^3+317s^2+182s+48)}{3(3 s+2)(2s-1)(3s+1)}.
\]
\end{proof}

\noindent In particular, when $s=3$, we obtain the formula for the class of $[\widetilde{\mathfrak{D}}_{22}]^{\mathrm{virt}}$ in Theorem~\ref{thm:slopes}.

\begin{remark}
Substituting $s=2$, we obtain $g=11$, $d=13$ and $r=4$ and the slope of the corresponding divisor $\widetilde{\mathfrak{D}}_{11}$ is equal to $s\bigl([\widetilde{\mathfrak{D}}_{11}]\bigr)=7$. Note that $\widetilde{\mathfrak{D}}_{11}$ is precisely the Koszul divisor considered both in \cite[Theorem 1.3]{FarkasOrtega12} and in \cite[Theorem 2]{BakkerFarkas18}. In particular, its slope has been determined in \cite{BakkerFarkas18} using geometric considerations and Theorem \ref{rho1virtual} matches that calculation.    
\end{remark}

\section{Tropicalizations of linear series}
\label{Sec:Slopes}

We continue to work over an algebraically closed field $K$ of characteristic zero.  For the remainder of the paper, as in \S\ref{sec:tropicalprelim}, we choose the field to be  spherically complete with respect to a surjective valuation $\nu \colon K^\times \to \mathbb{R}$.  Let $R \subseteq K$ be the valuation ring, and $\kappa$ its residue field.  We begin by discussing properties of tropicalizations of not necessarily complete linear series, when the skeleton is an arbitrary tropical curve.

\bigskip

Let $X$ be a curve over $K$ with a skeleton $\Gamma \subseteq X^{\an}$.  Let $D_X$ be a divisor on $X$, with $V \subseteq H^0(X, \mathcal{O} (D_X))$ a linear series of rank $r$.  We consider $$ \trop (V) := \{ \trop(f) \in \PL(\Gamma) \mid f \in V \smallsetminus \{0 \} \}.$$  Let $D = \trop(D_X)$.

\begin{lemma}
For any effective divisor $E$ on $\Gamma$ of degree $r$, there is some $\varphi \in \trop(V)$ such that $\ddiv (\varphi) + D - E$ is effective.
\end{lemma}

\begin{proof}
We follow the standard argument showing that the rank of the tropicalization of a divisor is greater than or equal to its rank on the algebraic curve \cite{Baker08}. Let $E_X$ be an effective divisor of degree $r$ on $X$ that specializes to $E$.  Since $V$ has rank $r$, there is a function $f \in V$ such that $\ddiv (f) + D_X - E_X$ is effective.  Setting $\varphi = \trop (f)$ yields the result.
\end{proof}

\begin{lemma}
Any subset of $\trop(V)$ of size $s > r + 1$ is tropically dependent.
\end{lemma}

\begin{proof}
Let $f_1, \ldots, f_{s} \in V$, with $s > r + 1$. Since $\dim_K V = r+1$, the set $\{ f_1, \ldots, f_{s} \}$ is linearly dependent, and hence $\{ \trop(f_1), \ldots, \trop(f_{s}) \}$ is tropically dependent.
\end{proof}

\begin{lemma}
Any subset of $\trop(V)$ of size $s \leq r$ is contained in the tropicalization of a linear subseries of rank $s-1$.  Moreover, if $S_1$ and $S_2$ are subsets of size $s_1$ and $s_2$, respectively, with $s_i \leq r$ and $s_1 + s_2 \geq r + 2$, then there are tropicalizations of linear subseries $\Sigma_1$ and $\Sigma_2$ of ranks $s_1 - 1$ and $s_2-1$ containing $S_1$ and $S_2$, respectively, such that $\Sigma_1 \cap \Sigma_2$ contains the tropicalization of a linear series of rank $s_1 + s_2 - r - 2$.
\end{lemma}

\begin{proof}
 The $K$-linear span of $\{f_1, \ldots, f_s \} \subset V$ has dimension at most $s$ and hence, if $s \leq r$, it is contained in a subspace $W \subset V$ of dimension exactly $r$.  Then $W$ is a linear series of rank $s-1$ whose tropicalization contains $\{ \trop(f_1), \ldots, \trop(f_s) \}$.

Similarly, if $s_1$ and $s_2$ satisfy the specified inequalities, then any subsets  $\{f_1, \ldots, f_{s_1} \}$ and $\{g_1, \ldots, g_{s_2} \}$ are contained in linear subseries of rank $s_1 - 1$ and $s_2 - 1$ respectively, and the intersection of these is a linear subseries of rank at least $s_1 + s_2 - r - 2$.
\end{proof}

\begin{proposition} \label{Prop:FG}
As a tropical module $\trop(V) \subset R(D)$ is finitely generated.
\end{proposition}

\begin{proof}
It will suffice to show that $\trop(V)$ is the homomorphic image of a tropical linear space.  Indeed, any tropical linear space is the set of vectors of a valuated matroid, and the set of vectors of a valuated matroid is generated as a tropical module by one vector with each possible minimal support, i.e., by one vector supported on each circuit of the underlying matroid \cite[Section~3]{MurotaTamura01}.

We may assume that $D_X$ is effective.  Choose a semistable vertex set $S \subset \Gamma$, in the sense of \cite[\S3]{BPR13}, that contains the support of $D$. Then $\Trop^{-1}(\Gamma \smallsetminus S)$ is a union of finitely many open annuli $\{U_1, \ldots, U_s\}$.  Let $f \in V$.  Then $f$ is regular on $X \smallsetminus D_X$ and hence regular on each $U_i$.  We can identify $U_i$ with a standard open annulus $\Trop^{-1}(a_i,b_i)$, as in \cite[\S2]{BPR13}.

Then $f_{|U_i}$ has a power series expansion
\[
f_{|U_i} = \sum_{n=-\infty}^\infty \alpha_n t^n,
\]
where $\big(\val (\alpha_n) + n c\big) \to \infty$ for $c \in (a_i,b_i)$, i.e., for $c \in (a_i,b_i)$ and $N \in \RR$, there are only finitely many $n$ such that $\val(\alpha_n) + n c \leq N$.  Identifying $\Trop(U_i)$ with $(a_i, b_i)$, we have
\begin{equation} \label{eq:powerseries}
\Trop(f_{|U_i})(c) = \min_{n=-\infty}^\infty n\cdot c + \val(\alpha_n).
\end{equation}
By \cite[Lemma~7]{HMY12}, the slope of $\Trop(f)$ along any segment is at most $d := \deg(D)$.  Hence, $\Trop(f|_{U_i})$ is determined by $(\val(\alpha_{-d}), \ldots, \val(\alpha_d)) \in \overline{\mathbb{R}}^{2d+1}$, where $\overline {\mathbb{R}} = \mathbb{R} \cup \infty$ and $\infty = \val(0)$.

Thus $\Trop(f)$ is determined by the valuations of the $2d+1$ coefficients $\alpha_{-d}, \ldots, \alpha_d$ in the respective power series expansions on $U_1, \ldots, U_s$.  Considering these coefficients algebraically gives a $K$-linear embedding $V \hookrightarrow K^{(2d+1)s}$.  The image of $V$ in $\overline{\mathbb{R}}^{(2d+1) \cdot s}$ under coordinatewise valuation is a tropical linear space that surjects onto $\trop(V)$, and the proposition follows.
\end{proof}

Motivated by these properties, we define an abstract tropical linear series as follows.

\begin{definition} \label{Def:TLS}
A tropical linear series of rank $r$ on $\Gamma$ is a divisor $D$ together with a finitely generated tropical submodule $\Sigma \subset R(D)$ such that,
\begin{enumerate}
\item \label{DefItem:Rank} for every effective divisor $E$ of degree $r$ there is some $\varphi \in \Sigma$ such that $D + \ddiv(\varphi) \geq E$;
\item \label{DefItem:Dependence} any subset  of $\Sigma$ of size $s > r + 1$ is tropically dependent;
\item \label{DefItem:Recursive} any subset  of $\Sigma$ of size $s \leq r$ is contained in a tropical linear subseries of rank $r-1$.
\item \label{DefItem:Intersection} if $S_1$ and $S_2$ are subsets of $\Sigma$ of size $s_1$ and $s_2$, respectively, with $s_i \leq r$ and $s_1 + s_2  \geq r + 2$ then there are tropical linear subseries $\Sigma_1$ and $\Sigma_2$ containing $S_1$ and $S_2$ of rank $s_1-1$ and $s_2-1$, respectively, such that $\Sigma_1 \cap \Sigma_2$ contains a tropical linear series of rank $s_1 + s_2 - r -2$.

\end{enumerate}
\end{definition}

\noindent Definition~\ref{Def:TLS} isolates the essential combinatorial properties of tropicalizations of linear series that are used in the proof of our main results and we believe that it points the way forward toward a systematic combinatorial study of abstract tropical linear series.
Note that conditions \eqref{DefItem:Recursive}-\eqref{DefItem:Intersection} make the definition recursive in $r$.  Conditions \eqref{DefItem:Rank}-\eqref{DefItem:Recursive} suffice for nearly all of the arguments in the paper.  In particular, the analogs of Theorems~\ref{thm:generalv1} and \ref{thm:independence} hold for $g = 22$ for a weaker version of tropical linear series that omits condition \eqref{DefItem:Intersection}. The last condition is used only in the proof of Proposition~\ref{Prop:Case3Fns}, which is part of the case analysis for $g = 23$.

\medskip

A \emph{tangent vector} in $\Gamma$ is a germ of a directed edge.  Given $\varphi \in \PL(\Gamma)$,  we write $s_\zeta(\varphi)$ for the slope of $\varphi$ along a tangent vector $\zeta$.

\begin{lemma}
\label{Lem:Slopes}
Let $\Sigma \subseteq R(D)$ be a tropical linear series of rank $r$.  For each tangent vector $\zeta$, the set of slopes $s_\zeta(\Sigma)$ has size exactly $r+1$.
\end{lemma}

\begin{proof}
Suppose $\zeta$ is a tangent vector based at $v \in \Gamma$, and let $I \subseteq \Gamma$ be a half-open interval with one endpoint at $v$ that contains $\zeta$.  By choosing $I$ sufficiently small, we may assume that $I \smallsetminus \{ v \}$ does not intersect the support of $D$,\ and each of the functions in a finite generating set for $\Sigma$ is linear on $I$.  Then $s_{\eta} (\Sigma) = s_{\zeta} (\Sigma)$ for all tangent vectors $\eta$ in $I$ that are oriented away from $v$.  By the rank property \eqref{DefItem:Rank}, if $E$ is the sum of $r$ distinct points of $I \smallsetminus \{ v \}$, then there exists $\varphi \in \Sigma$ such that $\ddiv (\varphi) + D \geq E$.  At each of the $r$ distinct points, the incoming slope of $\varphi$ must be greater than the outgoing slope, and thus the function $\varphi$ has at least $r+1$ distinct slopes on $I$.  It follows that $|s_\zeta(\Sigma)| \geq r+1$.

On the other hand, any set of functions in $\Sigma$ with distinct slopes along $\zeta$ is tropically independent.  Hence, by the dependence property \eqref{DefItem:Dependence}, we also have $|s_\zeta(\Sigma)| \leq r+1$.
\end{proof}

We write
\[
s_\zeta(\Sigma) := (s_\zeta [0], \ldots, s_\zeta [r])
\]
for the vector of slopes in $\{ s_\zeta(\varphi) \mid \varphi \in \Sigma\}$, ordered so that  $s_\zeta [0] <  \cdots < s_\zeta[r]$.

\medskip

The recursive structure of tropical linear series guarantees the existence of functions in $\Sigma$ satisfying slope conditions at multiple  tangent vectors, as follows.

\begin{lemma} \label{Lemma:InductiveExistence}
Let $\Sigma \subset R(D)$ be a tropical linear series of rank $r$ on $\Gamma$ and let $\zeta$ be a tangent vector.  Then, for any $0 \leq i \leq r$, the tropical submodule $\{ \varphi \in \Sigma \mid s_\zeta(\varphi) \leq s_\zeta[i] \}$ contains a tropical linear series of rank $i$.  Similarly, $\{ \varphi \in \Sigma \mid s_\zeta(\varphi) \geq s_\zeta[i] \}$ contains a tropical linear series of rank $r -i$.
\end{lemma}

\begin{proof}
Choose $\varphi_0, \ldots, \varphi_r$ in $\Sigma$  so that $s_{\zeta}(\varphi_j) = s_{\zeta}[j]$.  By (3), the set $\{\varphi_0, \ldots, \varphi_i \}$ is contained in a tropical linear subseries of rank $i$, and similarly $\{\varphi_i, \ldots, \varphi_{r} \}$ is contained in a tropical linear subseries of rank $r - i$.
\end{proof}

\begin{lemma}  \label{Lemma:Existence}
Let $\Sigma \subset R(D)$ be a tropical linear series of rank $r$ on $\Gamma$, and let $\zeta$ and $\zeta'$ be tangent vectors. Then, for any $0 \leq i \leq r$, there is a function $\varphi \in \Sigma$ such that
\[
 s_\zeta (\varphi) \leq s_\zeta [i] \mbox{ \ and \ } s_{\zeta'} (\varphi) \geq s_{\zeta'} [i].
\]
\end{lemma}

\begin{proof}
There is a tropical linear series $\Sigma'$ of rank $i$ contained in the set of functions $\varphi \in \Sigma$ with $s_{\zeta}(\varphi) \leq s_{\zeta}[i]$. The functions in $\Sigma'$ have precisely $i+1$ different slopes along $\zeta'$, all of which are contained in the set $\{ s_{\zeta'}[0], \ldots, s_{\zeta'}[r]\}$.  Hence, the largest of these must be $s_{\zeta'}[j]$ for some $j \geq i$.
\end{proof}

\begin{proposition}
\label{Prop:Restriction}
Let $\Sigma \subseteq R(D)$ a tropical linear series of rank $r$.  If $\Gamma' \subseteq \Gamma$ is a metric subgraph, let $D'$ be the divisor on $\Gamma'$ given by
\[
D'(w) := D(w) - \min_{\varphi \in \Sigma} \bigg\{ \sum_{\zeta} s_{\zeta}(\varphi) \bigg\} \mbox{ for all } w \in \Gamma',
\]
where the sum is over all tangent vectors $\zeta$ from $w$ into the complement of $\Gamma'$.  Then the restriction
\[
\Sigma _{\vert \Gamma'} := \{ \varphi _{\vert \Gamma'} \mid \varphi \in \Sigma \} \subseteq R(D')
\]
is a tropical linear series of rank $r$ on $\Gamma'$.
\end{proposition}

\begin{proof}
We verify conditions~\eqref{DefItem:Rank}-\eqref{DefItem:Recursive} from Definition~\ref{Def:TLS}.  For \eqref{DefItem:Rank}, let $E$ be an effective divisor of degree $r$ on $\Gamma'$.  Since $\Sigma$ is a tropical linear series of rank $r$ on $\Gamma$, there exists a function $\varphi \in \Sigma$ such that $\ddiv (\varphi) + D \geq  E$.  For any $w \in \Gamma'$, we have
\[
\ord_w (\varphi_{\vert \Gamma'}) = \ord_w (\varphi) + \sum_{\zeta} s_{\zeta} (\varphi),
\]
where the sum is over all tangent vectors $\zeta$ from $w$ into the complement of $\Gamma'$.  It follows that
\[
\ord_w (\varphi \vert_{\Gamma'})  \geq \ord_w (\varphi) + \min_{\varphi \in \Sigma} \bigg \{ \sum_{\zeta} s_{\zeta}(\varphi) \bigg \} \geq E(w) - D'(w),
\]
hence $\ddiv (\varphi_{\vert \Gamma'}) + D'_{\vert \Gamma'} \geq E$.

For \eqref{DefItem:Dependence}, let $\{ \varphi_0 , \ldots , \varphi_{r+1} \}$ be a set of $r+2$ functions on $\Gamma$.  Since $\Sigma$ is a tropical linear series of rank $r$ on $\Gamma$, there exist coefficients $b_0 , \ldots , b_{r+1}$ such that $\min \{ \varphi_i + b_i \}$ achieves the minimum at least twice at every point of $\Gamma$.  But then $\min \{ (\varphi_i) _{\vert \Gamma'} + b_i \}$ achieves the minimum at least twice at every point of $\Gamma'$.

For \eqref{DefItem:Recursive}, we argue by induction on $r$.  The base case $r=0$ is trivial.  Let $\varphi_1 , \ldots, \varphi_r \in \Sigma_{\vert \Gamma'}$.  By definition, there exist functions $\varphi'_1 , \ldots , \varphi'_r$ such that $(\varphi'_i)_{\vert \Gamma'} = \varphi_i$, and a tropical linear subseries $\Sigma' \subset \Sigma$ of rank $r-1$ containing $\varphi'_1 , \ldots , \varphi'_r$.  By induction, the restriction $\Sigma'_{\vert \Gamma'}$ is a tropical linear series of rank $r-1$ on $\Gamma'$, and the result follows.

For \eqref{DefItem:Intersection}, if $S'_1$ and $S'_2$ are subsets of $\Sigma _{\vert \Gamma'}$ then we can lift them to subsets $S_1$ and $S_2$ of $\Sigma$.  By definition, there are tropical linear subseries $\Sigma_1$ and $\Sigma_2$ containing $S_1$ and $S_2$ of rank $s_1-1$ and $s_2-1$, respectively, such that $\Sigma_1 \cap \Sigma_2$ contains a tropical linear series of rank $s_1 + s_2 - r -2$. Then ${\Sigma_1} _{\vert \Gamma'}$ and ${\Sigma_2} _{\vert \Gamma'}$ are tropical linear subseries of $\Sigma _{\vert \Gamma'}$ with the required properties.
\end{proof}

We now discuss the simplest nontrivial case: the tropical linear series of degree 2 and rank 1 on an interval.  We treat this case in detail because many of the features of tropical linear series that are essential to the proof of our main results are already visible in this example.

\begin{example}
\label{Ex:Interval}
Let $\Gamma$ be an interval with left endpoint $w$ and right endpoint $v$.  Let $D = 2w$, and let $\Sigma \subset R(D)$ be a rank 1 tropical linear series.  For each rightward tangent vector $\zeta$, we consider the vector of slopes $s_\zeta = (s_{\zeta}[0], s_\zeta[1])$.  Similarly, we write $s_w = (s_w[0], s_w[1])$ and $s_v = (s_v[0], s_v[1])$ for the rightward slopes of functions in $\Sigma$ at the endpoints $w$ and $v$, respectively.

Note that $R(D)$ consists of all PL functions on $\Gamma$ whose rightward slopes are nonincreasing and bounded between 0 and 2.  At each rightward tangent vector $\zeta$, $s_\zeta$ is either $(1,2)$, $(0,2)$, or $(0,1)$.  We divide the interval $\Gamma$ into regions, according to these three possibilities.  Because the slopes of functions in $R(D)$ do not increase from left to right, the region where $s_\zeta$ is equal to $(1,2)$ must be to the left of the region where it is equal to $(0,2)$, which in turn must be to the left of the region where it is equal to $(0,1)$. Note that any of these regions may be empty.

\medskip

\textbf{Case 1:}   First, consider the case where there is a non-empty region where $s_\zeta$ is equal to $(0,2)$.  Choose functions $\psi_0$ and $\psi_1$ in $\Sigma$ such that
\[
s_w(\psi_0) = s_w[0] \mbox{ \ \ and \ \ } s_v(\psi_1) = s_v[1].
\]
Because the slopes of functions in $R(D)$ cannot increase from left to right, $s_\zeta(\psi_1)$ must be positive at all rightward tangent vectors $\zeta$.
It follows that $$s_\zeta(\psi_1) = s_\zeta[1]$$ for all rightward tangent vectors $\zeta$.  By a similar argument, $s_\zeta(\psi_0) = s_\zeta[0]$ for all $\zeta$.
See Figure~\ref{Fig:EasyCase}.

\begin{figure}[H]
\begin{center}

\begin{tikzpicture}
\tikzmath{ \y = .6;}
\draw (-0.3,2*\y) node {{\tiny $\psi_1$}};
\draw[dashed] (0,2*\y)--(4,6*\y);
\draw[dashed] (4,6*\y)--(6,7*\y);

\draw (-0.3,3.5*\y) node {{\tiny $\psi_0$}};
\draw (0,3.5*\y)--(2,4.5*\y);
\draw (2,4.5*\y)--(6,4.5*\y);

\draw (0,1.5*\y)--(6,1.5*\y);
\draw [ball color=white] (0,1.5*\y) circle (0.55mm);
\draw [ball color=white] (6,1.5*\y) circle (0.55mm);
\draw [ball color=white] (2,1.5*\y) circle (0.55mm);
\draw [ball color=white] (4,1.5*\y) circle (0.55mm);
\draw (-0.25, 1.25*\y) node {{\tiny $w$}};
\draw (6.25, 1.25*\y) node {{\tiny $v$}};
\draw (1,1.7*\y) node {{\tiny $(1,2)$}};
\draw (3,1.7*\y) node {{\tiny $(0,2)$}};
\draw (5,1.7*\y) node {{\tiny $(0,1)$}};
\end{tikzpicture}

\end{center}

\caption{The functions $\psi_0$ and $\psi_1$, when the $(0,2)$ region is nonempty.}
\label{Fig:EasyCase}
\end{figure}

Similarly, for any $\psi \in \Sigma$, there is some point $v_\psi \in \Gamma$ such that $s_\zeta(\psi) = s_\zeta[1]$ at a rightward tangent vector $\zeta$ if and only if the basepoint of $\zeta$ is to the right of $v_\psi$.  Then $\psi = \min\{ \psi_0 + a_0, \psi_1 + a_1 \}$ is a tropical linear combination of $\psi_0$ and $\psi_1$, with coefficients chosen so that $\psi_i(v_\psi) + a_i = \psi(v_\psi)$.  It follows that $\Sigma$ is the tropical submodule of $R(D)$ generated by $\psi_0$ and $\psi_1$, and hence is completely determined by the nonempty region where $s_\zeta = (0,2)$.

\medskip

\textbf{Case 2:} Suppose the region where $s_\zeta = (0,2)$ is empty.   Then there is a distinguished point $x \in \Gamma$ such that $s_\zeta$ is equal to $(1,2)$ to the left of $x$ and $(0,1)$ to the right of $x$.  We now consider functions $\psi_A$ and $\psi_B$ such that
\[
s_w(\psi_A) = s_w[0] \mbox{\ \ and \ \ } s_v(\psi_B) = s_v[1].
\]
Note that $\psi_A$ must have slope $1$ at all points to the left of $x$.  It continues to have slope 1 for some distance $t$ to the right of $x$, and then its slope is 0 the rest of the way.  Similarly, $\psi_B$ has slope 1 at all points to the right of $x$, and for some distance $t'$ to the left of $x$.  Then at all points of distance greater than $t'$ to the left of $x$, the slope of $\psi_B$ is 2.

By taking a tropical linear combination of two functions that agree at $x$, one with rightward slope 0 and the other with leftward slope $-2$, we get a third function $\psi_C \in \Sigma$ with rightward slope $2$ everywhere to the left of $x$ and $0$ everywhere to the right of $x$.
See Figure~\ref{Fig:HardCase}.

\begin{figure}[h!]
\begin{center}
\begin{tikzpicture}
\tikzmath{ \y = .75;}
\draw (-0.3,4*\y) node {{\tiny $\psi_A$}};
\draw (0,4*\y)--(4,6*\y);
\draw (4,6*\y)--(6,6*\y);

\draw (-0.3,3.35*\y) node {{\tiny $\psi_B$}};
\draw[dashed] (0,3.5*\y)--(1.5,5*\y);
\draw[dashed] (1.5,5*\y)--(6,7.25*\y);

\draw (-0.3,2.7*\y) node {{\tiny $\psi_C$}};
\draw[dotted] (0,2.9*\y)--(3,5.9*\y);
\draw[dotted] (3,5.9*\y)--(6,5.9*\y);

\draw (0,2.3*\y)--(6,2.3*\y);
\draw [ball color=white] (0,2.3*\y) circle (0.55mm);
\draw [ball color=white] (6,2.3*\y) circle (0.55mm);
\draw [ball color=white] (3,2.3*\y) circle (0.55mm);
\draw [ball color=black] (1.4,2.3*\y) circle (0.55mm);
\draw [ball color=black] (4,2.3*\y) circle (0.55mm);
\draw (-0.25, 2.1*\y) node {{\tiny $w$}};
\draw (6.25, 2.1*\y) node {{\tiny $v$}};
\draw (3,2.5*\y) node {{\tiny $x$}};
\draw [<->] (2.95,2.1*\y) -- (1.45,2.1*\y);
\draw [<->] (3.05,2.1*\y) -- (3.95,2.1*\y);
\draw (2.25,1.9*\y) node {{\tiny $t'$}};
\draw (3.5,1.9*\y) node {{\tiny $t$}};
\end{tikzpicture}

\end{center}

\caption{The functions $\psi_A$, $\psi_B$, and $\psi_C$ when the $(0,2)$ region is empty.}
\label{Fig:HardCase}
\end{figure}

We claim that the distances $t$ and $t'$ \emph{must be equal}.   This follows from Definition~\ref{Def:TLS}\eqref{DefItem:Dependence}, which says that $\{\psi_A, \psi_B, \psi_C\}$ is tropically dependent.  On each region where all three functions are linear, there are exactly two with the same slope, and this determines the combinatorial type of the tropical dependence, i.e., which functions achieve the minimum on which regions, as shown in Figure~\ref{Fig:IntervalDependence}.

\begin{figure}
\begin{center}
\begin{tikzpicture}
\tikzmath{ \y = .75;}
\draw (0,0)--(6,0);
\draw [ball color=white] (0,0) circle (0.55mm);
\draw [ball color=white] (6,0) circle (0.55mm);
\draw (-0.25, -.15*\y) node {{\tiny $w$}};
\draw (6.25, -.15*\y) node {{\tiny $v$}};

\draw (-0.3,2.15*\y) node {{\tiny $\psi_A$}};
\draw (0,2.15*\y)--(4,4.15*\y);
\draw (4,4.15*\y)--(6,4.15*\y);

\draw (-0.3,0.9*\y) node {{\tiny $\psi_B$}};
\draw[dashed] (0,1*\y)--(2.05,3.05*\y);
\draw[dashed] (2.16,3.16*\y)--(6,5.08*\y);

\draw (-0.3,1.15*\y) node {{\tiny $\psi_C$}};
\draw[dotted] (0,1.1*\y)--(3,4.1*\y);
\draw[dotted] (3,4.1*\y)--(6,4.1*\y);

\draw [ball color=black] (2,0) circle (0.55mm);
\draw [ball color=white] (3,0) circle (0.55mm);
\draw [ball color=black] (4,0) circle (0.55mm);
\draw (1,0.3*\y) node {{\tiny $BC$}};
\draw (3,0.3*\y) node {{\tiny $AB$}};
\draw (5,0.3*\y) node {{\tiny $AC$}};
\draw (3,-0.2*\y) node {{\tiny $x$}};

\end{tikzpicture}
\caption{The tropical dependence that shows $t=t'$.}
\label{Fig:IntervalDependence}
\end{center}
\end{figure}

Note that all three functions in this dependence achieve the minimum together at two points, as shown in Figure~\ref{Fig:IntervalDependence}:  the point of distance $t$ to the right of $x$, and the point of distance $t'$ to the left of $x$.   Comparing the slopes of $\psi_C$ to those of $\psi_A$ shows that $\psi_C - \psi_A$ is equal to $t'$ at $x$, and also equal to $t$.  This proves that $t = t'$, as claimed.  Analogous arguments show that an arbitrary function $\psi \in \Sigma$ is a tropical linear combination of $\psi_A$, $\psi_B$, and $\psi_C$, so $\Sigma$ is determined by $x$ and $t$.
\end{example}

\noindent Note that one may realize the interval as a skeleton of the analytification of a rational curve, and thereby interpret Example~\ref{Ex:Interval} as a study of the possible tropicalizations of degree $2$ pencils on $\PP^1$.

\section{Divisors and linear series on chains of loops}

\subsection{Admissible edge lengths} Our main results are proved using a detailed study of tropical linear series on a chain of $g$ loops with bridges whose edge lengths are constrained as follows.  Let $\Gamma$ be a graph with $2g+2$ vertices, labeled $w_0, \ldots , w_{g}$, and $v_1, \ldots, v_{g+1}$. There are two edges connecting $v_k$ to $w_k$, whose lengths are denoted $\ell_k$ and $m_k$, for $1 \leq k \leq g$, along with a bridge $\beta_k$ of length $n_k$ connecting $w_{k-1}$ to $v_k$ for $1 \leq k \leq g + 1$, as shown in Figure~\ref{Fig:TheGraph}.

{\centering
\begin{figure}[H]
\scalebox{.9}{
\begin{tikzpicture}
\draw (-1.75,0) node {\footnotesize $w_0$};
\draw(-.9,.3) node {\footnotesize $\beta_1$};
\draw (-1.5,0)--(-0.5,0);
\draw (0,0) circle (0.5);
\draw (-0.25,0) node {\footnotesize $v_1$};
\draw (0.25,0) node {\footnotesize $w_1$};
\draw (0.5,0)--(1.5,0);
\draw (2,0) circle (0.5);
\draw (1.75,0) node {\footnotesize $v_2$};
\draw (2.5,0)--(3.5,0);
\draw (4,0) circle (0.5);
\draw (4.5,0)--(5.5,0);
\draw (6,0) circle (0.5);
\draw (6.5,0)--(7.5,0);
\draw (8,0) circle (0.5);
\draw (7.75,0) node {\footnotesize $v_{g}$};
\draw (8.25,0) node {\footnotesize $w_{g}$};
\draw (8.5,0)--(9.5,0);
\draw (10,0) node {\footnotesize $v_{g+1}$};
\draw [<->] (2.6,0.15)--(3.3,0.15);
\draw [<->] (4.6,.25) arc[radius = 0.65, start angle=20, end angle=160];
\draw [<->] (4.61, -.15) arc[radius = 0.63, start angle=-9, end angle=-173];
\draw (3,0.4) node {\footnotesize$n_k$};
\draw (4,1) node {\footnotesize$\ell_k$};
\draw (4,-1) node {\footnotesize$m_k$};
\end{tikzpicture}
}
\caption{The chain of loops $\Gamma$.}
\label{Fig:TheGraph}
\end{figure}
}

\noindent We write $\gamma_k$ for the $k$th loop, formed by the two edges connecting $v_k$ and $w_k$, for $1 \leq k \leq g$.

Throughout, we assume that the edge lengths of $\Gamma$ satisfy the following condition for some large positive number $C$.

\begin{definition} \label{Def:Admissible}
The graph $\Gamma$ has \emph{$C$-admissible edge lengths} if
\[
\ell_k > C \cdot m_k, \ \ \ \ell_k > C \cdot \ell_{k+1}, \ \ \ n_k > C \cdot n_{k+1}, \ \ \mbox{ and } \ \ n_{k+1} > C \cdot \ell_k, \ \ \mbox{ for all $k$.}
\]
\end{definition}

\noindent The required size of the parameter $C$ depends on the degree $d$ of the divisors under consideration and the genus $g$ of the chain of loops; taking $C > 12 d g$ suffices for the constructions in this paper.

\begin{remark}
These conditions on edge lengths are more restrictive than those in \cite{MRC}.
Our arguments here, e.g., in Lemma~\ref{Lem:VANotToRight}, require not only that the bridges are much longer than the loops, but also that each loop is much larger than the loops that come after it.  For illustrative purposes, we will generally draw the loops and bridges as if they are the same size.
\end{remark}

\subsection{Break divisors}
Every divisor of degree $d$ on $\Gamma$ is equivalent to a unique \emph{break divisor} $D$, with multiplicity $d-g$ at $w_0$ and precisely one point of multiplicity 1 on each loop $\gamma_k$; see, for instance, \cite{ABKS}.  In this way, $\Pic^d(\Gamma)$ is naturally identified with $\prod_{k=1}^g \gamma_k$.  We choose a coordinate on each loop $\gamma_k$, and hence coordinates on $\Pic^d(\Gamma)$, as in Definition~\ref{def:Pic-coords} below.  Since the top and bottom edges have lengths $\ell_k$ and $m_k$, respectively, the loop has length $\ell_k + m_k$.

\begin{remark}
There are other natural choices for representatives of divisor classes on $\Gamma$.  For example, \cite{tropicalBN} uses $w_0$-reduced divisors.  The combinatorial advantages of break divisors on chains of loops were illustrated by Nathan Pflueger in \cite{Pflueger17a}.  One such advantage is that the natural map $\Sym^g (\Gamma) \to \Pic^g (\Gamma)$ admits a continuous section, taking a divisor class to its unique break divisor representative \cite{MikhalkinZharkov08}.  The map taking a divisor class to its $w_0$-reduced representative, on the other hand, is not continuous.
\end{remark}

\begin{definition} \label{def:Pic-coords}
For a break divisor $D$, let $x_k(D) \in \mathbb{R} / (\ell_k + m_k ) \cdot \mathbb{Z}$ be the counterclockwise distance from $v_k$ to the unique point in the support of $D_{|{\gamma_k}}$.
\end{definition}

\noindent Thus the class $[D]$ is determined uniquely by its degree $d$ and the coordinates $(x_1(D), \ldots, x_g(D))$.  When $D$ is fixed, we omit it from the notation, and write simply $x_k$.

Recall that $\PL(\Gamma)$ is the additive group of continuous functions on $\Gamma$ that are piecewise-linear with integer slopes.  The \emph{order of $\varphi \in \PL(\Gamma)$ at $v$} is the sum of the incoming slopes along all edges incident to $v$.  Note that $\ord_v(\varphi)$ is zero for all but finitely many $v \in \Gamma$.  The \emph{divisor of $\varphi$} is then
\begin{equation} \label{eq:divphi}
\ddiv(\varphi) := \sum_{v \in \Gamma} \ord_v(\varphi) \cdot v,
\end{equation}
and $R(D) := \{ \varphi \in \PL(\Gamma) \mid D + \ddiv(\varphi) \geq 0 \}$.

\subsection{Analysis on a single loop} \label{sec:singleloop}
We often  analyze divisors $D$ and functions $\varphi \in \PL(\Gamma)$ by restricting to one loop $\gamma_k$ at a time.  The slopes along adjacent bridges play a special role.

\begin{definition}
For $\varphi \in \PL(\Gamma)$, let $s_k(\varphi)$ denote the rightward slope of $\varphi$ at $v_k$ along the bridge $\beta_k$.  Similarly, let $s'_k(\varphi)$ denote the rightward slope of $\varphi$ at $w_k$ along the bridge $\beta_{k+1}$.
\end{definition}

\begin{figure}[h!]
\begin{tikzpicture}
\begin{scope}[grow=right, baseline]
\draw (-1,0) circle (1);
\draw (-3.5,0)--(-2,0);
\draw (0,0)--(1.5,0);
\draw (-1,-1.25) node {{$\gamma_k$}};
\draw (-2.25,-.25) node {{$v_k$}};
\draw (.25,-.25) node {{$w_k$}};
\draw (-2.5,.45) node {{$s_k$}};
\draw [->] (-3,.2)--(-2.15,0.2);
\draw (.6,.45) node {{$s'_k$}};
\draw [->] (.15,.2)--(1,0.2);
\end{scope}
\end{tikzpicture}
\caption{The slopes $s_k$ and $s'_k$.}
\label{Fig:Slopes}
\end{figure}

Note that the operations of restricting to a loop and computing the divisor of a $\PL$ function do not commute. The difference is naturally expressed in terms of the slopes $s_k$ and $s'_k$ as follows.

\begin{lemma} \label{Lem:RestrictOneLoop}
Let $\varphi \in \PL(\Gamma)$.  Then
$
\ddiv(\varphi_{|\gamma_k}) = (\ddiv(\varphi))_{|\gamma_k} - s_k(\varphi) \cdot v_k + s'_k(\varphi) \cdot w_k.
$
\end{lemma}

\begin{proof}
The lemma follows directly from \eqref{eq:divphi} and the definition of $\ord_v(\varphi)$.
\end{proof}

\begin{lemma} \label{Lem:DegreeOneLoop}
Let $\varphi \in \PL(\Gamma)$.  Then $\deg\big(\ddiv(\varphi)_{|\gamma_k}\big) = s_k(\varphi) - s'_k(\varphi)$.
\end{lemma}

\begin{proof}
This is a consequence of Lemma~\ref{Lem:RestrictOneLoop}, since $\deg(\ddiv(\varphi_{|\gamma_k}))= 0$.
\end{proof}

Because the difference between $s'_k(\varphi)$ and $s_k(\varphi)$ appears so frequently, we introduce notation:
\[
\delta_k(\varphi) := s'_k(\varphi) - s_k(\varphi).
\]
Thus, if $\varphi$ has constant slope along each bridge then $s_{k+1}(\varphi) = s_0(\varphi) + \sum_{i=1}^k \delta_k(\varphi)$.  Also, with this notation Lemma~\ref{Lem:DegreeOneLoop} says that $$\deg( \ddiv(\varphi)_{|\gamma_k} ) = -\delta_k.$$

\medskip

We recall that every divisor of degree $g$ on $\Gamma$ is equivalent to a unique break divisor, whose support contains one point on each edge in the complement of a spanning tree \cite{ABKS}. More generally, we say that a divisor $D$ of degree $d$ is a \emph{break divisor} if $D + (g-d)w_0$ is a break divisor of degree $g$.  Concretely, the break divisors on $\Gamma$ are exactly those of the form $(d-g)w_0 + p_1 + \cdots + p_g$, where $p_k \in \gamma_k$.

\begin{lemma} \label{Lem:Slope++}
Let $D$ be a break divisor on $\Gamma$ and let $\varphi \in R(D)$.  Then $\delta_k(\varphi) \leq 1$ for all $k$. Moreover, $\delta_k(\varphi) = 1$ if and only if $(D + \ddiv(\varphi))_{|\gamma_k} = 0$.
\end{lemma}

\begin{proof}
First, $\deg(D_{|\gamma_k}) = 1$ for all $k$, since $D$ is a break divisor.  Also, $D + \ddiv(\varphi)$ is effective, since $\varphi \in R(D)$.  Thus $\deg(\ddiv(\varphi)_{|\gamma_k}) \geq -1$ with equality if and only if $(D + \ddiv(\varphi))_{|\gamma_k} = 0$.
\end{proof}

\begin{lemma} \label{Lem:EquivOneLoop}
Let $p_1, \ldots, p_s$ be points on $\gamma_k$ and let $x_i$ be the counterclockwise distance from $v_k$ to $p_i$.  Then $m_1 p_1 + \cdots + m_s p_s$ is a principal divisor on $\gamma_k$ if and only if $m_1 + \cdots + m_s = 0$ and $m_1 x_1 + \cdots + m_s x_s \equiv 0 \mod (\ell_k + m_k)$.
\end{lemma}

\begin{proof}
This is the tropical Abel-Jacobi Theorem on the loop $\gamma_k$.  Specifically, the map from $\gamma_k$ to $\mathbb{R}/(\ell_k + m_k)\mathbb{Z}$, sending a point to its clockwise distance from $v_k$, is precisely the tropical Abel-Jacobi map with basepoint $v_k$.
\end{proof}

\begin{lemma}
\label{Lem:PointsOfSlopeIncrease}
Let $D$ be a break divisor of degree $d$ on $\Gamma$ and let $\varphi \in R(D)$.  If $(D + \ddiv(\varphi))_{|\gamma_k} = 0$ then $x_k(D) = (s_k(\varphi) + 1) m_k$.
\end{lemma}

\begin{proof}
If $(D + \ddiv(\varphi))_{|\gamma_k} = 0$ then $\delta_k(\varphi) = 1$, by Lemma~\ref{Lem:Slope++}, and hence $s'_k(\varphi) = s_k(\varphi) + 1$.  By Lemma~\ref{Lem:RestrictOneLoop}, we have
$$ D_{|\gamma_k} + \ddiv(\varphi_{|\gamma_k}) = -s_k(\varphi) \cdot v_k + (s_k(\varphi) + 1) \cdot w_k.$$  Now, apply Lemma \ref{Lem:EquivOneLoop} to $\ddiv(\varphi_{|\gamma_k})$, and use the fact that $\Gamma$ has $C$-admissible edge lengths.
\end{proof}

\subsection{Classification of special divisors}
By the Riemann-Roch Theorem, every divisor class of degree $d$ on $\Gamma$ has rank at least $d - g$ \cite{BakerNorine07}.  A divisor is \emph{special} if its rank is non-negative and strictly greater than $d - g$.  When $0 \leq \rho(g,r,d) < g$, the Brill-Noether locus $W^r_d (\Gamma) \subseteq \Pic^d(\Gamma)$ parametrizing special divisor classes of degree $d$ and rank at least $r$ is a union of translates of $\rho(g,r,d)$-dimensional coordinate subtori, each of which corresponds to a standard Young tableau $T$ on a $(r+1) \times (g-d+r)$ rectangle, with entries from $\{1, \ldots, g \}$. This is part of the classification of special divisors on chains of loops, from \cite{tropicalBN}.

\begin{remark}
The metric graphs in \cite{tropicalBN} are chains of loops without bridges. However, the classification in that paper, and its proof, carries over to chains of loops with bridges essentially without change. This is because pushing forward divisors under the continuous map obtained by contracting all bridges induces a degree and rank preserving isomorphism on divisor class groups.  Note, in particular, that any two points on the same bridge of $\Gamma$ are linearly equivalent to each other.  Thus, if $D$ is a divisor on $\Gamma$ whose support intersects the open bridge $\beta_k^\circ = (w_{k-1}, v_k)$ then there is an equivalent divisor $D'$ with $D'_{|(\Gamma \smallsetminus \beta_k)} = D_{|(\Gamma \smallsetminus \beta_k)}$ whose support does not intersect $\beta_k^\circ$; such a $D'$ can be obtained by moving points of $D$ in $\beta_k^\circ$ to either of bridge's endpoints $w_{k-1}$ or $v_k$.
\end{remark}

Let $c_i(T)$ denote the set of entries in the $i$th column of a Young tableau $T$, ordered from right to left, starting with 0. Thus, if $T$ has $r + 1$ columns then the set of entries in the leftmost column is $c_r (T)$, the set of entries in the next column is $c_{r-1} (T)$, and so on.

\begin{proposition}
\label{Prop:DivisorsOnGamma}
Let $D$ be a special divisor of degree $d$ and rank $r$ on $\Gamma$. Then there is a $(r + 1) \times (g-d+r)$ rectangular standard Young tableau $T$ with entries from $\{1, \ldots, g \}$ such that, for each $0 \leq i \leq r$ there is a divisor $D_i \sim D$ satisfying
\begin{enumerate}
\item \label{it:eff} $D_i - (r-i)w_0 - i v_{g+1}$ is effective,
\item $(D_i)_{|\beta_k^\circ} = 0$ for all $k$
\item $\deg(D_i)_{|\gamma_k} \leq 1$ for all $k$,
\item \label{it:zero} if $k \in c_i(T)$ then $(D_i)_{|\gamma_k} = 0$.
\end{enumerate}
\end{proposition}

\begin{proof}
The proposition follows from \cite[Theorem~4.6]{tropicalBN}, via the correspondence between lattice paths in Weyl chambers and Young tableaux given in the proof of \cite[Theorem~1.4]{tropicalBN}.
\end{proof}

\subsection{Vertex avoiding divisors}
\label{Sec:VADivisors}

The set of divisor classes that satisfy conditions \eqref{it:eff}-\eqref{it:zero} of Proposition~\ref{Prop:DivisorsOnGamma} for a fixed tableau $T$ is a coordinate subtorus of dimension $\rho(g,r,d) = g - (r+1)(g-d+r)$ in $\Pic^d(\Gamma)$.  When $\rho$ is positive, these subtori intersect nontrivially. However, each contains a dense open subset of classes of \emph{vertex avoiding divisors} that are contained in exactly one of these subtori.

\begin{definition}
Let $D$ be a divisor of rank $r$ on $\Gamma$.  Then $D$ is \emph{vertex avoiding} if there is a unique divisor $D_i \sim D$ such that $D_i - (r-i) w_0 - i v_{g+1}$ is effective, for each $0 \leq i \leq r$.
\end{definition}

By Proposition~\ref{Prop:DivisorsOnGamma}, the existence of such a divisor $D_i \sim D$ holds for any divisor $D$ of rank $r$; the defining property of vertex avoiding divisors is the uniqueness.  Recall that the set of vertices of $\Gamma$ is $\{w_0, v_1, w_1, \ldots, v_g, w_g, v_{g+1}\}$.

\begin{proposition}
\label{Prop:VAiff}
Let $D$ be a divisor of rank $r$ on $\Gamma$.  Then $D$ is vertex avoiding if and only if, for each $0 \leq i \leq r$, there is a divisor $D_i$ satisfying properties \eqref{it:eff}-\eqref{it:zero} in Proposition~\ref{Prop:DivisorsOnGamma} such that the support of $D_i - (r-i) w_0 - i v_{g+1}$ does not contain any of the vertices of $\Gamma$.
\end{proposition}

\begin{proof}
Suppose $D_i$ satisfies properties \eqref{it:eff}-\eqref{it:zero} in Proposition~\ref{Prop:DivisorsOnGamma} and let $D'_i := D_i - (r-i) w_0 - i v_{g+1}$.  If the support of $D'_i$ does not contain any of the vertices of $\Gamma$, then $D'_i$ has at most one point on each loop, and that point is on either the top or the bottom edge (but not both).  Thus, we can choose a spanning tree disjoint from the support of $D'_i$. By \cite[Lemma~3.5]{ABKS}, we conclude that $D'_i$ is a rigid effective divisor, and hence $D_i$ is unique.

Conversely, if the support of $D'_i$ contains at least one of the vertices of $\Gamma$, then moving this point along the bridge adjacent to this vertex gives an infinite family of divisors equivalent to $D$ that satisfy \eqref{it:eff}-\eqref{it:zero} and hence $D_i$ is not unique.
\end{proof}

Fix a vertex avoiding break divisor $D$ of degree $d$ and rank $r$.

\begin{definition}
For $0 \leq i \leq r$, we define $\varphi_i \in \PL(\Gamma)$ to be the unique function such that $D + \ddiv(\varphi_i) = D_i$ and $\varphi_i(w_0) = 0$.
\end{definition}

Let $T$ be the tableau satisfying Proposition~\ref{Prop:DivisorsOnGamma}\eqref{it:eff}-\eqref{it:zero}; it is unique because $D$ is vertex avoiding.

\begin{proposition}
\label{Prop:VASlopes}
The function $\varphi_i$ has constant slope $s_k(\varphi_i)$ along the bridge $\beta_k$.  These slopes are given by
\[
s_k(\varphi_i) =  i - (g-d+r) + \#\{ \ell \in c_i(T) \mid \ell < k \}
\]
\end{proposition}

In other words, moving from left to right across the graph, we start with
\[
s_1(\varphi_i) = i - (g-d+r) = i - \# c_i(T),
\]
and $s_k(\varphi_i)$ is nondecreasing as a function of $k$.  When passing the loop $\gamma_k$, it increases by
\[
\delta_k(\varphi_i) = \left\{ \begin{array}{ll} 1 & \mbox{ if } k \in c_i(T), \\
							0 & \mbox{ otherwise.} \end{array} \right.
\]
When we arrive at the final bridge, the slopes are $s_{g+1}(\varphi_i) = i$ for all $i$.

\begin{proof}
The fact that $\varphi_i$ has constant slope along the bridge $\beta_k$ follows from Proposition~\ref{Prop:VAiff}.  Since $D(w_0) = d-g$ and $D_i (w_0) = r-i$, we see that $s_1 (\varphi_i) = i - (g-d+r)$.  Now, recall that $\deg(D_{|\gamma_k}) = 1$ for all $k$.  By Proposition~\ref{Prop:DivisorsOnGamma}, $\deg(D_i)_{|\gamma_k} \leq 1$ for all $k$ with equality if and only if $k \notin c_i(T)$.  Since $\deg( \ddiv({\varphi_i}_{|\gamma_k}) ) = -\delta_k$, we see that
\begin{equation} \label{Eq:OneIncrease}
\delta_k(\varphi_i) = \left\{ \begin{array}{ll} 1 & \mbox{ if } k \in c_i(T), \\
							0 & \mbox{ otherwise,} \end{array} \right.
\end{equation}
and the result follows.
\end{proof}

\begin{corollary}
\label{Cor:VABreak}
Fix a tableau $T$ on a $(r+1) \times (g-d+r)$ rectangle with entries from $\{1, \ldots, g \}$, and set
\[
s_k[i] := i - (g-d+r) + \# \{ \ell \in c_i(T) \mid \ell < k \}.
\]
Then the vertex avoiding break divisors of degree $d$ and rank $r$ associated to the tableau $T$ are precisely those break divisors $D$ of degree $d$ such that
$
x_k(D) = (s_k[i] + 1) m_k \mbox{ if and only if } k \in c_i(T).
$
\end{corollary}

\begin{proof}
This follows from Proposition~\ref{Prop:VASlopes} and Lemma~\ref{Lem:PointsOfSlopeIncrease}.
\end{proof}

\begin{corollary} \label{Cor:DistinctSlopes}
Let $D$ be a vertex avoiding break divisor of degree $d$ and rank $r$.  On each bridge $\beta_k$, the slopes $\{ s_k(\varphi_i) \mid 0 \leq i \leq r \}$ are distinct and satisfy
\[
s_k(\varphi_0) < \cdots < s_k(\varphi_r).
\]
\end{corollary}

\begin{proof}
This follows from the previous corollary, since
\[
\# \{ \ell \in c_i(T) \mid \ell < k \} \leq \# \{ \ell' \in c_{j}(T) \mid \ell' < k \}
\]
for $i < j$.
\end{proof}

\begin{definition}
Let $D$ be a fixed vertex avoiding break divisor of degree $d$ and rank $r$, and let $T$ be the associated tableau.  Then $\gamma_k$ is a \emph{lingering loop} if $k$ does not appear in $T$.
\end{definition}

Equivalently, $\gamma_k$ is a lingering loop if $\delta_k(\varphi_i) = 0$ for all $i$.  If $\gamma_k$ is not a lingering loop, then there is precisely one index $i$ such that $\delta_k(\varphi_i) = 1$; it is the index $i$ such that $k \in c_i(T)$.  The restriction of $\varphi_i$ to a non-lingering loop may be pictured schematically as in Figure~\ref{Fig:FnShapes}.
\begin{figure}[H]
\begin{tikzpicture}[thick, scale=0.8]

\begin{scope}[grow=right, baseline]
\draw (-1,0) circle (1);
\draw (-3,0)--(-2,0);
\draw (0,0)--(1,0);
\draw [ball color=white] (-.5,.866) circle (0.55mm);
\draw [ball color=black] (-1.866,0.5) circle (0.55mm);
\draw (-2.5,0.25) node {\tiny$s_k[i]$};
\draw (0.5,0.25) node {\tiny$s_k[i]$};
\draw (-1,-1.25) node {\tiny$s_k[i] - 1$};
\draw (-0.35,0.45) node {\tiny$1$};
\draw (-1.8,0.2) node {\tiny$1$};
\draw (-1.2,0.75) node {\tiny$0$};
\draw (-1,-2) node {$x_k < (s_{k}[i]+1)m_k$};

\draw (4,0) circle (1);
\draw (2,0)--(3,0);
\draw (5,0)--(6,0);
\draw [ball color=white] (4.5,.866) circle (0.55mm);
\draw (2.5,0.25) node {\tiny$s_k[i]$};
\draw (5.7,0.25) node {\tiny$s_k[i] + 1$};
\draw (4,-1.25) node {\tiny$s_k[i]$};
\draw (4.65,0.45) node {\tiny$1$};
\draw (3.5,0.6) node {\tiny$0$};
\draw (4,-2) node {$x_k = (s_{k}[i]+1)m_k$};

\draw (9,0) circle (1);
\draw (7,0)--(8,0);
\draw (10,0)--(11,0);
\draw [ball color=white] (9.5,.866) circle (0.55mm);
\draw [ball color=black] (9.866,0.5) circle (0.55mm);
\draw (7.5,0.25) node {\tiny$s_k[i]$};
\draw (10.5,0.25) node {\tiny$s_k[i]$};
\draw (9,-1.25) node {\tiny$s_k[i]$};
\draw (8.45,0.55) node {\tiny$0$};
\draw (9.8,0.2) node {\tiny$0$};
\draw (9.55,0.6) node {\tiny$1$};
\draw (9,-2) node {$x_k > (s_{k}[i]+1)m_k$};
\end{scope}
\end{tikzpicture}
\caption{The function $(\varphi_i)_{|\gamma_k}$ when $\gamma_k$ is not lingering.}
\label{Fig:FnShapes}
\end{figure}

\noindent The point of $D$ on $\gamma_k$ is indicated by a white dot, and the point of $D_i$, if any, is indicated by a black dot.  Each edge is labeled with the rightward slope of $\varphi_i$.  Note that if $x_k = (s_k[i] +1)m_k$, then $(D_i)_{|\gamma_k} = 0$.  Otherwise, $(D_i)_{|\gamma_k}$ is a point on the top edge.  If $x_k < (s_k[i] +1)m_k$, then this point is on the left half of the top edge, near $v_k$, and if $x_k < (s_k[i] +1)m_k$, then it is on the right half.

\begin{remark} \label{Rem:NotToScale}
Note that Figure~\ref{Fig:FnShapes}, and other similar figures in this paper, is schematic, and not to scale. In particular, because we assume that $\Gamma$ has $C$-admissible edge lengths, the top edge of $\gamma_k$ is much longer than the bottom edge, by a factor of at least $12dg$. It follows that the region where $\varphi_i$ has slope zero occupies most of the top edge, as shown. The point of $D$ on $\gamma_k$ is relatively close to $w_k$; its distance from $w_k$ is a small integer multiple of $m_k$. In particular, it is not near the middle of the top edge.  The point of $D_i$ on $\gamma_k$, if any, is relatively close to either $w_k$ or $v_k$, with the distance being a small integer multiple of $m_k$.
\end{remark}

\begin{lemma}
\label{Lem:HalfInteger}
Let $D$ be a vertex avoiding break divisor of degree $d$ and rank $r$ and let $\gamma_k$ be a non-lingering loop.  Suppose $x$ is a point on $\gamma_k$ whose distance from $w_k$ is a half-integer multiple of $m_k$, and choose constants $c_i$ such that $\varphi_i(w_k) + c_i = \varphi_j(w_k) + c_j$ for all $j$.  Then $\varphi_i(x) + c_i - \varphi_j(x) - c_j$ is a half-integer multiple of $m_k$.
\end{lemma}

\begin{proof}
Let $\gamma$ be the shortest path from $w_k$ to $x$.  The restriction of $\varphi_i(x) + c_i - \varphi_j(x) - c_j$ to $\gamma$ is a piecewise linear function with integer slopes.  By assumption, its value at $w_k$ is 0.  By examining Figure~\ref{Fig:FnShapes}, we see that the length of each domain of linearity is a half-integer multiple of $m_k$, and the result follows.
\end{proof}

\begin{example}
\label{Ex:Genus3}
As an example, we consider the canonical divisor $K_{\Gamma}$ on the chain of 3 loops.  By Riemann-Roch, $[K_{\Gamma}]$ is the only divisor class of degree 4 and rank 2 on $\Gamma$.  The corresponding tableau $T$ is the unique tableau on a rectangle with 3 columns and 1 row:

\begin{figure}[h!]
\begin{ytableau}
1 & 2 & 3
\end{ytableau}
\caption{The unique tableau $T$ on a $3 \times 1$ rectangle.}
\label{Fig:CanonicalTableau}
\end{figure}

We can use Corollary~\ref{Cor:VABreak} to determine the unique break divisor $D$ equivalent to $K_{\Gamma}$.  We see that $x_1 (D) =2$, $x_2 (D) = 1$, and $x_3 (D) = 0$, as depicted schematically in Figure~\ref{Fig:BreakDivisor}.

\begin{figure}[H]
\begin{tikzpicture}[thick, scale=0.8]
\begin{scope}[grow=right, baseline]
\draw (-1,0) circle (1);
\draw (-3,0)--(-2,0);
\draw (0,0)--(1,0);
\draw (2,0) circle (1);
\draw (3,0)--(4,0);
\draw (5,0) circle (1);
\draw (6,0)--(7,0);
\draw [ball color=white] (-3,0) circle (0.75mm);
\draw [ball color=white] (-0.13,0.5) circle (0.75mm);
\draw [ball color=white] (3,0) circle (0.75mm);
\draw [ball color=white] (4,0) circle (0.75mm);
\draw (-3.5,.5) node {$D$};

\end{scope}
\end{tikzpicture}
\caption{Break divisor equivalent to the canonical divisor}
\label{Fig:BreakDivisor}
\end{figure}

We can then use Proposition~\ref{Prop:VASlopes} to explicitly determine the divisors $D_i$ and the functions $\varphi_i$, for $i \in \{0, 1, 2\}$. The result is schematically illustrated in Figure~\ref{Fig:CanonicalExample}. The support of $D_i$ is indicated with black dots, and the remaining points in the support of $D$ are indicated with white dots.  All of these points have multiplicity $1$, except that the left endpoint has multiplicity 2 in $D_0$, and the right endpoint has multiplicity 2 in $D_2$.

\begin{figure}[H]
\begin{tikzpicture}[thick, scale=0.63]
\begin{scope}[grow=right, baseline, shift={(-6,0)}]
\draw (-1,6) circle (1);
\draw (-3,6)--(-2,6);
\draw (0,6)--(1,6);
\draw (2,6) circle (1);
\draw (3,6)--(4,6);
\draw (5,6) circle (1);
\draw (6,6)--(7,6);
\draw [ball color=black] (-3,6) circle (0.75mm);
\draw (-3.25,7) node {$D_0$};
\draw [ball color=white] (-0.13,6.5) circle (0.75mm);
\draw [ball color=white] (3,6) circle (0.75mm);
\draw [ball color=white] (4,6) circle (0.75mm);
\draw [ball color=black] (-.5,6.87) circle (0.75mm);
\draw [ball color=black] (2.87,6.5) circle (0.75mm);

\draw (-2.5,5.75) node {\tiny -1};
\draw (0.5,5.75) node {\tiny -1};
\draw (3.5,5.75) node {\tiny -1};
\draw (6.5,5.75) node {\tiny 0};
\draw (-1,4.75) node {\tiny -1};
\draw (2,4.75) node {\tiny -1};
\draw (5,4.75) node {\tiny 0};
\draw (-1.85,6.85) node {\tiny 0};
\draw (1.6,7.15) node {\tiny 0};
\draw (5,7.25) node {\tiny 0};
\draw (0.15,6.35) node {\tiny 0};
\draw (-0.1,6.9) node {\tiny -1};
\draw (3.25,6.4) node {\tiny -1};

\end{scope}

\begin{scope}[grow=right, baseline, shift={(6,3)}]
\draw (-1,3) circle (1);
\draw (-3,3)--(-2,3);
\draw (0,3)--(1,3);
\draw (2,3) circle (1);
\draw (3,3)--(4,3);
\draw (5,3) circle (1);
\draw (6,3)--(7,3);
\draw [ball color=black] (-3,3) circle (0.75mm);
\draw [ball color=black] (-0.13,3.5) circle (0.75mm);
\draw [ball color=white] (3,3) circle (0.75mm);
\draw [ball color=white] (4,3) circle (0.75mm);
\draw [ball color=black] (4.13,3.5) circle (0.75mm);
\draw [ball color=black] (7,3) circle (0.75mm);

\draw (-3.25,4) node {$D_1$};

\draw (-2.5,2.75) node {\tiny 0};
\draw (0.5,2.75) node {\tiny 0};
\draw (3.5,2.75) node {\tiny 1};
\draw (6.5,2.75) node {\tiny 1};
\draw (-1,1.75) node {\tiny 0};
\draw (2,1.75) node {\tiny 0};
\draw (5,1.75) node {\tiny 1};
\draw (-1,4.25) node {\tiny 0};
\draw (2,4.25) node {\tiny 0};
\draw (5,4.25) node {\tiny 0};
\draw (0.15,3.3) node {\tiny 0};
\draw (3.85,3.3) node {\tiny 1};

\end{scope}

\begin{scope}[grow=right, baseline, shift={(0,2)}]
\draw (-1,0) circle (1);
\draw (-3,0)--(-2,0);
\draw (0,0)--(1,0);
\draw (2,0) circle (1);
\draw (3,0)--(4,0);
\draw (5,0) circle (1);
\draw (6,0)--(7,0);
\draw [ball color=white] (-3,0) circle (0.75mm);
\draw [ball color=white] (-0.13,0.5) circle (0.75mm);
\draw [ball color=white] (3,0) circle (0.75mm);
\draw [ball color=white] (4,0) circle (0.75mm);
\draw [ball color=black] (1.13,0.5) circle (0.75mm);
\draw [ball color=black] (4.5,0.86) circle (0.75mm);
\draw [ball color=black] (7,0) circle (0.75mm);
\draw (-3.25,1) node {$D_2$};

\draw (-2.5,-.25) node {\tiny 1};
\draw (0.5,-.25) node {\tiny 2};
\draw (3.5,-.25) node {\tiny 2};
\draw (6.5,-.25) node {\tiny 2};
\draw (-1,-1.25) node {\tiny 1};
\draw (2,-1.25) node {\tiny 1};
\draw (5,-1.25) node {\tiny 2};
\draw (-1.5,1.15) node {\tiny 0};
\draw (2.5,1.15) node {\tiny 0};
\draw (5.75,1) node {\tiny 0};
\draw (0.15,0.3) node {\tiny 1};
\draw (0.85,0.3) node {\tiny 1};
\draw (4,0.65) node {\tiny 1};

\end{scope}

\end{tikzpicture}
\caption{The divisors $D_0$, $D_1$, and $D_2$, and the slopes of $\varphi_0$, $\varphi_1$, and $\varphi_2$}
\label{Fig:CanonicalExample}
\end{figure}

Each edge in the illustration of $D_i$ is labeled with the rightward slope of $\varphi_i$.
\end{example}

\section{The vertex avoiding case}
\label{Sec:VertexAvoiding}

We continue the notation from the previous section.  In particular, $D$ is a vertex avoiding break divisor of degree $d$ and rank $r$ on $\Gamma$, which is a chain of $g$ loops with bridges whose edge lengths are $C$-admissible for some $C > 12dg$.  For $0 \leq i \leq r$, let $D_i$ be the unique divisor such that $D_i \sim D$ and $D_i - i w_0 - (r-i) v_{g+1}$ is effective, and let $\varphi_i \in \PL(\Gamma)$ be the unique function such that $\varphi_i(w_0) = 0$ and  $D + \ddiv(\varphi_i) = D_i$.

We use the functions $\varphi_i$ on $\Gamma$ to study ranks of multiplication maps on algebraic curves.  Set
\[
\varphi_{i,j} := \varphi_i + \varphi_j.
\]

\begin{proposition} \label{prop:rankmu2}
Let $X$ be a smooth projective curve of genus $g$ over $K$ with skeleton $\Gamma$, let $D_X$ be a divisor on $X$ such that $\trop(D_X) = D$, and let $V \subset H^0(X, \OO(D_X))$ be a linear series of rank $r$.  Then the rank of the multiplication map
\[
\mu_2 \colon \Sym^2 V \to H^0(X, \OO(2D_X))
\]
is bounded below by the size of the largest tropically independent subset of $\{ \varphi_{i,j} \mid 0 \leq i \leq j \leq r \}$.
\end{proposition}

\begin{proof}
First, $\trop(V) \subset R(D)$ is a tropical linear series of rank $r$ and hence contains $\varphi_0, \ldots, \varphi_r$.  Choose $f_i \in V$ such that $\trop(f_i) = \varphi_i$.  Then $\varphi_{i,j}$ is the tropicalization of $\mu_2(f_i \otimes f_j)$. The proposition follows, since the tropicalization of the image of $\mu_2$ is a tropical linear series of rank equal to $\mathrm{rank}(\mu_2) - 1$.
\end{proof}

Motivated by this proposition, we present an algorithm for constructing a certificate of independence for subsets of $\{ \varphi_{i,j} \mid 0 \leq i \leq j \leq r\}$.

\subsection{Algorithm for the vertex avoiding case} \label{sec:basicalg}

We present our algorithm in greater generality than needed for the main results of this paper, to make the underlying ideas more transparent and readily adaptable to other situations.  (A minor variant has already been applied to a different problem of similar flavor in genus 13  \cite{M13}.)

Recall that $\Gamma$ and $D$ are fixed.
The remaining input for the algorithm is as follows:
\begin{itemize}
\item A subset $\cB \subset \{ \varphi_{i,j} \mid 0 \leq i \leq j \leq r \}$, and
\item A non-increasing sequence of integers $\sigma := (\sigma_1, \ldots, \sigma_{g+1})$.
\item A positive real number $\epsilon < \frac{1}{144 dr^2}$.
\end{itemize}

The output is a collection of coefficients $\{ c(\psi) \in \RR \cup \infty \mid \psi \in \cB\}$ and a function $$\alpha \colon \cB \to \{ \beta_k \} \cup \{ \gamma_\ell \} \cup \{ \emptyset \}$$ that ``assigns'' functions to a bridge or loop.  We say that $\psi$ is ``assigned'' to $\alpha(\psi)$ if $\alpha(\psi) \in \{ \beta_k \} \cup \{ \gamma_\ell \}$. If $\alpha(\psi) = \emptyset$, we say that $\psi$ is ``unassigned.'' The output has the following properties:
\begin{itemize}
\item The function
$
\theta := \min \{ \psi + c(\psi) \mid \alpha(\psi) \neq \emptyset \}
$
is a certificate of independence for the collection of functions in $\cB$ that are assigned to some bridge or loop.
\item  More precisely, if $\alpha(\psi) \neq \emptyset$ then $\psi + c(\psi)$ achieves the minimum uniquely on an open subset of the bridge or loop $\alpha(\psi)$ to which $\psi$ is assigned.
\item If $\gamma_\ell$ is to the right of $\alpha(\psi)$ and $\alpha^{-1}(\gamma_\ell) \neq \emptyset$, then $\psi + c(\psi)$ does not achieve the minimum anywhere on $\gamma_\ell$.
\end{itemize}
Note that the collection of functions in $\cB$ that are assigned to some bridge or loop depends on $\sigma$. In the setting of Proposition~\ref{prop:rankmu2}, the size of this collection is a lower bound for the rank of $\mu_2$.  The proofs of our main results depend on a choice of $\sigma$ that maximizes the size of this collection. The coefficients in the algorithm's output depend on $\epsilon$, but the assignment function $\alpha$ does not.

\begin{remark} \label{rem:casesofinterest}
In the cases of primary interest for our main result (when $g \geq 21$, $d = g + 3$, and $r = 6$), cf. \S\ref{sec:VArelevant}, we take $\cB$ to be the sumset $2 \cA := \{ \varphi + \varphi' \mid \varphi, \varphi' \in \cA \}$ and choose $\sigma$ so that:
\begin{itemize}
\item Every function in $\cB$ is assigned to a bridge or loop.
\item One function is assigned to each non-lingering loop.
\item Two functions are assigned to $\beta_1$ and three are assigned to $\beta_{g+1}$.
\item One function is assigned to each of the two bridges $\beta_k$, for $2 \leq k \leq g$, where $\sigma_k < \sigma_{k-1}$.
\end{itemize}
Moreover, in these cases, the slope of $\theta$ equals $\sigma_k$ on a sub-interval of $\beta_k$ of length greater than $(1-\epsilon) \ell_k$, and the average slope of $\theta$ on $\beta_k$ is between $\sigma_k - \epsilon$ and $\sigma_k+\epsilon$, for all $k$.
\end{remark}

The algorithm for constructing the coefficients and assignment function works from left to right across the graph.  At the beginning, we set all coefficients $c(\psi)$ to be infinite, and all functions are unassigned, i.e., $c(\psi) = \infty$ and $\alpha(\psi) = \emptyset$ for all $\psi$.    Once a coefficient $c(\psi)$ has a finite value, this value will never decrease.  It may increase as we progress through the algorithm, as long as $\psi$ is unassigned.  However, when $\psi$ is assigned to a bridge or loop, its coefficient $c(\psi)$ is fixed and never changes again.  From that point forward, $\psi + c(\psi)$ achieves the minimum uniquely on an open subset of the bridge or loop $\alpha(\psi)$ to which it is assigned.

\bigskip

Fix a subset $\cB \subset \{ \varphi_{i,j} \mid 0 \leq i \leq j \leq r \}$ and a non-increasing integer sequence $\sigma = (\sigma_1, \ldots, \sigma_{g+1})$.

\begin{definition} \label{Def:VAPermissible}
A function $\psi \in \cB$ is $\sigma$-\emph{permissible} on the loop $\gamma_k$ if
\[
s_{k} (\psi) \leq \sigma_k \leq s_{k+1} (\psi).
\]
\end{definition}

\begin{remark} \label{rem:permissible}
Suppose $\theta = \min \{ \psi + c(\psi) \mid \psi \in \cB \}$ and the average slope of $\theta$ on $\beta_k$ is in the range $(\sigma_k - \epsilon, \sigma_k + \epsilon)$ for all $k$, as it will be in the cases of interest for our main results (cf. Remark~\ref{rem:casesofinterest}).  Then, since the edge lengths are $C$-admissible, the sequence $\sigma_k$ is non-increasing, and the slopes $s_k(\varphi)$ are non-decreasing, the inequalities $s_{k} (\psi) \leq \sigma_k \leq s_{k+1} (\psi)$ are a necessary condition for $\psi + c(\psi)$ to achieve the minimum on $\gamma_k$.  This motivates Definition~\ref{Def:VAPermissible}, and the algorithm below assigns functions to loops only when they are $\sigma$-permissible.
\end{remark}

When $\sigma$ is fixed and no confusion seems possible, we refer to $\sigma$-permissible functions as \emph{permissible}.

\begin{lemma}
\label{Lem:VAEverythingIsPermissible}
Let $\psi \in \cB$. Then either
$s_1 (\psi) > \sigma_1$,
$s_{g+1} (\psi) < \sigma_g$, or
there is a $k$ such that $\psi$ is permissible on $\gamma_k$.
\end{lemma}

\begin{proof}
Suppose $s_1 (\psi) \leq \sigma_1$ and $s_{g+1} (\psi) \geq \sigma_g$.  If $s_k (\psi) \leq \sigma_k$ for all $k$, then $\psi$ is permissible on $\gamma_{g}$.  Otherwise, consider the smallest $k$ such that $s_k (\psi) > \sigma_k$.  Then, since $(\sigma_1, \ldots, \sigma_{g+1})$ is non-increasing and $(s_1(\psi), \ldots, s_{g+1}(\psi))$ is non-decreasing, $\psi$ is permissible on $\gamma_{k-1}$.
\end{proof}

Our algorithm is organized around keeping track of which functions in $\cB$ are permissible on each loop, as we move from left to right across the graph.  The loops on which $\psi$ is permissible are consecutive and we pay particular attention to the last loop on which each function is permissible.

\begin{lemma} \label{lem:lastloop}
Suppose $\psi$ is permissible on $\gamma_k$.  Then $\gamma_k$ is the last loop on which $\psi$ is permissible if and only if either $s_{k+1}(\psi) > \sigma_k$, $\sigma_k > \sigma_{k+1}$, or $k = g$.
\end{lemma}

\begin{proof}
This follows immediately from Definition~\ref{Def:VAPermissible}, since $(\sigma_1, \ldots, \sigma_{g+1})$ is non-increasing and $(s_1(\psi), \ldots, s_{g+1}(\psi))$ is non-decreasing.
\end{proof}

The last two conditions apply to all functions that are permissible on $\gamma_k$. The first condition is specific to $\psi$ and permissible functions that satisfy this condition are prioritized for assignment to $\gamma_k$ in the algorithm.

\begin{definition}
\label{Def:VADeparting}
A permissible function $\psi$ is \emph{departing} if $s_{k+1}(\psi) > \sigma_k$.
\end{definition}

\begin{lemma}
\label{Lem:VAAtMost1}
There is at most one departing permissible function on each loop.
\end{lemma}

\begin{proof}
Let $\psi \in \cB$ be a departing permissible function on $\gamma_k$.  By definition,
\[
s_k (\psi) \leq \sigma_k < s_{k+1} (\psi) .
\]
By Proposition~\ref{Prop:VASlopes}, there is at most one value of $i$ such that $s_k (\varphi_i) < s_{k+1} (\varphi_i)$, and it satisfies
\[
s_{k+1} (\varphi_i) = s_k (\varphi_i) + 1.
\]
(This is the unique $i$ such that $k \in c_i(T)$.) Then $\psi = \varphi_{i,j}$ for some $j$.  There is at most one such $j$ with $s_k(\psi) = \sigma_k$.  Also, if there is a departing function $\psi'$ such that $s_k(\psi') < \sigma_k$, then it must be $\psi' = 2 \varphi_i$, and $2s_k(\varphi_i) = \sigma_k - 1$.  It remains to show that both cannot occur on the same loop.

Suppose $2s_k(\varphi_i) = \sigma_k - 1$ and also $s_k(\varphi_{i,j}) = \sigma_k$ for some $j \neq i$.  Then $s_k(\varphi_{j}) = s_k(\varphi_i) + 1$ and hence $s_{k+1} (\varphi_j) = s_{k+1}(\varphi_i)$.  This is impossible, because the functions $\varphi_0, \ldots, \varphi_r$ have distinct slopes on every bridge (Corollary~\ref{Cor:DistinctSlopes}).
\end{proof}

\begin{remark} \label{rem:atmostone}
A similar argument shows that at most one function $\psi \in \cB$ is permissible on $\gamma_k$ and satisfies  $s_k(\psi) < \sigma_k$.  This observation is the essence of Lemma~\ref{lem:onenew}, below.
\end{remark}

\noindent We now provide the algorithm for producing the coefficients $c(\psi)$ and assignment function $\alpha$.

\bigskip

\noindent \textbf{Initialize.} Start by setting $c(\psi) = \infty$ and $\alpha(\psi) = \emptyset$ for all $\psi \in \cB$.
Define $$\theta' := \min\{ \psi + c(\psi) \mid \psi \in \cB \} \mbox{ \ \  and \ \ } \theta := \min \{ \psi + c(\psi) \mid \alpha(\psi) \neq \emptyset \}.$$

\medskip

\noindent \textbf{Start at the first bridge.}
Start at $\beta_1$.  Consider the set of slopes $\{ s_1(\psi) \mid \psi \in \cB \}$.  For each such slope that is strictly greater than $\sigma_1$, choose a function $\psi \in \cB$ with this slope, and give it a finite coefficient $c(\psi)$ so that $\vartheta'(w_0) = 0$ and each such $\psi + c(\psi)$ achieves the minimum in $\theta'$ on a subinterval of $\beta_1$ of length $\epsilon^2 \cdot \ell_1$.  Assign each of these functions to $\beta_1$, i.e., set $\alpha(\psi) = \beta_1$ for each chosen function.  Proceed to the first loop.

\medskip

\noindent \textbf{Loop subroutine.}
Each time we arrive at a loop $\gamma_k$, check whether there are any unassigned permissible functions. If not, proceed to $\beta_{k+1}$.  Otherwise, apply the following steps.

\medskip



\noindent \textbf{Loop subroutine, Step 1:  Align the unassigned permissible functions at $w_k$.} If $\psi$ is an unassigned permissible function, then either $c(\psi) = \infty$ or $\psi(w_k) + c(\psi)$ is strictly less than $\psi'(w_k) + c(\psi')$ for any previously assigned or non-permissible function $\psi'$.   (See Lemma~\ref{Lem:VANotToRight}.) If $c(\psi) = \infty$, then set a finite coefficient to that $\psi + c(\psi)$ is equal to $\theta'$ at $v_k$.  Then adjust the coefficient of each unassigned permissible function upward, the smallest amount possible, so that all of these terms are equal to $\theta'$ at $w_k$.

\medskip

\noindent \textbf{Loop subroutine, Step 2: Skip lingering loops.}
If $\gamma_k$ is a lingering loop, proceed to $\beta_{k+1}$.

\medskip

\noindent \textbf{Loop subroutine, Step 3:  Assign departing functions.}
Suppose there is an unassigned permissible function $\psi$ on $\gamma_k$ that is \emph{departing}, in the sense of Definition~\ref{Def:VADeparting}. (There is at most one, by Lemma~\ref{Lem:VAAtMost1}.)  For each unassigned, non-departing permissible function $\psi'$, adjust the coefficient $c(\psi')$ upward so that $\psi + c(\psi) = \psi' + c(\psi')$ at a point on $\beta_{k+1}$ at distance $\epsilon \cdot \ell_{k+1}$ to the right of $w_k$, and $\psi + c(\psi)$ achieves the minimum in $\theta'$ uniquely at $w_k$.  Assign the departing function $\psi$ to the loop, i.e., set $\alpha(\psi) = \gamma_k$. Proceed to $\beta_{k+1}$.

\medskip

\noindent \textbf{Loop subroutine, Step 4: Assign a function that achieves the minimum uniquely, if possible.}
If at least one unassigned permissible function $\psi$ achieves the minimum uniquely at some point in $\gamma_k$, then choose one such function and assign it to $\gamma_k$.  Increase its coefficient $c(\psi)$ by $\frac{1}{3} m_k$.  Note that $\psi + c(\psi)$ still achieves the minimum uniquely on a (smaller) open subset of $\gamma_k$; see Lemma~\ref{Lem:VAConfig}.  Proceed to $\beta_{k+1}$.

\medskip

\noindent \textbf{Internal bridge subroutine:} Upon arrival at $\beta_k$, for $1 < k < g + 1$, apply the following steps.

\medskip

 \noindent \textbf{Internal bridge subroutine, Step 1: If the slopes are steady, carry on.} If $\sigma_k = \sigma_{k-1}$ then proceed to the next loop $\gamma_k$.

\medskip

\noindent \textbf{Internal bridge subroutine, Step 2: Otherwise, be greedy.} If $2 \leq k \leq g$ and $\sigma_k < \sigma_{k-1}$, then assign as many functions to $\beta_k$ as possible.  Consider the slopes of unassigned functions $\psi \in \cB$ with $\sigma_k > s_k(\psi) \geq \sigma_{k-1}$, and choose one with each such slope.  Set the coefficients of the chosen functions so that each is greater than $\theta'$ on $\gamma_{k-1}$ and each achieves the minimum uniquely on a subinterval of $\beta_{k}$ of length $\epsilon^2 \cdot \ell_k$. Assign these functions to $\beta_{k}$.  Proceed to $\gamma_k$.

\medskip

\noindent \textbf{Final bridge subroutine:} When we arrive at the final bridge $\beta_{g+1}$, be greedy and assign as many functions to the bridge as possible.  Consider the slopes of unassigned functions $\psi \in \cB$ with $\sigma_{g+1} \geq  s_k(\psi) \geq \sigma_{g}$, and choose one with each such slope.  Set the coefficients of the chosen functions so that each is greater than $\theta'$ on $\gamma_{k-1}$ and achieves the minimum uniquely on a subinterval of $\beta_{g+1}$ of length $\epsilon^2 \cdot \ell_{g+1}$. Assign each of these functions to $\beta_{g+1}$.  Output the assignment function $\alpha$ and the certificate of independence $\theta = \min \{ \psi + c(\psi) \mid \alpha(\psi) \neq 0 \}$.

\subsection{Examples and verification} In Lemma~\ref{Lem:VAConfig} and Corollary~\ref{Cor:VAIndependence}, we verify that the tropical linear combination $\theta = \min \{ \psi + c(\psi) \mid \alpha(\psi) \neq \emptyset \}$ is indeed a certificate of independence with the desired properties, i.e., $\psi + c(\psi)$ achieves the minimum uniquely on an open subset of the bridge or loop $\alpha(\psi)$ and does not achieve the minimum on any loop to the right of $\alpha(\psi)$.   Before proving that the output of the algorithm has these properties, we illustrate with an example.

\medskip

\begin{example}
\label{Ex:CanonicalIndependence}
Suppose $g = 3$, and let $D$ be the break divisor in the canonical class $[K_\Gamma]$, as discussed in Example~\ref{Ex:Genus3}.  We apply the algorithm with $\cB = \{ \varphi_{i,j} \mid 0 \leq i \leq j \leq 2 \}$ and $\sigma_k = 0$ for $1 \leq k \leq 4$. For simplicity, we denote $c_{i,j} := c(\varphi_{i,j})$.

Using Proposition~\ref{Prop:VASlopes},
we compute the slopes of the functions $\varphi_i$ along each of the bridges.  These slopes are shown in Figure~\ref{Fig:CanonicalExample} and listed in the following table.

\medskip

\begin{figure}[h!]
\begin{tabular}{|c|c|c|c|c|}
\hline
 & $\beta_1$ & $\beta_2$ & $\beta_3$ & $\beta_4$ \\
\hline
$\varphi_2$ & 1 & 2 & 2 & 2 \\
\hline
$\varphi_1$ & 0 & 0 & 1 & 1 \\
\hline
$\varphi_0$& -1 & -1 & -1 & 0 \\
\hline
\end{tabular}
\caption{The slopes $\{s_k(\varphi_i)\}$ of the functions $\varphi_i$ on the each bridge $\beta_k$.}
\end{figure}

We then use this table to determine which functions are permissible on each loop.

\medskip

\begin{figure}[h!]
\begin{tabular}{|c|c|c|}
\hline
$\gamma_1$ & $\gamma_2$ & $\gamma_3$ \\
\hline
$\varphi_{1,1}$ & $\varphi_{1,1}$ &
$\varphi_{0,1}$ \\
$\varphi_{0,2}$ & $\varphi_{0,1}$ & $\varphi_{0,0}$ \\
\hline
\end{tabular}
\caption{The functions $\varphi_{i,j}$ that are permissible on each loop $\gamma_k$.}
\end{figure}

On the first bridge, there are 2 functions with strictly positive slope: $\varphi_{2,2}$ and $\varphi_{1,2}$.  Since $\varphi_{2,2}$ has the highest slope 2 on $\beta_1$, if $\varphi_{2,2} + c_{2,2}$ does not achieve the minimum at the left endpoint $w_0$, then it will never achieve the minimum.  So we set $c_{2,2} = 0$ and $\alpha (\varphi_{2,2}) = \beta_1$.  Similarly, $\varphi_{1,2}$ is the unique function with slope 1 on $\beta_1$.  We choose $c_{1,2}$ so that $\varphi_{1,2} + c_{1,2}$ is equal to $\varphi_{2,2}$ at a point a short distance to the right of $w_0$ and assign $\alpha (\varphi_{1,2}) = \beta_1$.

On the first loop $\gamma_1$, there are 2 permissible functions: $\varphi_{0,2}$ and $\varphi_{1,1}$.  The function $\varphi_{0,2}$ is departing (Definition~\ref{Def:VADeparting}), so we set $\alpha (\varphi_{0,2}) = \gamma_1$.  We choose coefficients $c_{0,2}$ and $c_{1,1}$ so that $\varphi_{0,2} + c_{0,2}$ achieves the minimum at $w_1$, and $\varphi_{1,1} + c_{1,1}$ achieves the minimum a short distance to the right of $w_1$.

On the second loop $\gamma_2$ there are again 2 permissible functions: $\varphi_{1,1}$ and $\varphi_{0,1}$.  Now, $\varphi_{1,1}$ is departing, so we set $\alpha (\varphi_{1,1}) = \gamma_2$ and choose the coefficient $c_{1,1}$ accordingly, so that $\varphi_{1,1}+ c_{1,1}$ achieves the minimum on $\gamma_2$.  Similarly, on the third loop $\gamma_3$ the 2 unassigned functions $\varphi_{0,1}$ and $\varphi_{0,0}$ are both permissible.  The function $\varphi_{0,1}$ is departing, so we set $\alpha (\varphi_{0,1}) = \gamma_3$ and set $c_{0,1}$ accordingly.  Finally, we set $\alpha (\varphi_{0,0}) = \beta_4$ and choose the coefficient $c_{0,0}$ so that $\varphi_{0,0} + c_{0,0}$ achieves the minimum uniquely on $\beta_4$, start from a point a short distance to the right of $w_3$.

Note that each $\varphi_{i,j}$ is assigned to a bridge or loop.  The resulting certificate of independence $\theta = \min_{i,j} \{ \varphi_{i,j} + c_{i,j} \}$ is schematically illustrated in Figure~\ref{Fig:CanonicalCertificate}.  The points in the support of $2D + \ddiv(\theta)$ are marked in black. (The point on $\beta_3$ appears with multiplicity 2, as indicated, and all others have multiplicity 1.)

\begin{figure}[H]
\begin{tikzpicture}[thick, scale=0.8]
\begin{scope}[grow=right, baseline]
\draw (-1,0) circle (1);
\draw (-4,0)--(-2,0);
\draw (0,0)--(1,0);
\draw (2,0) circle (1);
\draw (3,0)--(4,0);
\draw (5,0) circle (1);
\draw (6,0)--(7,0);
\draw [ball color=black] (-3.5,0) circle (0.75mm);
\draw [ball color=black] (-2.5,0) circle (0.75mm);
\draw [ball color=black] (-0.29,0.71) circle (0.75mm);
\draw [ball color=black] (0.3,0) circle (0.75mm);
\draw [ball color=black] (3.3,0) circle (0.75mm);
\draw (3.3,-0.3) node {\footnotesize $2$};
\draw [ball color=black] (4.29,0.71) circle (0.75mm);
\draw [ball color=black] (6.3,0) circle (0.75mm);

\draw (-3.9,0.2) node {\tiny $2,\!2$};
\draw (-3,0.2) node {\tiny $1,\!2$};
\draw (-1,1.2) node {\tiny $0,\!2$};
\draw (2,1.2) node {\tiny $1,\!1$};
\draw (5,1.2) node {\tiny $0,\!1$};
\draw (6.7,0.2) node {\tiny $0,\!0$};

\end{scope}
\end{tikzpicture}
\caption{A certificate of independence $\theta = \min_{i,j} \{ \varphi_{i,j} + c_{i,j} \}$ on a chain of 3 loops. The function $\varphi_{i,j} + c_{i,j}$ achieves the minimum uniquely on the region labeled $i,j$ in the complement of the support of $2D + \ddiv(\theta)$.}
\label{Fig:CanonicalCertificate}
\end{figure}
\end{example}

\begin{example}  \label{ex:randomtableau}
We now explain how the algorithm proceeds and illustrate the output in an example that is relevant to our main results, with $(g,r,d) = (22,6,25)$.
Let $\Gamma$ be a chain of $22$ loops with admissible edge lengths, and let $D$ be a vertex avoiding break divisor of degree $25$ and rank $6$ associated to the tableau $T$ in Figure~\ref{Fig:RandomTableau}. \begin{figure}[h!]
\begin{ytableau}
1 & 3 & 6 & 9 & 10 & 13 & 15 \\
2 & 5 & 7 & 12 & 16 & 19 & 20 \\
4 & 8 & 11 & 14 & 17 &21 & 22
\end{ytableau}
\caption{A randomly generated $3 \times 7$ tableau $T$.}
\label{Fig:RandomTableau}
\end{figure}

The additional input is $\cB = \{ \varphi_{i,j} \mid 0 \leq i \leq j \leq 6 \}$ and the slope function $\sigma$ given by:
\begin{displaymath}
\sigma_k = \left\{ \begin{array}{ll}
4 & \textrm{if $k \leq 7$,} \\
3 & \textrm{if $7 < k \leq 15$,}\\
2 & \textrm{if $16 < k \leq 23$.}
\end{array} \right.
\end{displaymath}
The tableau was chosen at random; the slope function $\sigma$ is specified according to a general rule given in Definition~\ref{def:z}, below.
In preparation for running the algorithm, it is helpful to pre-compute the slopes of the functions $\varphi_i$ along each of the bridges, using Proposition~\ref{Prop:VASlopes}. These slopes are shown in Figure~\ref{fig:g22slopes}.
\begin{figure}[h!]
\scalebox{.78}{
\begin{tabular}{|c||c|c|c|c|c|c|c|c|c|c|c|c|c|c|c|c|c|c|c|c|c|c|c|}
\hline
& $\beta_1$ & $\beta_2$ & $\beta_3$ & $\beta_4$ & $\beta_5$ & $\beta_6$ & $\beta_7$ & $\beta_8$ & $\beta_9$ & $\beta_{10}$ & $\beta_{11}$ & $\beta_{12}$ & $\beta_{13}$ & $\beta_{14}$ & $\beta_{15}$ & $\beta_{16}$ & $\beta_{17}$ & $\beta_{18}$ & $\beta_{19}$ & $\beta_{20}$ & $\beta_{21}$ & $\beta_{22}$ & $\beta_{23}$ \\
\hline
$\varphi_6$ &3 & 4 & 5 & 5 & 6 & 6 & 6 & 6 & 6 & 6 & 6 & 6 & 6 & 6 & 6 & 6 & 6 & 6 & 6 & 6 & 6 & 6 & 6 \\
\hline
$\varphi_5$ & 2 & 2 & 2 & 3 & 3 & 4 & 4 & 4 & 5 & 5 & 5 & 5 & 5 & 5 & 5 & 5 & 5 & 5 & 5 & 5 & 5 & 5 & 5 \\
\hline
$\varphi_4$ & 1 & 1 & 1 & 1 & 1 & 1 & 2 & 3 & 3 & 3 & 3 & 4 & 4 & 4 & 4 & 4 & 4 & 4 & 4 & 4 & 4 & 4 & 4\\
\hline
$\varphi_3$ & 0 & 0 & 0 & 0 & 0 & 0 & 0 & 0 & 0 & 1 & 1 & 1 & 2 & 2 & 3 & 3 & 3 & 3 & 3 & 3 & 3 & 3 & 3\\
\hline
$\varphi_2$ & -1 & -1 & -1 & -1 & -1 & -1 & -1 & -1 & -1 & -1 & 0 & 0 & 0 & 0 & 0 & 0 & 1 & 2 & 2 & 2 & 2 & 2 & 2 \\
\hline
$\varphi_1$ & -2 & -2 & -2 & -2 & -2 & -2 & -2 & -2 & -2 & -2 & -2 & -2 & -2 & -1 & -1 & -1 & -1 & -1 & -1 & 0 & 0 & 1 & 1\\
\hline
$\varphi_0$ & -3 & -3 & -3 & -3 & -3 & -3 & -3 & -3 & -3 & -3 & -3 & -3 & -3 & -3 & -3 & -2 & -2 & -2 & -2 & -2 & -1 & -1 & 0\\
\hline
\end{tabular}
}
\caption{The slopes $s_k(\varphi_i)$ of the distinguished functions $\varphi_i$.}
\label{fig:g22slopes}
\end{figure}
We can then use these slopes to determine the permissible functions $\varphi_{i,j}$ on each loop, as shown in Figure~\ref{fig:permiss22}.
\begin{figure}[h!]
\scalebox{.7}{
\begin{tabular}{|c|c|c|c|c|c|c|c|c|c|c|c|c|c|c|c|c|c|c|c|c|c|c|c|}
\hline
$\gamma_1$ & $\gamma_2$ & $\gamma_3$ & $\gamma_4$ & $\gamma_5$ & $\gamma_6$ & $\gamma_7$ & $\gamma_8$ & $\gamma_9$ & $\gamma_{10}$ & $\gamma_{11}$ & $\gamma_{12}$ & $\gamma_{13}$ & $\gamma_{14}$ & $\gamma_{15}$ &  $\gamma_{16}$ & $\gamma_{17}$ & $\gamma_{18}$ & $\gamma_{19}$ & $\gamma_{20}$ & $\gamma_{21}$ & $\gamma_{22}$ \\
\hline
 $\varphi_{5,5}$ & $\varphi_{5,5}$ & $\varphi_{5,5}$ & $\varphi_{4,5}$ & $\varphi_{4,5}$ & $\varphi_{4,4}$ & $\varphi_{4,4}$ & $\varphi_{3,4}$ & $\varphi_{3,4}$ & $\varphi_{2,4}$ & $\varphi_{2,4}$ & $\varphi_{3,3}$ & $\varphi_{1,5}$ & $\varphi_{1,5}$ & $\varphi_{2,3}$ & $\varphi_{2,2}$ & $\varphi_{2,2}$ & $\varphi_{1,3}$ & $\varphi_{1,3}$ & $\varphi_{1,2}$ & $\varphi_{1,2}$ & $\varphi_{1,1}$ \\
  $\varphi_{4,6}$ & $\varphi_{3,6}$ & $\varphi_{4,5}$ & $\varphi_{2,6}$ & $\varphi_{3,5}$ & $\varphi_{3,5}$ & $\varphi_{3,5}$ & $\varphi_{2,5}$ & $\varphi_{1,5}$ & $\varphi_{1,5}$ & $\varphi_{1,5}$ & $\varphi_{1,5}$ & $\varphi_{1,4}$ & $\varphi_{1,4}$ & $\varphi_{0,6}$ & $\varphi_{1,3}$ & $\varphi_{1,3}$ & $\varphi_{0,4}$ & $\varphi_{1,2}$ & $\varphi_{0,4}$ & $\varphi_{1,1}$ & $\varphi_{0,3}$ \\
    $\varphi_{3,6}$ & $\varphi_{2,6}$ & $\varphi_{2,6}$ & $\varphi_{1,6}$ & $\varphi_{1,6}$ & $\varphi_{1,6}$ & $\varphi_{1,6}$ & $\varphi_{1,5}$ & $\varphi_{0,6}$ & $\varphi_{0,6}$ & $\varphi_{0,6}$ & $\varphi_{0,6}$ & $\varphi_{0,6}$ & $\varphi_{0,6}$ & $\varphi_{0,5}$ & $\varphi_{0,4}$ & $\varphi_{0,4}$ &  & $\varphi_{0,4}$ & $\varphi_{0,3}$ & $\varphi_{0,3}$ & $\varphi_{0,2}$ \\
 &&&&&&& $\varphi_{3,4}$ & & & & & & &&&&&&&&\\
\hline
 \end{tabular}
 }
 \caption{The functions $\varphi_{i,j}$ that are permissible on each loop.}
  \label{fig:permiss22}
  \end{figure}
The output of the algorithm, i.e., the certificate of independence $\theta = \min_{i,j} \{ \varphi_{i,j} + c_{i,j} \}$ together with the assignment function $\alpha \colon \{ \varphi_{i,j} \} \to \{ \beta_k \} \cup \{ \gamma_\ell \}$ is depicted schematically in Figure~\ref{Fig:Config}.

\begin{figure}[H]

\begin{center}
\scalebox{.8}{
\begin{tikzpicture}
\draw (-1.5,12)--(0,12);
\draw (-1.25,12.2) node {\footnotesize $66$};
\draw [ball color=black] (-1,12) circle (0.55mm);
\draw (-0.75,12.2) node {\footnotesize $56$};
\draw [ball color=black] (-0.5,12) circle (0.55mm);

\draw (0.5,12) circle (0.5);
\draw (0.5,12.7) node {\footnotesize $46$};
\draw [ball color=black] (0.75,12.43) circle (0.55mm);

\draw (1,12)--(2,12);
\draw [ball color=black] (1.15,12) circle (0.55mm);
\draw (2.5,12) circle (0.5);
\draw (2.5,12.7) node {\footnotesize $36$};
\draw [ball color=black] (2.75,12.43) circle (0.55mm);

\draw (3,12)--(4,12);
\draw [ball color=black] (3.15,12) circle (0.55mm);
\draw (4.5,12) circle (0.5);
\draw (4.5,12.7) node {\footnotesize $55$};

\draw (5,12)--(6,12);
\draw [ball color=black] (5.15,12) circle (0.55mm);
\draw (5.15,11.8) node {\tiny $2$};
\draw (6.5,12) circle (0.5);
\draw (6.5,12.7) node {\footnotesize $26$};
\draw [ball color=black] (6.75,12.43) circle (0.55mm);

\draw (7,12)--(8,12);
\draw [ball color=black] (7.15,12) circle (0.55mm);
\draw (8.5,12) circle (0.5);
\draw (8.5,12.7) node {\footnotesize $45$};
\draw [ball color=black] (8.75,12.43) circle (0.55mm);
\draw (9,12)--(10,12);

\draw [ball color=black] (9.15,12) circle (0.55mm);
\draw (10.5,12) circle (0.5);
\draw (10.5,12.7) node {\footnotesize $44$};
\draw [ball color=black] (10.75,12.43) circle (0.55mm);
\draw [ball color=black] (10.25,12.43) circle (0.55mm);

\draw (11,12)--(12,12);
\draw (12.5,12) circle (0.5);
\draw (12.5,12.7) node {\footnotesize $16$};
\draw (13.25,12.2) node {\footnotesize $35$};
\draw [ball color=black] (12.75,12.43) circle (0.55mm);
\draw [ball color=black] (12.25,12.43) circle (0.55mm);
\draw (13,12)--(14,12);
\draw [ball color=black] (13.5,12) circle (0.55mm);
\draw (14.5,12) node {$\cdots$};
\end{tikzpicture}
}

\bigskip

\scalebox{.78}{
\begin{tikzpicture}
\draw (-7,6) node {$\cdots$};
\draw (-6.5,6)--(-6,6);
\draw (-5.5,6) circle (0.5);
\draw (-5.5,6.7) node {\footnotesize $25$};
\draw [ball color=black] (-5.25,6.43) circle (0.55mm);

\draw (-5,6)--(-4,6);
\draw [ball color=black] (-4.85,6) circle (0.55mm);
\draw (-3.5,6) circle (0.5);
\draw (-3.5,6.7) node {\footnotesize $34$};
\draw [ball color=black] (-3.75,6.43) circle (0.55mm);

\draw (-3,6)--(-2,6);
\draw [ball color=black] (-2.85,6) circle (0.55mm);
\draw (-1.5,6) circle (0.5);
\draw (-1.5,6.7) node {\footnotesize $24$};
\draw [ball color=black] (-1.75,6.43) circle (0.55mm);
\draw [ball color=black] (-1.25,6.43) circle (0.55mm);

\draw (-1,6)--(0,6);
\draw (0.5,6) circle (0.5);
\draw (0.5,6.7) node {\footnotesize $15$};
\draw [ball color=black] (0.75,6.43) circle (0.55mm);
\draw [ball color=black] (0.25,6.43) circle (0.55mm);

\draw (1,6)--(2,6);
\draw (2.5,6) circle (0.5);
\draw (2.5,6.7) node {\footnotesize $33$};

\draw (3,6)--(4,6);
\draw [ball color=black] (3.15,6) circle (0.55mm);
\draw (4.5,6) circle (0.5);
\draw (4.5,6.7) node {\footnotesize $14$};
\draw (3.15,5.8) node {\tiny $2$};
\draw [ball color=black] (4.75,6.43) circle (0.55mm);
\draw [ball color=black] (4.25,6.43) circle (0.55mm);

\draw (5,6)--(6,6);
\draw (6.5,6) circle (0.5);
\draw (6.5,6.7) node {\footnotesize $23$};
\draw [ball color=black] (6.25,6.43) circle (0.55mm);
\draw [ball color=black] (6.75,6.43) circle (0.55mm);

\draw (7,6)--(8,6);
\draw (8.5,6) circle (0.5);
\draw (8.5,6.7) node {\footnotesize $06$};
\draw [ball color=black] (8.25,6.43) circle (0.55mm);
\draw (9,6)--(9.5,6);
\draw [ball color=black] (9.15,6) circle (0.55mm);
\draw (9.35,6.2) node {\footnotesize $05$};
\draw (10,6) node {$\cdots$};
\end{tikzpicture}
}
\bigskip

\scalebox{.8}{
\begin{tikzpicture}
\draw (-4.5,3) node {$\cdots$};
\draw (-4,3)--(-3,3);
\draw [ball color=black] (-3.5,3) circle (0.55mm);
\draw (-2.5,3) circle (0.5);
\draw (-2.5,3.7) node {\footnotesize $22$};
\draw [ball color=black] (-2.75,3.43) circle (0.55mm);
\draw [ball color=black] (-2.25,3.43) circle (0.55mm);

\draw (-2,3)--(-1,3);
\draw (-.5,3) circle (0.5);
\draw (-.5,3.7) node {\footnotesize $04$};
\draw [ball color=black] (-.75,3.43) circle (0.55mm);
\draw [ball color=black] (-.25,3.43) circle (0.55mm);

\draw (0,3)--(1,3);
\draw (1.5,3) circle (0.5);
\draw [ball color=black] (1.25,3.43) circle (0.55mm);
\draw [ball color=black] (1.75,3.43) circle (0.55mm);

\draw (2,3)--(3,3);
\draw (3.5,3) circle (0.5);
\draw (3.5,3.7) node {\footnotesize $13$};
\draw [ball color=black] (3.25,3.43) circle (0.55mm);

\draw (4,3)--(5,3);
\draw [ball color=black] (4.15,3) circle (0.55mm);
\draw [ball color=black] (5.75,3.43) circle (0.55mm);
\draw (5.5,3) circle (0.5);
\draw (5.5,3.7) node {\footnotesize $12$};
\draw [ball color=black] (5.25,3.43) circle (0.55mm);

\draw (6,3)--(7,3);
\draw [ball color=black] (7.75,3.43) circle (0.55mm);
\draw (7.5,3) circle (0.5);
\draw (7.5,3.7) node {\footnotesize $11$};
\draw [ball color=black] (7.25,3.43) circle (0.55mm);

\draw (8,3)--(9,3);
\draw (9.5,3) circle (0.5);
\draw (9.5,3.7) node {\footnotesize $03$};
\draw [ball color=black] (9.25,3.43) circle (0.55mm);
\draw [ball color=black] (10.15,3) circle (0.55mm);

\draw (10,3)--(11.75,3);
\draw [ball color=black] (10.65,3) circle (0.55mm);
\draw [ball color=black] (11.2,3) circle (0.55mm);
\draw (10.4,3.2) node {\footnotesize $02$};
\draw (10.95,3.2) node {\footnotesize $01$};
\draw (11.5,3.2) node {\footnotesize $00$};
\end{tikzpicture}
}
\end{center}
\caption{The divisor $D' = 2D + \ddiv (\theta)$.  The function $\varphi_{i,j}$ is assigned to the bridge or loop directly under the label $ij$, and $\varphi_{i,j} + c_{i,j}$ achieves the minimum uniquely on the labeled region in $\Gamma \smallsetminus \mathrm{Supp}(D')$.}
\label{Fig:Config}
\end{figure}

The graph should be read from left to right and top to bottom, so the first 7 loops appear in the first row, with $\gamma_1$ on the left and $\gamma_7$ on the right, and $\gamma_{22}$ is the last loop in the third row.   The function $\varphi_{i,j}$ is assigned to the bridge or loop directly under the label $ij$. The 48 filled dots indicate the support of the divisor $D' = 2D + \ddiv (\theta)$. Note that $\deg(D') = 50$; the points on the bridges $\beta_4$ and $\beta_{13}$ appear with multiplicity 2, as marked.  Each of the 28 functions $\varphi_{i,j} + c_{i,j}$ achieves the minimum uniquely on the connected component of the complement of $\mathrm{Supp}(D')$ labeled $ij$.

This example has several interesting features that do not appear in Example~\ref{Ex:CanonicalIndependence}. Several functions, such as $\varphi_{35}$ and $\varphi_{13}$, achieve the minimum on parts of some loops to the left of the bridge or loop to which they are assigned. Nevertheless, no function achieves the minimum on any loop to the right of where it is assigned. There are loops with no permissible departing functions, the first of which is $\gamma_6$.  There are also loops with a departing permissible function that is assigned to an earlier loop. The first of these is $\gamma_{11}$; the function $\varphi_{2,4}$ is departing on $\gamma_{11}$, but it is assigned to $\gamma_{10}$. On such loops, we follow Step 4 in the loop subroutine to adjust the coefficients and assign a function.  Note also that there are functions assigned to the internal bridges $\beta_8$ and $\beta_{16}$.  These are exactly the bridges $\beta_k$, with $2 \leq k \leq g$, for which $\sigma_k < \sigma_{k-1}$.
\end{example}

\begin{remark} \label{rem:variants}
Many natural variants of this algorithm are possible, even in the vertex avoiding case. For instance, our algorithm for constructing a template from pairwise sums of building blocks in the general case (\S\ref{sec:template}), although closely modeled on this algorithm, varies slightly in its treatment of certain special loops. In particular, when applied to the 28 functions $\varphi_{i,j}$ in Example~\ref{ex:randomtableau}, the template algorithm produces an assignment function and a certificate of independence but not necessarily the same one constructed here.  More specifically, depending on $x_{18}(D)$, the lingering loop $\gamma_{18}$ may or may not be ``skippable'' in the sense of Definition~\ref{Def:Skippable}.  When $\gamma_{18}$ is not skippable, the general algorithm will assign a function to achieve the minimum on this loop.

In our work on curves of genus 13, we introduced a different variant of the basic algorithm, in which functions are systematically assigned to lingering loops \cite{M13}. In view of such applications, we see the flexibility of the basic construction as a desirable feature.
\end{remark}

We now resume our discussion of the basic algorithm in the vertex avoiding case and prove that the output
$
\theta =  \min \{ \psi + c(\psi) \}
$
is a certificate of independence.

\begin{lemma}
\label{Lem:VANotToRight}
Suppose that $\psi$ is assigned to the loop $\gamma_k$ or the bridge $\beta_k$, and let $\ell > k$.  If there is an unassigned permissible function on $\gamma_\ell$, then $\psi + c(\psi)$ is strictly greater than $\theta$ on $\gamma_\ell$.
\end{lemma}

\begin{proof}
Let $v \in \gamma_{\ell}$.  Let $\psi'$ be an unassigned permissible function on $\gamma_{\ell}$, and let $\ell'$ be the smallest integer in the range $\ell \leq \ell' \leq k$ such that $\psi'$ is permissible on $\gamma_{\ell'}$.  If $s_{k'} (\psi) > s_{k'} (\psi')$ for some $k'$ in the range $\ell' < k' \leq \ell$, then because $\Gamma$ has $C$-admissible edge lengths, we see that $\psi (v) + c(\psi) > \psi' (v) + c(\psi') \geq \theta (v)$.

Otherwise, by construction, we have
\[
\psi (v) + c(\psi) - \psi'(v) + c(\psi') \geq \frac{1}{3}m_k - d \sum_{t=\ell'}^{\ell} m_t .
\]
Since $\Gamma$ has $C$-admissible edge lengths, this expression is positive.
\end{proof}

\begin{lemma}
\label{Lem:VAConfig}
Suppose that $\psi$ is assigned to the loop $\gamma_k$.  Then there is a point $v \in \gamma_k$ where $\psi + c(\psi)$ achieves the minimum uniquely.
\end{lemma}

\begin{proof}
If there is an unassigned departing function $\psi$ on $\gamma_k$, then by construction $\psi + c(\psi)$ is the only function that achieves the minimum at $w_k$.  Otherwise, by Lemma~\ref{Lem:HalfInteger}, any two permissible functions differ by an integer multiple of $\frac{1}{2}m_k$ at points whose distance from $w_k$ is a half-integer multiple of $m_k$.  By construction, there is such a point where $\psi + c(\psi)$ achieves the minimum uniquely and, after increasing the coefficient by $\frac{1}{3}m_k$, it still does.
\end{proof}

\begin{corollary}
\label{Cor:VAIndependence}
The tropical linear combination $\theta = \min \{ \psi + c(\psi) \mid \alpha(\psi) \neq \emptyset \}$ is a certificate of independence for $\{ \psi \in \cB\mid \alpha(\psi) \neq \emptyset \}$.
\end{corollary}

\begin{proof}
By Lemma~\ref{Lem:VAConfig}, each function $\psi$ assigned to a loop achieves the minimum uniquely at some point of that loop.  Similarly, since the set of functions assigned to a bridge have distinct slopes along that bridge, each achieves the minimum uniquely at some point of that bridge.  It follows that $\theta = \min \{ \psi + c(\psi) \mid \alpha(\psi) \neq \emptyset \}$ is a certificate of independence for $\{ \psi \in \cB\mid \alpha(\psi) \neq \emptyset \}$.
\end{proof}

\subsection{The cases where $r = 6$ and $g = 21 + \rho$}  \label{sec:VArelevant}
We now address how the algorithm applies in cases specific to the main results of this paper.  We continue the assumption that $\Gamma$ is a chain of $g$ loops, $D$ is a fixed vertex avoiding break divisor of degree $d$ and rank $r$, and the edge lengths of $\Gamma$ are $C$-admissible for some $C > 12 dg$.

Now we also assume  that $r = 6$, $g = 21 + \rho$, and $d = 24 + \rho$, for some non-negative integer $\rho$. We explain how to choose the function $\sigma$ so that the algorithm produces an independence among all 28 functions $\varphi_{i,j}$, for $0 \leq i \leq j \leq 6$.  By Corollary~\ref{Cor:VAIndependence}, it suffices to show that, for our choice of $\sigma$, every function $\varphi_{i,j}$ is assigned to some bridge or loop.

\medskip

Let $T$ be the tableau associated to $D$, as discussed in \S\ref{Sec:VADivisors}, above.

\begin{definition}\label{def:z}
Let $z$ be the 6th smallest entry appearing in the first two rows of $T$, and let $z'$ be the 10th smallest entry appearing in the last two rows.  We then set
\begin{displaymath}
\sigma_k = \left\{ \begin{array}{ll}
4 & \textrm{if $1 \leq k \leq z$,} \\
3 & \textrm{if $z + 1 \leq k \leq z'-2$,}\\
2 & \textrm{if $z' -1 \leq k \leq g+1$.}
\end{array} \right.
\end{displaymath}
\end{definition}

\begin{proposition} \label{Prop:VAVerification}
When applied to $\Gamma$, $D$, and $\sigma$ as above, with $\cB = \{ \varphi_{i,j} \mid 0 \leq i \leq j \leq 6 \}$, the algorithm outputs a certificate of independence together with a function $\alpha$ that assigns exactly one function to each of the 21 non-lingering loops.  The remaining 7 functions are assigned to the bridges $\beta_1$, $\beta_{z+1}$, $\beta_{z'-1}$, and $\beta_{g+1}$.  More precisely,
\begin{equation} \label{eq:assignments}
|\alpha^{-1}(\beta_k)| = \left\{ \begin{array}{ll}
2 & \textrm{if $k = 1$,} \\
1 & \textrm{if $k \in \{ z + 1, z' -1 \}$}, \\
3 & \textrm{if $k = g+1$,} \\
0 & \textrm{otherwise.}
\end{array} \right.
\end{equation}
\end{proposition}

\noindent In Example~\ref{ex:randomtableau}, we have $z = 7$ and $z' = 17$, and the functions assigned to bridges are:
\[
\alpha^{-1}(\beta_1) = \{ \varphi_{6,6}, \varphi_{5,6} \}, \quad \alpha^{-1}(\beta_8) = \{ \varphi_{3,5} \}, \quad \alpha^{-1}(\beta_{16}) = \{ \varphi_{0,5} \}, \quad \alpha^{-1}(\beta_{23}) = \{\varphi_{0,2}, \varphi_{0,1}, \varphi_{0,0} \}.
\]

We now state and prove several lemmas in preparation for the proof of Proposition~\ref{Prop:VAVerification}.

\begin{lemma}
\label{Lem:VAAtMostThree}
There are at most 3 non-departing permissible functions on each loop.
\end{lemma}

\begin{proof}
Fix $k$.  If $\varphi_{i,j}$ is a non-departing permissible function on $\gamma_k$, then $s_{k+1} (\varphi_{i,j}) = \sigma_k$.  For each $i$, this equality holds for at most one $j$.
It follows that there are most $\big \lceil \frac{r+1}{2} \big \rceil = 4$ non-departing permissible functions on $\gamma_k$, and 4 is possible only if $2s_{k+1} (\varphi_{3}) = \sigma_k$.  We claim that this never happens. Indeed, if $k \leq z+1$, then $2s_{k+1} (\varphi_{3}) \leq 2 < \sigma_k.$   If $z+1 < k \leq z'-1$, then $\sigma_k$ is odd.  And if $k \geq z'-1$ then $2s_{k+1}(\varphi_3) \geq 4 > \sigma_k$.  This proves the claim, and the lemma follows.
\end{proof}

\begin{lemma}  \label{lem:threeshape}
Let $S \subset \{ \varphi_{i,j} \}$ be a set of two or three non-departing permissible functions on $\gamma_k$. Suppose the coefficients $c(\psi)$ are chosen so that the functions $\{ \psi + c(\psi) \mid \psi \in S \}$ are all equal at $w_k$.  Then there is a point of $\gamma_k$ at which one of these functions is strictly less than the others.
\end{lemma}

\begin{proof}
We consider the case where $|S| = 3$.  The case where $|S| = 2$ is similar, but simpler. Since each $\psi \in S$ is permissible on $\gamma_k$ and non-departing, we have $s_{k} (\psi) \leq \sigma_k = s_{k+1}(\psi)$.

Assume the coefficients $c(\psi)$, for $\psi \in S$, are chosen so that the functions $\psi + c(\psi)$ all agree at $w_k$.  Let $\Upsilon = \min \{ \psi + c(\psi) \mid \psi \in S \}$. Note that $s_k(\Upsilon) \leq \sigma_k = s_{k+1}(\Upsilon)$.  By Lemma~\ref{Lem:DegreeOneLoop}, $\deg\big(\ddiv(\Upsilon)_{|\gamma_k}\big) = s_k(\Upsilon) - s_{k+1}(\Upsilon)$, and it follows that the restriction $(2D + \ddiv (\Upsilon))_{|{\gamma_k}}$ has degree at most 2.  Hence $\gamma_k \smallsetminus \mathrm{Supp} (2D + \ddiv (\Upsilon))$ consists of at most two connected components.  By \cite[Lemma~3.4]{tropicalGP}, any boundary point of a region where one of the tropical summands in $\Upsilon$ achieves the minimum is contained in $\mathrm{Supp} (2D + \ddiv (\Upsilon))$.  Therefore, the region where any summand achieves the minimum is either one of these connected components, or the union of both.

Since all three summands agree at $w_k$, and no two agree on the whole loop $\gamma_k$, we can narrow down the combinatorial possibilities as follows:  either all three agree on one region that contains $w_k$ and one of the three achieves the minimum uniquely on the other region, or two different pairs agree on the two different regions, and $w_k$ is in the boundary of both.  These two possibilities are illustrated in Figure~\ref{Fig:ThreeFunctions}.

\begin{figure}[H]
\begin{tikzpicture}
\matrix[column sep=0.5cm] {

\begin{scope}[grow=right, baseline]
\draw (-1,0) circle (1);
\draw (-3,0)--(-2,0);
\draw (0,0)--(1,0);
\draw [ball color=black] (-0.29,0.71) circle (0.55mm);
\draw [ball color=black] (-1.71,0.71) circle (0.55mm);
\draw (-1,-1.25) node {{\tiny $\psi, \psi', \psi''$}};
\draw (-1,1.25) node {{\tiny $\psi$}};

\draw (4,0) circle (1);
\draw (2,0)--(3,0);
\draw (5,0)--(6,0);
\draw [ball color=black] (5,0) circle (0.55mm);
\draw [ball color=black] (3.29,0.71) circle (0.55mm);
\draw (4,-1.25) node {{\tiny $\psi, \psi'$}};
\draw (4,1.25) node {{\tiny $\psi, \psi''$}};

\end{scope}

\\};
\end{tikzpicture}

\caption{Two possibilities for where 3 functions may achieve the minimum.}
\label{Fig:ThreeFunctions}
\end{figure}
In either case, there is one summand $\psi$ that achieves the minimum on all of $\gamma_k$.  Furthermore, all three have the same slope along the bridge $\beta_{k+1}$, so $\psi + c(\psi)$ also achieves the minimum on $\beta_{k+1}$.  Therefore $\Upsilon$ is equal to $\psi$ in a neighborhood of $w_k$.  Since $D$ is vertex avoiding, $2D+ \ddiv (\psi)$ does not contain $w_k$ (Proposition~\ref{Prop:VAiff}).  Therefore, $2D+ \ddiv ( \Upsilon )$ does not contain $w_k$ either.  This rules out the second case; we conclude that all three summands achieve the minimum on a region that includes $w_k$, and one achieves the minimum uniquely on the nonempty complementary region, as shown on the left in Figure~\ref{Fig:ThreeFunctions}.  This proves the lemma.
\end{proof}

Recall that the loops $\gamma_k$ on which $\psi \in \cB$ is permissible are consecutive.

\begin{definition} \label{def:new}
We say that $\psi \in \cB$ is a \emph{new permissible function} on $\gamma_k$ if $\psi$ is permissible on $\gamma_k$ and is not permissible on $\gamma_j$ for $j < k$.
\end{definition}

\begin{lemma}
\label{lem:onenew}
If $k \not \in \{ 1, z+1, z'-1 \}$ then there is at most one new permissible function on $\gamma_k$.  If, furthermore, $\gamma_k$ is lingering then there are none.
\end{lemma}

\begin{proof}
Recall that $\psi \in \cB$ is permissible on $\gamma_k$ if and only if $s_{k} (\psi) \leq \sigma_k \leq s_{k +1}(\psi)$.  Suppose $k \not \in \{ 1, z+1, z'-1 \}$.  Then $\sigma_k = \sigma_{k-1}$, so if $s_k (\psi) = \sigma_k$, then $\psi$ is permissible on $\gamma_{k-1}$.  Thus, if $\psi$ is a new permissible function, then $s_k(\psi) < \sigma_k \leq s_{k+1}(\psi)$.  Hence, if $\psi = \varphi_{i,j}$ then $k$ must be in the $i$th or $j$th column of $T$.

Suppose $k$ is in the $i$th column of $T$.  Then there is at most one $j$ (possibly equal to $i$) such that $s_{k}(\varphi_{i,j}) < \sigma_k$ and $s_{k+1}(\varphi_{i,j}) \geq \sigma_k$, and hence at most one new permissible function $\varphi_{i,j}$.
\end{proof}

In order to prove Proposition~\ref{Prop:VAVerification}, we must keep track of the non-lingering loops where there are no new permissible functions. These will be the loops numbered $b$ and $b'$, characterized as follows.

\begin{definition} \label{def:b}
Let $b$ be the 7th smallest entry appearing in the first two rows of the tableau $T$ and let $b'$ be the 8th smallest symbol appearing in the union of the first and third row.
\end{definition}

\noindent We note that $z < b < b' \leq z'-2.$  The first two inequalities are straightforward.  To see the last inequality, recall from Definition~\ref{def:z} that $z'$ is the 10th smallest entry that appears in the union of the second and third row.  Therefore, the 9th smallest symbol appearing in the union of the second and third row must be strictly between $b'$ and $z'$.

\begin{example}
In Example~\ref{ex:randomtableau}, we have $b = 9$ and $b' = 11$.   Note that the number of unassigned permissible functions dropped from $4$ to $3$ and from $3$ to $2$ on $\gamma_9$ and $\gamma_{11}$, respectively.
\end{example}

\begin{proposition}  \label{prop:nonew}
If $b \neq z+1$, then the non-lingering loops with no new permissible functions are exactly $\gamma_z$, $\gamma_b$, $\gamma_{b'}$, and $\gamma_{z'-2}$ and there are exactly 4 permissible functions on $\gamma_{z+1}$.  Otherwise, if $b = z+1$, the non-lingering loops with no new permissible functions are exactly $\gamma_z$, $\gamma_{b'}$, and $\gamma_{z'-2}$, and there are exactly 3 permissible functions on $\gamma_{z+1}$.
\end{proposition}

\begin{proof}
We begin by showing that there are no new permissible functions on $\gamma_z$.    Suppose $\varphi_{i,j}$ is a new permissible function on $\gamma_z$, i.e., $s_{z} (\varphi_{i,j}) < \sigma_z = 4 \leq s_{z+1} (\varphi_{i,j})$.  We show that this is impossible.

Recall that $z$ is the 6th smallest entry appearing in the first two rows of $T$ (Definition~\ref{def:z}).  There are 4 possibilities for the location of these entries, corresponding to partitions of 6 into no more than 2 parts.  We consider the case where the partition is $(4,2)$; the other three cases are similar.
\begin{figure}[H]
\begin{tikzpicture}
\matrix[column sep=0.7cm, row sep = 0.7cm] {
\begin{scope}[node distance=0 cm,outer sep = 0pt]
	      \node[bsq] (11) at (2.5,0) {};
	      \node[bsq] (21) [below = of 11] {};
	      \node[bsq] (12) [right = of 11] {};
	      \node[bsq] (22) [below = of 12] {};
	      \node[bsq] (13) [right = of 12] {};
	      \node[bsq] (14) [right = of 13] {};
\end{scope}
\\};
\end{tikzpicture}
\caption{The Young diagram corresponding to the partition $(4,2)$.}
\label{Fig:Partition}
\end{figure}
When the 6 smallest entries in the first two rows occupy the Young diagram corresponding to $(4,2)$, then $z$ is either the 4th entry in the first row, or the 2nd entry in the second row.  Suppose $z$ is the $2$nd entry in the second row. Then $s_{z+1} (\Sigma) = s_z (\Sigma) + (0,0,0,0,0,1,0)$, and $s_{z+1}(\Sigma)$ is either
\[
(-3,-2,-1,1,2,4,5) \mbox{ \  or \ } (-3,-2,-1,1,2,4,6),
\]
depending on whether the 3rd box in the first column is greater than $z$.  By inspection, there is no pair of indices $(i,j)$ such that $s_z[i] + s_z[j] < 4$ and $s_{z+1}[i] + s_{z+1}[j] \geq 4$, so there is no new permissible function on $\gamma_z$.  Similarly, if $z$ is the $4$th entry in the first row, then $s_{z+1} (\Sigma)= s_z + (0,0,0,1,0,0,0)$, and $s_{z+1}(\Sigma)$ is one of the following:
\[
(-3,-2, -1,1, 2,4,5), \ (-3,-2,-1,1,2,4,6), \mbox{ \ or \ } (-3,-2,-1,1,2,5,6),
\]
depending on whether the 3rd box in each of the first two columns is greater than $x$.  Once again, by inspection, there is no new permissible function on $\gamma_z$.

The proofs that $\gamma_{b'}$ and $\gamma_{z'-2}$ have no new permissible functions are similar, as is the proof that $\gamma_b$ has no new permissible functions if $b \neq z+1$.  If $b = z+1$, then a similar argument shows that there is no permissible function $\varphi_{i,j}$ with $s_{b} (\varphi_{i,j}) < 3$, and a case-by-case examination shows that there are only 3 permissible functions on $\gamma_b$.

It remains to prove that these are the only loops with no new permissible functions.  We do so by a counting argument.  There are 4 functions that are not permissible on any loop: $\varphi_{0,0}$, $\varphi_{0,1}$, $\varphi_{5,6}$, and $\varphi_{6,6}$.  There are exactly 3 permissible functions on the first non-lingering loop: $\varphi_{3,6}$, $\varphi_{4,6}$, and $\varphi_{5,5}$. If $b \neq z+1$ and $k$ is the smallest integer such that $\gamma_k$ is not a lingering loop and $k \geq z+1$, then there are 4 permissible functions on $\gamma_{k}$.  If $b=z+1$, there are 3.  Finally, $k'$ is the smallest integer such that $\gamma_{k'}$ is not a lingering loop and $k' \geq z-1$, then there are 3 permissible functions on $\gamma_{k'}$.  Thus, if $b \neq z+1$, there are $28-4-3-4-3 = 14$ functions that are each permissible on some loop, but not permissible on any of these three.  Similarly, if $b = z+1$, there are $28-4-3-3-3 = 15$ such functions.  By Lemma~\ref{lem:onenew}, each of these functions must be new on a distinct non-lingering loop.  The number of non-lingering loops is 21.  None of these is new on $\gamma_1$, $\gamma_k$, or $\gamma_{k'}$, and no function is new on $\gamma_z$, $\gamma_b$, $\gamma_{b'}$, or $\gamma_{z'}$.  Thus, the number non-lingering loops in the complement of this set is 14 when $b \neq z+1$ and 15 when $b=z+1$.  Hence each has a new permissible function, as required.
\end{proof}

\begin{proof}[Proof of Proposition~\ref{Prop:VAVerification}.]
 By Corollary~\ref{Cor:VAIndependence}, it suffices to show that, for our choice of $\sigma$, every function $\varphi_{i,j}$ is assigned to some bridge or loop.  By construction, $\alpha(\varphi_{6,6}) = \alpha(\varphi_{5,6}) = \beta_1$, and $\alpha(\varphi_{0,1}) = \alpha(\varphi_{0,0}) = \beta_{g+1}$.  On the first non-lingering loop, there are 3 permissible functions by Proposition~\ref{prop:nonew}.  On each non-lingering loop $\gamma_k$ for $k<z$, there is one new permissible function.  To each such loop, we assign a function $\psi \in \cB$.  Moreover, if there is an unassigned departing permissible function on $\gamma_k$, we assign it to $\gamma_k$.  It follows that there are 3 unassigned permissible functions on each non-lingering loop $\gamma_k$ with $k<z$, and on $\gamma_z$, there are 2.  These two functions are assigned to the loop $\gamma_z$ and the bridge $\beta_{z+1}$.

A similar analysis, using Proposition~\ref{prop:nonew}, shows that we assign a function to every non-lingering loop $\gamma_k$ with $k > z$, and we assign one function to $\beta_{z'-1}$ and one additional function (aside from $\varphi_{0,1}$ and $\varphi_{0,0}$) to $\beta_{g+1}$. Thus the number of functions assigned to each bridge or loop is as specified in \eqref{eq:assignments}. In particular, since there are precisely 21 non-lingering loops, the total number of functions assigned to a bridge or loop is 28. Hence, the output of the algorithm is a certificate of independence among all 28 functions in $\cB = \{ \varphi_{i,j} \mid 0 \leq i \leq j \leq 6\}$.
\end{proof}

\section{Beyond the vertex avoiding case}

As in the previous sections, we consider a break divisor $D$ of degree $d$ on a chain of loops $\Gamma$ with $C$-admissible edge lengths for some $C > 12dg$, and a tropical linear series $\Sigma \subset R(D)$ of rank $r$.  Here, we consider the case where $D$ is not necessarily vertex avoiding.  Let $2\Sigma$ denote the sumset $2 \Sigma := \{ \varphi + \varphi' \mid \varphi, \varphi' \in \Sigma \}$.

\begin{theorem} \label{thm:generalv1}
Assume $g = 22$ or $23$, $d = g+3$, and $r = 6$.  Then there is a tropically independent set $\cT \subset 2\Sigma$ of size $|\cT| = 28$.
\end{theorem}

We have already proved this in the special case where $D$ is vertex avoiding by algorithmically producing a certificate of independence among the 28 pairwise sums of distinguished functions $\varphi_{i,j} = \varphi_i + \varphi_j$.  Our proof in the general case follows that construction as much as possible, but is necessarily more technical. Note that Theorem~\ref{Thm:MainThm} is an immediate consequence of Theorem~\ref{thm:generalv1}.

\medskip

In this section, in preparation for the proof of Theorem~\ref{thm:generalv1}, we:
\begin{itemize}
\item introduce multiplicities of loops that quantify how the behavior of functions in $\Sigma$, on each loop of $\Gamma$, differs from that on typical (non-lingering) loops in the vertex-avoiding case
\item introduce multiplicities of bridges that similarly quantify how the behavior of functions in $\Sigma$, on each bridge of $\Gamma$, differ from the vertex-avoiding case
\item define ``switching bridges'' and ``switching loops'' exhibiting one essential new phenomenon that does not appear in the vertex-avoiding case
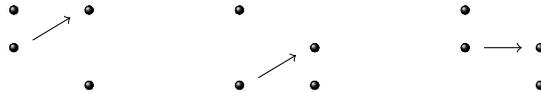
\item classify switching bridges and  switching loops of multiplicity at most 2.
\end{itemize}

In \S\ref{sec:adjustedBN}, we also define ramification weights of tropical linear series at the endpoints of $\Gamma$ and adjusted Brill-Noether numbers that take these weights into account. We show that the adjusted Brill-Noether number of a tropical linear series on $\Gamma$ is equal to the sum of the multiplicities of all loops and bridges.  We then state a variant of Theorem~\ref{thm:generalv1} for certain tropical linear series on chains of loops of smaller genus with prescribed ramification, which is needed for our proof of Theorem~\ref{thm:genfinite}.  The tropical linear series that appear have adjusted Brill-Noether at most 2, and the classification of loops and bridges of multiplicity at most 2 is used in the proofs of these results. The section concludes  in \S\ref{sec:strategy} with an overview of these proofs.

\subsection{Slope vectors}
Recall that we write $s_k(\varphi)$ and $s'_k(\varphi)$ for the rightward slopes of $\varphi \in \PL(\Gamma)$ along the bridges incident to $v_k$ and $w_k$, respectively, as in Figure~\ref{Fig:Slopes}.

\begin{definition} \label{def:slopevectors}
Let
\[
s_k(\Sigma) := (s_k[0], \ldots, s_k[r]) \mbox{ \ \ and \ \ } s'_k(\Sigma) := (s'_k[0], \ldots, s'_k[r] )
\]
be the vectors of slopes in $\{ s_k(\varphi) \mid \varphi \in \Sigma\}$ and $\{ s'_k(\varphi) \mid \varphi \in \Sigma\}$, respectively, ordered so that  $s_k [0] <  \cdots < s_k[r]$ and $s'_k [0] <  \cdots < s'_k[r]$.
\end{definition}

We denote the changes in these slopes, as one moves from left to right across the graph, by
\[
\delta_k[i] := s'_k[i] - s_k[i] \quad \quad \mbox{and} \quad \quad \delta'_k[i] := s_k[i]-s'_{k-1}[i].
\]
These changes in slopes are bounded above as follows.

\begin{proposition} \label{Prop:LingeringLatticePath}
Let $D$ be a break divisor of degree $d$ on $\Gamma$ with coordinates $x_1, \ldots, x_g$, and let $\Sigma \subset R(D)$ be a tropical linear series of rank $r$.  Then $\delta_k[i] \leq 1$ and $\delta'_k[i] \leq 0$, for all $k$ and $i$.  Moreover, if $\delta_k[i] = 1$ then $x_k(D) = (s_{k}[i]+1)m_k$.
\end{proposition}

\begin{proof}
First, by Lemma~\ref{Lemma:Existence}, there is a function $\varphi \in \Sigma$ such that $s_k(\varphi) \leq s_k[i]$ and $s'_k(\varphi) \geq s_k[i]$.  Let $\delta_k(\varphi) = s'_k(\varphi) - s_k(\varphi)$.  Then $\delta_k(\varphi) \geq \delta_k[i]$, and $\deg(D + \ddiv(\varphi))_{|\gamma_k} = - \delta_k(\varphi)$.  Since $D + \ddiv(\varphi)$ is effective, we conclude that $\delta_k[i] \leq \delta_k(\varphi) \leq 1$.  The proof that $\delta'_k(\varphi) \leq 0$ is similar.  If $\delta_k[i] = 1$, then $\delta_k(\varphi) = 1$ and hence, by Lemma~\ref{Lem:PointsOfSlopeIncrease}, we have $x_k(D) = (s_{k}[i]+1)m_k$, as required.
\end{proof}

\subsection{Multiplicities of loops and bridges}
It follows from Proposition~\ref{Prop:LingeringLatticePath} that  $\sum_{i=0}^r \delta_k[i] \leq 1$ and $\sum_{i=0}^r \delta'_k[i] \leq 0$.  We define the \emph{multiplicities} of $\gamma_k$ and $\beta_k$ to be the amount by which these sums differ from their respective upper bounds.

\begin{definition}\label{def:multiplicities}
The multiplicity of $\gamma_k$ and $\beta_k$ are defined, respectively, to be
\[
\mu (\gamma_k) := 1 - \sum_{i=0}^r \delta_k[i] \mbox{ \ and \ }
\mu (\beta_k) := -\sum_{i=0}^r \delta'_k[i].
\]
\end{definition}

\begin{example}
If $D$ is vertex avoiding then $\mu(\beta_k) = 0$ for all $k$ and $\mu(\gamma_k) = 0$ unless $\gamma_k$ is lingering, in which case $\mu(\gamma_k) = 1$.
\end{example}

\subsection{Ramification weights}

The slopes of functions in $\Sigma$ at the left and right endpoints of $\Gamma$ are bounded as follows.

\begin{lemma}
The slope vectors at the left and right endpoints of $\Gamma$ are bounded by $s'_1[i] \leq (d-g-r + i)$ and $s_{g+1}[i] \geq i$, respectively.
\end{lemma}

\begin{proof}
Since $D$ is a break divisor, we have $D(w_0) = d - g$. By Lemma~\ref{Lemma:InductiveExistence}, $\{ \varphi \in \Sigma \mid s'_1(\varphi) \geq s'_1[i] \}$ contains a tropical linear series of rank $r - i$. Then, by the definition of a tropical linear series of rank $r-i$, there is a function $\varphi \in \Sigma$ such that $(D + \ddiv(\varphi))(w_0) \geq r -i $. Hence $s'_1[i] \leq d-g- r + i$.  The bound on $s_{g+1}[i]$ is proved similarly.
\end{proof}

\begin{definition} \label{def:ramificationweights}
The \emph{ramification weights} of $\Sigma$ at $w_0$ and $v_{g+1}$ are
\[
\mathrm{wt} (w_0) := \sum_{i=0}^r (d-g -r + i - s'_0 [i]) \mbox{ \ \ and \ \ } \mathrm{wt} (v_{g+1}) := \sum_{i=0}^r (s_{g+1} [i] - i).
\]
\end{definition}

\begin{example}
If $D$ is vertex avoiding, then $\mathrm{wt}(w_0)$ and $\mathrm{wt}(v_{g+1})$ are both 0.
\end{example}

\begin{example} \label{ex:ramified}
Suppose $X$ is a curve with skeleton $\Gamma$ and $\Sigma$ is the tropicalization of a linear series $V \subset H^0(X, \OO(D_X))$ of degree $d$ and rank $r$. Recall that the nonzero sections in $V$ have exactly $r +1$ distinct orders of vanishing at any given point $p \in X$, denoted $a^V_0(p) < \cdots < a^V_r(p)$.  If $p$ specializes to $v_{g+1}$, then an induction on $r$ shows that $s_{g+1} [i] \geq a^V_i (p),$ for all $i$.  It follows that $\mathrm{wt}(v_{g+1})$ is greater than or equal to the ramification weight $\sum_i a^V_i(p) - i$ of $V$ at $q$.  A similar argument shows that $\mathrm{wt}(w_0)$ is bounded below by the ramification weight of any point of $X$ specializing to $w_0$.
\end{example}

\subsection{The adjusted Brill-Noether number} \label{sec:adjustedBN}

If $D$ is vertex avoiding then both ramification weights are $0$, as is the multiplicity of every bridge and every non-lingering loop, and there are exactly $\rho$ lingering loops, each of multiplicity 1. The analogous statement in the general case is as follows.

\begin{proposition}
\label{Thm:BNThm}
The sum of the multiplicities of all loops and bridges plus the ramification weights at $w_0$ and $v_{g+1}$ is equal to the Brill-Noether number $\rho = g-(r+1)(g-d+r)$.
\end{proposition}

\begin{proof}
Starting from the definitions of the multiplicities of the loops and bridges and then collecting and canceling terms, we have
\[
\sum_{k=1}^{g} \mu (\gamma_k) + \sum_{k=1}^{g+1} \mu (\beta_k) =
g + \sum_{i=0}^r (s'_0 [i] - s_{g+1} [i]) .
\]
Moreover,
\[
\mathrm{wt} (w_0) = (r+1)(d-g) - \binom{r+1}{2} - \sum_{i=0}^r s'_0 [i], \mbox { \ \ and \ \ }
\mathrm{wt} (v_{g+1}) = \sum_{i=0}^r s_{g+1}[i] - \binom{r+1}{2}.
\]
Adding these together and again collecting and canceling terms gives $g-(r+1)(g-d+r)$. \end{proof}

\begin{definition} \label{def:adjustedBN}
The adjusted Brill-Noether number of a linear series $\Sigma \subset R(D)$ on $\Gamma$ is
\[
\rho' := g - (r+1)(g-d+r) - \wt(w_0) - \wt(v_{g+1}).
\]
\end{definition}

This is a two-pointed tropical analogue of the adjusted Brill-Noether number of a linear series on a pointed curve, considered in \S\ref{virtualis_divizorok}. By Proposition~\ref{Thm:BNThm}, the adjusted Brill-Noether number $\rho'$ is equal to the sum of the multiplicities of the loops and bridges.

\begin{theorem} \label{thm:generalv2}
Let $g = 22$ or $23$, let $d = g + 3$ and let $g' \in \{ g, g-1, g-2 \}$.  Let $\Gamma'$ be a chain of $g'$ loops with $C$-admissible edge lengths for some $C > 12dg$, and let $D$ be a break divisor on $\Gamma'$ with $\Sigma \subset R(D)$ a tropical linear series of rank $6$.  Assume furthermore that
\begin{enumerate}
\item if $g' = g-1$ then $s'_0[5] \leq 2$;
\item if $g' = g-2$ then either $s'_0[5] \leq 2$ or $s'_0[6] + s'_0[4] \leq 5$.
\end{enumerate}
Then there is a tropically independent subset $\cT \subset 2 \Sigma$ of size $|\cT| = 28$.
\end{theorem}

\noindent In all of the cases covered by Theorem~\ref{thm:generalv2}, the adjusted Brill-Noether number $\rho'$ is at most 2. Theorem~\ref{thm:generalv1} is the special case of Theorem~\ref{thm:generalv2} where $g' = g$.  This generalization is used in the proof of Theorem~\ref{thm:genfinite}; the proof is essentially the same, with just a little more bookkeeping. The additional cases arise when studying the tropicalization of a limit linear series of degree $d$ and rank $6$ on a 2-component curve of genus $g$, with one component of genus $1$ or $2$ and another component $X$ of genus $g'$.  The $X$ aspect of this limit linear series is required to have some ramification at the node, which induces inequalities on the slopes of its tropicalization at one endpoint.  Then, just as Theorem~\ref{Thm:MainThm} follows from Theorem~\ref{thm:generalv1}, the analogous statements for linear series with ramification on pointed curves of genus $g'$ follow from Theorem~\ref{thm:generalv2}, and are used to prove Theorem~\ref{thm:genfinite} in \S\ref{sec:genfinite}.  The overall strategy of proof is outlined in \S\ref{sec:strategy}. Before presenting this outline, we introduce the essential notation and terminology that will be used throughout the argument.

\subsection{Slope indices}
\label{sec:switch}

In the vertex avoiding case, we focused our attention on distinguished functions $\varphi_i$, for $0 \leq i \leq r$, with the property that $s_k(\varphi_i) = s_k[i]$ and $s'_k(\varphi_i) = s'_k[i]$ for all $k$.  In the general cases that we need to consider for our main results, there typically exist functions with this property for some, but not all, values of $i$.  We will use the following notation to discuss how the slopes of functions in $\Sigma$ vary, relative to the slope vectors $s_k(\Sigma)$ and $s'_k(\Sigma)$ (Definition~\ref{def:slopevectors}).

\begin{definition}
The \emph{slope index} of $\varphi \in \Sigma$ at $v_k$, denoted $\iota_k(\varphi) \in \{0, \ldots, r \}$, is characterized by $s_k(\varphi) = s_k[ \iota_k(\varphi)]$.  Similarly, the slope index $\iota'_k(\varphi)$ of $\varphi$ at $w_k$, is characterized by $s'_k(\varphi) = s'_k[\iota'_k(\varphi)]$.
\end{definition}

\begin{remark}
For any $i$ and $k$, by Lemma~\ref{Lemma:Existence}, there is a function $\varphi_{i,k} \in \Sigma$ such that $\iota_k[\varphi_{i,k}] \leq i$ and $\iota'_k[\varphi_{i,k}] \geq i$.  If $\mathrm{mult}(\gamma_k) = 0$ then $(\varphi_{i,k})_{|\gamma_k}$ is unique, up to an additive constant.  The analogous statement holds for bridges as well.
\end{remark}

\begin{proposition} \label{prop:non-increasing}
If $D$ is vertex avoiding, then the slope indices of any $\varphi \in \Sigma$ form a non-increasing sequence: $\iota'_0(\varphi) \geq \iota_1(\varphi) \geq \iota'_1(\varphi) \geq \cdots\geq \iota'_g(\varphi) \geq \iota_{g+1}(\varphi)$.
\end{proposition}
\begin{proof}
This follows from the single loop analysis in \S\ref{sec:singleloop}, and the fact that the the restriction of $\varphi \in \Sigma$ to each bridge is convex.
\end{proof}

In the general case, slope indices can and do increase.  We have already seen this in Example~\ref{Ex:Interval}, on a chain of zero loops.

\begin{example}
Let $\Sigma$ be the tropical linear series of degree $2$ and rank $1$ on the interval $[v,w]$ in Case 2 of Example~\ref{Ex:Interval}.  Assume the distance $t$ is positive, and that $v$ is at distance $t$ from $x$.  In particular, the rightward slope of $\psi_A$ at $v$ is $1$.  Note that the interval is a chain of zero loops with bridges; the entire graph consists of the bridge $\beta_1$, and the multiplicity of the bridge is 2.  The slope of $\psi_A$ is $1$ on the entire bridge, but its slope index increases:
\[
\iota'_0(\psi_A) = 0 \quad \quad \mbox{ and } \quad \quad \iota_1(\psi_A) = 1.
\]
\end{example}

\subsection{Switching loops and bridges} \label{sec:switchingclassification}

In the vertex avoiding case, the slope indices of functions in $\Sigma$ form a non-increasing sequence (Proposition~\ref{prop:non-increasing}).  In the general case, functions in $\Sigma$ can have slopes that switch from a lower index to a higher index.

\begin{definition}
A loop $\gamma_k$ is a $j$-\emph{switching loop} if there is some $\varphi \in \Sigma$ such that
\[
\iota_k(\varphi) = j \quad \mbox{ and } \quad \iota'_k(\varphi) > j.
\]
Similarly, $\beta_k$ is a $j$-\emph{switching bridge} if there is some $\varphi \in \Sigma$ such that
\[
\iota'_{k-1}(\varphi) = j \quad \mbox{ and } \quad \iota_k(\varphi) > j,
\]
\end{definition}

We will say that $\gamma_k$ is a switching loop if it is a $j$-switching loop for some $j$, and similarly $\beta_k$ is a switching bridge if it is a $j$-switching bridge for some $j$.

\begin{example}
\label{Ex:SwitchLoop}
Suppose $g = 1$, and let $X$ be a curve of genus 1 with skeleton $\Gamma$.  Let $p,q \in X$ be points specializing to $w_0$ and $v_2$, respectively, and consider the tropicalization $\Sigma$ of the complete linear series $\vert \cO_X (p+q) \vert$.  Figure~\ref{Fig:BreakSwitchLoop} depicts the divisor $D_0 = \Trop (p+q)$ in black and the break divisor $D$ equivalent to $D_0$ in white.

\begin{figure}[H]
\begin{tikzpicture}[thick, scale=0.8]
\begin{scope}[grow=right, baseline]
\draw (-1,0) circle (1);
\draw (-3,0)--(-2,0);
\draw (0,0)--(1,0);
\draw [ball color=white] (-3,0) circle (0.75mm);
\draw [ball color=white] (0,0) circle (0.75mm);

\draw (3,0)--(4,0);
\draw (5,0) circle (1);
\draw (6,0)--(7,0);
\draw [ball color=black] (3,0) circle (0.75mm);
\draw [ball color=black] (7,0) circle (0.75mm);

\end{scope}
\end{tikzpicture}
\caption{The divisors $D$ (in white) and $D_0$ (in black) of Example~\ref{Ex:SwitchLoop}.}
\label{Fig:BreakSwitchLoop}
\end{figure}

\noindent The tropical linear series $\Sigma$ has rank 1, and we have
\[
s'_0 (\Sigma) = s_1 (\Sigma) = s'_1 (\Sigma) = s_2 (\Sigma) = (0,1) .
\]
By construction, there is a function $\varphi \in \Sigma$ with $\ddiv (\varphi) = D_0 - D$.  This function is unique up to an additive constant and has
\[
s_1 (\varphi) = 0 = s_1 [0] \mbox{ and } s'_1 (\varphi) = 1 = s'_1 [1] .
\]
It follows that $\iota_1 (\varphi) = 0$ and $\iota'_1 (\varphi) = 1$, so $\gamma_1$ is a $0$-switching loop.

\end{example}

We now classify switching loops and bridges of multiplicity at most 2; this is all that will be needed for the proof of our main results. We start with the classification of switching loops.

\begin{proposition} \label{Prop:SwitchLoopMult1}
There are no switching loops of multiplicity 0.  If $\gamma_k$ is an $h$-switching loop of multiplicity 1 then $\delta_k[i] = 0$ for all $i$, $x_k(D) = (s_k[h] + 1) m_k$, and $s_k[h+1] = s_k[h] + 1$.
\end{proposition}

\begin{proof}
Suppose $\mu(\gamma_k) = 0$.  By the single loop analysis in \S\ref{sec:singleloop}, there is a unique index $i$ such that $\delta_k[i] = 1$, and $\delta_k[j] = 0$ for $j \neq i$.  This index $i$ is characterized by $x_k(D) = (s_k[i] + 1)m_k$.  Furthermore, if $\iota(\varphi) \neq i$ then $\delta_k(\varphi) \leq 0$. It follows that $\iota'_k(\varphi) \leq \iota_k(\varphi)$ for all $\varphi \in \Sigma$, i.e., $\gamma_k$ is not a switching loop for $\Sigma$.

Suppose $\mu(\gamma_k) = 1$ and $\delta_k[i] = 1$ for some $i$.  Then there is a unique $j \neq i$ such that $\delta_k[j] = -1$.  A similar argument to the multiplicity 0 case then shows that $\gamma_k$ is not a switching loop.

It remains to consider the case where $\mu(\gamma_k) = 1$ and $\delta_k[i] = 0$ for all $i$. Suppose $\delta_k[i] =0$ for all $i$ and  there is some $\varphi \in \Sigma$ with $s_k(\varphi) = s_k[h]$ and $s'_k(\varphi) > s_k[h]$.  By the single loop analysis in \S\ref{sec:singleloop}, it follows that $s'_k(\varphi) = s_k(\varphi) + 1$, $x_k(D) = (s_k[h] + 1) m_k$, and $s_k[h+1] = s_k[h] + 1$.
\end{proof}

In Example~\ref{Ex:SwitchLoop}, the loop $\gamma_1$ is a $0$-switching loop of multiplicity 1.  Note that, in that example $\delta_1[i] = 0$ for all $i$ and $s_1[1] = s_1[0] + 1$, in accordance with Proposition~\ref{Prop:SwitchLoopMult1}.

\begin{proposition}
\label{Prop:SwitchLoopMult2}
Suppose $\gamma_k$ is a switching loop of multiplicity 2.  Then there is a unique $h$ such that $\gamma_k$ is an $h$-switching loop.  Moreover, exactly one of the following holds:
\begin{enumerate}
\item there is a unique $i \neq h, h+1$ such that $\delta_k[i] = 1$, $\delta_k[h] = \delta_k[h+1] = -1$, and $s'_k[h+1] = s_k[h]$;
\item $\delta_k[i] \leq 0$  for all $i$, there is a unique $j$ such that $\delta_k[j] = -1$, and $s'_k[h+1] = s_k[h] + 1$.
\end{enumerate}
\end{proposition}

\noindent Note that, in case (1), $\delta_k[i] = 0$ for $i \neq h, h+1$.  In case (2), $j$ may be equal to $h$ or $h +1$.

\begin{proof}
Suppose $\delta_k[i] = 1$ for some $i$.  By Lemma~\ref{Lem:Slope++}, this $i$ is unique.  If there is a unique index $h$ such that $s'_k[h] < s_{k}[h]$, then an argument similar to the proof of Proposition~\ref{Prop:SwitchLoopMult1} shows that $\gamma_k$ is not a switching loop.  It follows that there are two values $h < h'$ such that $s'_k[h] = s_{k}[h]-1$ and $s'_k[h'] = s_{k}[h']-1$.  The same argument as in multiplicity 1 (Proposition~\ref{Prop:SwitchLoopMult1}) shows that $\gamma_k$ is not a $j$-switching loop for any $j \neq h$.  Moreover, if $\gamma_k$ is an $h$-switching loop then $s'_k[h+1] \leq s_{k}[h]$ and it follows that $h' = h+1$ and $s'_k[h+1] = s_k[h]$, as required.

Otherwise, $\delta_k[i] \leq 0 $ for all $i$.  Since $\gamma_k$ has multiplicity 2, there is a unique $j$ such that $s'_k[j] = s_k[j]-1$.  Let $\varphi \in \Sigma$ satisfy $\iota'_k (\varphi) > \iota_k (\varphi)$.  By Lemma~\ref{Lem:PointsOfSlopeIncrease}, we see that  $x_k(D) = s_k[\iota_k(\varphi)] + 1$, which uniquely determines $\iota_k(\varphi)$.
\end{proof}

\begin{figure}[h!]
\begin{tabular}{|c|c|}
\hline
$s_k $ & $s'_k$  \\
\hline
3 & 4 \\
\hline
2 & 1 \\
\hline
1 & 0 \\
\hline
-1 & -1 \\
\hline
\end{tabular}
\quad \quad
\begin{tabular}{|c|c|}
\hline
$s_k $ & $s'_k$  \\
\hline
3 & 2 \\
\hline
1 & 1 \\
\hline
0 & 0 \\
\hline
-1 & -1 \\
\hline
\end{tabular}
\quad \quad
\begin{tabular}{|c|c|}
\hline
$s_k $ & $s'_k$  \\
\hline
3 & 3 \\
\hline
1 & 1 \\
\hline
0 & 0 \\
\hline
-1 & -2 \\
\hline
\end{tabular}
\quad \quad
\begin{tabular}{|c|c|}
\hline
$s_k $ & $s'_k$  \\
\hline
3 & 3 \\
\hline
2 & 1 \\
\hline
0 & 0 \\
\hline
-1 & -1 \\
\hline
\end{tabular}
\quad \quad
\begin{tabular}{|c|c|}
\hline
$s_k $ & $s'_k$  \\
\hline
3 & 3 \\
\hline
2 & 2 \\
\hline
1 & 0 \\
\hline
-1 & -1 \\
\hline
\end{tabular}
\caption{Some possible tables of slopes $s_k[i]$ and $s'_k[i]$ when $\Sigma$ has rank 3 and $\gamma_k$ is a $1$-switching loop of multiplicity 2. The leftmost table reflects case (1) in Proposition~\ref{Prop:SwitchLoopMult2}, the rest reflect case (2). (In each case, the first column of the table has $s_k[0]$ in the bottom row, $s_k[1]$ in the next row, and so on.)}
\end{figure}

\begin{proposition}
\label{Prop:SwitchBridgeComputation}
There are no switching bridges of multiplicity less than 2.  If $\beta_k$ is a switching bridge of multiplicity 2 then there is a unique $h$ such that $\beta_k$ is an $h$-switching bridge.  Moreover,
\[
\delta'_k[h] = \delta'_k[h+1] = -1,  \quad \delta'_k[i] = 0 \mbox{ for } i \neq h, h+1, \quad  \mbox{ and } \quad s'_k[h+1] = s_k[h] + 1.
\]
\end{proposition}

\begin{proof}
The proof is similar to the classification of switching loops, except that the single loop analysis from \S\ref{sec:singleloop} is replaced with the observation that the restriction of any $\varphi \in \Sigma$ to a bridge $\beta_k$ is convex, and hence $s_k(\varphi) \leq s'_{k-1}(\varphi)$.  This immediately rules out switching on a bridge of multiplicity 0.

Suppose $\beta_k$ is a bridge of multiplicity 1.  Then there is a unique index $h$ such that $\delta'_k[h] = -1$. The proof that $\beta_k$ is not $j$-switching for any $j \neq h$ is the same as in multiplicity 0.  We now show that it is not $h$-switching. Since the multiplicity is exactly 1, we have $\delta'_k[i] = 0$ for $i \neq h$ and hence there is no function $\varphi$ in $\Sigma$ with $s_k(\varphi) = s'_{k-1}[h]$.  It follows that if $\iota'_{k-1}(\varphi) = h$ then $s_k(\varphi)$ is strictly less than $s'_{k-1}[h]$, so $\beta_k$ is not $h$-switching.

Now, let $\beta_k$ be a bridge of multiplicity 2.  If there is a unique value $h$ such that $s_k[h] < s'_{k-1}[h]$, then as in the multiplicity 1 argument, $\beta_k$ is not a switching bridge.  It follows that there are two values $h < h'$ such that $s_k[h] = s'_{k-1}[h]-1$, $s_k[h'] = s'_{k-1}[h']-1$, and $s_k[i] = s'_{k-1}[i]$ for all $i \neq h, h'$. The remainder of the proof is just as for switching loops (case (1) of Proposition~\ref{Prop:SwitchLoopMult2}).
\end{proof}

\begin{corollary}\label{Cor:IndexIncrease}
Suppose $\Sigma$ is a tropical linear series on $\Gamma$ with adjusted Brill-Noether number $\rho' \leq 2$ and $\varphi \in \Sigma$.  Then, for all $k$,
\[
\iota_k(\varphi) \leq \iota'_{k-1}(\varphi) + 1 \quad  \mbox{ and } \quad \iota'_k(\varphi) \leq \iota_k(\varphi) + 1.
\]
\end{corollary}

\subsection{Functions with constant slope index}

 In the vertex avoiding case, we studied functions $\varphi_i \in R(D)$, for $0 \leq i \leq r$, with \emph{constant slope index} $i$, i.e., with the property that
 \[
\iota_k(\varphi_i) = \iota'_{k-1}(\varphi_i) = i \mbox { \  for \ } 1 \leq k \leq g+1.
 \]
 In the general case, it is useful to keep track of the indices $i$ such that there exists a function $\varphi_i$ with constant slope index $i$.

\begin{lemma}
\label{Lem:GenericFns}
Let $\Sigma \subset R(D)$ be a tropical linear series on $\Gamma$ with adjusted Brill-Noether number $\rho' \leq 2$.  Then, for each $0 \leq i \leq r$, either there is a function $\varphi_i \in \Sigma$ such that
\[
\iota_k(\varphi_i) = \iota'_{k-1}(\varphi_i) = i \mbox { \  for \ } 1 \leq k \leq g+1,
\]
or there is a loop or bridge that is either $i$-switching or $(i-1)$-switching.
\end{lemma}

\begin{proof}
By Lemma~\ref{Lemma:Existence}, there is a function $\varphi_i \in \Sigma$ with $\iota'_0(\varphi_i) \leq i$ and $\iota_{g+1} (\varphi_i) \geq i$.  By Corollary~\ref{Cor:IndexIncrease}, the sequence $(\iota'_0(\varphi_i),  \iota_1(\varphi_i), \ldots, \iota_{g+1}(\varphi_i))$ increases by at most 1 at each step.  By assumption, the first term in this sequence is at most $i$ and the final term is at least $i$. Thus, either the sequence is constant and equal to $i$, or there is a step where it increases from $i-1$ to $i$, or there is a step where it increases from $i$ to $i+1$.
\end{proof}

\noindent Note that this is an existence statement.  If $D$ is not vertex avoiding, then these functions $\varphi_i$ with constant slope index $i$ need not be unique, even after normalizing so that $\varphi_i(w_0) = 0$.

\subsection{Overview of the proofs of Theorems~\ref{thm:generalv1} and \ref{thm:generalv2}} \label{sec:strategy}

Our proofs of these theorems are completed via case-by-case considerations in \S\ref{Sec:Generic},  according to the switching loops and bridges of the tropical linear series $\Sigma$.  The proof in each case involves three steps.  First, we algorithmically construct a ``template'' $\theta \in R(2D)$.  Then, we choose a collection $\cT \subset 2\Sigma$ of size $|\cT| = 28$.  Finally, we show that the ``best approximation of $\theta$ by $\cT$'' is a certificate of independence.

\subsubsection{The template}
The template is a function in $R(2D)$ (but typically not in $2\Sigma$) that is constructed in \S\ref{sec:template} via an algorithm closely analogous to the algorithm for the vertex avoiding case in \S\ref{sec:basicalg}.  The input for the algorithm is once again a collection $\cB$ of functions in $R(2D)$, each with constant slope along every bridge, and a non-increasing integer sequence $\sigma = (\sigma_1, \ldots, \sigma_{g+1})$.  The output is again assignment function $\alpha \colon \cB \to \{ \beta_k \} \cup \{ \gamma_\ell \} \cup \{ \emptyset \}$ and a collection of coefficients $\{ c(\psi) \mid \psi \in \cB \}$. The template is the tropical linear combination
\[
\theta := \min \{ \psi + c(\psi) \mid \alpha(\psi) \neq \emptyset \},
\]
and each term $\psi + c(\psi)$ achieves the minimum on an open subset of the loop or bridge to which it is assigned.  Unlike the vertex avoiding case, $\psi + c(\psi)$ may not achieve the minimum uniquely on the loop or bridge to which it is assigned. However, there is an open subset of the assigned bridge or loop where all of the functions that achieve the minimum are assigned to that bridge or loop, and all of them agree on the entire bridge or loop.  (The template is a technical tool for building a certificate of independence, but is itself typically far from being such a certificate.)

The functions in $\cB$ are not necessarily in $2\Sigma$, but they are all of a special form.  We introduce a finite set $\cA \subset R(D)$ consisting of certain functions with constant slope along each bridge that we call ``building blocks.''  The building blocks model the behavior of the distinguished functions $\varphi_0, \ldots, \varphi_r$ in the vertex avoiding case.  Then we choose $\cB$ from the sumset $2\cA$.  In order to guarantee that the template $\theta$ has certain desirable technical properties, which are needed for the final step (best approximation from above) we require that $\cB$ satisfies two technical properties that we denote $(\mathbf B)$ and $(\mathbf B')$.  See Definition~\ref{def:properties} and Theorem~\ref{Thm:ExtremalConfig}.

\subsubsection{The collection $\cT$ of 28 functions in $2\Sigma$}
Just as the functions $\cB$ used to form the template are all in the sumset $2\cA$, the functions $\cT$ that appear in the certificate of independence are in a sumset $2\cS$ for a small finite subset $\cS \subset \Sigma$.  This set $\cS$ includes one function $\varphi_i$ with constant slope index $i$, for each $i$ such that $\Sigma$ has no $i$-switching or $(i-1)$-switching loops or bridges.  (Such functions exist, by Corollary~\ref{Cor:IndexIncrease}.) To this collection of functions with constant slope index, we add additional functions that reflect the switching patterns of $\Sigma$.  For instance, in the case where $\Sigma$ has an $h$-switching bridge $\beta_k$ (\S\ref{Sec:SwitchBridgeCase}), we identify a tropical linear subseries $\Sigma' \subset \Sigma$ of rank $1$ such that
\begin{itemize}
\item Every function $\varphi \in \Sigma'$ has $s_k(\varphi) \in \{ s_k[h], s_k[h+1] \}$, and
\item The bridge $\beta_k$ is a $0$-switching bridge for $\Sigma'$.
\end{itemize}
We then identify three functions $\varphi_A, \varphi_B, \varphi_C$ in $\Sigma'$, analogous to the functions $\psi_A$, $\psi_B$, and $\psi_C$ in Example~\ref{Ex:Interval}, and define $\cA$ to be the union of this set of three functions together with the $5$ functions $\varphi_i$ for $i \neq h, h+1$. Then $2\cA$ has size $\binom{9}{2} = 36$, and from these we choose a subset of size $|\cT| = 28$ that includes all 15 functions of the form $\varphi_i + \varphi_j$, plus 13 more that involve one or two of the functions from the tropical pencil $\Sigma'$.  Each $\varphi_i$, for $i \neq h, h+1$ is a building block, the set $\cB$ contains all of the pairwise sums $\varphi_i + \varphi_j$, and the template is constructed so that $\varphi_i + \varphi_j$ achieves the minimum uniquely on an open subset of the bridge or loop to which it is assigned in the template algorithm.  Thus, the essential difference between the construction in this case and the construction in the vertex avoiding case is how to choose and handle the 13 functions that involve the three functions $\varphi_A, \varphi_B, \varphi_C$.

The way these 13 functions are chosen and handled depends on a parameter analogous to the parameter $t$ in Example~\ref{Ex:Interval}.  Note that construction of the template does not depend on this parameter.  Thus, a single template is used to construct many different certificates of independence, for many different tropical linear series.  For this reason, we present the construction of the template separately, in \S\ref{sec:template}, in advance of the case-by-case analysis.

\subsubsection{The best approximation from above} \label{sec:bestapprox}

Once we have the template and the set $\cT \subset 2\Sigma$, the certificate of independence is obtained via ``best approximation from above.''  Let us explain the idea of this construction.  Fix a real-valued function $\theta$ on $\Gamma$.  Imagine trying to approximate $\theta$ by tropical linear combinations of functions in a finite set $\cT$, while imposing the condition that this approximation is greater than or equal to $\theta$.  We claim that there is a best possible such approximation.  Indeed, for $\psi \in \cT$, the function $\psi-\theta$ is continuous and bounded, so it achieves its minimum $b(\psi)$ on $\Gamma$.  Then $\psi - b(\psi) \geq \theta$, with equality at some point $v \in \Gamma$, and
\[
\Upsilon = \min \{ \psi - b(\psi) \mid \psi \in \cT \}
\]
is the smallest tropical linear combination of functions in $\cT$ that is greater than or equal to $\theta$, i.e., the best approximation of $\theta$ from above.

Note that there are no choices to be made in this final step. All of the hard work is already done constructing the template $\theta$, proving that it has the desired technical properties, and then choosing the set $\cT$ based on the switching properties of the linear series $\Sigma$.  The final verification that $\Upsilon$ is a certificate of independence is relatively easy in each case, using the technical properties of the template established in Theorem~\ref{Thm:ExtremalConfig}.

\section{The template algorithm}
\label{Sec:Construction}

Here we carry out the first step in the proof of Theorems~\ref{thm:generalv1} and \ref{thm:generalv2}, presenting an algorithmic construction of the template $\theta$ as a tropical linear combination of pairwise sums of certain functions in $R(D)$ with constant slope along each bridge that we call building blocks.

As in the vertex avoiding case, our algorithm for constructing the template requires some additional input, namely a collection $\cB$ of pairwise sums of building blocks, and a non-increasing sequence of integers $\sigma = (\sigma_1, \ldots, \sigma_{g'+1})$.
We first present the general algorithm, and then we explain how to choose the input in the specific cases needed for the proof of Theorems \ref{thm:generalv1} and \ref{thm:generalv2}. Finally we verify that, with this input, the algorithm produces a tropical linear combination with a few key technical properties that will streamline our case-by-case construction of certificates of independence in \S\ref{Sec:Generic}; see Theorem~\ref{Thm:ExtremalConfig}.

\medskip

\noindent{\textbf{Notation.}} We maintain the notation from the preceding section: $D$ is a break divisor of degree $d$ on a chain of $g$ loops $\Gamma$ with $C$-admissible edge lengths for some $C > 12dg$, and $\Sigma \subset R(D)$ is a tropical linear series of rank $r$. The break divisor $D$ is not necessarily vertex avoiding and the linear series $\Sigma$ may have ramification at the endpoints $w_0$ and $v_{g+1}$.

From now on, we assume furthermore that the adjusted Brill-Noether $\rho'(\Sigma)$ (Definition~\ref{def:adjustedBN}) is at most 2.  Thus, any switching loops or bridges for $\Sigma$ are as classified in \S\ref{sec:switchingclassification}.

\subsection{Building blocks}
\label{Sec:Extremals}

We introduce a class of functions in $R(D)$ with constant slope along bridges that behave sufficiently similarly to the distinguished functions $\varphi_i$ in the vertex avoiding case so that we can run an algorithm analogous to the one in \S\ref{Sec:VertexAvoiding} using them in place of $\{ \varphi_0, \ldots, \varphi_r \}$.  The definition of these \emph{building blocks} is motivated by the extremals of \cite{HMY12} and by Example~\ref{Ex:Interval}.

We write $D(v)$ for the coefficient of a point $v \in \Gamma$ in the divisor $D$.  Then, we define
\[
d_v (\varphi) := \ord_v (\varphi) + D(v),
\]
for $\varphi \in \PL(\Gamma)$.  By definition, $\varphi \in R(D)$ if and only if $d_v (\varphi) \geq 0$ for all $v$.

We fix finitely many possibilities for the restriction to each loop, as follows.  For each $k$ in the range $1 \leq k \leq g$ and each $i$ in the range $0 \leq i \leq r$, choose a function $f_{k,i} \in \PL(\gamma_k)$ such that
\begin{equation} \label{eq:fki}
d_{w_k}(f_{k,i}) \geq s'_k [i], \quad  \quad d_{v_k}(f_{k,i}) \geq \begin{cases} -s_k[i-1] & \mbox{ if $\gamma_k$ is an $(i-1)$-switching loop} \\ -s_k[i] & \mbox{ otherwise, } \end{cases}
\end{equation}
and
\begin{equation} \label{eq:eff}
D_{|\gamma_k} + \ddiv(f_{k,i}) \mbox{ is effective on } \gamma_k \smallsetminus \{ v_k, w_k \}.
\end{equation}
Note that $f_{k,i}$ is uniquely determined by \eqref{eq:fki} and \eqref{eq:eff}, up to an additive constant, except in the cases where $s'_k[i] < s_k[i]$; this follows from Lemma~\ref{Lem:EquivOneLoop}.

\begin{definition} \label{def:buildingblock}
A sequence of non-negative integers $\tau = (\tau'_0 , \tau_1 , \tau'_1 , \ldots , \tau'_g , \tau_{g+1})$ is a \emph{building sequence} if it satisfies $0 \leq \tau'_0 \leq \tau_1 \leq \cdots \leq \tau_{g+1} \leq r$,
\[
\tau'_k = \begin{cases} \tau_{k} \mbox{ or } \tau_{k} + 1 & \mbox{if $\gamma_k$ is $\tau_{k}$-switching} \\  \tau_{k}  & \mbox{otherwise.} \end{cases}
\]
and
\[
\tau_k = \begin{cases} \tau'_{k-1} \mbox{ or } \tau'_{k-1} + 1 & \mbox{if $\beta_k$ is $\tau'_{k-1}$-switching} \\  \tau'_{k-1}  & \mbox{otherwise.} \end{cases}
 \]
Given a building sequence $\tau$ and a sequence of integers $s = (s_1 , \ldots , s_{g+1})$ satisfying
\begin{equation} \label{eq:xibounds}
s_k[\tau_k] \leq s_k \leq s'_{k-1}[\tau'_{k-1}],
\end{equation}
the associated $(\Sigma, \{f_{k,i}\})$-\emph{building block} is the unique function $\varphi = \varphi_{\tau,s} \in \PL (\Gamma)$ such that
\[
\varphi (w_0) = 0,  \ \ \  \varphi - f_{k,\tau'_k} \mbox{ is constant on $\gamma_k$}, \ \ \mbox{ and } \ \  \varphi \mbox{ has constant slope } s_k \mbox{ on } \beta_k.
\]
 \end{definition}
\noindent Note that, by construction, we have $\varphi_{\tau,s} \in R(D)$.

\begin{remark}
In comparison with the vertex avoiding case, the requirement that $\tau$ is nondecreasing replaces the fact that the slopes of each distinguished function $s_k(\varphi_i)$ are nondecreasing as a function of $k$.  The requirement that $\tau$ increases by at most 1 at each step reflects the classification of switching bridges and loops from \S\ref{sec:switchingclassification}; since all bridges and loops have multiplicity at most 2, the slope index of a function in $\Sigma$ can never increase by more than 1 at any step.
\end{remark}

When $\Sigma$ and $\{f_{k,i}\}$ are fixed, we refer to the functions $\varphi_{\tau,s}$ as \emph{building blocks}.  Given a building block $\varphi$, we write $\tau_k (\varphi)$ and $\tau'_k (\varphi)$ for the terms in the corresponding building sequence.

\begin{example}
\label{Ex:VABB}
Suppose $D$ is vertex avoiding.  Then $s'_{k-1}[i] = s_k[i]$ and hence each $f_{k,i}$ is uniquely determined, up to an additive constant, by \eqref{eq:fki}.  It follows that $\varphi_i - f_{k,i}$ is constant on $\gamma_k$, and hence the distinguished functions $\{ \varphi_i \}$ are precisely the $(\Sigma, \{f_{k,i}\})$-building blocks.
\end{example}

\begin{example}
\label{Ex:Genus5}
Suppose $g = 5$ and let $D$ be the break divisor of degree $7$ on $\Gamma$ with $x_1 (D) = 3m_1$, $x_2 (D) = 2m_2$, and $x_3 (D) = m_3$, $x_4(D) = 2m_4$, and $x_5 (D) = 0$.  Note that $D$ has rank $r = 3$, $\rho(g,r,d) = 1$, and  $[D]$ is in the locus of rank 3 divisor classes associated to the following tableau $T$.
\begin{figure}[h!]
\begin{ytableau}
1 & 2 & 3 & 5
\end{ytableau}
\caption{Tableau $T$ corresponding to the divisor $D$ of Example~\ref{Ex:Genus5}.}
\label{Fig:SwitchingTableau}
\end{figure}

Let $\Sigma \subset R(D)$ be a tropical linear series of rank $3$. Since $4$ is not in $T$, the loop $\gamma_4$ has positive multiplicity.  Since $\rho(g,r,d) = 1$, if follows that $\mu(\gamma_4) = 1$ and all other loops and bridges have multiplicity 0.  It follows that $s'_{k-1}[i] = s_{k}[i]$ for all $i$ and $k$ and the slope vectors $s_i(\Sigma)$ are as shown in Figure~\ref{fig:slopesg5}.
\begin{figure}[h!]
\begin{tabular}{|c|c|c|c|c|c|}
\hline
 $s_1$ & $s_2$ & $s_3$ & $s_4$ & $s_5$ & $s_6$ \\
\hline
2 & 3 & 3 & 3 & 3 & 3 \\
\hline
1 & 1 & 2 & 2 & 2 & 2 \\
\hline
0 & 0 & 0 & 1 & 1 & 1 \\
\hline
-1 & -1 & -1 & -1 & -1 & 0 \\
\hline
\end{tabular}
\caption{The slope vectors $s_i(\Sigma)$ of Example~\ref{Ex:Genus5}.} \label{fig:slopesg5}
\end{figure}
Note that there is a function $\varphi \in R(D)$ with $s_4(\varphi) = 1$ and $s'_4(\varphi) = 2$; indeed, Figure~\ref{Fig:SwitchExample} schematically illustrates such a function $\varphi_\tau$ with constant slope along each bridge, such that $(s_1(\varphi_\tau), \ldots, s_6(\varphi_\tau)) = (0,0,0,1,2,2)$.

Let $X$ be a curve of genus 5 with skeleton $\Gamma$, and let $p,q \in X$ be points specializing to $w_0$ and $v_6$, respectively.  We briefly show that there exists a divisor $D_X$ on $X$ of degree 7 and rank 3 such that $\Trop (D_X) = D$ and $\gamma_4$ is $1$-switching for $\Sigma = \trop (R(D_X))$.  To see this, note that $K_X$ has rank 4, so $K_X - 2p - 2q$ is equivalent to an effective divisor $E_X$.  Additionally, $K_{\Gamma} - 2w_0 - 2v_6$ is equivalent to a unique effective divisor $E$, so we must have $\Trop (E_X) = E$.  Let $x \in \Gamma$ be the point on $\gamma_4$ of distance $2m_4$ counterclockwise from $v_4$.  Then $x \in \mathrm{Supp} (E)$, and there exists $y \in \mathrm{Supp} (E_X)$ such that $\Trop (y) = x$.  By construction, $D_X = K_X - y$ is a divisor of degree 7 and rank 3 such that $\Trop(D_X) = D$, and $D_X - 2p - 2q$ is effective.  It follows that there is a function $\varphi \in \Sigma$ that vanishes to order at least 2 at both $w_0$ and $v_6$.  Up to tropical scaling, there is a unique such function in $R(D)$; it is the function $\varphi_{\tau}$ depicted in Figure~\ref{Fig:SwitchExample}.  Note that $s_4 (\varphi_{\tau}) = s_4 [1]$ and $s_5 (\varphi_{\tau}) = s_5 [2]$, hence $\gamma_4$ is $1$-switching.

In this example, there are 5 building sequences: the constant sequences $i$ for $0 \leq i \leq 3$, and the sequence $\tau$ given by
\[
\tau_k = \tau'_{k-1} = \begin{cases} 1 & \mbox{if $k \leq 4$} \\  2 & \mbox{if $k \geq 5$.} \end{cases}
\]
Since $s'_{k-1}[i]  = s_k[i]$ for all $i$ and $k$, there is a unique building block for each building sequence; we denote them by $\varphi_0 , \varphi_1 , \varphi_2 , \varphi_3$, and $\varphi_{\tau}$.  The functions $\varphi_1$ and $\varphi_{\tau}$ are illustrated in Figure~\ref{Fig:SwitchExample}.

\begin{figure}[H]
\begin{tikzpicture}[thick, scale=0.63]
\begin{scope}[grow=right, baseline, shift={(-6,0)}]
\draw (-3.25,7) node {$\varphi_1$};
\draw (-3,6)--(-2,6);
\draw (-1,6) circle (1);
\draw (0,6)--(1,6);
\draw (2,6) circle (1);
\draw (3,6)--(4,6);
\draw (5,6) circle (1);
\draw (6,6)--(7,6);
\draw (8,6) circle (1);
\draw (9,6)--(10,6);
\draw (11,6) circle (1);
\draw (12,6)--(13,6);
\draw [ball color=black] (-3,6) circle (0.75mm);
\draw (-3.5,6) node {\small 2};
\draw [ball color=black] (-.5,6.87) circle (0.75mm);
\draw [ball color=black] (2.87,6.5) circle (0.75mm);
\draw [ball color=white] (6,6) circle (0.75mm);
\draw [ball color=white] (8.87,6.5) circle (0.75mm);
\draw [ball color=white] (10,6) circle (0.75mm);
\draw [ball color=black] (9,6) circle (0.75mm);
\draw [ball color=black] (10.13,6.5) circle (0.75mm);
\draw [ball color=black] (13,6) circle (0.75mm);

\draw (6.5,5.75) node {\tiny 1};
\draw (9.5,5.75) node {\tiny 1};
\draw (12.5,5.75) node {\tiny 1};
\draw (8,4.75) node {\tiny 1};
\draw (11,4.75) node {\tiny 1};
\draw (9.1,6.3) node {\tiny 1};
\draw (9.9,6.3) node {\tiny 1};
\draw (8,7.25) node {\tiny 0};
\draw (11,7.25) node {\tiny 0};
\end{scope}

\begin{scope}[grow=right, baseline, shift={(-6,-4)}]
\draw (-3.25,7) node {$\varphi_{\tau}$};
\draw (-3,6)--(-2,6);
\draw (-1,6) circle (1);
\draw (0,6)--(1,6);
\draw (2,6) circle (1);
\draw (3,6)--(4,6);
\draw (5,6) circle (1);
\draw (6,6)--(7,6);
\draw (8,6) circle (1);
\draw (9,6)--(10,6);
\draw (11,6) circle (1);
\draw (12,6)--(13,6);
\draw [ball color=black] (-3,6) circle (0.75mm);
\draw (-3.5,6) node {\small 2};
\draw [ball color=black] (-.5,6.87) circle (0.75mm);
\draw [ball color=black] (2.87,6.5) circle (0.75mm);
\draw [ball color=white] (6,6) circle (0.75mm);
\draw [ball color=white] (8.87,6.5) circle (0.75mm);
\draw [ball color=white] (10,6) circle (0.75mm);
\draw [ball color=black] (10.5,6.87) circle (0.75mm);
\draw [ball color=black] (13,6) circle (0.75mm);
\draw (13.5,6) node {\small 2};

\draw (6.5,5.75) node {\tiny 1};
\draw (9.5,5.75) node {\tiny 2};
\draw (12.5,5.75) node {\tiny 2};
\draw (8,4.75) node {\tiny 1};
\draw (11,4.75) node {\tiny 2};
\draw (9.1,6.3) node {\tiny 1};
\draw (10.05,6.7) node {\tiny 1};
\end{scope}

\end{tikzpicture}
\caption{Schematic illustration of the functions $\varphi_1$ and $\varphi_{\tau}$ of Example~\ref{Ex:Genus5}}
\label{Fig:SwitchExample}
\end{figure}
\noindent In Figure~\ref{Fig:SwitchExample}, the divisors $D + \ddiv(\varphi_1)$ and $D + \ddiv(\varphi_\tau)$ are shown in black.  The point $w_0$ has multiplicity 2 in $D + \ddiv(\varphi_1)$ and both $w_0$ and $v_{g+1}$ have multiplicity 2 in $D + \ddiv(\tau)$, as indicated; all other points in the support of these divisors have multiplicity 1. The white dots indicate the points in the support of $D$ on $\gamma_k$ for $k \geq 3$.  Note that both of the building blocks are identically zero to the left of $w_3$ and are equal to each other to the left of $w_4$.  The slopes along segments to the right of $w_3$ are as indicated.

We will use this as a relatively simple running example to illustrate the technical steps in our construction of the template.  See Examples~\ref{Ex:Genus5Properties} and \ref{Ex:Genus5Template}.
\end{example}

\subsection{Input for the template algorithm}
\label{Sec:Input}

As discussed in \S\ref{sec:strategy}, we prove Theorems~\ref{thm:generalv1} and \ref{thm:generalv2} by first constructing a template $\theta$ and then finding a certificate of independence for 28 functions in $2\Sigma$ via best approximation of $\theta$ from above.  Here, we present the algorithm for constructing the template, closely following our algorithm for the vertex avoiding case (\S\ref{sec:basicalg}).

Recall that the graph $\Gamma$, break divisor $D$, and tropical linear series $\Sigma \subset R(D)$ are fixed. We also fix  functions $f_{k,i}$, as in \S\ref{Sec:Extremals}. Let $\cA \subset R(D)$ be the set of $(\Sigma, \{f_{k,i}\})$-building blocks (Definition~\ref{def:buildingblock}). We write $2\cA$ for the sumset
$$2 \cA = \{ \varphi + \varphi' \mid \varphi, \varphi' \in \cA \}.$$
The additional input to run the algorithm is:
\begin{itemize}
\item A nonempty subset $\cB \subset 2 \cA$.
\item A non-increasing sequence of integers $(\sigma_1, \ldots, \sigma_{g+1})$.
\item A positive real number $\epsilon < \frac{1}{144dr^2}$.
\end{itemize}
The output is a collection of coefficients $\{c(\psi) \mid \psi \in \cB\}$ and an assignment function
\[
\alpha \colon \cB \to \{ \beta_k \} \cup \{ \gamma_\ell \} \cup \{ \emptyset \}.
\]
As in the vertex avoiding case, the assignment function $\alpha$ depends on $\sigma$ but not $\epsilon$.

As in \S\ref{sec:basicalg}, we consider the tropical linear combinations
\[
\theta' := \min\{ \psi + c(\psi) \mid \psi \in \cB \} \mbox{ \ \  and \ \ } \theta := \min \{ \psi + c(\psi) \mid \alpha(\psi) \neq \emptyset \}.
\]

\begin{definition} \label{def:equivalentS}
Let $\varphi, \varphi' \in \PL(\Gamma)$. We say that $\varphi$ is \emph{equivalent} to $\varphi'$ on a subset $S \subset \Gamma$, and write $\varphi \sim_S \varphi'$, if $(\varphi - \varphi')_{|S}$ is a constant function on $S$.
\end{definition}
\noindent We most often consider this equivalence relation when $S = \gamma_k$ is a single loop. In a few places (notably in  Definition~\ref{def:properties}) we also consider equivalence on larger subgraphs.

\subsection{Properties of the template} \label{sec:templateprops} The tropical linear combination $\theta$, which we call the ``template'' has several notable properties, analogous to the properties of the output of the algorithm from \S\ref{sec:basicalg} in the vertex avoiding case.

\begin{enumerate}[label=(P\arabic*)]
\item \label{it:achieves} If $\alpha(\psi) = \beta_k$ then $\psi + c(\psi)$ achieves the minimum on a nonempty open subset of $\beta_k$.
\item \label{it:achieveloop} If $\alpha^{-1}(\gamma_k) \neq \emptyset$ then there is a nonempty open subset of $\gamma_k$ where $\psi + c(\psi)$ achieves the minimum if and only if $\alpha(\psi) = \gamma_k$.
\item \label{it:equivalent} If $\psi, \psi' \in \alpha^{-1}(\gamma_k)$ then $\psi \sim_{\gamma_k} \psi'$.
\item \label{it:nottoright} If $\alpha(\psi) \neq \emptyset$ then $\psi + c(\psi)$ does not achieve the minimum on any loop to the right of $\alpha(\psi)$ to which a function is assigned.
\end{enumerate}

The algorithm proceeds from left to right across the graph.  At the beginning, we initialize $c(\psi) = \infty$ and $\alpha(\psi) = \emptyset$ for all $\psi \in \cB$.  Once a coefficient $c(\psi)$ has a finite value, it may be increased (but never decreased) as long as $\alpha(\psi) = \emptyset$.  When $\psi$ is assigned to a loop or bridge, $c(\psi)$ is fixed and never changes again, and $\psi + c(\psi)$ achieves the minimum on an open subset of the loop or bridge to which it is assigned.  If $\psi$ is assigned to a loop then all functions assigned to that loop are equivalent on the loop and hence achieve the minimum together on the same open set.

For $\psi \in \cB$ the slopes $s_k(\psi)$ are not necessarily non-decreasing. Because these slopes may decrease, the natural extension of Definition~\ref{Def:VAPermissible} from the vertex avoiding case to the general case is as follows.

\begin{definition} \label{def:permissible}
Let $\psi \in \cB$.  We say that $\psi$ is \emph{$\sigma$-permissible} on $\gamma_k$ if
\begin{itemize}
\item  $s_{j} (\psi) \leq \sigma_j$ for all $j\leq k$,
\item $s_{k+1}(\psi) \geq \sigma_k$, and
\item  if $s_{\ell} (\psi) < \sigma_\ell$ for some $\ell > k$, then $s_{k'} (\psi) > \sigma_{k'}$ for some $k'$ such that $k < k' < \ell$.
\end{itemize}
\end{definition}

\noindent When $\sigma$ is fixed, we simply say that $\psi$ is permissible.
Note that Definition~\ref{def:permissible} agrees with Definition~\ref{Def:VAPermissible} when $D$ is vertex avoiding. Also, if $\theta$ has average slope in $(\sigma_j - \epsilon, \sigma_j + \epsilon)$ on each bridge $\beta_j$, and if $\psi$ achieves the minimum at some point of $\gamma_k$, then $\psi$ is permissible on $\gamma_k$ (cf. Remark~\ref{rem:permissible}).

\medskip

If $\psi \in \cB$ is permissible on $\gamma_k$ then $\delta_k (\psi) \geq 0$.
The proofs of Lemmas~\ref{Lem:VAEverythingIsPermissible} and \ref{lem:lastloop} go through essentially without change: if $\psi \in \cB$ then either $s_1(\psi) < \sigma_1$, $s_{g+1}(\psi) > \sigma_{g+1}$, or $\psi$ is permissible on some loop $\gamma_k$. We retain Definition~\ref{Def:VADeparting}: a permissible function on $\gamma_k$ is \emph{departing} if $s_{k+1}(\psi) > \sigma_k$.    If $\psi$ is departing on $\gamma_k$, then $\gamma_k$ is the last loop on which $\psi$ is permissible.

\begin{lemma}
\label{Lem:OneDepartingFunction}
All departing permissible functions are equivalent on $\gamma_k$.
\end{lemma}

\begin{proof}
Let $\psi \in R(2D)$ be a departing permissible function on $\gamma_k$.  By Lemma~\ref{Lem:RestrictOneLoop},
\[
\ddiv(\psi_{|\gamma_k}) = (\ddiv(\psi))_{|\gamma_k} - s_k(\psi) \cdot v_k + s'_k(\psi) \cdot w_k.
\]
Since $\psi$ is departing, we have
\[
s_k (\psi) \leq \sigma_k \leq s'_k (\psi) - 1,
\]
hence $\ddiv(\psi_{|\gamma_k}) + 2D_{|\gamma_k} + \sigma_k v_k - (\sigma_k +1) w_k$ is effective.  But $2D_{|\gamma_k} + \sigma_k v_k - (\sigma_k +1) w_k$ is a divisor on $\gamma_k$ of degree 1, so by Lemma~\ref{Lem:EquivOneLoop} it is equivalent to a unique effective divisor.  It follows that $\psi_{|\gamma_k}$ is uniquely determined up to an additive constant.
\end{proof}

\subsection{Skippable loops} In the vertex avoiding case, the loops with positive multiplicity are precisely the lingering loops. The algorithm presented in \S\ref{sec:basicalg} skips over these loops without assigning any functions,.  In the general case, there is a closely related class of loops of positive multiplicity that we skip over without assigning a function, characterized as follows.

\begin{definition}
\label{Def:Skippable}
We say that the loop $\gamma_k$ is \emph{skippable} if there are no unassigned departing functions, not all unassigned permissible functions are equivalent on $\gamma_k$, and there is an unassigned permissible function $\psi = \varphi + \varphi' \in \cB$ such that $D + \ddiv(\varphi)$ contains either:
\begin{itemize}
\item  more than one point of $\gamma_k \smallsetminus \{ v_k \}$, or
\item  the vertex $w_k$, or
\item a point on $\gamma_k$ whose distance from $w_k$ is not an integer multiple of $m_k$.
\end{itemize}
\end{definition}
\noindent In the vertex avoiding case, the first two conditions are never satisfied, and the last is satisfied only on lingering loops.  In this case, if $D_{|\gamma_k}$ is a point whose distance from $w_k$ is a non-zero integer multiple of $m_k$, but this integer is not equal to $s_k[i] + 1$ for any $i$, then $\gamma_k$ is a lingering loop but not a skippable loop.  We could have written the algorithm in \S\ref{sec:basicalg} to assign functions to these loops, but it is simpler to skip all of the lingering loops (see Remark~\ref{rem:variants}.)  Note that, if $\gamma_k$ is skippable, then $\gamma_k$ or $\beta_{k+1}$ has positive multiplicity.  Also, whether a loop is skippable depends on which functions have been previously assigned.  In particular, if there is an unassigned departing function on $\gamma_k$, then $\gamma_k$ is not skippable.

\subsection{The template algorithm} \label{sec:template}
The procedure that we follow is closely modeled on the vertex avoiding case, but the details are somewhat different.  In particular, since we are only concerned with the template function $\theta$ and do not care whether it is a certificate of independence, we may assign multiple functions to achieve the minimum together on a loop $\gamma_k$.  In all cases, if $\psi$ and $\psi'$ are both assigned to $\gamma_k$, then $\psi \sim_{\gamma_k} \psi'$.  We also sometimes assign multiple functions with the same slope on a bridge to achieve the minimum together on that bridge. The details are as follows.

\medskip

\noindent \textbf{Start at the first bridge.}
Start at $\beta_1$.  Consider the set of slopes $\{s_1(\psi) \mid \psi \in \cB \}$.  For each such slope $s$ that is strictly greater than $\sigma_1$, choose all of the functions $\psi \in \cB$ with this slope and give each of them a finite coefficient $c(\psi)$ so that $\{ \psi + c(\psi) \mid s_1(\psi) = s \}$ achieve the minimum together on an open subset of $\beta_1$ of length $\epsilon^2 \cdot \ell_1$.  Assign all of these functions to $\beta_1$, i.e., set $\alpha(\psi) = \beta_1$.  Proceed to the first loop.

\medskip

\noindent \textbf{Loop subroutine.}
Each time we arrive at a loop $\gamma_k$, check whether there are any unassigned permissible functions. If not, proceed to $\beta_{k+1}$.  Otherwise, apply the following steps.



\medskip

\noindent \textbf{Loop subroutine, Step 1:  Align the unassigned permissible functions at $w_k$}
If $\psi$ is an unassigned permissible function, then either $c(\psi) = \infty$ or $\psi(w_k) + c(\psi)$ is less than $\psi'(w_k) + c(\psi')$ for any previously assigned or non-permissible function $\psi'$.  If $c(\psi) = \infty$, then set a finite coefficient to that $\psi + c(\psi)$ is equal to $\theta'$ at $v_k$.  Then adjust the coefficient of each unassigned permissible function upward, the smallest amount possible, so that all of these terms are equal to $\theta'$ at $w_k$.

\medskip

\noindent \textbf{Loop subroutine, Step 2:  Assign departing functions.}
If there are departing functions on $\gamma_k$, assign all of them to $\gamma_k$.  Note that any two departing functions are equivalent, by Lemma~\ref{Lem:OneDepartingFunction}.  Adjust the coefficients of the non-departing unassigned permissible functions upward so that they are all equal to the departing function with the smallest slope on $\beta_{k+1}$ at a point at distance $\epsilon \ell_k$ to the right of $w_k$. Then the departing functions are the only ones to achieve the minimum at $w_k$.  
Proceed to $\beta_{k+1}$.

\medskip

\noindent \textbf{Loop subroutine, Step 3: Skip skippable loops.} If $\gamma_k$ is skippable then proceed to $\beta_{k+1}$.

\medskip

\noindent \textbf{Loop Subroutine, Step 4: Assign an equivalence class of functions that achieves the minimum uniquely.}  Otherwise, there are unassigned permissible functions, none are departing, and the loop is not skippable.  We have aligned the unassigned permissible functions at $w_k$, so if $\psi \sim_{\gamma_k} \psi'$ are equivalent permissible functions, then $\psi + c(\psi)$ and $\psi' + c(\psi')$ achieve the minimum in $\theta$ on the same subset of $\gamma_k$.  Suppose there is an open subset $U$ where all of the functions that achieve the minimum are equivalent on $\gamma_k$.  Then choose one such subset, assign all of the functions in this equivalence class to $\gamma_k$, and increase the coefficients of these functions by $\frac{1}{3} m_k$. Note that they still achieve the minimum together on a smaller open subset of $\gamma_k$.  If there is no such subset $U$, do not assign any function to $\gamma_k$.  Proceed to $\beta_{k+1}$.

\medskip

\noindent \textbf{Internal bridge subroutine:} Upon arrival at $\beta_k$, for $1 < k < g + 1$, apply the following steps.

\medskip

 \noindent \textbf{Internal bridge subroutine, Step 1: If the slopes are steady, carry on.} If $\sigma_k = \sigma_{k-1}$ then proceed to the next loop $\gamma_k$.

\medskip

\noindent \textbf{Internal bridge subroutine, Step 2: Otherwise, be greedy.} If $1 < k < g + 1$ and $\sigma_k < \sigma_{k-1}$, then assign every possible function to $\beta_k$. More precisely, assign every unassigned function $\psi \in \cB$ with $\sigma_k < s_k(\psi) \leq \sigma_{k-1}$ to $\beta_k$.  Set the coefficients of the assigned functions so that each is greater than $\theta$ on $\gamma_{k-1}$ and achieves the minimum, together with all of the other assigned functions of the same slope, on a subinterval of $\beta_{k}$ of length $\epsilon \cdot \ell_k$s. Proceed to $\gamma_k$.

\medskip

\noindent \textbf{Final bridge subroutine.} When we arrive at the final bridge $\beta_{g+1}$, assign every function $\psi \in \cB$ with $\sigma_{g+1} < s_k(\psi) \leq \sigma_g$.  Set the coefficients of the assigned functions so that each is greater than $\theta$ on $\gamma_g$ and achieves the minimum, together with all of the other assigned functions of the same slope, on a subinterval of $\beta_{g+1}$ of length $\epsilon^2 \cdot \ell_{g+1}$.  Output the assignment function $\alpha$ and the template $\theta = \min \{ \psi + c(\psi) \mid \alpha(\psi) \neq 0 \}$.

\subsection{Choosing the slopes of the template}  \label{sec:choosingslopes}

For the rest of this section, we adopt the hypotheses of Theorem~\ref{thm:generalv2}.  Specifically, we let $g = 22$ or $23$, let $d = g + 3$ and let $g' \in \{ g, g-1, g-2 \}$.  Let $\Gamma'$ be a chain of $g'$ loops with $C$-admissible edge lengths for some $C > 12dg$, and let $D$ be a break divisor on $\Gamma'$ with $\Sigma \subset R(D)$ a tropical linear series of rank $6$.  Assume furthermore that
\begin{enumerate}
\item if $g' = g-1$ then $s'_0[5] \leq 2$;
\item if $g' = g-2$ then either $s'_0[5] \leq 2$ or $s'_0[6] + s'_0[4] \leq 5$.
\end{enumerate}

\medskip

We choose the integer sequence $\sigma = (\sigma_1, \ldots, \sigma_{g+1})$ to input into the template algorithm as follows.  First, we define a sequence of partitions (with possibly negative parts)
\[
\lambda'_0, \lambda_1, \lambda'_1 \ldots, \lambda'_{g'}, \lambda_{g'+1},
\]
each with at most $r + 1$ columns, numbered from 0 to $r$, that depends only on the slope vectors $s_k(\Sigma) = (s_k[0], \ldots, s_k[r])$.  The $(r-i)$th column of $\lambda_k$ contains $(g-d+r)+ s_k[i]-i$ boxes.  (Note that this depends on $g$, not $g'$.)  Similarly, the $(r-i)$th column of $\lambda'_k$ contains $(g-d+r) + s'_k[i]-i$ boxes.  For more on partitions with negative parts, also known as signed partitions, see \cite{Andrews07}.

As an example, consider the case where $g'=g$.  If $\wt(w_0) = 0$, then $s'_0[i] \geq (g-d+r)-i$ for all $i$, so $\lambda'_0$ has no negative parts.  On the other hand, if $\wt (w_0) \geq 1$, then $s'_0[0] \leq -4$.  The last column of $\lambda'_0$ therefore contains $-1$ boxes.  In this case, the partition $\lambda'_0$ has a negative part.

By Proposition~\ref{Prop:LingeringLatticePath}, $\lambda_k$ is a subset of $\lambda'_{k-1}$, and $\lambda'_k$ contains at most one box that is not contained in $\lambda_k$.  In particular, $\mu (\beta_k) = \vert \lambda'_{k-1} \vert - \vert \lambda_k \vert$, and similarly, $\mu (\gamma_k) = \vert \lambda_k \vert - \vert \lambda'_k \vert - 1$.   Moreover, $\lambda_{g'+1}$ contains the $(r+1) \times (g-d+r)$ rectangle.  In the vertex avoiding case, each partition contains the partition that precedes it, and the sequence corresponds to the associated tableau discussed in \S\ref{Sec:VertexAvoiding}.

Next, as in the vertex avoiding case, we specify four indices $z$, $z'$, $b$, and $b'$ (cf. Definitions~\ref{def:z} and \ref{def:b}).  These indices depend only on the sequence of partitions.

\begin{definition}
Let $z$ be the largest integer such that $\lambda'_z$ contains exactly 6 boxes in the union of the first two rows, and $\lambda_z$ does not.  Similarly, let $z'$ be the largest integer such that $\lambda'_{z'}$ contains exactly 10 boxes in the union of the second and third row, and $\lambda_{z'}$ does not.  Let $b$ be the largest integer such that $\lambda'_b$ contains exactly 7 boxes in its first two rows, and $\lambda_b$ does not.  Similarly, let $b'$ be the largest integer such that $\lambda'_{b'}$ contains exactly 8 boxes in the union of its first and third row, and $\lambda_{b'}$ does not.
\end{definition}

\noindent Since each partition in the sequence contains at most 1 box not contained in the previous partition, such indices exist.  In the vertex avoiding case, this definition agrees with Definitions~\ref{def:z} and \ref{def:b}.

As in \S\ref{Sec:VertexAvoiding}, the slopes of the master template $\theta$ are given in terms of $z$ and $z'$ by:

\begin{equation} \label{eq:sigma}
\sigma_k = \left\{ \begin{array}{ll}
4 & \textrm{if $1 \leq k \leq z$} \\
3 & \textrm{if $z+1 \leq k \leq z'-2$} \\
2 & \textrm{if $z'-1 \leq k \leq g'+1$.}
\end{array} \right.
\end{equation}

\subsection{Shiny functions} \label{sec:shiny}
Our study of the template algorithm in the cases needed to prove Theorem~\ref{thm:generalv2} depends heavily on the following technical notion.  The functions in $2\cA$ that we call \emph{shiny} play a role analogous to the new permissible functions on loops $\gamma_k$ for $k \not \in \{ 1, z + 1, z'-1 \}$ in the vertex avoiding case, cf. Definition~\ref{def:new}, Lemma~\ref{lem:onenew} and Proposition~\ref{prop:nonew}.

The slopes $s_k[\tau_k(\varphi)]$ and $s'_k[\tau'_k(\varphi)]$ appeared once earlier, in the definition of building blocks. 
They appear frequently in the arguments that follow, so we introduce the  separate notation:
\[
\xi_k(\varphi) := s_k[\tau_k(\varphi)] \quad \quad \mbox{ and } \quad \quad \xi'_k(\varphi) := s'_k[\tau'_k(\varphi)].
\]
With this notation, formula \eqref{eq:xibounds} of Definition~\ref{def:buildingblock} says that if $\varphi$ is a building block then
\begin{equation} \label{eq:xi-inequality}
\xi_k(\varphi) \leq s_k(\varphi) \leq \xi'_{k-1}(\varphi).
\end{equation}

\begin{definition}
We say that $\psi = \varphi + \varphi' \in 2\cA$ is \emph{shiny} on $\gamma_k$ if it is permissible on $\gamma_k$, and
\[
\xi_k (\varphi) + \xi_k (\varphi') < \sigma_k.
\]
\end{definition}

\begin{lemma} \label{lem:smallslopeshiny}
Let $\psi = \varphi + \varphi' \in 2\cA$.  Assume that $\psi$ is permissible on $\gamma_k$, and $s_k (\varphi) > \xi_k (\varphi)$.  Then $\psi$ is shiny on $\gamma_k$.
\end{lemma}

\begin{proof}
If $\psi$ is permissible on $\gamma_k$, then
\[
\sigma_k \geq s_k (\psi) = s_k (\varphi) + s_k (\varphi') > \xi_k (\varphi) + \xi_k (\varphi'),
\]
so $\psi$ is shiny.
\end{proof}

We retain Definition~\ref{def:new}; a function $\varphi \in \cB$ is a \emph{new permissible function} on $\gamma_k$ if it is permissible on $\gamma_k$ and is not permissible on $\gamma_j$ for any $j < k$.

\begin{lemma}
\label{Lem:NewIsShiny}
If $k \notin \{ 1, z+1, z'-1 \}$, then any new permissible function on $\gamma_k$ is shiny.
\end{lemma}

\begin{proof}
Let $\psi = \varphi + \varphi' \in 2\cA$ be a new function on $\gamma_k$.  If $k \notin \{ 1, z+1, z'-1 \}$, we have $\sigma_{k-1} = \sigma_k$.  It follows that
\[
\sigma_k > s_k (\psi) = s_k (\varphi) + s_k (\varphi') \geq \xi_k (\varphi) + \xi_k (\varphi'),
\]
hence $\psi$ is shiny.
\end{proof}

In the remainder of this subsection, we prove two technical propositions about shiny functions.  We start by establishing a lemma about the divisors associated to building blocks.

\begin{lemma}
\label{Lem:VDegree}
If $\varphi$ is a $(\Sigma, \{f_{k,i}\})$-building block, then
\[
d_{v_k} (\varphi) \geq s_k (\varphi) - \xi_k (\varphi),
\]
and similarly,
\[
d_{w_k} (\varphi) \geq \xi'_k  (\varphi) - s'_k (\varphi).
\]
\end{lemma}

\begin{proof}
By definition, the functions $\varphi$ and $f_{k,\tau'_k(\varphi)}$ are equivalent on the loop $\gamma_k$.  It follows from (\ref{eq:fki}) that
\[
d_{v_k} (\varphi) = d_{w_k} (f_{k,\tau'_k(\varphi)}) - s'_k (\varphi) \geq \xi'_k (\varphi) - s'_k (\varphi).
\]
Similarly,
\[
d_{v_k} (\varphi) = d_{v_k} (f_{k,\tau'_k}) + s_k (\varphi) .
\]
If $\tau'_k (\varphi) = \tau_k (\varphi)$, then by (\ref{eq:fki}), $d_{v_k} (f_{k,\tau'_k}) \geq - \xi_k (\varphi)$, and the result follows.  On the other hand, if $\tau'_k (\varphi) = \tau_k (\varphi) +1$, then $\gamma_k$ is a $\tau_k (\varphi)$-switching loop.  Again, by (\ref{eq:fki}), $d_{v_k} (f_{k,\tau'_k}) \geq - \xi_k (\varphi)$, and the result follows.
\end{proof}

We now show that all shiny functions on $\gamma_k$ are equivalent.  Lemma~\ref{Lem:NewIsShiny} and Proposition~\ref{prop:equivtonew} together give a strong  analogue of Lemma~\ref{lem:onenew} from the vertex avoiding case.

\begin{proposition} \label{prop:equivtonew}
Any two shiny functions on $\gamma_k$ are equivalent on $\gamma_k$.  Moreover, if $\psi \in 2\cA$ is shiny on $\gamma_k$, then the restriction of $2D+\ddiv(\psi)$ to $\gamma_k \smallsetminus \{ v_k \}$ has degree at most 1.
\end{proposition}

\begin{proof}
Let $\psi = \varphi_1 + \varphi_2 \in 2\cA$ be shiny on $\gamma_k$.  Consider the function $\varphi'_1 \in \PL(\Gamma)$ that is equivalent to $\varphi_1$ on each of the two connected components of $\Gamma \smallsetminus \beta_k$, and with constant slope $\xi_k (\varphi_1)$ on $\beta_k$.  By Lemma~\ref{Lem:VDegree}, $\xi_k (\varphi_1) \geq s_k (\varphi_1) - d_{v_k} (\varphi_1)$, so $\varphi'_1 \in R(D)$.  Similarly, the function $\varphi'_2 \in \PL(\Gamma)$ that is equivalent to $\varphi_2$ on each of the two connected components of $\Gamma \smallsetminus \beta_k$, and with constant slope $\xi_k (\varphi_2)$ on $\beta_k$, is contained in $R(D)$.  Let $\psi' = \varphi'_1 + \varphi'_2$.

By defintion, $\psi \sim_{\gamma_k} \psi'$, and
\[
s_k (\psi') = \xi_k (\varphi_1) + \xi_k (\varphi_2) \leq \sigma_k - 1.
\]
By Lemma~\ref{Lem:RestrictOneLoop},
\[
\ddiv(\psi'_{|\gamma_k}) = (\ddiv(\psi'))_{|\gamma_k} - s_k(\psi') \cdot v_k + s'_k(\psi') \cdot w_k.
\]
Hence $\ddiv(\psi'_{|\gamma_k}) + 2D_{|\gamma_k} + (\sigma_k -1) v_k - \sigma_k w_k$ is effective.  But $2D_{|\psi'_k} + (\sigma_k -1) v_k - \sigma_k w_k$ is a divisor on $\gamma_k$ of degree 1, so by Lemma~\ref{Lem:EquivOneLoop} it is equivalent to a unique effective divisor.  It follows that $\psi'_{|\gamma_k}$, and hence $\psi_{|\gamma_k}$, is uniquely determined up to an additive constant.

Finally, note that the restriction of $2D+\ddiv(\psi)$ to $\gamma_k$ differs from the restriction of $2D+\ddiv(\psi')$ to $\gamma_k$ only at $v_k$, and the latter has degree at most 1.  It follows that the restriction of $2D+\ddiv(\psi)$ to $\gamma_k \smallsetminus \{ v_k \}$ has degree at most 1.
\end{proof}

\begin{proposition}
\label{Prop:Shiny}
If $\psi = \varphi + \varphi' \in 2\cA$ is shiny on $\gamma_k$, then the restriction of either $D+\ddiv (\varphi)$ or $D+\ddiv (\varphi')$ to $\gamma_k \smallsetminus \{ v_k \}$ has degree 0.  Moreover, either
\[
s_{k+1} (\varphi) > \xi_k (\varphi) \mbox{ or } s_{k+1} (\varphi') > \xi_k (\varphi').
\]
\end{proposition}

\begin{proof}
By Proposition~\ref{prop:equivtonew}, the restriction of $2D+\ddiv(\psi)$ to $\gamma_k \smallsetminus \{ v_k \}$ has degree at most 1.  It follows that the restriction of either $D+\ddiv(\varphi)$ or $D+\ddiv(\varphi')$ to $\gamma_k \smallsetminus \{ v_k \}$ must have degree 0.

Now, since $\psi = \varphi + \varphi'$ is permissible, we have
\[
s_{k+1} (\varphi) + s_{k+1} (\varphi') \geq \sigma_k .
\]
By definition, since $\psi$ is shiny, we have
\[
\xi_k (\varphi) + \xi_k (\varphi') < \sigma_k.
\]
It follows that either $
s_{k+1} (\varphi) > \xi_k (\varphi) \mbox{ or } s_{k+1} (\varphi') > \xi_k (\varphi')$, as required.
\end{proof}

\subsection{Two technical conditions on the set of functions $\cB$}  \label{sec:BB'}

We write $\Gamma_{\leq k}$ for the subgraph of $\Gamma$ to the left of $w_k$, i.e., $\Gamma_{\leq k}$ is the union of the loops $\gamma_i$ and bridges $\beta_j$ for $1 \leq i,j \leq k$.

\begin{definition}  \label{def:properties}
Let $\cA$ be a set of building blocks, and let $\cB \subseteq 2\cA$.  We consider the following two properties:
\begin{itemize}
\item[$(\mathbf B)$]  Whenever there is a permissible function $\psi = \varphi + \varphi' \in \cB$ on $\gamma_k$ such that $2D+\ddiv(\psi)$ contains $w_k$, and either $\gamma_k$ is $\tau_k (\varphi)$-switching or $s_{k+1} (\varphi) < \xi'_k  (\varphi)$, then there is some permissible function $\psi' \in \cB$ that is equal to $\psi$ on $\Gamma_{\leq k}$ such that $s_{k+1} (\psi') > s_{k+1} (\psi)$.
\item[$(\mathbf B')$]  Whenever there are permissible functions in $\cB$ that are equivalent on $\gamma_k$ with different slopes on $\beta_{k+1}$, and either $\gamma_k$ is a switching loop or $\beta_{k+1}$ is a switching bridge, then no function in $\cB$ is shiny on $\gamma_k$.
\end{itemize}
\end{definition}

\begin{example}
\label{Ex:Genus5Properties}
Consider the set of functions $\cA = \{ \varphi_0 , \varphi_1 , \varphi_2, \varphi_3, \varphi_{\tau} \}$ from Example~\ref{Ex:Genus5}.  Let $\sigma_k = 2$ for all $k$.  We show that the set $\cB = 2\cA$ satisfies property $(\mathbf B)$.  To see this, note that all $\varphi \in \cA$ satisfy $s_{k+1} (\varphi) = \xi'_k (\varphi)$ for all $k$, so it suffices to check property $(\mathbf B)$ on the switching loop $\gamma_4$.  The functions $\psi = 2\varphi_1$ and $\psi' = \varphi_1 + \varphi_{\tau}$ are the permissible functions on $\gamma_4$ such that $2D + \ddiv (\psi), 2D + \ddiv(\psi')$ contain $w_4$.  Both functions are equivalent to the function $\psi'' = 2\varphi_{\tau}$ on $\Gamma_{\leq 4}$, and $s_5 (\psi'') > s_5 (\psi') > s_5 (\psi)$.  Thus, $\cB$ satisfies property $(\mathbf B)$.  (This contrasts with the subsets $\cB \smallsetminus \{\psi''\}$ and $\cB \smallsetminus \{ \psi' \}$, which do not satisfy property $(\mathbf B)$.)

The set $\cB$ also satisfies property $(\mathbf B')$, because no functions are shiny on $\gamma_4$.  To see this, note that all $\varphi \in \cA$ satisfy $s_4 (\varphi) = \xi_4 (\varphi)$, so a shiny function must have slope strictly less than 2 on $\beta_4$.  The only such functions are $2\varphi_0 , \varphi_0 + \varphi_1 , \varphi_0 + \varphi_2$, and $\varphi_0 + \varphi_{\tau}$.  By inspection, if $\psi$ is one of these functions, then $s_5 (\psi) \leq 1$, so $\psi$ is not permissible on $\gamma_4$.
\end{example}

\begin{example}
\label{Ex:Genus22Properties}
We consider the following modification of Example~\ref{ex:randomtableau}, which we will use as a second example (following the simpler genus 5 example above) to illustrate the template algorithm and its output in a situation relevant to the proof of our main results; see Example~\ref{Ex:Genus22Template}.  Let $\Sigma$ be a tropical linear series of rank 6 on the chain of 22 loops $\Gamma$, with $s_k[i] = s'_{k-1}[i]$ the same as in Example~\ref{ex:randomtableau} for all $i$ and $k$, but where the loop $\gamma_{18}$ is 3-switching.  As in Example~\ref{ex:randomtableau}, we set
\begin{displaymath}
\sigma_k = \left\{ \begin{array}{ll}
4 & \textrm{if $k \leq 7$,} \\
3 & \textrm{if $7 < k \leq 15$,}\\
2 & \textrm{if $16 < k \leq 23$.}
\end{array} \right.
\end{displaymath}
There are 8 building sequences:  the constant sequence $i$ for $0 \leq i \leq 6$, and the sequence
\[
\tau_k = \tau'_{k-1} = \begin{cases} 3 & \mbox{if $k \leq 18$} \\  4 & \mbox{if $k \geq 19$.} \end{cases}
\]
As in Example~\ref{Ex:Genus5}, because $s_k[i] = s'_{k-1}[i]$ for all $i$ and $k$, each building sequence corresponds to a unique building block.  We denote the building block corresponding to the constant sequence $i$ by $\varphi_i$, and the building block corresponding to $\tau$ by $\varphi_{\tau}$.

The set $\cB = 2\cA$ does not satisfy property $(\mathbf B')$, because the two permissible functions $\varphi_1 + \varphi_3$ and $\varphi_1 + \varphi_{\tau}$ are equivalent on $\gamma_{18}$ and have different slopes along $\beta_{19}$, and $\varphi_0 + \varphi_{\tau}$ is shiny on $\gamma_{18}$.  However, the sets
\[
\cB' = \cB \smallsetminus \{ \varphi_0 + \varphi_{\tau} \} \quad \quad \cB'' = \cB \smallsetminus \{ \varphi_1 + \varphi_3 \}
\]
do satisfy property $(\mathbf B')$.  Both sets also satisfy property $(\mathbf B)$ (as does $\cB$ itself), because the only permissible function $\psi \in \cB$ such that $2D + \ddiv (\psi)$ contains $w_{18}$ is $\psi = \varphi_1 + \varphi_3$.  The function $\psi' = \varphi_1 + \varphi_{\tau} \in \cB' \cap \cB''$ is equivalent to $\psi$ on $\Gamma_{\leq 18}$, and $s_{19} (\psi') > s_{19} (\psi)$.  The fact that $\cB'$ and $\cB''$ satisfy both properties contrasts, e.g., with $\cB' \smallsetminus \{ \psi'\}$, which satisfies $(\mathbf B')$ but not $(\mathbf B)$.
\end{example}

\subsection{Properties of the template, revisited}

We now state the main result of this section, which gives the essential technical properties of the output of template algorithm in the specific cases needed for the proof of our main results.  The rest of this section will be devoted to proving it.

\begin{theorem}
\label{Thm:ExtremalConfig}
Let $g = 22$ or $23$, let $d = g + 3$ and let $g' \in \{ g, g-1, g-2 \}$.  Let $\Gamma'$ be a chain of $g'$ loops with $C$-admissible edge lengths for some $C > 12dg$, and let $D$ be a break divisor on $\Gamma'$ with $\Sigma \subset R(D)$ a tropical linear series of rank $6$.  Assume furthermore that
\begin{itemize}
\item if $g' = g-1$ then $s'_0[5] \leq 2$, and
\item if $g' = g-2$ then either $s'_0[5] \leq 2$ or $s'_0[6] + s'_0[4] \leq 5$.
\end{itemize}
Let $\cB \subseteq 2\cA$ satisfy properties $(\mathbf B)$ and $(\mathbf B')$, and define $\sigma_k$ for all $k$ as in (\ref{eq:sigma}).  Then:
\begin{enumerate}
\item  \label{it:everyfunction} every function in $\cB$ is assigned to a loop $\gamma_k$ or a bridge $\beta_k$;
\item  \label{it:assignment} the function $\psi + c(\psi)$ achieves the minimum on an open susbet of the loop or bridge $\alpha (\psi)$;
\item  \label{it:sameloop} if $\varphi_1 + \varphi'_1$ and $\varphi_2 + \varphi'_2$ are assigned to the same loop $\gamma_k$, then, after possibly reordering, $\varphi_1 \sim_{\gamma_k} \varphi_2$ and $\varphi'_1 \sim_{\gamma_k} \varphi'_2$,
\item \label{it:oneloop} if $\varphi_1 + \varphi'_1$ and $\varphi_2 + \varphi'_2$ are assigned to the same loop $\gamma_k$ and $\gamma_k$ is not $i$-switching for $i \in \{ \tau_k (\varphi_1), \tau_k (\varphi_2), \tau_k (\varphi'_1) , \tau_k (\varphi'_2) \}$, then $\{ \tau'_k (\varphi_1) , \tau'_k (\varphi'_1) \} = \{ \tau'_k (\varphi_2) , \tau'_k (\varphi'_2) \}$, and
\item  \label{it:samebridge} if $\varphi_1 + \varphi'_1$ and $\varphi_2 + \varphi'_2$ are assigned to the same bridge $\beta_k$ and have the same slope on $\beta_k$, for $k > 1$, then $\{ \tau'_{k-1} (\varphi_1), \tau'_{k-1} (\varphi'_1) \} = \{ \tau'_{k-1} (\varphi_2), \tau'_{k-1} (\varphi'_2) \}$.
\end{enumerate}
Moreover, if $k = 1$, then \eqref{it:samebridge} holds as long as $\mu (\beta_1) < 2$.
\end{theorem}

Before we begin the proof of Theorem~\ref{Thm:ExtremalConfig}, we describe the template constructed by our algorithm in Examples~\ref{Ex:Genus5Properties} and~\ref{Ex:Genus22Properties}. We illustrate first with a simpler case, in genus 5.

\begin{example}
\label{Ex:Genus5Template}
Consider the set $\cB$ of Example~\ref{Ex:Genus5Properties}.  We saw above that, when $\sigma_k =2$ for all $k$, then the set $\cB$ satisfies properties $(\mathbf B)$ and $(\mathbf B')$.  We describe the sets of functions assigned to each loop and bridge by the template algorithm.

Since $s_1 (2\varphi_3) > s_1 (\varphi_2 + \varphi_3) > 2$, we set $\alpha (2\varphi_3) = \alpha (\varphi_2 + \varphi_3) = \beta_1$.  Proceeding to $\gamma_1$, we see that $\varphi_1 + \varphi_3 \sim_{\gamma_1} \varphi_{\tau} + \varphi_3$ are departing, so we set $\alpha (\varphi_1 + \varphi_3) = \alpha (\varphi_{\tau} + \varphi_3) = \gamma_1$.  On $\gamma_2$, the function $2\varphi_2$ is departing, so $\alpha (2\varphi_2) = \gamma_2$.  On $\gamma_3$, $\varphi_1 + \varphi_2 \sim_{\gamma_3} \varphi_{\tau} + \varphi_2$ are departing, so $\alpha (\varphi_1 + \varphi_2) = \alpha (\varphi_{\tau} + \varphi_2) = \gamma_3$.  On $\gamma_4$, $2\varphi_1 \sim_{\gamma_4} \varphi_1 + \varphi_{\tau} \sim_{\gamma_4} 2\varphi_{\tau}$, but only the latter two functions are departing.  Hence, $\alpha ( \varphi_1 + \varphi_{\tau} ) = \alpha (2\varphi_{\tau}) = \gamma_4$.  On $\gamma_5$, $\varphi_0 + \varphi_3$ is departing, so $\alpha (\varphi_0 + \varphi_3) = \gamma_5$.  The remaining functions are assigned to the final bridge $\beta_6$.  The template $\theta$ is pictured in Figure~\ref{Fig:Genus5Template}.

\begin{figure}[H]
\begin{tikzpicture}[thick, scale=0.63]
\begin{scope}[grow=right, baseline, shift={(-6,0)}]
\draw (-3.25,7) node {$\theta$};
\draw (-5,6)--(-2,6);
\draw (-1,6) circle (1);
\draw (0,6)--(1,6);
\draw (2,6) circle (1);
\draw (3,6)--(4,6);
\draw (5,6) circle (1);
\draw (6,6)--(7,6);
\draw (8,6) circle (1);
\draw (9,6)--(10,6);
\draw (11,6) circle (1);
\draw (12,6)--(15,6);
\draw [ball color=black] (-4,6) circle (0.75mm);
\draw (-4.5,5.75) node {\tiny 33};
\draw [ball color=black] (-3,6) circle (0.75mm);
\draw (-3.5,5.75) node {\tiny 23};
\draw [ball color=black] (-1,7) circle (0.75mm);
\draw [ball color=black] (0.2,6) circle (0.75mm);
\draw (-1,4.75) node {\tiny 13, $\tau$3};
\draw [ball color=black] (3.2,6) circle (0.75mm);
\draw (3.2,5.75) node {\tiny 2};
\draw (2,4.75) node {\tiny 22};
\draw [ball color=black] (4.13,6.5) circle (0.75mm);
\draw [ball color=black] (6.2,6) circle (0.75mm);
\draw (5,4.75) node {\tiny 12, $\tau$2};
\draw [ball color=black] (9,6) circle (0.75mm);
\draw [ball color=black] (9.2,6) circle (0.75mm);
\draw (8,4.75) node {\tiny 1$\tau$, $\tau \tau$};
\draw [ball color=black] (10.5,6.87) circle (0.75mm);
\draw [ball color=black] (12.2,6) circle (0.75mm);
\draw (11,4.75) node {\tiny 03};
\draw (12.5,4.75) node {\tiny 11, 02, 0$\tau$};
\draw (12.5,5.5) node {$\uparrow$};
\draw [ball color=black] (13,6) circle (0.75mm);
\draw (13.5,5.75) node {\tiny 01};
\draw [ball color=black] (14,6) circle (0.75mm);
\draw (14.5,5.75) node {\tiny 00};
\end{scope}

\end{tikzpicture}
\caption{The template $\theta$ of Example~\ref{Ex:Genus5Template}.}
\label{Fig:Genus5Template}
\end{figure}

The functions assigned to the bridge $\beta_6$ satisfy the following inequalities on slopes:
\[
s_6 (2\varphi_1) = s_6 (\varphi_0 + \varphi_2) = s_6 (\varphi_0 + \varphi_{\tau}) > s_6 (\varphi_0 + \varphi_1) > s_6 (2\varphi_0) .
\]
Thus, the functions $2\varphi_1$, $\varphi_0 + \varphi_2$, and $\varphi_0 + \varphi_{\tau}$ achieve the minimum on the same open subset of the bridge $\beta_6$.  It follows that the template $\theta$ does not satisfy part (\ref{it:samebridge}) of Theorem~\ref{Thm:ExtremalConfig}.  It does satisfy the other parts, however:  every function in $\cB$ is assigned to a loop or a bridge, each function achieves the minimum on an open subset of the loop or bridge to which it is assigned, and functions assigned to the same loop have equivalent summands on that loop.
\end{example}

\begin{example}
\label{Ex:Genus22Template}
Consider the sets $\cB'$ and $\cB''$ from Example~\ref{Ex:Genus22Properties}.  We describe the template $\theta$ for each of the two sets by explicitly comparing it to the template constructed in Example~\ref{ex:randomtableau}.  Using either $\cB'$ or $\cB''$, the functions assigned to loops $\gamma_k$ or bridges $\beta_k$ with $k \leq 18$ are the same.  Specifically, to each loop $\gamma_k$ with $k \leq 17$ or bridge $\beta_k$ with $k \leq 18$, we assign the function $\varphi_i + \varphi_j$ if and only if it is assigned to that loop or bridge in Example~\ref{ex:randomtableau}.  If $\varphi_3 + \varphi_i$ is assigned to a loop $\gamma_k$ with $k \leq 17$ or a bridge $\beta_k$ with $k \leq 18$, then because $\varphi_{\tau}$ is equivalent to $\varphi_3$ on $\Gamma_{\leq 18}$, we assign the function $\varphi_{\tau} + \varphi_i$ to that loop or bridge as well.  Because $\varphi_1 + \varphi_{\tau} \in \cB' \cap \cB''$ is the only departing function on $\gamma_{18}$, we have $\alpha (\varphi_1 + \varphi_{\tau}) = \gamma_{18}$.

We first consider the set $\cB'$.  Every function $\varphi_i + \varphi_j$ that is assigned to a loop $\gamma_k$ or bridge $\beta_k$ with $k \geq 19$ in Example~\ref{ex:randomtableau} is assigned to that same loop or bridge.  If $\varphi_4 + \varphi_i$ is assigned to a loop $\gamma_k$ or bridge $\beta_k$ with $k \geq 19$ in Example~\ref{ex:randomtableau}, then $\varphi_{\tau} + \varphi_1$ is assigned to that same loop or bridge.

Next, consider the set $\cB''$.  Since $\varphi_1 + \varphi_3 \notin \cB''$, it is not assigned to $\gamma_{19}$.  Instead, $\varphi_0 + \varphi_{\tau}$ is assigned to $\gamma_{19}$.  Every function $\varphi_i + \varphi_j$ that is assigned to a loop $\gamma_k$ or bridge $\beta_k$ with $k \geq 20$ in Example~\ref{ex:randomtableau} is then assigned to that same loop or bridge.  If $\varphi_4 + \varphi_i$ is assigned to a loop $\gamma_k$ or bridge $\beta_k$ with $k \geq 20$ in Example~\ref{ex:randomtableau}, then $\varphi_{\tau} + \varphi_1$ is assigned to that same loop or bridge.  In both cases, we see that the template $\theta$ and assignment function $\alpha$ satisfies Theorem~\ref{Thm:ExtremalConfig}.

Recall that the set $\cB$ of Example~\ref{Ex:Genus22Properties} does not satisfy property $(\mathbf B')$.  We show that, if we run the algorithm on this set $\cB$, the resulting template $\theta$ and assignment function $\alpha$ do not satisfy Theorem~\ref{Thm:ExtremalConfig}.  For every $k \geq 19$, there is an unassigned departing function in $\cB$ on $\gamma_k$.  Since departing functions are always assigned, we have $\alpha (\varphi_1 + \varphi_3) = \gamma_{19}$, $\alpha (\varphi_0 + \varphi_{\tau}) = \gamma_{20}$, $\alpha (\varphi_1 + \varphi_2) = \gamma_{21}$, and $\alpha (\varphi_0 + \varphi_3) = \gamma_{22}$.  Thus, when we reach the last bridge, the functions $\varphi_0 + \varphi_2$ and $2\varphi_1$ have not been assigned.  They are therefore both assigned to the last bridge $\beta_{23}$, but since they have the same slope along $\beta_{23}$, they both achieve the minimum on the same open subset of this bridge, contradicting part (\ref{it:samebridge}) of Theorem~\ref{Thm:ExtremalConfig}
\end{example}

For the rest of this section, we assume the hypotheses of the theorem; in particular, $\cA$ is a set of building blocks and $\cB \subseteq 2\cA$ satisfies $(\mathbf B)$ and $(\mathbf B')$.  Part (\ref{it:assignment}) of Theorem~\ref{Thm:ExtremalConfig} was proved at the end of Section~\ref{sec:template}.

The main difficulty is proving that every function is assigned to a bridge or loop.  This is a counting argument and is somewhat more intricate than in the vertex avoiding case.  The proofs of properties \eqref{it:sameloop}, \eqref{it:oneloop}, and \eqref{it:samebridge} are shorter; we prove \eqref{it:sameloop} and \eqref{it:oneloop} at the beginning and \eqref{it:samebridge} at the end.

\subsubsection{Functions assigned to the same loop}
\begin{proof}[Proof of Theorem~\ref{Thm:ExtremalConfig}\eqref{it:sameloop}]
Since $\psi_1 = \varphi_1 + \varphi'_1$ and $\psi_2 = \varphi_2 + \varphi'_2$ are assigned to the same loop $\gamma_k$, we have $\psi_1 \sim_{\gamma_k} \psi_2$.  It therefore suffices to show that, after possibly reordering, $\varphi_1 \sim_{\gamma_k} \varphi_2$.

The fact that $\psi_1$ is equivalent to $\psi_2$ on $\gamma_k$ is equivalent to the statement that the restrictions of $2D + \ddiv(\psi_1)$ and $2D + \ddiv(\psi_2)$ to $\gamma_k \smallsetminus \{ v_k , w_k \}$ are the same.  Thus, if $D+\ddiv(\varphi_1)$ contains a point $v \in \gamma_k \smallsetminus \{ v_k , w_k \}$, then one of $D+\ddiv(\varphi_2)$ or $D+\ddiv(\varphi'_2)$ must contain $v$ as well.  If $v$ is the only point of $\gamma_k \smallsetminus \{ v_k , w_k \}$ contained in both $D+\ddiv(\varphi_1)$ and $D+\ddiv(\varphi_2)$, then $\varphi_1 \sim_{\gamma_k} \varphi_2$.  Note that, since a function is assigned to $\gamma_k$, it is not a skippable loop.  Thus, the restrictions of $D+\ddiv(\varphi_1), D+\ddiv(\varphi'_1), D+\ddiv(\varphi_2)$, and $D+\ddiv(\varphi'_2)$ to $\gamma_k \smallsetminus \{ v_k , w_k \}$ all have degree at most 1, and the conclusion holds.
\end{proof}

\begin{proof}[Proof of Theorem~\ref{Thm:ExtremalConfig}\eqref{it:oneloop}]
Let $\psi_1 = \varphi_1 + \varphi'_1 , \psi_2 = \varphi_2 + \varphi'_2$.  By Theorem~\ref{Thm:ExtremalConfig}\eqref{it:sameloop}, after possibly reordering, we have $\varphi_1 \sim_{\gamma_k} \varphi_2$ and $\varphi'_1 \sim_{\gamma_k} \varphi'_2$.  Suppose that $\tau'_k (\varphi_1) \neq \tau'_k (\varphi_2)$.   Without loss of generality, suppose that $\tau'_k (\varphi_1) < \tau'_k (\varphi_2)$.  We first show that either $d_{w_k} (\varphi_1) \geq 1$ or $d_{w_k} (\varphi_2) \geq 1$.  If either $s'_k (\varphi_1) < \xi'_k (\varphi_1)$ or $s'_k (\varphi_2) < \xi'_k (\varphi_2)$, this follows from Lemma~\ref{Lem:VDegree}.  Otherwise, we have $s'_k (\varphi_1) < s'_k (\varphi_2)$.  Since $\varphi_1 \sim_{\gamma_k} \varphi_2$, it follows that $d_{w_k} (\varphi_1) > d_{w_k} (\varphi_2) \geq 0$.

Without loss of generality, assume that $d_{w_k} (\varphi_1) \geq 1$.  Since $\alpha (\psi_1) = \gamma_k$, the loop $\gamma_k$ is not skippable.  Since $2D + \ddiv(\psi_1)$ contains $w_k$, there must be unassigned departing functions on $\gamma_k$.  Because $\alpha (\psi_1) = \alpha (\psi_2) = \gamma_k$, $\psi_1$ and $\psi_2$ must be departing functions.  This implies that $(2D + \ddiv(\psi_1))_{|\gamma_k}$ has degree at most 1, hence $(2D + \ddiv(\psi_1))_{|\gamma_k} = w_k$.  Hence, $\varphi_1 \sim_{\gamma_k} \varphi'_1$, and since $d_{v_k} (\varphi_1) = d_{v_k} (\varphi'_1) = 0$, we have $\xi_k (\varphi_1) = s_k (\varphi_1) = s_k (\varphi'_1) = \xi_k (\varphi'_1)$.  It follows that $\tau_k (\varphi_1) = \tau_k (\varphi'_1)$, and since $\gamma_k$ is not $\tau_k(\varphi_1)$-switching, $\tau'_k (\varphi_1) = \tau'_k (\varphi'_1)$ as well.

We have now demonstrated the equivalences: $\varphi_1 \sim_{\gamma_k} \varphi'_1$, $\varphi_1 \sim_{\gamma_k} \varphi_2,$ and $\varphi'_1 \sim_{\gamma_k} \varphi'_2$.  It follows that $\varphi_2 \sim_{\gamma_k} \varphi'_2$ as well.  Since $\psi_2$ is departing, either $d_{w_k} (\psi_2) = 0$ or $d_{v_k} (\psi_2) = 0$.  If $d_{v_k} (\psi_2) = 0$, then $\tau'_k (\varphi_2) = \tau'_k (\varphi'_2)$ by the same argument as the previous paragraph.  If $d_{w_k} (\psi_2) = 0$, then $\xi'_k (\varphi_2) = s'_k (\varphi_2) = s'_k (\varphi'_2) = \xi'_k (\varphi'_2)$, hence $\tau'_k (\varphi_2) = \tau'_k (\varphi'_2)$.  Finally, since $\psi_2$ is departing, we see that either $(D + \ddiv(\varphi_2))_{|\gamma_k} = 0$ or $(D + \ddiv(\varphi'_2))_{|\gamma_k} = 0$.  Since $(D + \ddiv(\varphi'_1))_{|\gamma_k} = 0$ as well, we see that $\tau'_k (\varphi_2) = \tau'_k (\varphi'_1)$, and the result follows.
\end{proof}

\subsubsection{Preparation for the counting argument}

In preparation for the counting argument used to prove Theorem~\ref{Thm:ExtremalConfig}, we show that, whenever we arrive at Step 4 in the loop subroutine, we are able to assign an equivalence class of functions to the given loop.

\begin{lemma}
\label{Lem:AtMostThree}
If there is no unassigned departing permissible function on $\gamma_k$, then the number of equivalence classes of permissible functions on $\gamma_k$ is at most 3.
\end{lemma}

\begin{proof}
Consider the set of building blocks $\varphi \in \cA$ such that there exists $\varphi' \in \cA$ with $\varphi + \varphi' \in \cB$ an unassigned permissible function on $\gamma_k$.  If every such building block satisfies $s_{k+1} (\varphi) = \xi'_k (\varphi)$, then the proof of Lemma~\ref{Lem:VAAtMostThree} goes through essentially unchanged.  Indeed, in this case any two functions in $\cA$ with the same slope along $\beta_{k+1}$ are equivalent on $\gamma_k$.  Thus, the number of equivalence classes of permissible functions on $\gamma_k$ is bounded above by the number of pairs $(i,j)$ such that $s'_k[i]+s'_k[j] = \sigma_k$.  This number of pairs is at most 3, exactly as in the proof of Lemma~\ref{Lem:VAAtMostThree}.

On the other hand, suppose that there is an unassigned permissible function $\psi = \varphi + \varphi' \in \cB$ such that $s_{k+1} (\varphi) < \xi'_k (\varphi)$.  By property $(\mathbf B)$, there is a function $\psi' \in \cB$ that agrees with $\psi$ on $\Gamma_{\leq k}$, with the property that
\[
s_{k+1} (\psi') > s_{k+1} (\psi) \geq \sigma_k .
\]
Since $\psi'$ is a departing function, it must have been assigned to a loop $\gamma_{\ell}$ with $\ell < k$ or a bridge $\beta_{\ell}$ with $\ell \leq k$.  But this is impossible, because $\psi$ is equivalent to $\psi'$ on all previous loops and bridges and $\psi$ is unassigned.
\end{proof}

\begin{proposition}
\label{Prop:NewThreeShape}
Consider a set of at most three equivalence classes of functions in $\cB$ on a non-skippable loop $\gamma_k$.  If all of the functions take the same value at $w_k$, then there is an open subset of $\gamma_k$ on which one of these equivalence classes is strictly less than the others.
\end{proposition}

\begin{proof}
The proof of Lemma~\ref{lem:threeshape} depends only on the restrictions of the functions to the loop $\gamma_k$, and the fact that, if $\psi$ is one of these functions, then $2D+\ddiv(\psi)$ does not contain $w_k$.  This latter fact is guaranteed by our assumption that $\gamma_k$ is not skippable.  The conclusion therefore continues to hold if we replace the functions with equivalence classes of functions.
\end{proof}

\subsubsection{Shiny functions and skippable loops}
The next two propositions are analogues of Lemma~\ref{lem:onenew} and Proposition~\ref{prop:nonew}, respectively, with shiny functions and skippable loops in place of new permissible functions and lingering loops.

\begin{proposition}
\label{Prop:Lingering}
If $\gamma_k$ is skippable, then no permissible function is shiny on $\gamma_k$.
\end{proposition}

\begin{proposition}
\label{Prop:GammaA}
The loops $\gamma_z, \gamma_b, \gamma_{b'},$ and $\gamma_{z'-2}$ are all non-skippable, and no permissible function is shiny on any of them.
\end{proposition}

The proofs of these propositions rely heavily on property $(\mathbf B)$, and use the following two technical lemmas about permissible functions on skippable loops.

\begin{lemma}
\label{Lem:Skippable1}
Let $\gamma_k$ be a skippable loop and let $\psi = \varphi + \varphi' \in \cB$ be an unassigned permissible function on $\gamma_k$.  Then $s_{k+1} (\psi) = \sigma_k$, $s_{k+1} (\varphi) = \xi'_k  (\varphi)$, and $s_{k+1} (\varphi') = \xi'_k  (\varphi')$.
\end{lemma}

\begin{proof}
By definition, no unassigned permissible function is departing on a skippable loop.  Therefore,
\[
s_{k+1} (\psi) = \sigma_k .
\]
It remains to show that $s_{k+1}(\varphi) = \xi'_k(\varphi)$.  Suppose not.  Then
\[
s_{k+1} (\varphi) < \xi'_k  (\varphi),
\]
and, by property $(\mathbf B)$, there is a function $\psi' \in \cB$ that agrees with $\psi$ on $\Gamma_{\leq k}$, with the property that
\[
s_{k+1} (\psi') > s_{k+1} (\psi) = \sigma_k .
\]
We claim that this is impossible.  Indeed, if $\psi'$ is unassigned on $\gamma_k$, then it would be a departing function, contradicting the hypothesis that $\gamma_k$ is skippable.  On the other hand, if $\psi'$ is assigned to a previous loop or bridge then $\psi$ would have been assigned to that loop or bridge as well.

We conclude that $s_{k+1} (\varphi) = \xi'_k  (\varphi)$, and, similarly,  $s_{k+1} (\varphi') = \xi'_k  (\varphi')$, as required.
\end{proof}

\begin{lemma}
\label{Lem:Skippable2}
Suppose $\gamma_k$ is a skippable loop, and there is a building block $\varphi_1$ such that $s_{k+1} (\varphi_1) > \xi_k (\varphi_1)$.  Then there is a permissible function $\psi = \varphi + \varphi' \in \cB$ such that
\begin{enumerate}
\item  $s_{k+1} (\varphi) = s_k (\varphi) + 1$,
\item  $s_{k+1} (\varphi') = s_k (\varphi') - 1$, and
\item $D + \ddiv(\varphi')$ contains either $w_k$, a point of $\gamma_k$ whose distance from $w_k$ is not an integer multiple of $m_k$, or two points of $\gamma_k \smallsetminus \{ v_k \}$.
\end{enumerate}
\end{lemma}

\begin{proof}
Since $\gamma_k$ is skippable, there is an unassigned permissible function $\psi = \varphi + \varphi' \in \cB$ such that $2D+\ddiv(\psi)$ contains either $w_k$, a point whose distance from $w_k$ is a non-integer multiple of $m_k$, or two points of $\gamma_k \smallsetminus \{ v_k \}$.
We first consider the case where one of the two functions $\varphi, \varphi'$ has smaller slope on $\beta_{k+1}$ than on $\beta_{k}$.  Suppose without loss of generality that
\[
s_{k+1} (\varphi') < s_k (\varphi').
\]
Since $\psi$ is permissible, we must have $s_{k+1} (\psi) \geq s_k (\psi)$.  It follows that
$
s_{k+1} (\varphi) > s_k (\varphi).
$
Since, by Lemma~\ref{Lem:Slope++}, the slope of a building block can increase by at most 1 from one bridge to the next, we see that
\[
s_{k+1} (\varphi) = s_k (\varphi) + 1 \mbox{ and } s_{k+1} (\varphi') = s_k (\varphi') -1.
\]
It follows that the restriction of $D + \ddiv(\varphi)$ to $\gamma_k$ is zero, and $D + \ddiv(\varphi')$ contains either $w_k$, a point of $\gamma_k$ whose distance from $w_k$ is not an integer multiple of $m_k$, or two points of $\gamma_k \smallsetminus \{ v_k \}$.

To complete the proof, we will rule out the possibility that neither function $\varphi, \varphi'$ has smaller slope on $\beta_{k+1}$ than on $\beta_{k}$.  Assume that
\[
s_{k+1} (\varphi) \geq s_k (\varphi) \mbox{ and } s_{k+1} (\varphi') \geq s_k (\varphi').
\]
By Lemma~\ref{Lem:DegreeOneLoop}, this immediately rules out the possibility that $D + \ddiv(\varphi')$ contains more than one point of $\gamma_k$.  We will reach a contradiction by showing that $\gamma_k$ is a switching loop and then applying property $(\mathbf B)$.  As a first step in this direction, we claim that the restriction of $D+\ddiv(\varphi_1)$ to $\gamma_k \smallsetminus \{ v_k \}$ has degree 0.  By Lemma~\ref{Lem:VDegree},
\[
d_{v_k} (\varphi_1) \geq s_k (\varphi_1) - \xi_k (\varphi_1).
\]
By Lemma~\ref{Lem:DegreeOneLoop},
\[
\deg\big(\ddiv(\varphi_1)_{|\gamma_k}\big) = s_k(\varphi_1) - s_{k+1}(\varphi_1) .
\]
It follows that
\[
\xi_k (\varphi_1) - s_{k+1} (\varphi_1) \geq \deg\big(\ddiv(\varphi_1)_{|\gamma_k}\big)  - d_{v_k} (\varphi_1) .
\]
The right hand side is the degree of the restriction of $\ddiv(\varphi_1)$ to $\gamma_k \smallsetminus \{ v_k \}$.  Since the left hand side is negative, the restriction of $D+\ddiv(\varphi_1)$ to $\gamma_k \smallsetminus \{ v_k \}$ has degree 0.

We now claim that $2D+ \ddiv \psi$ contains $w_k$.  Since there exists a building block $\varphi_1$ such that the restriction of $D+\ddiv(\varphi_1)$ to $\gamma_k \smallsetminus \{ v_k \}$ has degree 0, the shortest distance from the point of $D$ on $\gamma_k$ to $w_k$ is an integer multiple of $m_k$ by Lemma~\ref{Lem:EquivOneLoop}. Combined with our assumption that the slopes of $\varphi$ and $\varphi'$ do not decrease from $\beta_{k}$ to $\beta_{k+1}$, we see that the shortest distances from $w_k$ to the point of $D+\ddiv(\varphi)$ on $\gamma_k$ and the point of $D+\ddiv(\varphi')$ on $\gamma_k$ are integer multiples of $m_k$ as well.  Therefore, $2D+\ddiv(\psi)$ cannot contain a point whose shortest distance from $w_k$ is a non-integer multiple of $m_k$, and must therefore contain $w_k$, as claimed.

Without loss of generality, we assume that $D+\ddiv(\varphi)$ contains $w_k$.  Since the restriction of $D+\ddiv(\varphi)$ to $\gamma_k$ has degree at most 1, we see that this restriction is equal to $w_k$.  It follows that
\[
s_{k+1} (\varphi) = s_k (\varphi).
\]
Since the restrictions of $D+\ddiv(\varphi)$ and $D+\ddiv(\varphi_1)$ to $\gamma_k$ are supported at the vertices $v_k$ and $w_k$, we must have $\varphi \sim_{\gamma_k} \varphi_1$.  However, since $D+\ddiv(\varphi)$ contains $w_k$ and $D+\ddiv(\varphi_1)$ does not, we must have $s_{k+1} (\varphi) = s_{k+1}(\varphi_1) - 1$.

We now show that $\gamma_k$ switches slope $\tau_k (\varphi_1)$.  By assumption, $s_{k+1} (\varphi_1) > \xi_k (\varphi_1)$.  Combining this with the two equations above, we see that
\[
s_k (\varphi) \geq \xi_k (\varphi_1).
\]
This implies $\tau_k (\varphi) \geq \tau_k (\varphi_1)$.  By Lemma~\ref{Lem:Skippable1}, however, we have
\[
s_{k+1} (\varphi) = \xi'_k  (\varphi),
\]
so $\tau'_k (\varphi) < \tau'_k (\varphi_1)$.  Combining these inequalities with the fact that slope index sequences are nondecreasing, we see that
\[
\tau_k (\varphi_1) \leq \tau_k (\varphi) \leq \tau'_k (\varphi) < \tau'_k (\varphi_1).
\]
Since $\varphi_1$ is a building block, by Definition~\ref{def:buildingblock}, it follows that $\gamma_k$ switches slope $\tau_k (\varphi_1)$.

We now apply property $(\mathbf B)$ again, in a similar way to the proof of Lemma~\ref{Lem:Skippable1}.  Specifically, there is a function $\psi' \in \cB$ that agrees with $\psi$ on $\Gamma_{\leq k}$, with the property that
\[
s_{k+1} (\psi') > s_{k+1} (\psi) = \sigma_k.
\]
If $\psi'$ has not been assigned to a previous loop or bridge, then it is an unassigned departing function on $\gamma_k$, and for this reason $\gamma_k$ is not skippable.  If $\psi'$ has been assigned to a previous loop or bridge, then since $\psi$ is equivalent to $\psi'$ on this previous loop or bridge, we see that $\psi$ must have been assigned to this previous loop or bridge as well.  We therefore arrive at a contradiction, which rules out the possibility that neither $\varphi$ nor $\varphi'$ has smaller slope on $\beta_{k+1}$ than on $\beta_{k}$ and completes the proof of the lemma.
\end{proof}

\begin{proof}[Proof of Proposition~\ref{Prop:Lingering}]
By Proposition~\ref{Prop:Shiny}, any shiny permissible function on $\gamma_k$ has a summand $\varphi_1 \in \cA$ satsifying $s_{k+1} (\varphi_1) > \xi_k (\varphi_1)$.  We will assume that such a function $\varphi_1$ exists, and consider unassigned permissible functions of the form $\varphi_1 + \varphi'_1 \in \cB$.  Since $\varphi_1$ exists, by Lemma~\ref{Lem:Skippable2}, there is an unassigned permissible function $\varphi + \varphi' \in \cB$ such that
\[
s_{k+1} (\varphi) = s_k (\varphi) + 1, \ \ \ \   s_{k+1} (\varphi') = s_k (\varphi') -1,
\]
and $D + \ddiv(\varphi')$ contains either $w_k$, or a point of $\gamma_k$ whose distance from $w_k$ is not an integer multiple of $m_k$, or two points of $\gamma_k \smallsetminus \{ v_k \}$.

Both $D+\ddiv(\varphi_1)$ and $D+\ddiv(\varphi)$ contain no points of $\gamma_k \smallsetminus \{ v_k \}$.  Any two functions with this property are equivalent on $\gamma_k$, and have the same slope along the bridge $\beta_{k+1}$.  It follows that $s_{k+1} (\varphi) = s_{k+1} (\varphi_1)$.  By Lemma~\ref{Lem:Skippable1}, we also have
\[
s_{k+1} (\varphi) + s_{k+1} (\varphi') = s_{k+1} (\varphi_1) + s_{k+1} (\varphi'_1) = \sigma_k .
\]
Subtracting, we see that
$
s_{k+1} (\varphi') = s_{k+1} (\varphi'_1).
$
Moreover, by Lemma~\ref{Lem:Skippable1}, we have
\[
s_{k+1} (\varphi') = \xi'_k  (\varphi') \mbox{ and } s_{k+1} (\varphi'_1) = \xi'_k  (\varphi'_1),
\]
so $\tau'_k (\varphi') = \tau'_k (\varphi'_1)$, which implies that $\varphi'$ and $\varphi'_1$ are equivalent on $\gamma_k$.

Since $\varphi'$ and $\varphi'_1$ are equivalent on $\gamma_k$ and their slopes on $\beta_{k+1}$ are equal, the difference between $\ddiv (\varphi')$ and $\ddiv (\varphi'_1)$ on $\gamma_k$ must be supported at $v_k$.  Since $s_{k+1} (\varphi') = s_k (\varphi') -1$, the restriction of $D+\ddiv(\varphi')$ to $\gamma_k$ has degree 2. Since $\varphi_1 + \varphi'_1$ is shiny, by Proposition~\ref{prop:equivtonew} the restriction of $2D+\ddiv(\varphi_1+\varphi'_1)$ to $\gamma_k \smallsetminus \{ v_k \}$ has degree at most 1.  It follows that $D+\ddiv(\varphi')$ contains $v_k$.  By Lemma~\ref{Lem:EquivOneLoop}, this forces $D+\ddiv(\varphi')$ to be supported at points whose shortest distance to $w_k$ is an integer multiple of $m_k$.  Recall, however, that $\varphi'$ was chosen so that $D + \ddiv(\varphi')$ contains either $w_k$ or a point of $\gamma_k$ whose distance from $w_k$ is not an integer multiple of $m_k$.  We therefore see that $D+\ddiv(\varphi')$ contains $w_k$, and hence
\[
\left[ D+\ddiv(\varphi') \right]_{\vert{\gamma_k}} = v_k + w_k .
\]
Thus, $\varphi'$ is equivalent to $\varphi$ and $\varphi_1$ on $\gamma_k$, but since $D+\ddiv(\varphi')$ contains $v_k$ and $D+\ddiv(\varphi)$ does not, we have
$
s_k (\varphi) < s_k (\varphi').
$
It follows that $\tau_k (\varphi) \leq \tau_k (\varphi')$.  Similarly, since $D+\ddiv(\varphi')$ contains $w_k$ and $D+\ddiv(\varphi)$ does not, we have
\[
s_{k+1} (\varphi) > s_{k+1} (\varphi').
\]
By Lemma~\ref{Lem:Skippable1}, however, we have
\[
s_{k+1} (\varphi) = \xi'_k  (\varphi) \mbox{ and } s_{k+1} (\varphi') = \xi'_k  (\varphi') .
\]
Thus, $\tau'_k (\varphi) > \tau'_k (\varphi')$.  Combining these inequalities with the fact that slope index sequences are nondecreasing, we see that
\[
\tau_k (\varphi) \leq \tau_k (\varphi') \leq \tau'_k (\varphi') < \tau'_k (\varphi).
\]
Since $\varphi$ is a building block, by Definition~\ref{def:buildingblock}, it follows that $\gamma_k$ switches slope $\tau_k (\varphi)$.

By property $(\mathbf B)$, there is a function $\psi' \in \cB$ that is equivalent to $\varphi_1+\varphi'_1$ on $\Gamma_{\leq k}$, with the property that
\[
s_{k+1} (\psi') > s_{k+1} (\varphi_1 + \varphi'_1) = \sigma_k .
\]
If $\psi'$ has not been assigned to a previous loop or bridge, then it is an unassigned departing function on $\gamma_k$, and for this reason $\gamma_k$ is not skippable.  If $\psi'$ has been assigned to a previous loop or bridge, then since $\varphi_1 + \varphi'_1$ is equivalent to $\psi'$ on this previous loop or bridge, we see that $\varphi_1+\varphi'_1$ must have been assigned to this previous loop as well.  This contradicts our assumption that $\varphi_1+\varphi'_1$ was unassigned.  We conclude that there are no shiny functions on $\gamma_k$, as required.
\end{proof}

\begin{proof}[Proof of Proposition~\ref{Prop:GammaA}]
Let $k \in \{ z,b,b',z'-2 \}$.  As in Proposition~\ref{prop:nonew}, these four choices for $k$ guarantee that there is an index $i$ such that $s_k [i] < s'_k [i]$, and there does not exist a value of $j$ such that $s'_k [i] + s'_k [j] = \sigma_k$.  We must show that $\gamma_k$ is not skippable, and that no permissible function is shiny on $\gamma_k$.  We begin by showing that $\gamma_k$ is not skippable.

Suppose $\gamma_k$ is skippable.  Then there is an unassigned permissible function $\psi = \varphi + \varphi'$ on $\gamma_k$ such that $2D + \ddiv(\psi)$ contains either $w_k$ or a point whose distance from $w_k$ is not an integer multiple of $m_k$, or $D + \ddiv(\varphi')$ contains two points of $\gamma_k \smallsetminus \{ v_k \}$.  By Lemma~\ref{Lem:Skippable1}, we have
\begin{equation} \label{eq:skippable}
s_{k+1} (\varphi) + s_{k+1}(\varphi') = \xi'_k  (\varphi) + \xi'_k  (\varphi') = \sigma_k .
\end{equation}
Moreover, by Lemma~\ref{Lem:Skippable2}, we have
$
s_{k+1} (\varphi) = s_k (\varphi) + 1,
$
hence the restriction of $D+\ddiv(\varphi)$ to $\gamma_k$ must have degree 0, which forces $\tau'_k (\varphi) = i$.  As in Proposition~\ref{prop:nonew}, the choice of $z$, $b$, $b'$, and $z'$ ensures that there does not exist a value of $j$ such that $s'_k [i] + s'_k [j] = \sigma_k$, contradicting \eqref{eq:skippable}.  Thus, $\gamma_k$ cannot be skippable.

It remains to show that there are no shiny permissible functions on $\gamma_k$.  Let $\varphi$ be a function satisfying $s_k (\varphi) = s_k [i]$ and $s_{k+1} (\varphi) = s'_k [i]$.  Any function that is shiny on $\gamma_k$ must be equivalent on $\gamma_k$ to a function of the form $\psi = \varphi + \varphi'$ with $s_{k+1} (\varphi + \varphi') = \sigma_k$.  Again, because there is no $j$ such that $s'_k [i] + s'_k [j] = \sigma_k$, we see that $s_{k+1}(\varphi')$ cannot equal $s'_k[j]$ for any $j$.  This means that
\[
s_{k+1} (\varphi') < \xi'_k (\varphi').
\]
Hence by Lemma~\ref{Lem:VDegree}, $D+\ddiv(\varphi')$ contains $w_k$.  Since $\psi$ is shiny, by Proposition~\ref{prop:equivtonew}, the restriction of $2D+\ddiv(\psi)$ to $\gamma_k \smallsetminus \{ v_k \}$ has degree at most 1, so this restriction is exactly $w_k$.  It follows that $\psi$ is equivalent to $2\varphi$ on $\gamma_k$, and
\[
s_{k+1} (\psi) = \sigma_k = s_{k+1} (2\varphi) -1.
\]
This implies that $\sigma_k$ is odd, so $k$ is either $b$ or $b'$ and $\sigma_k = 3$.  However, $b$ and $b'$ were chosen so that $s_{k+1}(\varphi)$ is at most 1 if the box contained in $\lambda'_k$ but not $\lambda_k$ is in the first row, and $s_{k+1}(\varphi)$ is at least 3 if this box is contained in the second or third row.  In the first case, we have $s_{k+1} (2\varphi) \leq 2$, and in the second case, we have $s_{k+1} (2\varphi) \geq 6$.  In either case, we obtain a contradiction to the displayed equation above, so there cannot be a shiny function on $\gamma_k$.
\end{proof}

\subsubsection{The counting argument}

In the vertex avoiding case, Proposition~\ref{Prop:VAVerification} follows from a counting argument.  Specifically, we count the number of permissible functions on $\gamma_1$, $\gamma_{z+1}$, and $\gamma_{z'-1}$.   Together with the fact that there is at most one new permissible function on every non-lingering loop, no new permissible functions on lingering loops, and no new permissible functions on $\gamma_k$ for $k \in \{z,b,b',z'-2 \}$, we derive the number of permissible functions on each loop.

In the general case, we may assign several functions to the same loop, so instead of counting individual unassigned permissible functions, we count certain equivalence classes of such functions, that we call \emph{cohorts}.  These cohorts are defined using the following auxiliary function $\Sha$.

\begin{definition}
Let $\Sha \colon \cB \to \{1, \ldots, g' \} \cup \{ \infty \}$ be the function taking $\psi$ to the maximal $k$ such that $\psi$ is a shiny or new permissible function on $\gamma_k$ and $\psi$ is not assigned to any $\gamma_j$ for $j < k$, or to $\infty$ if there is no such $k$.
\end{definition}

Note that $\Sha$ depends on which functions are assigned to which loops, so it is something that we compute after the algorithm has run.

In the vertex avoiding case, $\Sha (\psi) = k \neq \infty$ if and only if $\psi$ is a new permissible function on $\gamma_k$, and $\Sha (\psi) = \infty$ if and only if $\psi$ is not permissible on any loop.

We define an equivalence relation $\equiv_{\Sha}$ on $\cB$, as follows.

\begin{definition}
We define $\psi \equiv_\Sha \psi'$ if $\Sha(\psi) = \Sha(\psi') = k$ for some $k \in \{1, \ldots, g'\}$ and $\psi \sim_{\gamma_k} \psi'$.
\end{definition}

\begin{definition}
A \emph{cohort} on $\gamma_\ell$ is a $\equiv_{\Sha}$-equivalence class of unassigned permissible functions $\psi$ with $\Sha(\psi) \leq \ell$.
\end{definition}

A cohort is \emph{new} on $\gamma_k$ if it consists of functions $\psi$ with $\Sha(\psi) = k$.

\begin{example}
\label{Ex:Genus5Cohorts}
In Example~\ref{Ex:Genus5Template}, we list the distinct cohorts on each loop $\gamma_{\ell}$:
\begin{figure}[h!]
\scalebox{.9}{
\begin{tabular}{|c|c|c|c|c|}
\hline
$\gamma_1$ & $\gamma_2$ & $\gamma_3$ & $\gamma_4$ & $\gamma_5$ \\
\hline
$\{ \varphi_1 + \varphi_3, \varphi_{\tau} + \varphi_3 \}$ & $2\varphi_2$ &
$\varphi_0 + \varphi_3$  & $\varphi_0 + \varphi_3$ & $\varphi_0 + \varphi_3$ \\
$2\varphi_2$ & $\varphi_0 + \varphi_3$ &  $\{ \varphi_1 + \varphi_2 , \varphi_{\tau} + \varphi_2 \}$  & $\{ 2\varphi_1 , \varphi_1 + \varphi_{\tau} , 2\varphi_{\tau} \}$ & $2\varphi_1$ \\
$\varphi_0 + \varphi_3$ & $\{ \varphi_1 + \varphi_2 , \varphi_{\tau} + \varphi_2 \}$ & $\{ 2\varphi_1 , \varphi_1 + \varphi_{\tau} , 2\varphi_{\tau} \}$ &  & $\{ \varphi_0 + \varphi_2 , \varphi_0 + \varphi_{\tau} \}$ \\
\hline
\end{tabular}
}
\caption{Cohorts on each loop $\gamma_{\ell}$ in Example~\ref{Ex:Genus5Template}.}
\end{figure}

In particular, if $\varphi_1 + \varphi_i$ is permissible on $\gamma_{\ell}$ for some $\ell \leq 4$, then $\varphi_{\tau} + \varphi_i$ is in the same cohort on $\gamma_{\ell}$.  Similarly, if $\varphi_2 + \varphi_i$ is permissible on $\gamma_{\ell}$ for $\ell \geq 4$, then $\varphi_{\tau} + \varphi_i$ is in the same cohort on $\gamma_{\ell}$.

On loops $\gamma_2$, $\gamma_3$, and $\gamma_5$, there is one new cohort, and there is a cohort such that every function in the cohort is assigned to the loop.  On $\gamma_4$, there is no new cohort, and there exists a cohort such that some elements of the cohort are assigned to $\gamma_4$ and others are not.  Because of this, the number of cohorts on $\gamma_{\ell}$ such that some element of the cohort is not assigned to $\gamma_{\ell}$ remains constant.  In particular, this number is 2 for all $\ell$.
\end{example}

\begin{example}
\label{Ex:Genus22Cohorts}
Similarly, in Example~\ref{Ex:Genus22Template}, if $\varphi_3 + \varphi_i$ is permissible on $\gamma_{\ell}$ for $\ell \leq 18$, then $\varphi_{\tau} + \varphi_i$ is in the same cohort on $\gamma_{\ell}$.  If $\varphi_4 + \varphi_i$ is permissible on $\gamma_{\ell}$ for $\ell \geq 18$, then $\varphi_{\tau} + \varphi_i$ is in the same cohort on $\gamma_{\ell}$.  When we run the alogrithm on the set $\cB'$, there is no new cohort on $\gamma_{18}$.  The functions $\varphi_1 + \varphi_3$ and $\varphi_1 + \varphi_{\tau}$ are in the same cohort on $\gamma_{18}$, and only one of them assigned to $\gamma_{18}$.  As in Example~\ref{Ex:Genus5Cohorts}, the number of cohorts on $\gamma_{18}$ such that some element of the cohort is not assigned to $\gamma_{18}$ is equal to the number of cohorts on $\gamma_{17}$ such that some element of the cohort is not assigned to $\gamma_{17}$.

If we run the algorithm on the set $\cB''$, then the cohort consisting of $\varphi_0 + \varphi_{\tau}$ is new on $\gamma_{18}$.  In this example, every time a function is assigned to a loop, every function in the same cohort is also assigned to the loop.

On the other hand, recall that the set $\cB$ does not satisfy property $(\mathbf B')$, and the corresponding assignment function $\alpha$ does not satisfy part (\ref{it:samebridge}) of Theorem~\ref{Thm:ExtremalConfig}.  This can be seen by counting the cohorts.  Note that $\varphi_1 + \varphi_3$ and $\varphi_1 + \varphi_{\tau}$ are in the same cohort on $\gamma_{18}$, and exactly one of these two functions is assigned to $\gamma_{18}$.  In addition, the cohort consisting of $\varphi_0 + \varphi_{\tau}$ is new on $\gamma_{18}$.  In this case, the number of cohorts on $\gamma_{18}$ such that some element of the cohort is not assigned to $\gamma_{18}$ is greater than the number of cohorts on $\gamma_{17}$ such that some element of the cohort is not assigned to $\gamma_{17}$.  For this reason, property $(\mathbf B')$ is essential for the proof of Proposition~\ref{Prop:Counting} below.
\end{example}

By Proposition~\ref{prop:equivtonew}, all shiny functions on $\gamma_k$ are equivalent, and by Lemma~\ref{Lem:NewIsShiny}, on any loop other than $\gamma_1$, $\gamma_{z+1}$, or $\gamma_{z'-1}$, all new functions are shiny.  Thus, on any loop other than $\gamma_1$, $\gamma_{z+1}$, or $\gamma_{z'-1}$, there is at most one new cohort.  This is one essential way in which counting cohorts is like counting permissible functions in \S\ref{Sec:VertexAvoiding}; there may be several new cohorts on $\gamma_k$ for $k \in \{ 1, z+1, z'-1 \}$, and then at most one new cohort on each subsequent loop.

Furthermore, there are no shiny functions on $\gamma_k$, for $k \in \{z,b,b',z'-2 \}$, so there are no new cohorts on these loops.  This is analogous to Proposition~\ref{prop:nonew} in the vertex avoiding case.

The next proposition says that, on each loop where a new cohort is created, and also on the special loops $\gamma_k$, for $k \in \{z,b,b',z'-2 \}$, an entire cohort is assigned.  This is essential for the proof of Theorem~\ref{Thm:ExtremalConfig}, where we bound the number of cohorts on each loop while moving from left to right to show that every function in $\cB$ is assigned to some loop or bridge.

\begin{remark}
On a non-skippable loop where no new cohort is created, the functions that are assigned may form a proper subset of a cohort.  This is visible on the loop $\gamma_4$ of Example~\ref{Ex:Genus5Cohorts}, or on the loop $\gamma_{18}$ in Example~\ref{Ex:Genus22Cohorts}, when we run the algorithm on the set $\cB'$.  In this way, cohorts may lose members as we move from left to right across a block, but the number of cohorts on each loop behaves just as predictably as the number of unassigned permissible functions in the vertex avoiding case.

In the vertex avoiding case, each cohort consists of a single unassigned permissible function, and the argument for counting cohorts in the proof of Theorem~\ref{Thm:ExtremalConfig} specializes to the argument for counting permissible functions in \S\ref{Sec:VertexAvoiding}
\end{remark}

\begin{proposition}
\label{Prop:Counting}
Suppose that $\Sha^{-1}(\ell) \neq \emptyset$ or $\ell \in \{ z,b,b',z'-2 \}$.  If $\psi \in \cB$ is assigned to $\gamma_{\ell}$, then any other function in the same cohort with $\psi$ is also assigned to $\gamma_{\ell}$.
\end{proposition}

\noindent To prove Proposition~\ref{Prop:Counting}, we will use property $(\mathbf B')$ along with some preliminary lemmas.  We start with a relatively simple one.

\begin{lemma}
\label{Lem:SameSlope}
Let $\varphi, \varphi'$ be building blocks such that
\[
\tau_k(\varphi) = \tau_k(\varphi') \mbox{ and } \tau_{k+1}(\varphi) = \tau_{k+1}(\varphi').
\]
Then
$
\tau'_k(\varphi) = \tau'_k(\varphi')
$
and $\varphi \sim_{\gamma_k} \varphi'$.
\end{lemma}

\begin{proof}
Since $\rho'(\Sigma) \leq 2$, a switching loop has multiplicity at least 1, and a switching bridge has multiplicity at least 2, we cannot have both that $\gamma_k$ is a switching loop and $\beta_{k+1}$ is a switching bridge.  It follows that either
\[
\tau'_k(\varphi) = \tau_k(\varphi) \mbox{ and } \tau'_k(\varphi') = \tau_k(\varphi')
\]
or
\[
\tau'_k(\varphi) = \tau_{k+1}(\varphi) \mbox{ and } \tau'_k(\varphi') = \tau_{k+1}(\varphi').
\]
Thus,
$
\tau'_k(\varphi) = \tau'_k(\varphi'),
$
hence $\varphi \sim_{\gamma_k} \varphi'$.
\end{proof}

\begin{lemma}
\label{Lem:ShinyEq}
Let $\psi_1 = \varphi_1 + \varphi'_1$ and $\psi_2 = \varphi_2 + \varphi'_2$.  Suppose $\Sha(\psi_1) = \Sha(\psi_2) = k$  and $\psi_1 \sim_{\gamma_k} \psi_2$.  Then $\{ \tau'_k (\varphi_1) , \tau'_k (\varphi'_1) \} = \{ \tau'_k (\varphi_2) , \tau'_k (\varphi'_2) \}$.
\end{lemma}

\begin{proof}
By Proposition~\ref{Prop:Shiny}, we may assume without loss of generality that the restrictions of both $D+\ddiv(\varphi_1)$ and $D+\ddiv(\varphi_2)$ to $\gamma_k \smallsetminus \{ v_k \}$ have degree zero.  By Lemma~\ref{Lem:PointsOfSlopeIncrease}, it follows that $\varphi_1$ and $\varphi_2$ are equivalent on $\gamma_k$.  Since $\psi_1 \sim_{\gamma_k} \psi_2$, it follows that $\varphi'_1 \sim_{\gamma_k} \varphi'_2$ as well.

Since $d_{w_k} (\varphi_1) = 0$, $s_{k+1} (\varphi_1)$ is equal to the sum of the slopes of $\varphi_1$ along the two tangent vectors coming into $w_k$ on $\gamma_k$.  Similarly, since $d_{w_k} (\varphi_2) = 0$, $s_{k+1} (\varphi_2)$ is equal to the sum of the slopes of $\varphi_2$ along the two tangent vectors coming into $w_k$ on $\gamma_k$.  Since $\varphi_1 \sim_{\gamma_k} \varphi_2$, it follows that $s_{k+1} (\varphi_1) = s_{k+1} (\varphi_2)$.  By Lemma~\ref{Lem:VDegree}, since $d_{w_k} (\varphi_1) = d_{w_k} (\varphi_2) = 0$, we have $s_{k+1} (\varphi_1) = \xi'_k  (\varphi_1)$ and $s_{k+1} (\varphi_2) = \xi'_k  (\varphi_2)$.  Thus, $\tau'_k (\varphi_1) = \tau'_k (\varphi_2)$.

If $\tau'_k (\varphi'_1) \neq \tau'_k (\varphi'_2)$, assume without loss of generality that $\tau'_k (\varphi'_1) < \tau'_k (\varphi'_2)$.  Since $\psi_1$ and $\psi_2$ are permissible on $\gamma_k$, we have
\[
\sigma_k \leq \xi'_k (\varphi_1) + \xi'_k (\varphi'_1) < \xi'_k (\varphi_2) + \xi'_k (\varphi'_2) .
\]
But since $\psi_1$ and $\psi_2$ are shiny, we have
\[
\xi_k (\varphi_1) + \xi_k (\varphi'_1) < \xi_k (\varphi_2) + \xi_k (\varphi'_2) < \sigma_k .
\]
Hence, both
\begin{align*}
\xi'_k (\varphi_1) + \xi'_k (\varphi'_1) & \geq \xi_k (\varphi_1) + \xi_k (\varphi'_1) + 2 \mbox{ and } \\
\xi'_k (\varphi_2) + \xi'_k (\varphi'_2) & \geq \xi_k (\varphi_2) + \xi_k (\varphi'_2) + 2.
\end{align*}
Recall that there is at most one value $j$ such that $s'_k[j] \geq s_k[j] +1$, and when this $j$ exists, we have $s'_k[j] = s_k[j] + 1$.  Thus, if $\gamma_k$ is not a switching loop, the inequalities above imply that $\tau_k (\varphi_i) = \tau'_k (\varphi_i) = \tau_k (\varphi'_i) = \tau'_k (\varphi'_i) = j$ for $i=1,2$.

Finally, suppose that $\gamma_k$ is a switching loop.  If $s'_k (\varphi'_1) = s'_k (\varphi_2)$, then
\[
\xi'_k  (\varphi'_2) > \xi'_k  (\varphi'_1) \geq s'_k (\varphi'_1) = s'_k (\varphi'_2) \geq \xi_{k+1} (\varphi'_2) \geq s_{k+1} [\tau'_k (\varphi'_2)] > s_{k+1} [\tau'_k (\varphi'_1)] .
\]
It follows that $\mu (\beta_{k+1}) \geq 2$.  Since $\gamma_k$ is a switching loop, $\mu (\gamma_k) \geq 1$.  Since $\rho \leq 2$, this is impossible, hence $s'_k (\varphi'_1) \neq s'_k (\varphi_2)$.  It follows that $\psi_1$ and $\psi_2$ are permissible functions that are equivalent on $\gamma_k$ with different slopes on $\beta_{k+1}$.  By property $(\mathbf B')$, no function is shiny on $\gamma_k$, contradicting our assumption that $\Sha(\psi_1) = \Sha(\psi_2) = k$.  Thus, $\{ \tau'_k(\varphi_1), \tau'_k(\varphi'_1) \} = \{ \tau'_k(\varphi_2), \tau'_k(\varphi'_2) \}$.
\end{proof}

If two functions $\psi , \psi' \in \cB$ are in the same cohort, and $\Sha (\psi) = \Sha (\psi') = k_0$, then by definition, $\psi \sim_{\gamma_{k_0}} \psi'$.  There may be a loop $\gamma_k$, with $k > k_0$, such that $\psi \not\sim_{\gamma_k} \psi'$.  The next lemma considers the smallest such $k$, and examines which functions are assigned to $\gamma_k$.

\begin{lemma}  \label{lem:slopeindexchange}
Let $\psi = \varphi_1 + \varphi_2$ and $\psi' = \varphi'_1 + \varphi'_2$ be in $\cB$.  Suppose $\Sha(\psi) = \Sha(\psi') = k_0$ and  $\psi \sim_{\gamma_{k_0}} \psi'$.  Let $k$ be the smallest integer such that  $k \geq k_0$ and the sets of slope indices $\{ \tau_{k+1}(\varphi_1), \tau_{k+1}(\varphi_2) \}$ and $\{ \tau_{k+1}(\varphi'_1), \tau_{k+1}(\varphi'_2) \}$ are different.  Suppose, furthermore, that $\psi$ and $\psi'$ are unassigned and permissible on $\gamma_k$.  Then
\begin{enumerate}
\item either $\gamma_k$ is a switching loop or $\beta_{k+1}$ is a switching bridge,
\item $k \notin \{ z,b,b',z'-2 \}$, and
\item either $\psi$ or $\psi'$ is assigned to $\gamma_k$.
\end{enumerate}
Furthermore, if one of $\psi$, $\psi'$ is assigned to $\gamma_k$ and the other is not, then no function is shiny on $\gamma_k$.
\end{lemma}

\begin{proof}
By assumption, $\psi \sim_{\gamma_{k_0}} \psi'$.  By Lemma~\ref{Lem:ShinyEq}, $\{ \tau'_{k_0}(\varphi_1), \tau'_{k_0}(\varphi_2) \} = \{ \tau'_{k_0}(\varphi'_1), \tau'_{k_0}(\varphi'_2) \}$.  By Lemma~\ref{Lem:SameSlope}, $\psi \sim_{\gamma_t} \psi'$ for all $t$ in the range $k_0 \leq t < k$.

Without loss of generality, we may assume that $\tau_k(\varphi_1) = \tau_k(\varphi'_1)$ and $\tau_{k}(\varphi_2) = \tau_{k}(\varphi'_2)$, and that
$
\tau_{k+1}(\varphi_1) > \tau_{k+1}(\varphi'_1).
$
Since slope indices of building blocks only change due to switching (Definition~\ref{def:buildingblock}), it follows that either $\gamma_k$ switches slope $\tau_k(\varphi_1)$ or $\beta_{k+1}$ switches slope $\tau'_k(\varphi_1)$.  It follows that $k \notin \{ z,b,b',z'-2 \}$.
Since a switching loop or bridge can switch at most one slope, we also have
$
\tau_{k+1}(\varphi_2) \geq \tau_{k+1}(\varphi'_2),
$
with equality if $\tau_k (\varphi_1) \neq \tau_k (\varphi_2)$.

\medskip

 \emph{Claim 1:  Either $\psi$ or $\psi'$ is assigned to $\gamma_k$.}  Suppose that neither $\psi$ nor $\psi'$ is assigned to $\gamma_k$.  Then $\psi$ and $\psi'$ are both permissible and not shiny on $\gamma_{k+1}$. Therefore, by Lemma~\ref{lem:smallslopeshiny},
\[
s_{k+1} (\varphi_1) = \xi_{k+1}(\varphi_1), \ \ s_{k+1} (\varphi_2) = \xi_{k+1}(\varphi_2),
\]
\[
s_{k+1} (\varphi'_1) = \xi_{k+1}(\varphi'_1), \ \ s_{k+1} (\varphi'_2) = \xi_{k+1}(\varphi'_2).
\]
Summing these, we obtain
\[
s_{k+1} (\psi) = \xi_{k+1}(\varphi_1) + \xi_{k+1}(\varphi_2) > \xi_{k+1} (\varphi'_1) + \xi_{k+1}(\varphi'_2) = s_{k+1} (\psi').
\]
It follows that $\psi$ is departing on $\gamma_k$.  This contradicts the supposition that neither $\psi$ nor $\psi'$ is assigned to $\gamma_k$ and proves the claim.

\medskip

It remains to show that if one of $\psi$, $\psi'$ is not assigned to $\gamma_k$, then no function is shiny on $\gamma_k$.

\medskip

\emph{Claim 2: If $\psi \sim_{\gamma_k} \psi'$, then no function is shiny on $\gamma_k$.}  This is straightforward.  Indeed, if $\psi$ and $\psi'$ are equivalent and one is assigned while the other is not, then the one that is assigned is departing and the other is not.  In this case, no function is shiny on $\gamma_k$ by property $(\mathbf B')$.

\medskip

For the remainder of the proof of the lemma, we assume $\psi$ and $\psi'$ are not equivalent on $\gamma_k$, and show that no function is shiny.  Note that this assumption implies that $k$ is strictly greater than $k_0$, because $\psi \sim_{k_0} \psi'$.

\medskip

\emph{Claim 3: The functions $\varphi_1$ and $\varphi'_1$ are not equivalent on $\gamma_k$.}  Since $\psi$ and $\psi'$ are not equivalent on $\gamma_k$, either $\varphi_1$ and $\varphi'_1$ are not equivalent on $\gamma_k$, or $\varphi_2$ and $\varphi'_2$ are not equivalent on $\gamma_k$.  If $\varphi_2$ and $\varphi'_2$ are not equivalent, then $\tau'_k (\varphi_2) \neq \tau'_k (\varphi'_2)$.  This implies that $\gamma_k$ switches slope $\tau_k (\varphi_2)$.  Since a switching loop can switch at most one slope, it follows that $\varphi_2 \sim_{\gamma_k} \varphi_1$, $\tau_k (\varphi_2) = \tau_k (\varphi_1)$, and $\tau_{k+1} (\varphi_2) = \tau_{k+1} (\varphi_1)$.  In this case, we may relabel $\varphi_1$ and $\varphi_2$ without loss of generality, and the result follows.

\medskip

\emph{Claim 4: The loop $\gamma_k$ switches slope $\tau_k(\varphi_1)$, and $s_{k+1} (\varphi'_1) = s_k (\varphi'_1) - 1$.}  To see this, note that neither $\psi$ nor $\psi'$ is shiny on $\gamma_k$, hence by Lemma~\ref{lem:smallslopeshiny},
\[
s_k (\varphi_1) = \xi_k(\varphi_1), \ \ s_k (\varphi_2) = \xi_k(\varphi_2), \  \  s_k (\varphi'_1) = \xi_k(\varphi'_1), \mbox{ \ and \ } s_k (\varphi'_2) = \xi_k(\varphi'_2).
\]
It follows that
$
s_k (\varphi_1) = s_k (\varphi'_1) \mbox{ and } s_k (\varphi_2) = s_k (\varphi'_2).
$

Recall that, on a loop, every divisor of degree 1 is equivalent to a unique effective divisor.  Thus, since $\varphi_1$ and $\varphi'_1$ have the same incoming slope, if the restrictions of $D+\ddiv(\varphi_1)$ and $D+\ddiv(\varphi'_1)$ to $\gamma_k \smallsetminus \{ w_k \}$ each have degree at most 1, then $\varphi_1\sim_{\gamma_k} \varphi'_1$.  Because we showed, in the previous claim, that $\varphi_1$ and $\varphi'_1$ are not equivalent on $\gamma_k$, the restriction of $D+\ddiv(\varphi'_1)$ to $\gamma_k \smallsetminus \{ w_k \}$ has degree 2.  Equivalently,
\[
s_{k+1} (\varphi'_1) = s_k (\varphi'_1) - 1,
\]
so $\gamma_k$ is a decreasing loop and hence has positive multiplicity.  We already showed that either $\gamma_k$ switches slope $\tau_k(\varphi_1)$ or $\beta_k$ is a switching bridge.  Since switching bridges have multiplicity 2 and the sum of all multiplicities is at most 2, it follows that $\gamma_k$ switches slope $\tau_k (\varphi_1)$, as claimed.

\medskip

We now complete the proof that no function is shiny on $\gamma_k$.  Suppose $\varphi''_1 + \varphi''_2$ is shiny on $\gamma_k$.  By Proposition~\ref{Prop:Shiny}, without loss of generality, the restriction of $D+\ddiv(\varphi''_2)$ to $\gamma_k \smallsetminus \{ v_k \}$ has degree 0.  Because $\psi'$ is permissible on $\gamma_k$, and $s_{k+1} (\varphi'_1) < s_k (\varphi'_1)$, as shown above, we must have
\[
s_{k+1} (\varphi'_2) = s_k (\varphi'_2) + 1.
\]
Thus the restriction of $D + \ddiv(\varphi'_2)$ to $\gamma_k \smallsetminus \{ v_k \}$ also has degree 0, hence $\varphi''_2 \sim_{\gamma_k} \varphi'_2$, and
\[
s_{k+1} (\varphi''_2) = s_{k+1} (\varphi'_2).
\]
Since $\varphi''_1 + \varphi''_2$ is permissible on $\gamma_k$, we have
$
s_{k+1} (\varphi''_1) \geq s_{k+1} (\varphi'_1).
$
Also, since $\varphi''_1 + \varphi''_2$ is shiny,
\[
\xi_k (\varphi''_1) + \xi_k (\varphi''_2) < \sigma_k = \xi_k (\varphi'_1) + \xi_k (\varphi'_2).
\]
It follows that either $\tau_k (\varphi''_1) < \tau_k (\varphi'_1)$, and $\gamma_k$ switches slope $\tau_k (\varphi''_1)$, or $\tau_k (\varphi''_2) < \tau_k (\varphi'_2)$, and $\gamma_k$ switches slope $\tau_k (\varphi''_2)$.  We will show that neither of these is possible.  Indeed, the first is impossible because $\gamma_k$ switches slope $\tau_k (\varphi'_1)$, and a loop can switch at most 1 slope.  The second requires
\[
\tau_k (\varphi''_2) = \tau_k (\varphi'_2)-1 = \tau_k (\varphi'_1).
\]
However, since
\[
\xi_k (\varphi''_2) < \xi_k (\varphi'_2) < \xi'_k  (\varphi'_2),
\]
we see that $\xi_k (\varphi''_2) \leq s'_k [\tau_k (\varphi''_2)+1]+2$.  Since the slope of a function in $R(D)$ can increase by at most one from the left side of $\gamma_k$ to the right side, we see that there is no function $\eta$ with $s_k (\eta) \leq \xi_k (\varphi''_2)$ and $s'_k (\eta) \geq s'_k [\tau_k (\varphi''_2)+1]$.  This shows that it is impossible for $\gamma_k$ to switch slope $\tau_k (\varphi'_2) - 1$, and completes the proof of the lemma.
\end{proof}

\begin{proof}[Proof of Proposition~\ref{Prop:Counting}]
Suppose that $\psi = \varphi_1 + \varphi_2$ and $\psi' = \varphi'_1 + \varphi'_2$ are in the same cohort on $\gamma_\ell$, and that the functions assigned to $\gamma_\ell$ include $\psi$ but not $\psi'$.  We must show that $\ell \not \in \{ z,b,b',z'-2 \}$ and $\Sha^{-1}(\ell) = \emptyset$.

\emph{Step 1:  Reduce to the case where the slope indices are the same.}  Suppose $\Sha(\psi) = \Sha(\psi') = k_0$.  By Lemma~\ref{lem:slopeindexchange}, if the set of slope indices $\{ \tau_{k+1}(\varphi_1), \tau_{k+1}(\varphi_2) \}$ and $\{ \tau_{k+1}(\varphi'_1), \tau_{k+1}(\varphi'_2) \}$ are different for some $k_0 \leq k < \ell$, then one of $\psi$ or $\psi'$ is assigned to $\gamma_k$, contradicting our assumption that they are in the same cohort on $\gamma_{\ell}$.  Furthermore, if $\{ \tau_{\ell+1}(\varphi_1), \tau_{\ell+1}(\varphi_2) \}$ and $\{ \tau_{\ell+1}(\varphi'_1), \tau_{\ell+1}(\varphi'_2) \}$ are different, then by Lemma~\ref{lem:slopeindexchange}, $\ell \notin \{ z,b,b',z'-2 \}$ and $\Sha^{-1}(\ell) = \emptyset$.  We may therefore assume that $\{ \tau_{k}(\varphi_1), \tau_{k}(\varphi_2) \} = \{ \tau_{k}(\varphi'_1), \tau_{k}(\varphi'_2) \}$ for $k_0 \leq k \leq \ell + 1$.   By Lemma~\ref{Lem:SameSlope}, we then have
\[
\tau'_{\ell}(\varphi_1) = \tau'_{\ell}(\varphi'_1) \mbox{ and } \tau'_{\ell}(\varphi_2) = \tau'_{\ell}(\varphi'_2)
\]
and $\varphi_1 \sim_{\gamma_{\ell}} \varphi'_1$, $\varphi_2 \sim_{\gamma_{\ell}} \varphi'_2$.

\emph{Step 2:  $\psi$ is departing on $\gamma_{\ell}$.}  Since $\psi \sim_{\gamma_{\ell}} \psi'$, and the functions assigned to $\gamma_\ell$ include $\psi$ but not $\psi'$, we see that $\psi$ is a departing function, but $\psi'$ is not.  Because $\psi$ is departing, either $\varphi_1$ or $\varphi_2$ must have higher slope on $\beta_{\ell + 1}$ than on $\beta_{\ell}$.  Without loss of generality we may assume that $s_{\ell+1} (\varphi_1) > s_{\ell} (\varphi_1)$.  Since $\Sha(\psi') = k_0$, we see that $\psi'$ is not shiny on $\gamma_{\ell+1}$, hence we must have
\[
s_{\ell+1} (\varphi'_1) = \xi_{\ell+1}(\varphi'_1) \mbox{ and } s_{\ell+1} (\varphi'_2) = \xi_{\ell+1}(\varphi'_2).
\]

\emph{Step 3:  Show that $\ell \notin \{ z,b,b',z'-2 \}$.}  If $\gamma_{\ell}$ is a switching loop or $\beta_{\ell}$ is a switching bridge, then $\ell \notin \{ z,b,b',z'-2 \}$, hence we may assume that $\tau_{\ell} = \tau'_{\ell} = \tau_{\ell+1}$ for each of the functions $\varphi_1, \varphi_2, \varphi'_1,$ and $\varphi'_2$.

If $\xi_{\ell+1}(\varphi'_1) < \xi'_{\ell} (\varphi'_1)$ or $\xi_{\ell+1}(\varphi'_2) < \xi'_{\ell} (\varphi'_2)$, then $\lambda_{\ell+1}$ is contained in $\lambda_{\ell}$, hence $\ell \notin \{ z,b,b',z'-2 \}$.   We may therefore assume that $\xi_{\ell+1}(\varphi'_1) = \xi'_{\ell} (\varphi'_1)$ and $\xi_{\ell+1}(\varphi'_2) = \xi'_{\ell} (\varphi'_2)$.  It follows that
\[
s'_{\ell} (\psi') = \xi'_{\ell} (\varphi_1) + \xi'_{\ell} (\varphi'_2).
\]
Since $\psi'$ is not departing, we have $\sigma_{\ell} = s'_{\ell} (\psi')$.  But $\xi'_{\ell} (\varphi'_1) = \xi'_{\ell} (\varphi_1)$, and $z,b,b',$ and $z'$ are chosen so that there is no integer $j$ such that $\sigma_{\ell} = \xi'_{\ell} (\varphi_1) + s'_{\ell} [j]$, so again $\ell \notin \{ z,b,b',z'-2 \}$.

\medskip

It remains to show that $\Sha^{-1}(\ell) = \emptyset$.  Assume that $\varphi''_1 + \varphi''_2 \in \cB$ is shiny on $\gamma_{\ell}$.  We will show that $\Sha(\varphi''_1 + \varphi''_2) \neq \ell$.

\medskip

\emph{Step 4:  Show that $\varphi''_1+\varphi''_2$ cannot be equivalent to a departing function on $\gamma_{\ell}$.}  Any function that is both departing and shiny on $\gamma_{\ell}$ is equivalent to $2\varphi_1$, and such a function exists only if $s_{\ell+1} (2\varphi_1) = \sigma_{\ell} +1$.  From this we see that both
\[
s_{\ell+1} (\varphi_1) > \xi_{\ell+1}(\varphi_1) \mbox{ and } s_{\ell+1} (\varphi_2) > \xi_{\ell+1}(\varphi_2).
\]
But, because
\[
s_{\ell+1} (\varphi'_1) = \xi_{\ell+1}(\varphi'_1) \mbox{ and } s_{\ell+1} (\varphi'_2) = \xi_{\ell+1}(\varphi'_2),
\]
we see that $s_{\ell+1} (\psi') = \sigma_{\ell} -1$.  This implies that $\psi'$ is not permissible on $\gamma_{\ell}$, a contradiction.

\emph{Step 5:  Reduce to inequalities on slopes.}
We will prove by case analysis that one of the following four inequalities holds:
\[
s_{\ell+1} (\varphi''_1) < \xi'_{\ell} (\varphi''_1), \ \ s_{\ell+1} (\varphi''_2) < \xi'_{\ell} (\varphi''_2), \ \
s_{\ell+1} (\varphi''_1 ) > \xi_{\ell+1}(\varphi''_1), \mbox{ \, or \, } s_{\ell+1} (\varphi''_2) > \xi_{\ell+1}(\varphi''_2).
\]
We claim that $\Sha(\varphi''_1 + \varphi''_2) \neq \ell$. To see this, note that if one of the first two inequalities holds, then $2D+\ddiv(\varphi''_1 +\varphi''_2)$ contains $w_{\ell}$, hence $\varphi''_1+\varphi''_2$ is equivalent to a departing function on $\gamma_{\ell}$, a contradiction.  If one of the second two inequalities holds, we see that $\varphi''_1 +\varphi''_2$ is shiny on $\gamma_{\ell+1}$, and the claim follows.

\emph{Step 6:  Case analysis.}  It remains to show that one of the inequalities above holds.  By Proposition~\ref{Prop:Shiny}, we may assume that the restriction of $D+\ddiv(\varphi''_1)$ to $\gamma_{\ell} \smallsetminus \{ v_{\ell} \}$ has degree 0.  It follows that $\varphi''_1 \sim_{\gamma_{\ell}} \varphi_1$, and $s_{\ell+1} (\varphi''_1) = s_{\ell+1} (\varphi_1)$.  Since $\varphi''_1 +\varphi''_2$ is not departing on $\gamma_{\ell}$, we have $s_{\ell+1} (\varphi''_1) = s_{\ell+1} (\varphi_2)-1$.  By property $(\mathbf B')$, the bridge $\beta_{\ell}$ is not a switching bridge, so $\tau_{\ell+1} (\varphi''_1) = \tau'_{\ell} (\varphi''_1)$ and $\tau_{\ell+1} (\varphi''_2) = \tau'_{\ell} (\varphi''_2)$.

We now consider several cases.  Our assumption that the slope indices of $\varphi_1$ and $\varphi_2$ agree with those of $\varphi'_1$ and $\varphi'_2$ implies that either
\[
s_{\ell+1} (\varphi_1) > \xi_{\ell+1}(\varphi_1) \mbox{ or } s_{\ell+1} (\varphi_2) > \xi_{\ell+1}(\varphi_2).
\]
First, suppose that $s_{\ell+1} (\varphi_1) > \xi_{\ell+1} (\varphi_1)$.  If $\tau'_{\ell} (\varphi''_1) > \tau'_{\ell} (\varphi_1)$, then
\[
\xi'_{\ell}  (\varphi''_1) > \xi'_{\ell}  (\varphi_1) \geq s_{\ell+1} (\varphi_1) = s_{\ell+1} (\varphi''_1).
\]
On the other hand, if $\tau'_{\ell} (\varphi''_1) \leq \tau'_{\ell} (\varphi_1)$, then
\[
s_{\ell+1} (\varphi''_1) = s_{\ell+1} (\varphi_1) > \xi_{\ell+1} (\varphi_1) \geq \xi_{\ell+1} (\varphi''_1).
\]
Next, suppose that $s_{\ell+1} (\varphi_2) > \xi_{\ell+1} (\varphi_2)$.  If $\tau'_{\ell} (\varphi''_2) \geq \tau'_{\ell} (\varphi_2)$, then
\[
\xi'_{\ell}  (\varphi''_2) \geq \xi'_{\ell}  (\varphi_2) \geq s_{\ell+1} (\varphi_2) > s_{\ell+1} (\varphi''_2).
\]
On the other hand, if $\tau'_{\ell} (\varphi''_2) < \tau'_{\ell} (\varphi_2)$, then
\[
s_{\ell+1} (\varphi''_2) = s_{\ell+1} (\varphi_2)-1 > \xi_{\ell+1} (\varphi_2) - 1 \geq \xi_{\ell+1} (\varphi''_2). \vspace{-15 pt}
\]
\end{proof}

We now complete the proof of Theorem~\ref{Thm:ExtremalConfig}\eqref{it:everyfunction}.

\begin{proposition}
\label{Prop:EverythingAssigned}
Every function in $\cB$ is assigned to a bridge or a loop.  Moreover, if $\psi_1$ and $\psi_2$ are assigned to the bridge $\beta_k$ for $k>1$ and $s_k (\psi_1 ) = s_k (\psi_2) = \sigma_k$, then $\psi_1$ and $\psi_2$ are in the same cohort on $\gamma_k$.
\end{proposition}

\begin{proof}
If $\psi \in \cB$ is not permissible on any loop, then $s_1 (\psi) >4$ or $s_{g'+1} (\psi) < 2$, and $\psi + c_\psi$ achieves the minimum on the first or last bridge, respectively.

On the first non-skippable loop, there are at most 3 cohorts.  To see this, it suffices to show that there are at most two pairs $(i,j)$, with $i \leq j$, such that $s_1 [i] + s_1 [j] = 4$.  Recall that, by assumption, we have either $s_1 [5] \leq 2$ or $s_1 [4] + s_1 [6] \leq 5$.  If $s_1 [5] \leq 2$, then $s_1 [i] + s_1 [j] < 4$ for all $i < j \leq 5$.  It follows that if $s_1 [i] + s_1 [j] = 4$, then either $i=j=5$, or $j=6$ and $i$ is uniquely determined.  On the other hand, if $s_1 [4] + s_1 [6] \leq 5$, then $s_1 [i] + s_1 [j] \leq 4$ for all pairs $i < j \leq 5$, with equality only if $i=4$, $j=5$.  It follows that if $s_1 [i] + s_1 [j] = 4$, then either $i=4$, $j=5$, or $j=6$ and $i$ is uniquely determined.

We show that if $\Sha(\varphi) \leq z$ then $\varphi$ is assigned to a loop $\gamma_{\ell}$ with $\ell \leq z$, or to the bridge $\beta_{z+1}$.  If there is a permissible function on the first non-skippable loop $\gamma_k$, then there is some $\psi \in \cB$ that is assigned to this loop.  Then $\Sha(\psi)= k$, so by Proposition~\ref{Prop:Counting}, any function in the same cohort is also assigned to $\gamma_k$.  It follows that there are at most 2 cohorts such that some function $\psi$ in the cohort is not assigned to $\gamma_k$.

As we proceed from left to right for $k \leq z$, every time we reach a new loop, there are two possibilities.  One possibility is that $\Sha^{-1}(k) = \emptyset$. In this case, there are no more cohorts on $\gamma_k$ than there are on $\gamma_{k-1}$.  The other possibility is that $\Sha(\psi) = k$  for some $\psi \in \cB$.  In this case, by assumption, there are permissible functions on $\gamma_k$, and $\gamma_k$ is not skippable, so some function $\psi'$ is assigned to $\gamma_k$.  By Proposition~\ref{Prop:Counting}, any function in the same cohort as $\psi'$ is also assigned to $\gamma_k$.  It follows that the number of cohorts on $\gamma_k$ such that some function in the cohort is not assigned to $\gamma_k$ is equal to the number of cohorts on $\gamma_{k-1}$ such that some function in the cohort is not assigned to $\gamma_{k-1}$.  Specifically, as we proceed from $\gamma_{k-1}$ to $\gamma_k$, we introduce the cohort of $\psi$, but we remove the cohort of $\psi'$.  By induction, therefore, the number of cohorts on $\gamma_k$ with the property that some function in the cohort is not assigned to $\gamma_k$ is at most 2.

By Proposition~\ref{Prop:GammaA}, no function is shiny on $\gamma_z$, and $\gamma_z$ is not skippable.  Combining this with our enumeration of cohorts in the preceding paragraph, we see that there are at most 2 cohorts on $\gamma_z$.  By assumption, there is a function $\psi \in \cB$ that is assigned to $\gamma_z$, and by Proposition~\ref{Prop:Counting}, any function in the same cohort on $\gamma_z$ is also assigned to $\gamma_z$.  After assigning this cohort, there is at most one cohort left.  Everything in the remaining cohort is assigned to the bridge $\beta_{z+1}$.

A similar analysis holds on the remaining two intervals where $\sigma$, as defined in \eqref{eq:sigma}, is constant.  In particular, every function $\varphi \in \cB$ with $z + 1 \leq \Sha(\varphi) \leq z'-2$ is assigned to a loop $\gamma_\ell$, with $z+1 \leq \ell \leq z'-2$ or to the bridge $\beta_{z'-1}$.  Likwewise, if $z'-1 \leq \Sha(\varphi) \leq g'$ then $\varphi$ is assigned to a loop $\gamma_m$ with $m \geq z'-1$ or to the final bridge $\beta_{g'+1}$.
\end{proof}

\subsection{Functions assigned to the same bridge}

In this subsection, we prove part (3) of Theorem~\ref{Thm:ExtremalConfig}.  We begin with the observation that, if the multiplicity of a bridge is at most 1, then the slope of a building block determines its slope index.

\begin{lemma}
\label{Lem:Rho2}
Let $\varphi, \varphi'$ be building blocks satisfying $s_{k+1} (\varphi) = s_{k+1} (\varphi')$, and suppose $\mu (\beta_{k+1}) \leq 1$.  Then $\tau'_k (\varphi) = \tau'_k (\varphi')$.
\end{lemma}

\begin{proof}
If $\tau'_k(\varphi) \neq \tau'_k(\varphi')$, then there is some $j \neq j'$ such that
\[
s'_k [j] \geq s_{k+1} (\varphi) \geq s_{k+1} [j] \mbox{ and } s'_k [j'] \geq s_{k+1} (\varphi') \geq s_{k+1} [j'].
\]
Assume without loss of generality that $j'>j$.  Then
\[
s'_k [j'] > s'_k [j] \geq s_{k+1} (\varphi) = s_{k+1} (\varphi') \geq s_{k+1} [j'] > s_{k +1} [j].
\]
It follows that the multiplicity of $\beta_{k+1}$ is at least 2.
\end{proof}

We now check that, if two functions in $\cB$ are assigned to the first bridge $\beta_1$ and have the same slope on $\beta_1$, then the slope indices of their summands in $\cA$ are the same.

\begin{lemma}
\label{lem:firstbridge}
Suppose $\varphi_1 + \varphi'_1$, and $\varphi_2 + \varphi'_2$ are elements of $\cB$ such that
$
s_1 (\varphi_1 + \varphi'_1) = s_1 (\varphi_2 + \varphi'_2) > 4.
$
Then $\{ \tau'_0 (\varphi_1), \tau'_0 (\varphi'_1) \} = \{ \tau'_0 (\varphi_2), \tau'_0 (\varphi'_2) \}$.
\end{lemma}

\begin{proof}
Recall that $\mu(\beta_1) < 2$.  By Lemma~\ref{Lem:Rho2}, it suffices to show that, if
\[
s'_0 [i] + s'_0 [i'] = s'_0 [j] + s'_0 [j'] > 4,
\]
then $(i,i') = (j,j')$.  Recall that either $s'_0 [5] \leq 2$ or $s'_0 [4] + s'_0 [6] \leq 5$.  If $s'_0 [5] \leq 2$, then $s'_0 [i] + s'_0 [i'] \leq 4$ for all pairs $i \leq i' \leq 5$.  Note that $s'_0 [6] + s'_0 [i] = s'_0 [6] + s'_0 [j]$ if and only if $i=j$, so the conclusion follows in this case.

On the other hand, if $s'_0 [5] \geq 3$ and $s'_0 [4] + s'_0 [6] \leq 5$, then $s'_0 [i] + s'_0 [i'] \leq 4$ for all pairs $i < i' \leq 5$.  It therefore suffices to show that there is no $i$ such that $s'_0 [i] + s'_0 [6] = 2s'_0 [5]$.  By assumption, however, we have
\[
s'_0 [5] + s'_0 [6] > 2s'_0 [5] \geq 6,
\]
which is greater than $s'_0 [4] + s'_0 [6] \geq s'_0 [i] + s'_0 [6]$ for all $i < 4$.
\end{proof}

For the last bridge, the argument is simpler.

\begin{lemma}
\label{lem:lastbridge}
Suppose $\varphi_1 + \varphi'_1$, and $\varphi_2 + \varphi'_2$ are elements of $\cB$ such that
$
s_{g'+1} (\varphi_1 + \varphi'_1) = s_{g'+1} (\varphi_2 + \varphi'_2) < 2.
$
Then $\{ \tau'_{g'} (\varphi_1), \tau'_{g'} (\varphi'_1) \} = \{ \tau'_{g'} (\varphi_2), \tau'_{g'} (\varphi'_2) \}$.
\end{lemma}

\begin{proof}
Since $\mu (\beta_{g'+1}) + \wt(v_{g'+1}) \leq 2$, we see that $s'_{g'} [i] = s_{g'+1} [i] = i$ for all $i \leq 3$.  It follows that, if $\varphi \in \cA$ satisfies $s_{g'+1} (\varphi) \leq 3$, then $\tau'_{g'} (\varphi) = s_{g'+1} (\varphi)$.  Hence, $s_{g'+1} (\varphi_1 + \varphi'_1) = 1$ if and only if $\{ \tau'_{g'} (\varphi_1), \tau'_{g'} (\varphi'_1) \} = \{ 0,1 \}$, and $s_{g'+1} (\varphi_1 + \varphi'_1) = 0$ if and only if $\{ \tau'_{g'} (\varphi_1), \tau'_{g'} (\varphi'_1) \} = \{ 0,0 \}$.
\end{proof}

\begin{lemma}
\label{Lem:OneEq}
Suppose that both $\psi_1 = \varphi_1 + \varphi'_1$ and $\psi_2 = \varphi_2 + \varphi'_2$ are assigned to the bridge $\beta_{k+1}$, and $s_{k+1} (\psi_1) = s_{k+1} (\psi_2)$.  Then $\{ \tau'_k (\varphi_1) , \tau'_k (\varphi'_1)  \} = \{  \tau'_k (\varphi_2) , \tau'_k (\varphi'_2) \}$.
\end{lemma}

\begin{proof}
If $k=0$, then this follows from Lemma~\ref{lem:firstbridge}.  If $k=g'$ and $s_{k+1} (\psi_1 ) = s_{k+1} (\psi_2) < 2$, then this follows from Lemma~\ref{lem:lastbridge}.

Otherwise, $k \in \{z,z'-2,g\}$ and $s_{k+1} (\psi_1 ) = s_{k+1} (\psi_2) = \sigma_k$.  By Proposition~\ref{Prop:EverythingAssigned}, $\psi_1$ and $\psi_2$ are in the same cohort on $\gamma_k$.  If $\{ \tau'_k (\varphi_1) , \tau'_k (\varphi'_1) \} \neq \{ \tau'_k (\varphi_2) , \tau'_k (\varphi'_2) \}$, then by Lemma~\ref{lem:slopeindexchange} either $\psi_1$ or $\psi_2$ is assigned to $\gamma_k$, contradicting our assumption that they are both assigned to $\beta_{k+1}$.  Thus, $\{ \tau'_k (\varphi_1) , \tau'_k (\varphi'_1) \} = \{ \tau'_k (\varphi_2) , \tau'_k (\varphi'_2) \}$.
\end{proof}

This completes the proof of Theorem~\ref{Thm:ExtremalConfig}.

\section{Constructing the tropical independence}
\label{Sec:Generic}

In this section we prove the following theorem.

\begin{theorem}
\label{thm:independence}
Let $g = 22$ or $23$, $d = g + 3$, and $g' \in \{ g, g-1, g-2 \}$.  Let $\Gamma'$ be a chain of $g'$ loops with $C$-admissible edge lengths for some $C > 12dg$, and let $D$ be a break divisor on $\Gamma'$ with $\Sigma \subset R(D)$ a tropical linear series of rank $6$. We assume furthermore that
\begin{itemize}
\item if $g' = g-1$ then $s'_0[5] \leq 2$, and
\item if $g' = g-2$ then either $s'_0[5] \leq 2$ or $s'_0[6] + s'_0[4] \leq 5$.
\end{itemize}
Then there is a tropically independent subset $\cT \subset 2\Sigma$ of size $|\cT| = 28$.
\end{theorem}

\noindent Theorem~\ref{Thm:MainThm} follows immediately from the case $g' = g$.  When $g'$ is $g-1$ or $g-2$, we have analogous consequences for multiplication maps for linear series with ramification on a general pointed curve of genus $g'$.  See Theorems~\ref{thm:MRCg-1} and \ref{thm:MRCg-2}, respectively.  All three are used in the proof of Theorem~\ref{thm:genfinite}.

\subsection{Overview of the proof} We prove Theorem~\ref{thm:independence} by considering cases depending on the properties of the bridges and loops of positive multiplicity.  Our strategy is the same in each case:

\begin{itemize}
\item Identify a finite set $\cS \subset \Sigma$ that contains functions with constant slope index $i$, when such functions exist, plus additional functions that account for switching bridges and loops.
\item Use $\cS$ to specify functions $f_{k,i}$, satisfying \eqref{eq:fki}-\eqref{eq:eff}, as in \S\ref{Sec:Extremals}, and consider the set $\cA$ of $(\Sigma, \{f_{k,i}\})$-building blocks (Definition~\ref{def:buildingblock}).
\item Choose a set $\cB \subset 2 \cA$ satisfying the technical conditions $(\mathbf B)$ and $(\mathbf B')$ (Definition~\ref{def:properties}).
\item Run the template algorithm using the set $\cB$ and the sequence $\sigma$ specified in \eqref{eq:sigma}.
\end{itemize}

By Theorem~\ref{Thm:ExtremalConfig}, the template algorithm assigns every function in $\cB$ to a bridge or loop.  In particular, the output is an assignment function $\alpha \colon \cB \to \{\beta_k\} \cup \{\gamma_k\}$ and a collection of coefficients $\{c(\psi) \mid \psi \in \cB \}$ such that the template $\theta = \min \{\psi + c(\psi) \mid \psi \in \cB \}$ has the following properties:
\begin{enumerate} [label=(T\arabic*)]
\item \label{T1} For each $\psi \in \cB$, the function $\psi + c(\psi)$ achieves the minimum in $\theta$ on a nonempty open subset of the bridge or loop $\alpha(\psi)$.
\item \label{T2} If $\alpha^{-1}(\gamma_k) \neq \emptyset$ then there is an open subset of $\gamma_k$ on which $\psi + c(\psi)$ achieves the minimum if and only if $\alpha(\psi) = \gamma_k$.
\item \label{T3} Writing elements of $\cB$ as sums of two building blocks, if $\varphi_1 + \varphi'_1$ and $\varphi_2 + \varphi'_2$ are assigned to the same loop $\gamma_k$, then, after possibly reordering, $\varphi_1 \sim_{\gamma_k} \varphi_2$ and $\varphi'_1 \sim_{\gamma_k} \varphi'_2$.
\item \label{T4} If $\varphi_1 + \varphi'_1$ and $\varphi_2 + \varphi'_2$ are assigned to the same loop $\gamma_k$, and $\gamma_k$ is not $i$-switching for $i \in \{ \tau_k (\varphi_1), \tau_k (\varphi_2), \tau_k (\varphi'_1) , \tau_k (\varphi'_2) \}$, then $\{ \tau'_k (\varphi_1) , \tau'_k (\varphi'_1) \} = \{ \tau'_k (\varphi_2) , \tau'_k (\varphi'_2) \}$.
\item \label{T5} If  $\varphi_1 + \varphi'_1$ and $\varphi_2 + \varphi'_2$ are assigned to the same bridge $\beta_k$ and have the same slope on $\beta_k$, for $k > 1$, then $\{ \tau'_{k-1} (\varphi_1), \tau'_{k-1} (\varphi'_1) \} = \{ \tau'_{k-1} (\varphi_2), \tau'_{k-1} (\varphi'_2) \}$.
\end{enumerate}

Given $\alpha$, $\{c(\psi)\}$, and $\theta$ with these properties, the final steps in the proof are relatively straightforward. In each case, we:

\begin{itemize}
\item Identify a subset $\cT \subset 2\cS$ of size $|\cT| = 28$, define $\Upsilon = \min\{ \psi - b(\psi) \mid \psi \in \cT \}$ to be the best approximation of $\theta$ from above (\S\ref{sec:bestapprox}), and use the properties \ref{T1}-\ref{T5} to show that $\Upsilon$ is a certificate of independence.
\end{itemize}

\noindent The underlying idea of this final step is that we choose a set $\cT$ of 28 functions in $2\cS$ that can ``replace'' equivalence classes of functions in $\cB$ in the template $\theta$ on the bridges and loops to which they are assigned.  More formally, in each case, we give a ``replacement'' function $R \colon \cT \to \cB$ with the following properties:
\begin{enumerate} [label=(R\arabic*)]
\item \label{it:R1} For each $k$, there is at most one $\psi \in \cT$ such that $\alpha(R(\psi)) = \gamma_k$.
\item \label{it:R2} For each $\beta_\ell$ and each slope $s$, there is at most one $\psi \in \cT$ such that $\alpha(R(\psi)) = \beta_\ell$ and $s_\ell(R(\varphi)) = s$.
\item \label{it:R3} Each $\psi \in \cT$ achieves the minimum in $\Upsilon$ on the entire region in $\alpha(R(\psi))$ where $R(\psi)$ achieves the minimum in $\theta$, and does not achieve the minimum on the interior of the region in $\alpha(R(\psi'))$ where $R(\psi')$ achieves the minimum in $\theta$, for $\psi' \neq \psi$.
\end{enumerate}
It follows that $\psi$ achieves the minimum in $\Upsilon$ uniquely on the interior of the region in $\alpha(R(\psi))$ where $R(\psi)$ achieves the minimum in $\theta$. In particular, $\Upsilon$ is a certificate of independence, as required.

When verifying \ref{it:R1}-\ref{it:R3}, we use the following lemma about best approximations of tropical linear combinations from above.

\begin{lemma}
\label{Lem:Replace}
Let $\theta = \min_{\psi \in \cB} \{ \psi + c_\psi \}$. Suppose $\varphi = \min_{\psi' \in \cB'} \{ \psi' + b_{\psi'} \}$, where $\cB' \subset \cB$.  Then the best approximation of $\theta$ by $\varphi$ from above achieves equality on the entire region where $\psi' + c_{\psi'}$ achieves the minimum in $\theta$, for some $\psi' \in \cB'$.
\end{lemma}

\begin{proof}
Let $c = \min_{\psi' \in \cB'} \{ b_{\psi'} - c_{\psi'}\}$. Choose $\psi' \in \cB'$ such that $c = b_{\psi'} - c_{\psi'}$.  Then $\varphi - c \geq \theta$, with equality at points where $\psi' + c_{\psi'}$ achieves the minimum in $\theta$.
\end{proof}

\begin{remark} \label{rem:choosef}
The choice of $\cS$ is the most important step in each case.  Given the hard work that has already been done in proving Theorem~\ref{Thm:ExtremalConfig}, the rest of the argument is more-or-less mechanical. This choice is more delicate when $\Sigma$ has two switching loops, and in particular when there is an $h$-switching loop followed by an $h'$-switching loop for $h' \in \{h-1, h, h+1\}$.

In the second step, we specify $f_{k,i}$ to be the restriction to $\gamma_k$ of a function in $\cS$. For instance, if there is a function $\varphi_i \in \cS$ with constant slope index $i$ then we set $f_{k,i} = (\varphi_i)_{|\gamma_k}$.  In the special case where there are no switching loops or bridges, then, for each $0 \leq i \leq 6$ there is a function $\varphi_i$ with constant slope index $i$ (Lemma~\ref{Lem:GenericFns}). In this case, we set $\cS = \{\varphi_0, \ldots, \varphi_6\}$ and prove that $\cT = 2\cS$ is tropically independent.

In the case where there is a single $h$-switching loop, then there is a function $\varphi_i$ with constant slope index $i$ for $i \not \in \{h, h+1\}$, and we choose 3 additional functions $\varphi_A$, $\varphi_B$, and $\varphi_C$ in $\Sigma$, whose slope indices are all in $\{h, h+1\}$, analogous to the functions $\psi_A$, $\psi_B$ and $\psi_C$ in Example~\ref{Ex:Interval}.  We then set $\cS = \{\varphi_i \mid i \neq h, h+1\} \cup \{ \varphi_A, \varphi_B, \varphi_C \}$ and find an independent set $\cT \subset 2\cS$ of size $|\cT| = 28$.  This set $\cT$ contains every function of the form $\varphi_i + \varphi_j$ for $i,j \not \in \{h, h+1\}$, two functions from the set $\{ \varphi_i + \varphi_A, \varphi_i + \varphi_B, \varphi_i + \varphi_C \}$ for $i \not \in \{h, h+1\}$, and 3 from the sumset $2 \{ \varphi_A, \varphi_B, \varphi_C \}$.
\end{remark}

We conclude this overview by outlining the cases and subcases to be considered in our proof of Theorem~\ref{thm:independence}.  Recall that switching loops have positive multiplicity and the sum of the multiplicities of all loops and bridges is at most 2. Therefore, $\Sigma$ falls into one of the following cases:
\renewcommand{\theenumi}{\arabic{enumi}}
\begin{enumerate}
\item  There are no switching loops and no bridges of multiplicity 2.
\begin{enumerate}
\item There are no bridges of multiplicity 1.
\item There is 1 bridge of multiplicity 1.
\item There are 2 bridges of multiplicity 1.
\end{enumerate}
\item  There is a bridge of multiplicity 2.
\begin{enumerate}
\item The bridge of multiplicity 2 is $\beta_1$.
\item The bridge of multiplicity 2 is not $\beta_1$ and is not a switching bridge.
\item The bridge of multiplicity 2 is not $\beta_1$ and is a switching bridge.
\end{enumerate}
\item  There is one switching loop.
\begin{enumerate}
\item There are no bridges of multiplicity 1.
\item There is 1 bridge of multiplicity 1.
\end{enumerate}
\item  There are two switching loops.
\end{enumerate}
In the case of two switching loops, there are no bridges of positive multiplicity, and we consider subcases depending on the relationship between the two switching loops. More precisely, when there is an $h$-switching loop and an $h'$ switching loop, we consider separate cases for $h' \in \{h-1, h, h+1\}$.

\medskip

In all cases, when we run the template algorithm, the input $\sigma$ is specified by \eqref{eq:sigma}.

\medskip

For the reader's convenience, we recall the statement of the technical properties $(\mathbf B)$ and $(\mathbf B')$ to be verified in each case (from Definition~\ref{def:properties}):

\begin{itemize}
\item[$(\mathbf B)$]  Whenever there is a permissible function $\psi = \varphi + \varphi' \in \cB$ on $\gamma_k$ such that $2D+\ddiv(\psi)$ contains $w_k$, and either $\gamma_k$ is $\tau_k (\varphi)$-switching or $s_{k+1} (\varphi) < \xi'_k  (\varphi)$, then there is some permissible function $\psi' \in \cB$ that is equal to $\psi$ on $\Gamma_{\leq k}$ such that $s_{k+1} (\psi') > s_{k+1} (\psi)$.
\item[$(\mathbf B')$]  Whenever there are permissible functions in $\cB$ that are equivalent on $\gamma_k$ with different slopes on $\beta_{k+1}$, and either $\gamma_k$ is a switching loop or $\beta_{k+1}$ is a switching bridge, then no function in $\cB$ is shiny on $\gamma_k$.
\end{itemize}

\subsection{Case 1:  no switching loops and no bridges of multiplicity 2} \label{sec:noswitching}

Suppose there are no switching loops or bridges.  By Lemma~\ref{Lem:GenericFns}, for each $0 \leq i \leq 6$, there is a function $\varphi_i \in \Sigma$ such that
\[
s_k (\varphi_i) = s_k [i] \mbox{ and } s'_k (\varphi_i) = s'_k [i], \mbox{ for all } k.
\]
We set $\cS = \{ \varphi_i \mid 0 \leq i \leq 6 \}$ and $f_{k,i} = (\varphi_i)_{|\gamma_k}$.  We will show that $\cT = 2 \cS$ is tropically independent.  In each case where there are no switching loops or bridges, we will use the following lemma.

\begin{lemma}
\label{Lem:EasyReplace}
Let $\cA$ be the set of $(\Sigma, \{f_{k,i}\})$-building blocks, let $\theta = \min \{\psi + c(\psi) \mid \psi \in 2\cA \}$ and let $\alpha \colon 2\cA \to \{\beta_k\} \cup \{\gamma_k\}$ be an assignment function satisfying \ref{T1}-\ref{T5}.  If each function $\varphi_i \in \cS$ is a tropical linear combination of building blocks with constant building sequence $i$, then the best approximation of $\theta$ by $\cT$ from above is a certificate of independence.
\end{lemma}

\begin{proof}
Consider the best approximation $\Upsilon = \min\{ \psi - b(\psi) \mid \psi \in \cT \}$ of $\theta$ by $\cT$ from above.  By Lemma~\ref{Lem:Replace}, each function $\varphi_i + \varphi_j \in \cT$ achieves equality on the region where $\varphi'_i + \varphi'_j + c_{\varphi'_i + \varphi'_j}$ achieves the minimum in $\theta$, where $\varphi'_i \in \cA$ has constant building sequence $i$ and $\varphi'_j \in \cA$ has constant building sequence $j$.  We define the replacement function $R \colon \cT \to 2\cA$ by $\varphi_i + \varphi_j \mapsto \varphi'_i + \varphi'_j$.  (If there exists more than one function $\psi \in 2\cA$ such that $\varphi_i + \varphi_j \in \cT$ achieves equality on the region where $\psi + c_{\psi}$ achieves the minimum in $\theta$, simply choose one such $\psi$.)  It suffices to show that the replacement function $R$ satisfies the conditions \ref{it:R1}-\ref{it:R3}.

If $\varphi_1 + \varphi'_1$ and $\varphi_2 + \varphi'_2$ are assigned to the same loop $\gamma_k$, then by \ref{T4}, we have $\{ \tau'_k (\varphi_1) , \tau'_k (\varphi'_1) \} = \{ \tau'_k (\varphi_2) , \tau'_k (\varphi'_2) \}$.  If $\{ \tau'_k (\varphi_1) , \tau'_k (\varphi'_1) \} = \{ i,j \}$, then by construction there is at most one function $\psi \in \cT$ such that $\alpha(R(\psi)) = \gamma_k$, namely, $\psi = \varphi_i + \varphi_j$.  It follows that $R$ satisfies condition \ref{it:R1}.  Similarly, if two functions in $2\cA$ are assigned to the same bridge $\beta_k$ and have the same slope along that bridge, then by the same argument using \ref{T5}, there is at most one function $\psi \in \cT$ such that $\alpha(R(\psi)) = \beta_k$.  It follows that $R$ satisfies condition \ref{it:R2}.  By Lemma~\ref{Lem:Replace}, the region where $\psi - b(\psi)$ achieves equality in $\Upsilon$ contains the region where $R(\psi) + c_{R(\psi)}$ achieves the minimum in $\theta$.  By \ref{T1}, each function in $\cB$ obtains the minimum on an open subset of the bridge or loop to which it is assigned, and by \ref{T2}, no other function obtains the minimum on this open subset.  Therefore, $R$ satisfies condition \ref{it:R3}.
\end{proof}

\subsubsection{Case 1a:  no bridges of positive multiplicity}
If there are no bridges of positive multiplicity then each $\varphi_i$ has constant slope along each bridge, and the slope index sequence $(\iota'_0(\varphi_i), \ldots, \iota_{g+1}(\varphi_i))$ is the constant sequence $i$.  It follows that the set of $(\Sigma, \{f_{k,i}\})$-building blocks is $\cA = \{\varphi_0, \ldots, \varphi_6\}$.  Let $\cB = 2 \cA$.  Properties $(\mathbf B)$ and $(\mathbf B')$ are satisfied vacuously. Therefore, the output of the template algorithm satisfies \ref{T1}-\ref{T5}.

We set $\cT = 2 \cS$. Note that $\cT = \cB$ and hence $\Upsilon = \theta$.  The replacement function that we consider is the identification $\cT = \cB$.  By Lemma~\ref{Lem:EasyReplace}, since each element of $\cS$ is a building block with constant building sequence, $\theta$ is a certificate of independence.
\qed

\subsubsection{Case 1b: one bridge of multiplicity one}  In this case, there is one index $h$ and one bridge $\beta_\ell$ such that the slope of $\varphi_h$ decreases on $\beta_{\ell}$, and it decreases by exactly 1, i.e., $\delta'_\ell(\varphi_h) = 1$. There are building blocks $\varphi_h^0$ and $\varphi_h^\infty$ with constant building sequence $(\tau'_0, \ldots, \tau_{g+1}) = (h, \ldots, h)$, and
\[
s_\ell(\varphi_h^0) = s_\ell[h]; \quad  \quad s_\ell(\varphi_h^\infty) = s_\ell[h] + 1.
\]
Note that $\varphi_h$ is a tropical linear combination of $\varphi_h^0$ and $\varphi_h^\infty$, as in Example~\ref{Ex:Interval}.

The set of $(\Sigma, \{f_{k,i})$-building blocks is $\cA = \{ \varphi_0, \ldots, \widehat \varphi_h, \ldots, \varphi_6 \} \cup \{\varphi_h^0, \varphi_h^\infty \}$. We claim that $\cB = 2\cA$ satisfies properties $(\mathbf B)$ and $(\mathbf B')$.  The property $(\mathbf B')$ holds vacuously.
To see that $\cB$ has property $(\mathbf B)$, note that $\varphi^0_h$ is the only function in $\cA$ satisfying $s_{k} (\varphi^0_h) < \xi'_{k-1}  (\varphi^0_h)$, and then only when $k = \ell$.  Then $(\mathbf B)$ holds because $\varphi^{\infty}_h$ is equal to $\varphi^0_h$ to the left of $\beta_{\ell}$ and has larger slope on $\beta_{\ell}$.  Therefore, the output of the template algorithm satisfies \ref{T1}-\ref{T5}.

For $i \neq h$, $\varphi_i \in \cS$ is a building block with constant building sequence $i$.  By the above, $\varphi_h$ is a tropical linear combination of $\varphi_h^0$ and $\varphi_h^\infty$.  Therefore, by Lemma~\ref{Lem:EasyReplace}, we see that the best approximation of $\theta$ by $\cT$ from above is a certificate of independence.
\qed

\subsubsection{Case 1c:  two bridges of multiplicty one}

In this case, there are two indices $h, h'$ and two bridges $\beta_\ell , \beta_{\ell'}$ such that the slope of $\varphi_h$ decreases by 1 on $\beta_{\ell}$, and the slope of $\varphi_{h'}$ decreases by 1 on $\beta_{\ell'}$.  It is possible that $h=h'$, but by assumption, $\ell \neq \ell'$.  We again let $\cA$ be the set of all $(\Sigma, \{f_{k,i}\})$-building blocks, and let $\cB = 2\cA$.  The set $\cB$ satisfies property $(\mathbf B')$ vacuously.   It satisfies property $(\mathbf B)$ because there is only one function $\varphi \in \cA$ satisfying $s_{\ell} (\varphi) < \xi'_{\ell-1}  (\varphi)$, and only one function $\varphi' \in \cA$ satisfying $s_{\ell'} (\varphi') < \xi'_{\ell-1}  (\varphi')$.  As in Case 1b, there is a function that is equal to $\varphi$ to the left of $\beta_{\ell}$ and has higher slope on $\beta_{\ell}$, and a function that is equal to $\varphi'$ to the left of $\beta_{\ell'}$ and has higher slope on $\beta_{\ell'}$.  The function $\varphi_h \in \cS$ is a tropical linear combination of the $(\Sigma, \{f_{k,i}\})$-building blocks in $\cA$ with constant building sequence $h$, and the function $\varphi_{h'} \in \cS$ is a tropical linear combination of the building blocks in $\cA$ with constant building sequence $h'$.  By Lemma~\ref{Lem:EasyReplace}, therefore, the best approximation of $\theta$ by $\cT$ from above is a certificate of independence.
\qed

\subsection{Case 2:  there is a bridge of multiplicity two}  \label{Sec:SwitchBridgeCase}

\subsubsection{Case 2a:  the bridge of multiplicity two is $\beta_1$}

Let $x \in \beta_1$ be a point such that $s_{\zeta} [i] = s_1 [i]$ for all rightward tangent vectors $\zeta$ on $\beta_1$ to the right of $x$.  Let $\Gamma'' \subset \Gamma'$ be the subgraph consisting of all points to the right of $x$.  The graph $\Gamma''$ is a chain of loops that may not have $C$-admissible edge lengths, only because the first bridge may be too short.  On $\Gamma''$, by Lemma~\ref{Lem:GenericFns}, for each $0 \leq i \leq 6$, there is a function $\varphi_i \in \Sigma$ such that
\[
s_k (\varphi_i) = s_k [i] \mbox{ and } s'_k (\varphi_i) = s'_k [i], \mbox{ for all } k.
\]
We set $\cS = \{ \varphi_i \mid 0 \leq i \leq 6 \}$ and $f_{k,i} = (\varphi_i)_{|\gamma_k}$.  We will show that $\cT = 2 \cS$ is tropically independent.

On $\Gamma''$, the multiplicity of every loop $\gamma_k$ and every bridge $\beta_k$ is zero.  Each function $\varphi_i$ therefore has constant slope along each bridge, and the slope index sequence $(\iota'_0(\varphi_i), \ldots, \iota_{g+1}(\varphi_i))$ is the constant sequence $i$.  It follows that the set of $(\Sigma, \{f_{k,i}\})$-building blocks is $\cA = \{\varphi_0, \ldots, \varphi_6\}$.  Let $\cB = 2 \cA$.  Properties $(\mathbf B)$ and $(\mathbf B')$ are satisfied vacuously.  Although $\Gamma''$ might not have $C$-admissible edge lengths, only because the first bridge may be too short, the output of the template algorithm still satisfies \ref{T1}-\ref{T5}.  This is because the proof of Theorem~\ref{Thm:ExtremalConfig} uses the length of $\beta_1$ in only one case: when there exists $\psi \in \cB$ such that $s_1 (\psi) \geq 5$, and $s_k (\psi) \leq 4$ for some $k>1$.  Because the loops and bridges of $\Gamma''$ all have multiplicity zero, no such $\psi \in \cB$ exists.

As in case 1a, we set $\cT = 2 \cS$.  Note that $\cT = \cB$ and hence $\Upsilon = \theta$, and we again consider the replacement function to be the identification $\cT = \cB$.  The template $\theta$ is a certificate of independence exactly as in Case 1a.  This proves Theorem~\ref{thm:independence} in Case 2a. \qed

\subsubsection{Case 2b:  the bridge of multiplicity two is not $\beta_1$ and is not a switching bridge}

In this case, there are no switching loops or bridges. Therefore, by Lemma~\ref{Lem:GenericFns}, for each $0 \leq i \leq 6$, there is a function $\varphi_i \in \Sigma$ such that
\[
s_k (\varphi_i) = s_k [i] \mbox{ and } s'_k (\varphi_i) = s'_k [i], \mbox{ for all } k.
\]

Let $\beta_{\ell}$ be the bridge of multiplicity two. Either there is one index $h$ such that the slope of $\varphi_h$ decreases by 2 on $\beta_{\ell}$, or there are two indices $h$ and $h'$ such that the slopes of both $\varphi_h$ and $\varphi_{h'}$ decrease by 1 on $\beta_{\ell}$. We set $\cS = \{ \varphi_i \mid 0 \leq i \leq 6 \}$ and $f_{k,i} = (\varphi_i)_{|\gamma_k}$.  We again show that $\cT = 2 \cS$ is tropically independent.

We again let $\cA$ be the set of all $(\Sigma, \{f_{k,i}\})$-building blocks, and let $\cB = 2\cA$.  The set $\cB$ satisfies property property $(\mathbf B')$ vacuously, and it satisfies property $(\mathbf B)$ just as in Case 1b.  The function $\varphi_h \in \cS$ is a tropical linear combination of the building blocks in $\cA$ with constant building sequence $h$, and the function $\varphi_{h'} \in \cS$ is a tropical linear combination of the building blocks in $\cA$ with constant building sequence $h'$.   Therefore,  by Lemma~\ref{Lem:EasyReplace}, the best approximation of $\theta$ by $\cT$ from above is a certificate of independence.
\qed

\subsubsection{Case 2c:  the bridge of multiplicity two is not $\beta_1$ and is a switching bridge}

This is the first case that does not follow from Lemma~\ref{Lem:EasyReplace}.  Suppose $\beta_\ell$ is a switching bridge.  By Proposition~\ref{Prop:SwitchBridgeComputation}, there is a unique index $h$ such that $\beta_{\ell}$ switches slope $h$.  Moreover, $\beta_\ell$ has multiplicity 2 and
\begin{align}
\label{Eq:BridgeSwitch}
s'_{\ell-1} [h+1] = s'_{\ell-1} [h] +1 = s_{\ell} [h+1] + 1 = s_{\ell} [h] + 2.
\end{align}
By Lemma~\ref{Lem:GenericFns}, for $j \notin \{h,h+1\}$, there is $\varphi_j \in \Sigma$ with
$
s_k (\varphi_j) = s_k [j] \mbox{ and } s'_k (\varphi_j) = s'_k [j] \mbox{ for all } k.
$

\begin{lemma}
\label{Lem:BridgeSwitch}
There is a unique point $x \in \beta_\ell$ where the incoming and outgoing slopes, denoted $s_x$ and $s'_x$, respectively, satisfy $
s_x [i] = s'_{\ell-1} [i] \mbox{ and } s'_x [i] = s_{\ell} [i] \mbox{ for all } i.
$
\end{lemma}

\begin{proof}
The argument is similar to Case 2 of Example~\ref{Ex:Interval}.
\end{proof}

We now identify a subset $\cS \subset \Sigma$.  It will consist of the functions $\varphi_j$ for $j \notin \{h,h+1\}$, plus three more functions that are contained in a rank 1 tropical linear subseries and characterized in Proposition~\ref{Prop:BridgeFns}.  They are closely analogous to the functions $\psi_A$, $\psi_B$ and $\psi_C$ in Example~\ref{Ex:Interval}.

\begin{proposition}
\label{Prop:BridgeFns}
There is a rank 1 tropical linear subseries $\Sigma' \subset \Sigma$ and functions $\varphi_A , \varphi_B,$ and $\varphi_C$ in $\Sigma'$ with the following properties:
\begin{enumerate}[label=(\roman*)]
\item  $s'_k (\varphi_A) = s'_k [h]$ for all $k < \ell$, and $s_x (\varphi_A) = s_x [h]$;
\item  $s'_x (\varphi_B) = s'_x [h+1]$, and $s_k (\varphi_B) = s_k[h+1]$ for all $k \geq \ell$;
\item  $s'_k (\varphi_C) = s'_k [h+1]$ for all $k < \ell$, and $s_k (\varphi_C) = s_k [h]$ for all $k \geq \ell$;
\item   $s_k (\varphi_\bullet) \in \{ s_k [h], s_k [h+1]\}$ and $s'_k (\varphi_\bullet) \in \{ s'_k [h], s'_k [h+1]\}$ for all $k$.
\end{enumerate}
\end{proposition}

We find it helpful to illustrate the essential properties of these functions in Figure~\ref{Fig:BridgeDependence}, which provides a ``zoomed out'' view in which the chain of loops looks like an interval.   A region labeled $h$ in this interval indicates that $\varphi$ has $s_k (\varphi) = s_k [h]$ and $s'_k (\varphi) = s'_k [h]$  for all $k$ in the given region.  More precisely, the interval depicted in Figure~\ref{Fig:BridgeDependence} is the subgraph $\Gamma_L \subset \Gamma$ formed by the union of all bridges and the lower half of each loop. For each function $\varphi \in \Sigma$ and rightward tangent vector $\zeta$ along this subgraph, there exists an $i$ such that $s_{\zeta} (\varphi) = s_{\zeta} [i]$.  The illustration to the right of a function $\varphi$ depicts the function $\zeta \mapsto i$ for the given $\varphi$.

\begin{proof}
By Lemma~\ref{Lemma:Existence}, there is $\varphi_A \in \Sigma$ such that $s'_0 (\varphi_A) \leq s'_0 [h]$, and $s_{g'+1} (\varphi_A) \geq s_{g'+1} [h]$.  Since $\beta_{\ell}$ is the only switching bridge, and there are no switching loops, we have $s'_k (\varphi_A) \leq s'_k [h]$ for $k < \ell$, and $s_k (\varphi_A) \geq s_k [h]$ for $k \geq \ell$.  In particular, $s_{\ell} (\varphi_A) \geq s_{\ell} [h]$, so $s_x (\varphi_A) \geq s_x [h]$, and it follows that $s_x (\varphi_A) = s_x [h]$.  This proves that $\varphi_A$ satisfies (i), because there are no switching loops or bridges to the left of $\beta_{\ell}$.  The construction of $\varphi_B$ satisfying (ii) is similar.

We now construct $\varphi_C$ satisfying (iii).
By Definition~\ref{Def:TLS}\eqref{DefItem:Recursive} there is a rank $1$ tropical linear subseries $\Sigma' \subset \Sigma$ that contains $\{ \varphi_A, \varphi_B\}$.   Arguments similar to the proof of (i) above show that $s_k(\Sigma') = (s_k[h], s_k[h+1])$, for all $k$, and $s_x(\Sigma')) = (s_x[h], s_x[h+1])$.  Choose $\varphi \in \Sigma'$ such that $s_x (\varphi) = s_x [h+1]$.  Then $s'_k (\varphi) = s'_k [h+1]$ for $k < \ell$.  Similarly, choose $\varphi' \in \Sigma'$ such that $s'_x (\varphi') = s'_x [h]$, and $s_k (\varphi') = s_k [h]$ for $k > \ell$.  By adding a scalar to $\varphi'$, we may assume that $\varphi(x) = \varphi'(x)$ and set $\varphi_C = \min \{ \varphi, \varphi' \}$.
\end{proof}

\begin{figure}[H]
\begin{tikzpicture}

\draw (-0.2,1) node {{\tiny $\varphi_C$}};
\draw (0,1)--(6,1);
\draw [ball color=white] (3,1) circle (0.55mm);

\draw (1.5,1.2) node {{\tiny $h+1$}};
\draw (4.5,1.2) node {{\tiny $h$}};

\draw (-0.2,3) node {{\tiny $\varphi_A$}};
\draw (0,3)--(6,3);
\draw (3,2.8) node {{\tiny $x$}};
\draw [ball color=white] (3,3) circle (0.55mm);
\draw [ball color=black] (4.5,3) circle (0.55mm);
\draw (1.5,3.2) node {{\tiny $h$}};
\draw (3.75,3.2) node {{\tiny $h+1$}};
\draw (5.25,3.2) node {{\tiny $h$}};

\draw (-0.2,2) node {{\tiny $\varphi_B$}};
\draw (0,2)--(6,2);
\draw [ball color=white] (3,2) circle (0.55mm);
\draw [ball color=black] (1.5,2) circle (0.55mm);
\draw (0.75,2.2) node {{\tiny $h+1$}};
\draw (2.25,2.2) node {{\tiny $h$}};
\draw (4.5,2.2) node {{\tiny $h+1$}};

\end{tikzpicture}
\caption{A schematic depiction of the three functions $\varphi_A$, $\varphi_B$, $\varphi_C$ from Proposition~\ref{Prop:BridgeFns}, restricted to the subgraph $\Gamma_L \subset \Gamma$ formed by the bridges and the lower half of each loop.}
\label{Fig:BridgeDependence}
\end{figure}

Let $\cS = \{ \varphi_j \mid j \neq h, h+1 \} \cup \{ \varphi_A, \varphi_B, \varphi_C \}$.  For $i \neq h,h+1$, let $f_{k,i} = (\varphi_i)_{|\gamma_k}$.  For $k < \ell$, let $f_{k,h} = (\varphi_A)_{|\gamma_k}$ and $f_{k,h+1} = (\varphi_C)_{|\gamma_k}$.  For $k \geq \ell$, let $f_{k,h} = (\varphi_C)_{|\gamma_k}$ and $f_{k,h+1} = (\varphi_B)_{|\gamma_k}$.  We now describe the set of building blocks $\cA$.  It will include $\{ \varphi_j \mid j \neq h, h+1 \}$ along with three additional functions, as follows.

\begin{lemma}
\label{Lem:BuildingBlocksBridge}
There are building blocks $\varphi^0_h, \varphi^0_{h+1},$ and $\varphi^{\infty}_h$ in $R(D)$ such that
\begin{enumerate}[label=(\roman*)]
\item  $s_k (\varphi^0_h) = s_k [h]$ for all $k$;
\item  $s_k (\varphi^0_{h+1}) = s'_{k-1} [h+1]$ for all $k$;
\item  $s_k (\varphi^{\infty}_h) = s_k[h]$ for all $k < \ell$, and $s_k (\varphi^{\infty}_h) = s_k [h+1]$ for all $k \geq \ell$.
\end{enumerate}
\end{lemma}

\begin{proof}
The function $\varphi^0_h$ is the building block with constant building sequence $h$ and $s_k = s_k[h]$ for all $k$.  Similarly, $\varphi^0_{h+1}$ is the building block with constant building sequence $h+1$ and $s_k = s'_{k-1}[h+1]$ for all $k$.  Finally, $\varphi^{\infty}_h$ is the unique building block with building sequence
\[
\tau_k = \tau'_k = \begin{cases} h & \mbox{if $k < \ell$} \\  h+1 & \mbox{if $k \geq \ell$.} \end{cases}
\]
\vskip -18pt
\end{proof}

Set $\cA := \{ \varphi_i \, \mid \, i \neq h,h+1 \} \cup \{ \varphi^0_h , \varphi^0_{h+1} , \varphi^{\infty}_h \}$.
Note that the slope of the function $\varphi^0_h$ along $\beta_{\ell}$ is $s_{\ell} [h]$, which is not in $s'_{\ell-1} (\Sigma)$.  Hence $\varphi^0_h$ cannot be in $\Sigma$.  Similarly, the function $\varphi^0_{h+1}$ cannot be in $\Sigma$.  However, the functions $\varphi_A$, $\varphi_B$, and $\varphi_C$ can be written as tropical linear combinations of $\varphi^0_h, \varphi^0_{h+1},$ and $\varphi^{\infty}_h$, as follows.

\begin{lemma}
\label{lem:TLC}
The restrictions of the functions $\varphi_A$, $\varphi_B$, and $\varphi_C$ to $\Gamma_L$ can be written as tropical linear combinations of the restrictions of the building blocks $\varphi^0_h, \varphi^0_{h+1},$ and $\varphi^{\infty}_h$, as follows:
\begin{enumerate}[label=(\roman*), leftmargin=.75cm]
\item  The function $(\varphi_A)_{|\Gamma_L}$ is uniquely expressible as a tropical linear combination of $\varphi^0_h$ and $\varphi^{\infty}_h$;
\item  The function $(\varphi_B)_{|\Gamma_L}$ is uniquely expressible as a tropical linear combination of $\varphi^0_{h+1}$ and $\varphi^{\infty}_h$;
\item  The function $(\varphi_C)_{\Gamma_L}$ is uniquely expressible as a tropical linear combination of  $\varphi^0_h$ and $\varphi^0_{h+1}$.
\end{enumerate}
\end{lemma}

\begin{proof}
This is very similar to Example~\ref{Ex:Interval}.  We will prove the claim about $\varphi_A$ in detail; the functions $\varphi_B$ and $\varphi_C$ are handled similarly.  We first show that $\varphi^{\infty}_h$ has slope $s_{\zeta} [h+1]$ on all tangent vectors $\zeta$ in $\Gamma_L$ to the right of $x$, and $\varphi^0_h$ has slope $s_{\zeta} [h]$ on all such tangent vectors.  For tangent vectors $\zeta$ in the bridges, this statement follows from the definition of $\varphi^0_h$ and $\varphi^{\infty}_h$.  It suffices to prove the statement for tangent vectors along the bottom edges of loops.  Since $\mu (\beta_{\ell}) = 2$, we have $\mu (\gamma_k) = 0$ for all $k$.  By Lemma~\ref{Lem:EquivOneLoop}, therefore, if a function $\varphi \in R(D)$ satisfies $s_k (\varphi) = s_k[i]$ and $s'_k (\varphi) = s'_k [i]$, the restriction $\varphi_{|\gamma_k}$ is unique up to tropical scaling.  In particular, it must be one of the functions pictured in Figure~\ref{Fig:FnShapes}, which satisfy $s_{\zeta} (\varphi) = s_{\zeta} [i]$ for all rightward tangent vectors $\zeta$ along the bottom edge of $\gamma_k$.

By definition of $f_{k,h}$, $\varphi_A$ is equivalent to $\varphi^{\infty}_h$ on the portion of $\Gamma$ to the left of $\beta_{\ell}$.  To the right of $\beta_{\ell}$, since the edge lengths of $\Gamma$ are $C$-admissible, for all tangent vectors $\zeta$ along the bottom edge of $\gamma_k$, we have
\[
s_{\zeta} [h] \leq s_{\zeta} (\varphi_A) \leq s_{\zeta} [h+1] .
\]
Since $\varphi_A \in \Sigma$, it must therefore have either slope $s_{\zeta}[h]$ or $s_{\zeta}[h+1]$ on tangent vectors $\zeta$ in $\Gamma_L$ to the right of $x$.  Since there are no other switching loops or bridges, and the restriction of $\varphi_A$ to each bridge or bottom edge is convex, there is a unique point in $\Gamma_L$ to the right of $x$ such that $\varphi_A$ has slope $s_{\zeta}[h+1]$ to the left of this point and slope $s_{\zeta}[h]$ to the right.  Since $\varphi^{\infty}_h$ has slope $s_{\zeta} [h+1]$ on all tangent vectors $\zeta$ in $\Gamma_L$ to the right of $x$, and $\varphi^0_h$ has slope $s_{\zeta} [h]$ on all such tangent vectors, it follows that $(\varphi_A)_{\Gamma_L}$ is a tropical linear combination of these two functions.
\end{proof}

\begin{lemma}
\label{Lem:TLC2}
In the expression of $(\varphi_A)_{\Gamma_L}$ as a tropical linear combination of $\varphi^0_h$ and $\varphi^{\infty}_h$, there is a unique point $y$ to the right of $x$ where the two summands are equal.  Similarly, in the expression of $(\varphi_B)_{\Gamma_L}$ as a tropical linear combination of $\varphi^0_{h+1}$ and $\varphi^{\infty}_h$, there is a unique point $y'$ to the left of $x$ where the two summands are equal.  Moreover, up to tropical scaling there is a unique tropical dependence among $\varphi_A$, $\varphi_B$, and $\varphi_C$ in which the three terms are all equal at $y$ and $y'$.
\end{lemma}

\begin{proof}
At every rightward tangent vector $\zeta$ in $\Gamma_L$ to the right of $x$, we have $s_{\zeta} (\varphi^{\infty}_h) > s_{\zeta} (\varphi^0_h)$.  Thus, in any tropical linear combination of $\varphi^{\infty}_h$ and $\varphi^0_h$, there is a unique point $y \in \Gamma_L$ to the right of $x$ where the two summands are equal.  Similarly, in any tropical linear combination of $\varphi^{\infty}_h$ and $\varphi^0_{h+1}$, there is a unique point $y' \in \Gamma_L$ to the left of $x$ where the two summands are equal.

Since $\varphi_A$, $\varphi_B$, and $\varphi_C$ are contained in the rank 1 tropical linear series $\Sigma'$, they must be tropically dependent.  If we consider the point $y$ at which the function $\varphi_A$ is equivalent to $\varphi^{\infty}_h$ to the left and equivalent to $\varphi^0_h$ to the right, we see that locally in a neighborhood of this point, $\varphi_B$ is equivalent to $\varphi^{\infty}_h$ and $\varphi_C$ is equivalent to $\varphi^0_h$.  Thus, in the tropical dependence between these three functions, all three must achieve the minimum at this point.  This uniquely determines the tropical dependence up to tropical scaling.  By the same reasoning, all three functions must also achieve the minimum at $y'$.  This dependence is illustrated in Figure~\ref{Fig:GammaLDependence}.
\end{proof}

\begin{figure}[H]
\begin{tikzpicture}
\draw (0,-.5)--(6,-.5);
\draw [ball color=white] (3,-.5) circle (0.55mm);
\draw [ball color=black] (1.5,-.5) circle (0.55mm);
\draw [ball color=black] (4.5,-.5) circle (0.55mm);
\draw (0.75,-0.3) node {{\tiny $BC$}};
\draw (3,-0.3) node {{\tiny $AB$}};
\draw (5.25,-0.3) node {{\tiny $AC$}};

\draw [<->] (2.95,-.6) -- (1.55,-.6);
\draw [<->] (3.05,-.6) -- (4.45,-.6);
\draw (2.25,-.8) node {{\tiny $t'$}};
\draw (3.75,-.8) node {{\tiny $t$}};
\draw (1.5,-.8) node {{\tiny $y'$}};
\draw (4.5,-.8) node {{\tiny $y$}};
\end{tikzpicture}
\caption{A schematic depiction of the tropical dependence among $\varphi_A$, $\varphi_B$, and $\varphi_C$ on $\Gamma_L$, analogous to the bottom line in Figure~\ref{Fig:IntervalDependence}}
\label{Fig:GammaLDependence}
\end{figure}

The expressions of $\varphi_A$, $\varphi_B$, and $\varphi_C$ as tropical linear combinations of building blocks on $\Gamma_L$ are also valid on loops that do not contain $y$ or $y'$.  For future reference we record the details.

\begin{lemma}
\label{Lem:DependenceAwayFromGammaL}
The tropical linear combinations above extend to loops that do not contain $y$ or $y'$:
\begin{enumerate}[label=(\roman*), leftmargin=.75cm]
\item  If $y \notin \gamma_k$, then $\varphi_A \sim_{\gamma_k} \varphi^{\infty}_h$ if $\gamma_k$ is to the left of $y$, and $\varphi_A \sim_{\gamma_k} \varphi^0_h$ if $\gamma_k$ is to the right of $y$.
\item  If $y' \notin \gamma_k$, then $\varphi_B \sim_{\gamma_k} \varphi^0_{h+1}$ if $\gamma_k$ is to the left of $y'$ and $\varphi_B \sim_{\gamma_k} \varphi^{\infty}_h$ if $\gamma_k$ is to the right of $y'$.
\item  $\varphi_C \sim_{\gamma_k} \varphi^0_h$ for all $k \leq \ell $ and $\varphi_C \sim_{\gamma_k} \varphi^0_{h+1}$ for all $k > \ell$.
\end{enumerate}
\end{lemma}

\begin{proof}
Recall that $\mu (\gamma_k) = 0$ for all $k$.  Thus, if a function $\varphi \in R(D)$ satisfies $s_k (\varphi) = s_k[i]$ and $s'_k (\varphi) = s'_k [i]$, the restriction $\varphi_{|\gamma_k}$ is unique up to tropical scaling.  The result follows.
\end{proof}

\begin{definition}
Let $t$ be the distance, measured along the bridges and bottom edges, from $x$ to $y$.  Similarly, let $t'$ be the distance, measured along the bridges and bottom edges, from $x$ to $y'$.  
\end{definition}

As in Example~\ref{Ex:Interval}, the tropical dependence among $\{ \varphi_A, \varphi_B, \varphi_C\}$ induces a relation between the parameters $t$ and $t'$.

\begin{proposition}
\label{Prop:BridgeDependence}
The distance $t'$ is an increasing piecewise affine function in $t$.
\end{proposition}

\begin{proof}
By Lemma~\ref{Lem:TLC2}, there is a tropical dependence among $\varphi_A$, $\varphi_B$, and $\varphi_C$ in which the three terms are all equal at $y$ and $y'$.  The condition that all three functions are equal at these two points yields a system of equations, and by solving for $t'$, we obtain an expression for $t'$ as an increasing piecewise affine function in $t$.
\end{proof}

Note that $\varphi_A$ is linear with slope $s'_{\ell - 1}[h] = s_{\ell}[h+1]$ on a subinterval of $\beta_{\ell}$.  This subinterval extends from the left endpoint $w_{\ell - 1}$ of $\beta_{\ell}$ to the point to the right of $x$ of distance $\min \{ t, d(x,v_{\ell}) \}$.

\begin{definition}
Let $I \subset \beta_\ell$ be the interval where $\varphi_A$ has slope $s_{\ell}[h+1]$.
\end{definition}

\begin{corollary}
\label{Cor:BoundOnT}
If $I$ has length less than $m_{\ell-1}$, then $t'=t$.
\end{corollary}

\begin{proof}
If $y$ is not contained in the bridge $\beta_{\ell}$, then $\varphi_A$ has slope $s'_{\ell-1}[h]$ on the entire bridge $\beta_{\ell}$.  The assumption therefore implies that $y$ is contained in the bridge $\beta_{\ell}$.  The point $y'$ is contained either in the bridge $\beta_{\ell}$ or the bottom edge of the loop $\gamma_{\ell-1}$.  We consider the case where $y'$ is contained in the bridge first.  Examining the tropical dependence described in Proposition~\ref{Prop:BridgeDependence}, we see that
\[
\left( s'_{\ell-1}[h+1]- s'_{\ell-1} [h] \right) t' = \left( s_{\ell} [h+1]- s_{\ell} [h] \right) t.
\]
But, by equation~\ref{Eq:BridgeSwitch}, we have $s'_{\ell-1}[h+1]- s'_{\ell-1} [h] = s_{\ell} [h+1]- s_{\ell} [h] = 1$, and the result follows.

We now consider the case where $y'$ is contained in the bottom edge of the loop $\gamma_{\ell-1}$.  Recall that $\mu(\gamma_{\ell-1}) = 0$.  It follows that $\varphi^0_{h+1}$ has slope one greater than $\varphi^0_h$ along this bottom edge.  The result then follows by the same argument as the previous case.
\end{proof}

\begin{definition} \label{Def:BBOnBridge}
Suppose
$
s_{\ell-1}[h] < s'_{\ell-1}[h],
$
and there are functions $\varphi, \varphi' \in \cA$ such that
$\sigma_{\ell} = s_{\ell} [h] + s_{\ell} (\varphi) = s_{\ell} [h] + s_{\ell} (\varphi') +1.$  Then
\begin{enumerate} [label=(\roman*)]
\item  if $I$ has length less than $m_{\ell-1}$, let
$
\cB = 2\cA \smallsetminus \{ \varphi^{\infty}_h + \varphi' \};
$
\item   if $I$ has length at least $m_{\ell-1}$, let
$
\cB = 2\cA \smallsetminus \{ \varphi^0_h + \varphi \} .
$
\end{enumerate}
Otherwise, let
$
\cB = 2\cA .
$
\end{definition}

\begin{lemma}
\label{Lem:BBOnBridge}
This set $\cB$ satisfies properties $(\mathbf B)$ and $(\mathbf B')$.
\end{lemma}

\begin{proof}
The function $\varphi^0_h$ is the only element of $\cA$ satisfying $s_k (\varphi^0_h) < \xi'_{k-1} (\varphi^0_h)$,  and then only when $k = \ell$.  If $\varphi^0_h + \varphi \in \cB$ is permissible on $\gamma_{\ell-1}$ for some $\varphi \in \cA$, then by Definition~\ref{Def:BBOnBridge} we see that $\varphi^{\infty}_h + \varphi$ is also in $\cB$.  But $\varphi^0_h + \varphi$ is equivalent to $\varphi^{\infty}_h + \varphi$ to the left of $\beta_{\ell}$, and the latter function has higher slope along $\beta_{\ell}$, so $\cB$ satisfies property $(\mathbf B)$.

To establish property $(\mathbf B')$, we must show that if two permissible functions in $\cB$ agree on $\gamma_{\ell-1}$ and have different slopes on $\beta_{\ell}$, then no function is shiny on $\gamma_{\ell-1}$.  We first show that every shiny function on $\gamma_{\ell-1}$ is in fact new.  Since both $\beta_{\ell-1}$ and $\gamma_{\ell-1}$ have multiplicity zero, the restriction of $D + \ddiv (\varphi)$ to $\gamma_{\ell-1} \smallsetminus \{ w_{\ell-1} \}$ has degree at most 1 for every function $\varphi \in \cA$.  Moreover, none of these divisors contain $v_{\ell-1}$ in their support, which implies that every shiny function on $\gamma_{\ell-1}$ is new.

Now, if two functions in $\cB$ are permissible on $\gamma_{\ell-1}$ and have different slopes on $\beta_{\ell}$, then the one with higher slope must be departing.  In addition, if the two functions are equivalent on $\gamma_{\ell-1}$, then they must be $\varphi^0_h + \varphi$ and $\varphi^{\infty}_h + \varphi$ for some $\varphi \in \cA$.  We may therefore assume that $\varphi^{\infty}_h + \varphi$ is departing on $\gamma_{\ell-1}$ and
\begin{equation} \label{eq:SlopeBeforeBridge}
\sigma_{\ell} = s_{\ell} [h] + s_{\ell} (\varphi) .
\end{equation}
Since $\varphi^{\infty}_h + \varphi$ is departing, either (i) $s_{\ell-1} (\varphi^{\infty}_h) < s_{\ell} (\varphi^{\infty}_h)$,  or (ii) $s_{\ell-1} (\varphi) < s_{\ell} (\varphi)$ and $\varphi \neq \varphi^\infty_h$.
If $s_{\ell-1} (\varphi^{\infty}_h) < s_{\ell} (\varphi^{\infty}_h)$, then any new function must be of the form $\varphi^{\infty}_h + \varphi'$, where $\sigma_{\ell} = s_{\ell} [h] + s_{\ell} (\varphi') + 1$.
Definition~\ref{Def:BBOnBridge} ensures that either $\varphi^{\infty}_h + \varphi'$ is not in $\cB$, in which case no function is shiny on $\gamma_{\ell-1}$, or $\varphi^0_h + \varphi$ is not in $\cB$, in which case no two permissible functions are equivalent on $\gamma_{\ell-1}$.

It remains to consider the case where $s_{\ell-1} (\varphi) < s_{\ell} (\varphi)$, and $\varphi \neq \varphi^\infty_h$. Then any new function must be of the form $\varphi + \varphi'$, where $\sigma_{\ell} = s_{\ell} (\varphi) + s_{\ell} (\varphi')$.
Combining this with \eqref{eq:SlopeBeforeBridge}, we see that $s_{\ell} (\varphi') = s_{\ell} [h]$.  The only function in $\cA$ with this slope is $\varphi^0_h$, so $\varphi' = \varphi^0_h$.  Since $\varphi \neq \varphi^\infty_h$, we have $s_{\ell-1} (\varphi^{\infty}_h) = s_{\ell} (\varphi^{\infty}_h)$.  It follows that $s_{\ell-1} (\varphi^0_h) > s_{\ell} (\varphi^0_h)$. Hence the function $\varphi + \varphi^0_h$ is not new, and no function is shiny on $\gamma_{\ell -1}$.
\end{proof}

Since the set $\cB$ satisfies properties $(\mathbf B)$, and $(\mathbf B')$, the output of the template algorithm satisfies \ref{T1}-\ref{T5}.  The next step in our argument is to describe the set $\cT \subset 2\cS$  from which we will construct a certificate of independence.  If $i, j \not \in \{ h, h+1\}$, then $\varphi_{i,j} \in \cT$.  The remaining functions in $\cT$ will be chosen depending on where the best approximation of $\theta$ by $\varphi_C + \varphi_j$ achieves equality, as follows.  For $j \not \in \{h , h+1 \}$, we denote $\varphi^0_{h,j} = \varphi^0_h + \varphi_j$, and similarly for $\varphi^0_{h+1,j}$ and $\varphi^\infty_{h,j}$.

\begin{lemma} \label{Lem:TwoRegions}
The best approximation of $\theta$ by $\varphi_C + \varphi_j$ from above achieves equality on the region where either $\varphi^0_{h,j} + c^0_{hj}$ or $\varphi^0_{h+1,j} + c^0_{h+1,j}$ achieves the minimum.
\end{lemma}

\begin{proof}
If $\cB$ contains both $\varphi^0_{h,j}$ and $\varphi^0_{h+1,j}$ then this follows from Lemmas~\ref{Lem:Replace}, \ref{lem:TLC}(iii) and \ref{Lem:DependenceAwayFromGammaL}(iii).  Otherwise, we are in the subcase of Definition~\ref{Def:BBOnBridge}(ii) where $\varphi = \varphi_j$. Then Lemma~\ref{Lem:Replace} does not apply, since $\varphi_C + \varphi_j$ is not a tropical combination of functions in $\cB$.  In this case, $\varphi_C + \varphi_j$ has slope greater than $\sigma_{\ell}$ on $\beta_{\ell}$, and so the best approximation cannot achieve equality to the right of $\gamma_{\ell-1}$.  Hence it must achieve equality on or to the left of $\gamma_{\ell-1}$, where $\varphi_C + \varphi_j$ is equivalent to $\varphi^0_{h+1,j}$.
\end{proof}

\noindent We note that, a priori, it is possible for this best approximation to achieve equality  on \emph{both} regions.  However, in our construction of the master template $\theta$, if we perturb the coefficients of all functions in $\cB$ that are assigned to the same loop or bridge by a sufficiently small value, this does not change the conclusion of Theorem~\ref{Thm:ExtremalConfig}.  We may therefore assume that it achieves equality on exactly one of these two regions.  If the best approximation of $\theta$ by $\varphi_C + \varphi_j$ achieves equality where $\varphi^0_{h,j} + c(\varphi^0_{h,j})$ achieves the minimum, then we put $ \varphi_B + \varphi_j, \varphi_C + \varphi_j$ in $\cT$.  Otherwise, it achieves equality where $\varphi^0_{h+1,j} + c(\varphi^0_{h+1,j})$ achieves the minimum, and we put $\varphi_A + \varphi_j, \varphi_C + \varphi_j$ in $\cT$.

Similarly, we add to $\cT$ three pairwise sums of elements of $\{ \varphi_A , \varphi_B, \varphi_C \}$.  In all cases, we put $\varphi_C + \varphi_C$ in $\cT$.  If the best approximation of $\theta$ by $\varphi_C + \varphi_C$ achieves equality on a region to the left of $x$, then we put $\varphi_A + \varphi_C$ in $\cT$.  Otherwise, we put $\varphi_B + \varphi_C$ in $\cT$.  If $\varphi_A + \varphi_C \in \cT$ and the best approximation of $\theta$ by $\varphi_A + \varphi_C$ achieves equality on a region to the left of $x$, then we put $\varphi_A + \varphi_A$ in $\cT$.  If it achieves equality on a region to the right of $x$, then we put $\varphi_A + \varphi_B$ in $\cT$.  Similarly, if  $\varphi_B + \varphi_C \in \cT$ and the best approximation of $\theta$ by $\varphi_B + \varphi_C$ achieves equality on a region to the left of $x$, then we put $\varphi_A + \varphi_B$ in $\cT$.  If it achieves equality on a region to the right of $x$, then we put $\varphi_B + \varphi_B$ in $\cT$.  These choices are made so that the three chosen pairwise sums of elements of $\{ \varphi_A , \varphi_B, \varphi_C \}$ are not equivalent on the regions where they achieve the minimum.  By an argument similar to that of Lemma~\ref{Lem:TwoRegions}, in the best approximation of $\theta$ by $\cT$, each of these functions will achieve equality on the region where one of the building blocks in $\cB$ achieves the minimum in $\theta$.

\begin{proof}[Proof of Theorem~\ref{thm:independence}, case 2c]
We define the replacement function $R \colon \cT \to \cB$ as follows.  Let $R(\varphi_{i,j}) = \varphi_{i,j}$, for $i,j \neq h$.  If the best approximation of $\theta$ by $\varphi_C + \varphi_j$ achieves equality on the region where $\varphi^0_{h,j} + c^0_{hj}$ achieves the minimum, we define $R(\varphi_C + \varphi_j) = \varphi^0_{h,j}$.  Otherwise, let $R(\varphi_C + \varphi_j) = \varphi^0_{h+1,j}$.  Similarly, if $\varphi_B + \varphi_j \in \cT$ and the best approximation of $\theta$ by $\varphi_B + \varphi_j$ achieves equality where $\varphi^0_{h+1,j} + c_{h+1,j}$ achieves the minimum, let $R(\varphi_B + \varphi_j) = \varphi^0_{h+1,j}$.  Otherwise, let $R(\varphi_B + \varphi_j) = \varphi^{\infty}_{h,j}$.  If $\varphi_A + \varphi_j \in \cT$ and the best approximation of $\theta$ by $\varphi_A + \varphi_j$ achieves equality where $\varphi^0_{h,j} + c_{h,j}$ achieves the minimum, let $R(\varphi_A + \varphi_j) = \varphi^0_{h,j}$.  Otherwise, let $R(\varphi_A + \varphi_j) = \varphi^{\infty}_{h,j}$.  We now show that the best approximation $\Upsilon$ is a certificate of independence.

We first consider the case where the set $\cB = 2\cA$.  If the point $y$ is contained in the loop $\gamma_k$, then either $s_{k-1} (\varphi_A + \varphi_j) > \sigma_k$ or $s_k (\varphi_A + \varphi_j) < \sigma_k$.  It follows that the best approximtion of $\theta$ by $\varphi_A + \varphi_j$ does not achieve equality at any point of $\gamma_k$.  Similarly, if $y'$ is contained in the loop $\gamma_k$, then the best approximtion of $\theta$ by $\varphi_B + \varphi_j$ does not achieve equality at any point of $\gamma_k$.  Hence, by Lemmas~\ref{Lem:Replace}, \ref{lem:TLC} and~\ref{Lem:DependenceAwayFromGammaL}, each of the 28 functions in $\cT$ achieves the minimum on a region where one of the functions in $\cB$ achieves the minimum in the template $\theta$.  We show that each function achieves the minimun \emph{uniquely} at some point of $\Gamma$.  By \ref{T4} and \ref{T5}, if two functions $\psi, \psi' \in \cB$ are assigned to the same loop $\gamma_{k-1}$ or bridge $\beta_k$, then $\psi = \varphi^{\infty}_h + \varphi$ for some $\varphi \in \cA$, and either
\[
\psi' = \varphi^0_h + \varphi, k \leq \ell, \mbox{ or } \psi' = \varphi^0_{h+1} + \varphi, k > \ell .
\]
Assume for simplicity that $\varphi = \varphi_j$ for some $j$.  The other cases are similar.

By construction, the best approximation of $\theta$ by $\varphi_C + \varphi_j$ achieves equality on the region where $R(\varphi_C + \varphi_j) + c(R(\varphi_C) + \varphi_j)$ achieves the minimum in $\theta$.  Suppose it does so on the region where $\varphi^0_{h,j} + c(\varphi^0_{h,j})$ achieves the minimum.  (The other case is similar.) In this case, by construction, the set $\cT$ does not contain $\varphi_A + \varphi_j$.  Since $\varphi_C + \varphi_j$ is not equivalent to any other pairwise sum of functions in $\cS$ on $\alpha (\varphi^0_{h,j})$, it must achieve the minimum uniquely.  A similar argument shows that $\varphi_B + \varphi_j$ achieves the minimum uniquely on $\alpha (R(\varphi_B + \varphi_j))$.  This completes the proof that every function in $\cT$ achieves the minimum uniquely, and hence $\Upsilon$ is a certificate of independence, when $\cB = 2\cA$.

We now turn to the cases where $\cB$ is strictly contained in $2\cA$.  In these cases it suffices to show that the best approximation of $\theta$ by $\psi \in \cT$ achieves equality on $\alpha (R(\psi))$ for each $\psi \in \cT$.  Fix functions $\varphi$ and $\varphi'$ as in Definition~\ref{Def:BBOnBridge}.  Suppose that $I$ has length greater than or equal to $m_{\ell-1}$.  In this case Lemma~\ref{Lem:Replace} does not apply, since the functions $\varphi_A + \varphi$ and $\varphi_C + \varphi$ are not tropical linear combinations of functions in $\cB$.   By Lemma~\ref{Lem:TwoRegions}, however, the best approximation of $\theta$ by $\varphi_C + \varphi_j$ achieves equality on the region where $R(\varphi_C + \varphi_j) = \varphi^0_{h+1} + \varphi_j$ achieves the minimum.  By an identical argument, the best approximation of $\theta$ by $\varphi_A + \varphi_j$ achieves equality where $R(\varphi_A + \varphi_j) = \varphi^{\infty}_h + \varphi_j$ achieves the minimum.

Now, suppose that $I$ has length less than $m_{\ell-1}$, so $\cB = 2\cA \smallsetminus \{\varphi^{\infty}_h + \varphi'\}$.  We will consider the case where $\varphi_A + \varphi' \in \cT$; the case where $\varphi_B + \varphi' \in \cT$ is similar.  Note that Lemma~\ref{Lem:Replace} does not apply, since the function $\varphi_A + \varphi'$ is not a tropical linear combination of functions in $\cB$.  The assumption that $I$ has length less than $m_{\ell-1}$ implies that $\varphi_A + \varphi'$ has smaller slope than $\theta$ on a large subinterval of $\beta_{\ell}$, and slope smaller than or equal to that of $\theta$ on every bridge to the left of $\beta_{\ell}$.  Thus, in the best approximation, $\varphi_A + \varphi'$ must obtain the minimum to the right of $\beta_{\ell}$.  The assumption on the length of of $I$ also implies that $\varphi_A + \varphi'$ is equivalent to $\varphi^0_h + \varphi'$ to the right of $\beta_{\ell}$, hence $\varphi_A + \varphi'$ achieves the minimum on $\alpha(R(\varphi_A + \varphi')) = \alpha(\varphi^0_h + \varphi')$.
\end{proof}

\subsection{Case 3: one switching loop}
\label{Sec:Switch}
We now consider the case where there is only one switching loop $\gamma_{\ell}$, which switches slope $h$. By Lemma~\ref{Lem:GenericFns}, for all $j \notin \{ h,h+1\}$, there is $\varphi_j \in \Sigma$ with
\[
s_k (\varphi_j) = s_k [j] \mbox{ and } s'_k (\varphi_j) = s'_k [j] \mbox{ for all } k.
\]

Once again, we work with a set $\cS \subset\Sigma$ consisting of the functions $\varphi_j$ for $j \notin \{ h,h+1 \}$, plus three more functions that are contained in a tropical linear subseries of rank 1.

\begin{proposition}
\label{Prop:LoopFns}
There is a rank 1 tropical linear subseries $\Sigma' \subset \Sigma$ and functions $\varphi_A , \varphi_B,$ and $\varphi_C \in \Sigma'$ with the following properties:
\begin{enumerate}
\item  $s'_k (\varphi_A) = s'_k [h]$ for all $k < \ell$;
\item  $s_k (\varphi_B) = s_k [h+1]$ for all $k > \ell$;
\item  $s_k (\varphi_C) = s_k [h+1]$ for all $k\leq\ell$, and $s'_k (\varphi_C) = s_k [h]$ for all $k \geq \ell$;
\item   $s_k (\varphi_\bullet) \in \{ s_k [h], s_k [h+1]\}$ and $s'_k (\varphi_\bullet) \in \{ s'_k [h], s'_k [h+1]\}$ for all $k$.
\end{enumerate}
\end{proposition}

\begin{proof}
The argument is identical to the proof of Proposition~\ref{Prop:BridgeFns}.
\end{proof}

Let $\cS = \{ \varphi_j \mid j \neq h, h+1 \} \cup \{ \varphi_A, \varphi_B, \varphi_C \}$.  For $i \neq h,h+1$, let $f_{k,i} = (\varphi_i)_{|\gamma_k}$.  For $k \leq \ell$, let $f_{k,h} = (\varphi_A)_{|\gamma_k}$ and $f_{k,h+1} = (\varphi_C)_{|\gamma_k}$.  For $k > \ell$, let $f_{k,h} = (\varphi_C)_{|\gamma_k}$ and $f_{k,h+1} = (\varphi_B)_{|\gamma_k}$.  We let $\cA$ be the set of all $(\Sigma, \{ f_{k,i} \})$-building blocks.

\subsubsection{Case 3a:  there are no bridges of multiplicity one}

As in the previous case, the functions $\varphi_A, \varphi_B$, and $\varphi_C$ can be written as tropical linear combinations of simpler functions in $R(D)$.  We have the following analogue of Lemmas~\ref{Lem:BuildingBlocksBridge} and \ref{lem:TLC}.

\begin{lemma}
\label{lem:LoopTLC}
There are building blocks $\varphi^0_h, \varphi^0_{h+1},$ and $\varphi^{\infty}_h \in R(D)$ with the following properties:
\begin{enumerate} 
\item  $s_k (\varphi^0_h) = s_k [h]$ and $s'_k (\varphi^0_h) = s'_k[h]$ for all $k$;
\item  $s_k (\varphi^0_{h+1}) = s_k [h+1]$ and $s'_k (\varphi^0_{h+1}) = s'_k [h+1]$ for all $k$;
\item  $s_k (\varphi^{\infty}_h) = s_k [h]$, $s'_{k-1} (\varphi^{\infty}_h) = s'_{k-1} [h]$ for all $k\leq\ell$, and $s_k (\varphi^{\infty}_h) = s_k [h+1]$, $s'_{k-1} (\varphi^{\infty}_h) = s'_{k-1} [h+1]$ for all $k > \ell$.
\end{enumerate}
\end{lemma}

\begin{proof}
The construction of these three functions is identical to that of Lemma~\ref{Lem:BuildingBlocksBridge}.  In particular, the function $\varphi^0_h$ is the unique building block with constant building sequence $h$ and the function $\varphi^0_{h+1}$ is the unique building block with constant building sequence $h+1$.  The function $\varphi^{\infty}_h$ is the unique building block with building sequence
\[
\tau_k = \tau'_k = \begin{cases} h & \mbox{if $k < \ell$} \\  h+1 & \mbox{if $k \geq \ell$.} \end{cases}
\]
\vskip -22pt
\end{proof}

\begin{lemma}
The restrictions of the functions $\varphi_A$, $\varphi_B$, and $\varphi_C$ to $\Gamma_L$ can be written as tropical linear combinations of the restrictions of the building blocks $\varphi^0_h, \varphi^0_{h+1},$ and $\varphi^{\infty}_h$, as follows:
\begin{enumerate} [label=(\roman*), leftmargin = .75cm]
\item  The function $(\varphi_A)_{|\Gamma_L}$ is uniquely expressible as a tropical linear combination of $\varphi^0_h$ and $\varphi^{\infty}_h$;
\item  The function $(\varphi_B)_{|\Gamma_L}$ is uniquely expressible as a tropical linear combination of $\varphi^0_{h+1}$ and $\varphi^{\infty}_h$;
\item  The function $(\varphi_C)_{|\Gamma_L}$ is uniquely expressible as a tropical linear combination of $\varphi^0_h$ and $\varphi^0_{h+1}$.
\end{enumerate}
\end{lemma}

\begin{proof}
The proof differs from that of Lemma~\ref{lem:TLC} only on loops of positive multiplicity.  If $\mu (\gamma_k) = 1$ and $\gamma_k$ is not a switching loop, then the conclusion remains the same.  Specifically, suppose that $\varphi \in R(D)$ is a building block satisfying $s_k (\varphi) = s_k [i]$ and $s'_k (\varphi) = s'_k[i]$.  There is at most one value of $i$ for which $\varphi_{|\gamma_k}$ is not one of the functions pictured in Figure~\ref{Fig:FnShapes}.  If such an $i$ exists, we have $s'_k[i] = s_k[i]-1$.  The slopes of $\varphi$ along the bottom edge of $\gamma_k$ are bounded between $s_k[i]-1$ and $s_k[i]$.  Since $\gamma_k$ is not a switching loop, for $j \neq i$ we have $s_k[i] \neq s_k[j]$ and $s_k[i]-1 \neq s_k[j]$.  Since $\varphi$ is defined to be equivalent to a function in $\Sigma$ on $\gamma_k$, it follows that $s_{\zeta} (\varphi) = s_{\zeta}[i]$ for all rightward tangent vectors $\zeta$ along the bottom edge.

It remains to consider the switching loop $\gamma_{\ell}$.  By the classification of switching loops in Section~\ref{sec:switchingclassification}, however, we see that each of $\delta_{\ell} (\varphi_A)$, $\delta_{\ell} (\varphi_B)$, and $\delta_{\ell} (\varphi_C)$ are at most 1.  It follows from Lemma~\ref{Lem:EquivOneLoop}, therefore, that $\varphi_A$, $\varphi_B$, and $\varphi_C$ are equivalent on $\gamma_{\ell}$.
\end{proof}

Now, we have
\[
\cA = \{ \varphi_i \, \mid \, i \neq h,h+1 \} \cup \{ \varphi^0_h , \varphi^0_{h+1} , \varphi^{\infty}_h \} ,
\]
and the argument is identical to Case 2c.  Specifically, we let $I \subseteq \beta_{\ell+1}$ be the interval where $\varphi_A$ has slope $s_{\ell+1}[h+1]$.  As in Definition~\ref{Def:BBOnBridge}, suppose $s_{\ell}[h] < s'_{\ell}[h]$, and there are functions $\varphi, \varphi' \in \cA$ such that $\sigma_{\ell+1} = s_{\ell+1} [h] + s_{\ell+1} (\varphi) = s_{\ell+1} [h] + s_{\ell+1} (\varphi') +1.$  Then
\begin{enumerate}
\item  if $I$ has length less than $m_{\ell}$, let $\cB = 2\cA \smallsetminus \{ \varphi^{\infty}_h + \varphi' \}$;
\item   if $I$ has length at least $m_{\ell}$, let $\cB = 2\cA \smallsetminus \{ \varphi^0_h + \varphi \}$.
\end{enumerate}
Otherwise, let $\cB = 2\cA$.  The set $\cB$ satisfies properties $(\mathbf B)$ and $(\mathbf B')$, exactly as in Lemma~\ref{Lem:BBOnBridge}.

If $i, j \not \in \{ h, h+1\}$, then $\varphi_{ij} \in \cT$.  As in Lemma~\ref{Lem:TwoRegions}, the best approximation of $\theta$ by $\varphi_C + \varphi_j$ achieves equality on the region where either $\varphi^0_{h,j} + c^0_{hj}$ or $\varphi^0_{h+1,j} + c^0_{h+1,j}$ achieves the minimum.  If the best approximation by $\varphi_C + \varphi_j$ achieves equality where $\varphi^0_{h,j}$ achieves the minimum, then we put $ \varphi_B + \varphi_j, \varphi_C + \varphi_j$ in $\cT$.  Otherwise, it achieves equality where $\varphi^0_{h+1,j}$ achieves the minimum, and we put $\varphi_A + \varphi_j, \varphi_C + \varphi_j$ in $\cT$.  Similarly, we add to $\cT$ three pairwise sums of elements of $\{ \varphi_A , \varphi_B, \varphi_C \}$, exactly as in Case 2c.  We define the replacement function $R$ exactly as in Case 2c, and the proof that the best approximation $\Upsilon$ is a certificate of independence is the same.

\subsubsection{Case 3b:  there is one bridge of multiplicity one}

Let $\beta_{\ell'}$ be a bridge with multiplicity 1.  In this case, we combine the construction of Case 3a with that from Case 1b.  Define functions $\varphi^0_h$, $\varphi^0_{h+1}$, and $\varphi^{\infty}_h$, as in Lemma~\ref{lem:LoopTLC} such that, for all rightward tangent vectors $\zeta$ in $\beta_{\ell'}$,
\begin{itemize}
\item  $s_{\zeta} (\varphi^0_h) = s_{\zeta} [h]$;
\item  $s_{\zeta} (\varphi^0_{h+1}) = s_{\zeta} [h+1]$;
\item  $s_{\zeta} (\varphi^{\infty}_h) = s_{\zeta} [h]$ if $\ell' < \ell$, and $s_{\zeta} (\varphi^{\infty}_h) = s_{\zeta} [h+1]$ if $\ell' \geq \ell$.
\end{itemize}
Let
\[
\cA' = \{ \varphi_i \, \mid \, i \neq h,h+1 \} \cup \{ \varphi^0_h , \varphi^0_{h+1} , \varphi^{\infty}_h \} .
\]
Note that the functions in $\cA'$ may not be $(\Sigma, \{ f_{k,i} \})$-building blocks because they do not have constant slope along $\beta_{\ell'}$.  However, each function in $\cA'$ is a tropical linear combination of functions in $\cA$.  In particular, for $i \notin \{ h, h+1 \}$, $\varphi_i$ is a tropical linear combination of building blocks with constant building sequence $i$.  The function $\varphi^0_h$ is a a tropical linear combination of building blocks with constant buidling sequence $h$, the function $\varphi^0_{h+1}$ is a tropical linear combination of building blocks with constant building sequence $h+1$, and the function $\varphi^{\infty}_h$ is a tropical linear combination of building blocks with building sequence
\[
\tau_k = \tau'_{k-1} = \begin{cases} h & \mbox{if $k \leq \ell$} \\  h+1 & \mbox{if $k > \ell$.} \end{cases}
\]
Let $\cB' \subseteq 2\cA'$ be the subset denoted $\cB$ in Case 3a.  Finally, let $\cB \subseteq 2\cA$ be the set of sums $\varphi_{\tau_1,s_1} + \varphi_{\tau_2,s_2}$ with the property that there exists $\varphi_1 + \varphi_2 \in \cB'$ such that $\varphi_i$ is a tropical linear combination of building blocks with building sequence $\tau_i$.

Combining Lemma~\ref{Lem:BBOnBridge} with the proof of property $(\mathbf B)$ in Case 1b, we see that $\cB$ satisfies properties $(\mathbf B)$ and $(\mathbf B')$.  By Theorem~\ref{Thm:ExtremalConfig}, there exists a template $\theta$ and an assignment function $\alpha$ satisfying \ref{T1}-\ref{T5}.
 We define $\cT$ exactly as in Case 3a.  We let $\theta'$ be the best approximation of $\theta$ by $\cB'$, and then let $\Upsilon$ be the best approximation of $\theta'$ by $\cT$.  As in Case 3a, $\theta'$ is a certificate of independence, in which each function $\psi \in \cB'$ obtains the minimum uniquely on $\alpha (R(\psi))$.  Then, by Lemma~\ref{Lem:EasyReplace} we see that $\Upsilon$ is a certificate of independence as well.  \qed

\subsection{Case 4: two switching loops}
We now consider the case where there are two switching loops, $\gamma_{\ell}$ and $\gamma_{\ell'}$, with $\ell < \ell'$.  We write $h$ and $h'$ for the slopes that are switched by $\gamma_{\ell}$ and $\gamma_{\ell'}$, respectively.  Note that both loops must have multiplicity 1.  By our classification of switching loops in \S\ref{sec:switch}, we have
\[
s'_{\ell} [i] = s_{\ell} [i] \mbox{ and } s'_{\ell'} [i] = s_{\ell'} [i] \text{ for all } i.
\]
Moreover,
\[
s_{\ell} [h+1] = s_{\ell} [h]+1 \mbox{ and } s_{\ell'} [h'+1] = s_{\ell'} [h']+1.
\]
Since $\rho=2$ and we have two loops with positive multiplicity, by Proposition~\ref{Thm:BNThm} there are no decreasing loops or bridges.  By Lemma~\ref{Lem:EquivOneLoop}, up to an additive constant, the functions $f_{k,i}$ are uniquely determined for all $k$ and $i$ by \eqref{eq:fki} and \eqref{eq:eff}.

We break our analysis into several subcases, depending on the relationship between $h$ and $h'$.  By Lemma~\ref{Lem:GenericFns}, for all $j \notin \{ h,h+1,h',h'+1\}$, there is a function $\varphi_j \in \Sigma$ with
\[
s_k (\varphi_j) = s_k [j] \mbox{ and } s'_k (\varphi_j) = s'_k [j] \mbox{ for all } k.
\]

\medskip

\subsubsection{Case 4a:  $h' \notin \{ h-1, h, h+1\}$}

This is the simplest subcase because, roughly speaking, the two switching loops do not interact with one another.  More precisely, there are functions $\varphi_A, \varphi_B$, and $\varphi_C$ in $\Sigma$ with slopes as defined in Proposition~\ref{Prop:LoopFns}, and similarly, replacing $\ell$ with $\ell'$ and $h$ with $h'$, there are analogous functions $\varphi'_A, \varphi'_B$, and $\varphi'_C$ in $\Sigma$.  We may then have $(\Sigma, \{ f_{k,i} \})$-building blocks $\varphi^0_{h}, \varphi^0_{h+1}, \varphi^{\infty}_{h}, \varphi^0_{h'}, \varphi^0_{h'+1}$, and $\varphi^{\infty}_{h'}$ as in Case 3a, and set
\[
\cA = \{ \varphi_i \, : \, i \neq h,h+1,h',h'+1 \} \cup \{ \varphi^0_h , \varphi^0_{h+1} , \varphi^{\infty}_h , \varphi^0_{h'} , \varphi^0_{h'+1} , \varphi^{\infty}_{h'} \}.
\]
Our construction of the set $\cB$ and the independence $\Upsilon$ now follow the exact same steps as in Case 3a, treating each switching loop separately.

\medskip

\subsubsection{Case 4b:  $h' = h$}

We first identify a subset $\cS \subset \Sigma$.  It will consist of the functions $\varphi_i$ for $i \not \in \{ h,h+1 \}$, together with a subset of the functions illustrated in Figure~\ref{Fig:Case2}.

\begin{figure}[H]
\begin{tikzpicture}

\draw (-0.2,4) node {{\tiny $\varphi_A$}};
\draw (0,4)--(6,4);
\draw [ball color=white] (2,4) circle (0.55mm);
\draw [ball color=white] (4,4) circle (0.55mm);
\draw [ball color=black] (3,4) circle (0.55mm);
\draw [ball color=black] (5,4) circle (0.55mm);
\draw (1,4.2) node {{\tiny $h$}};
\draw (2.5,4.2) node {{\tiny $h+1$}};
\draw (3.5,4.2) node {{\tiny $h$}};
\draw (4.5,4.2) node {{\tiny $h+1$}};
\draw (5.5,4.2) node {{\tiny $h$}};

\draw (-0.2,3) node {{\tiny $\varphi_B$}};
\draw (0,3)--(6,3);
\draw [ball color=white] (2,3) circle (0.55mm);
\draw [ball color=white] (4,3) circle (0.55mm);
\draw [ball color=black] (1,3) circle (0.55mm);
\draw [ball color=black] (3,3) circle (0.55mm);
\draw (0.5,3.2) node {{\tiny $h+1$}};
\draw (1.5,3.2) node {{\tiny $h$}};
\draw (2.5,3.2) node {{\tiny $h+1$}};
\draw (3.5,3.2) node {{\tiny $h$}};
\draw (5,3.2) node {{\tiny $h+1$}};

\draw (-0.2,2) node {{\tiny $\varphi_C$}};
\draw (0,2)--(6,2);
\draw [ball color=white] (2,2) circle (0.55mm);
\draw [ball color=white] (4,2) circle (0.55mm);
\draw [ball color=black] (5,2) circle (0.55mm);
\draw (1,2.2) node {{\tiny $h+1$}};
\draw (3,2.2) node {{\tiny $h$}};
\draw (4.5,2.2) node {{\tiny $h+1$}};
\draw (5.5,2.2) node {{\tiny $h$}};

\draw (-0.2,1) node {{\tiny $\varphi_D$}};
\draw (0,1)--(6,1);
\draw [ball color=white] (2,1) circle (0.55mm);
\draw [ball color=white] (4,1) circle (0.55mm);
\draw [ball color=black] (1,1) circle (0.55mm);
\draw (0.5,1.2) node {{\tiny $h+1$}};
\draw (1.5,1.2) node {{\tiny $h$}};
\draw (3,1.2) node {{\tiny $h+1$}};
\draw (5,1.2) node {{\tiny $h$}};

\draw (-0.2,0) node {{\tiny $\varphi_E$}};
\draw (0,0)--(6,0);
\draw [ball color=white] (2,0) circle (0.55mm);
\draw [ball color=white] (4,0) circle (0.55mm);
\draw [ball color=black] (3,0) circle (0.55mm);
\draw (1,0.2) node {{\tiny $h+1$}};
\draw (2.5,0.2) node {{\tiny $h+1$}};
\draw (3.5,0.2) node {{\tiny $h$}};
\draw (5,0.2) node {{\tiny $h$}};

\draw (2,-0.25) node {{\tiny $\gamma_{\ell}$}};
\draw (4,-0.25) node {{\tiny $\gamma_{\ell'}$}};

\end{tikzpicture}
\caption{A schematic depiction of the five functions of Proposition~\ref{Prop:Case2Fns}.}
\label{Fig:Case2}
\end{figure}

\begin{proposition}
\label{Prop:Case2Fns}
There is a rank $1$ tropical linear subseries $\Sigma' \subset \Sigma$ containing functions $\varphi_A , \varphi_B , \varphi_C , \varphi_D , \varphi_E$ with the following properties:
\begin{enumerate} 
\item  $s'_k (\varphi_A) = s'_k [h]$ for all $k <\ell$;
\item  $s_k (\varphi_B) = s_k [h+1]$ for all $k > \ell'$;
\item  $s_k (\varphi_C) = s_k [h+1]$ for all $k \leq \ell$ and $s'_k (\varphi_C) = s'_k [h]$ for all $\ell \leq k \leq \ell'$;
\item  $s_k (\varphi_D) = s_k [h+1]$ for all $\ell < k \leq \ell'$ and $s'_k (\varphi_D) = s'_k [h]$ for all all $k \geq \ell'$;
\item  $s_k (\varphi_E) = s_k [h+1]$ for all $k\leq\ell$ and $s'_k (\varphi_E) = s'_k [h]$ for all $k \geq \ell'$;
\item $s_k (\varphi_\bullet) \in \{ s_k [h], s_k [h+1]\}$ and $s'_k (\varphi_\bullet) \in \{ s'_k [h], s'_k [h+1]\}$, for all $k$.
\end{enumerate}

\end{proposition}

\begin{proof}
Applying Lemma~\ref{Lemma:InductiveExistence} twice, we see that there is a rank $1$ tropical linear subseries
\[
\Sigma' \subset \{ \varphi \in \Sigma \mid s'_0(\varphi) \leq s'_0[h+1] \mbox{ and } s_{g+1}(\varphi) \geq s_{g+1}[h] \} .
\]
  If $\varphi \in \Sigma'$ then, for each $k$, $s_k (\varphi)$ is equal to either $s_k[h]$ or $s_k [h+1]$, exactly as in Proposition~\ref{Prop:BridgeFns}.

Choose $\varphi_A$ and $\varphi_B$ as in Proposition~\ref{Prop:BridgeFns}.  Next, choose  $\varphi \in \Sigma'$ such that $s_{\ell} (\varphi) = s_{\ell} [h+1]$, and $\varphi' \in \Sigma'$ such that $s'_{\ell} (\varphi') = s'_{\ell} [h]$.  By adding a scalar to $\varphi'$, we may assume that $\varphi$ and $\varphi'$ are equal on $\gamma_{\ell}$. Set $\varphi_C = \min \{ \varphi, \varphi' \}$.  The constructions of $\varphi_D$ and $\varphi_E$ are similar to that of $\varphi_C$.
\end{proof}

We now characterize two more functions in $R(D)$, depicted schematically in Figure~\ref{Fig:TwoOptions}.

\begin{lemma}
\label{lem:TwoOptions}
There are functions $\psi, \psi' \in R(D)$, unique up to additive constants, with the following properties:
\begin{enumerate}
\item  $s_k (\psi) = s_k [h+1]$ for all $k \leq \ell'$, and $s'_k (\psi) = s'_k [h]$ for all $k \geq \ell'$;
\item  $s_k (\psi') = s_k [h+1]$ for all $k \leq \ell$, and $s'_k (\psi') = s'_k [h]$ for all $k \geq \ell$;
\item  $\mathrm{supp} (D + \ddiv(\psi))$ contains $v_{\ell'}$ and $w_{\ell'}$;
\item $\mathrm{supp} (D + \ddiv(\psi'))$ contains $v_{\ell}$ and $w_{\ell}$.
\end{enumerate}
\end{lemma}

\begin{figure}[H]
\begin{tikzpicture}
\tikzmath{ \y = .7;}

\draw (0,4*\y)--(6,4*\y);
\draw (-0.5,4*\y) node {{\small $\psi$}};
\draw [ball color=white] (2,4*\y) circle (0.55mm);
\draw [ball color=white] (4,4*\y) circle (0.55mm);
\draw (1,4.2*\y) node {{\tiny $h+1$}};
\draw (3,4.2*\y) node {{\tiny $h+1$}};
\draw (5,4.2*\y) node {{\tiny $h$}};

\draw (0,3*\y)--(6,3*\y);
\draw (-0.5,3*\y) node {{\small $\psi'$}};
\draw [ball color=white] (2,3*\y) circle (0.55mm);
\draw [ball color=white] (4,3*\y) circle (0.55mm);
\draw (1,3.2*\y) node {{\tiny $h+1$}};
\draw (3,3.2*\y) node {{\tiny $h$}};
\draw (5,3.2*\y) node {{\tiny $h$}};

\end{tikzpicture}
\caption{A schematic depiction of the functions $\psi$ and $\psi'$ from Lemma~\ref{lem:TwoOptions}.}
\label{Fig:TwoOptions}
\end{figure}

\begin{lemma}
\label{Lem:Case2Dependence}
Either $\psi$ or $\psi'$ is in $\Sigma'$.
\end{lemma}

\begin{proof}
If $s_{\ell} (\varphi_D) = s_{\ell} [h+1]$, then $s_k (\varphi_D) = s_k [h+1]$ for all $k \leq \ell$, and we see that $\varphi_D = \psi$.  Now, suppose that $s_{\ell} (\varphi_D) \neq s_{\ell} [h+1]$.  Because $\Sigma'$ has rank $1$, the functions $\varphi_C, \varphi_D, $ and $\varphi_E$ from Proposition~\ref{Prop:Case2Fns} are tropically dependent.  Since $s_{\ell} (\varphi_D) \neq s_{\ell} [h+1]$, in this dependence the functions $\varphi_C$ and $\varphi_E$ must achieve the minimum at $v_{\ell}$.  All three functions agree on the loop $\gamma_{\ell}$, and since $s'_{\ell} (\varphi_C) = s'_{\ell} [h]$, it follows that one of the other two functions must also have slope $s'_{\ell} [h]$ along the bridge $\beta_{\ell+1}$.  By definition, this function cannot be $\varphi_D$, so we must have $s'_{\ell} (\varphi_E) = s'_{\ell} [h]$.  This implies that $s'_k (\varphi_E) = s'_k [h]$ for all $k \geq \ell$, hence $\varphi_E = \psi'$.
\end{proof}

\begin{lemma}
\label{Lem:ReduceToOne}
If $\psi \in \Sigma'$, then $\gamma_\ell$ is not a switching loop for $\Sigma'$. Similarly, if $\psi' \in \Sigma'$, then $\gamma_{\ell'}$ is not a switching loop for $\Sigma'$.
\end{lemma}

\begin{proof}
Suppose that $\psi \in \Sigma'$, and let $\varphi \in \Sigma'$ be a function with $s_{\ell} (\varphi) = s_{\ell} [h]$.  Because $\Sigma'$ has rank $1$, the functions $\varphi, \psi,$ and $\varphi_C$ are tropically dependent.  Because $s_{\ell} (\varphi) = s_{\ell} [h]$, we see that in this dependence $\varphi_C$ and $\psi$ must achieve the minimum at $w_{\ell}$.  Since $s'_{\ell} (\varphi_C) = s'_{\ell} [h]$, it follows that one of the other two functions must also have slope $s'_{\ell} [h]$ along the bridge $\beta_{\ell+1}$.  By definition, this function cannot be $\psi$, so it must be $\varphi$.  The other case, where $\psi' \in \Sigma'$, is similar.
\end{proof}

\begin{proof}[Proof of Theorem~\ref{thm:independence}, Case 4b]
If $\psi' \in \Sigma'$, we construct our certificate of independence $\Upsilon$ as though $\gamma_{\ell'}$ is not a switching loop.  Specifically, let $\cS = \{ \varphi_j \mid j \neq h, h+1 \} \cup \{ \varphi_A, \varphi_B, \psi' \}$.  We set
\[
\cA = \{ \varphi_i \, \mid \, i \neq h,h+1 \} \cup \{ \varphi^0_h , \varphi^0_{h+1} , \varphi^{\infty}_h \} ,
\]
The argument is then the same as Case 3a.
Similarly, if $\psi \in \Sigma'$, we construct our certificate of independence $\Upsilon$ as though $\gamma_{\ell}$ is not a switching loop.
\end{proof}

\subsubsection{Case 4c:  $h' =h+1$}

We first identify a subset $\cS \subseteq \Sigma$, consisting of the functions $\varphi_i$, for $i \notin \{ h, h+1, h + 2 \}$, together with the functions $\varphi_A, \ldots, \varphi_E$ illustrated in Figure~\ref{Fig:Case3}.

\begin{figure}[H]
\scalebox{.92}{
\begin{tikzpicture}

\draw (-0.2,4) node {{\tiny $\varphi_A$}};
\draw (0,4)--(6,4);
\draw [ball color=white] (2,4) circle (0.55mm);
\draw [ball color=white] (4,4) circle (0.55mm);
\draw [ball color=black] (4.75,4) circle (0.55mm);
\draw [ball color=black] (5.6,4) circle (0.55mm);
\draw (1,4.2) node {{\tiny $h$}};
\draw (3,4.2) node {{\tiny $h+1$}};
\draw (4.4,4.2) node {{\tiny $h+2$}};
\draw (5.2,4.2) node {{\tiny $h+1$}};
\draw (5.8,4.2) node {{\tiny $h$}};

\draw (-0.2,3) node {{\tiny $\varphi_B$}};
\draw (0,3)--(6,3);
\draw [ball color=white] (2,3) circle (0.55mm);
\draw [ball color=white] (4,3) circle (0.55mm);
\draw [ball color=black] (0.75,3) circle (0.55mm);
\draw [ball color=black] (1.6,3) circle (0.55mm);
\draw (0.4,3.2) node {{\tiny $h+2$}};
\draw (1.2,3.2) node {{\tiny $h+1$}};
\draw (1.8,3.2) node {{\tiny $h$}};
\draw (3,3.2) node {{\tiny $h+1$}};
\draw (5,3.2) node {{\tiny $h+2$}};

\draw (-0.2,2) node {{\tiny $\varphi_C$}};
\draw (0,2)--(6,2);
\draw [ball color=white] (2,2) circle (0.55mm);
\draw [ball color=white] (4,2) circle (0.55mm);
\draw (1,2.2) node {{\tiny $h+1$}};
\draw (3,2.2) node {{\tiny $h$}};
\draw (5,2.2) node {{\tiny $h$}};

\draw (-0.2,1) node {{\tiny $\varphi_D$}};
\draw (0,1)--(6,1);
\draw [ball color=white] (2,1) circle (0.55mm);
\draw [ball color=white] (4,1) circle (0.55mm);
\draw (1,1.2) node {{\tiny $h+2$}};
\draw (3,1.2) node {{\tiny $h+2$}};
\draw (5,1.2) node {{\tiny $h+1$}};

\draw (-0.2,0) node {{\tiny $\varphi_E$}};
\draw (0,0)--(6,0);
\draw [ball color=white] (2,0) circle (0.55mm);
\draw [ball color=white] (4,0) circle (0.55mm);
\draw [ball color=black] (1.6,0) circle (0.55mm);

\draw (0.5,0.2) node {{\tiny $h+1$}};
\draw (1.8,0.2) node {{\tiny $h$}};
\draw (3,0.2) node {{\tiny $h+1$}};
\draw (5,0.2) node {{\tiny $h+2$}};

\draw (2,-0.25) node {{\tiny $\gamma_{\ell}$}};
\draw (4,-0.25) node {{\tiny $\gamma_{\ell'}$}};

\end{tikzpicture}
}
\caption{A schematic illustration of the five functions of Proposition~\ref{Prop:Case3Fns}.}
\label{Fig:Case3}
\end{figure}

\begin{proposition}
\label{Prop:Case3Fns}
There is a rank 2 tropical linear subseries $\Sigma' \subseteq \Sigma$, rank $1$ tropical linear subseries $\Sigma_1 , \Sigma_2 \subseteq \Sigma'$, and functions $\varphi_A, \varphi_C \in \Sigma_1$, $\varphi_B , \varphi_D \in \Sigma_2$, and $\varphi_E \in \Sigma_1 \cap \Sigma_2$ with the following properties:
\begin{enumerate}
\item  $s'_k (\varphi_A) = s'_k [h]$ for all $k <\ell$;
\item  $s_k (\varphi_B) = s_k [h+2]$ for all $k > \ell'$;
\item  $s_k (\varphi_C) = s_k [h+1]$ for all $k \leq \ell$ and $s'_k (\varphi_C) = s'_k [h]$ for all $k \geq \ell$;
\item  $s_k (\varphi_D) = s_k [h+2]$ for all $k \leq \ell'$, and $s'_k (\varphi_D) = s'_k [h+1]$ for all all $k \geq \ell'$;
\item  $s'_{k-1} (\varphi_E) = s_k (\varphi_E) = s_k [h+1]$ for all $\ell < k \leq \ell'$, and $s_k (\varphi_E) = s_k[h+2]$ for all $k > \ell'$;
\item$s_k (\varphi_\bullet) \in \{ s_k [h], s_k [h+1], s_k [h+2]\}$ and $s'_k (\varphi_\bullet) \in \{ s'_k [h], s'_k [h+1], s'_k [h+2]\}$ for all $k$.
\end{enumerate}
\end{proposition}

\begin{proof}
As in Proposition~\ref{Prop:Case2Fns}, we apply Lemma~\ref{Lemma:InductiveExistence} to construct a rank $2$ tropical linear series
\[
\Sigma' \subset \{ \varphi \in \Sigma \mid s'_0(\varphi) \leq s'_0[h+2] \mbox{ and } s_{g+1}(\varphi) \geq s_{g+1}[h] \}.
\]
By choosing functions in $\Sigma'$ with specified slopes at $v_{g+1}$ and $w_0$, respectively, and applying Definition~\ref{Def:TLS}\eqref{DefItem:Intersection}, we obtain rank 1 tropical linear subseries
\[
\Sigma_1\subset \{ \varphi \in \Sigma' \mid s_{g+1} \geq s_{g+1}[h+1] \}, \mbox{ \ \ and \ \ }
\Sigma_2\subset \{ \varphi \in \Sigma' \mid s'_0 \leq s'_0[h+1] \},
\]
with nontrivial intersection.  Choose $\varphi_A$ and $\varphi_B$ as in Proposition~\ref{Prop:BridgeFns}.  Then choose $\varphi_C$ and $\varphi_D$ as in Proposition~\ref{Prop:Case2Fns}.  Finally, let $\varphi_E$  be a function in $\Sigma_1 \cap \Sigma_2$.

By arguments analogous to the proofs of Propositions~\ref{Prop:BridgeFns} and~\ref{Prop:Case2Fns}, the functions $\varphi_A , \varphi_B , \varphi_C$, and $\varphi_D$ have the required slopes.  We now describe the slopes of $\varphi_E$.  Since $\varphi_E \in \Sigma_1$, we have $s_k (\varphi_E) \in \{ s_k [h], s_k[h+1] \}$ for all $\ell < k \leq \ell'$, and since $\varphi_E \in \Sigma_2$, we have $s_k (\varphi_E) \in \{ s_k [h+1], s_k [h+2] \}$ for all $\ell < k \leq \ell'$.  It follows that $s_k (\varphi_E) = s_k [h+1]$ for all $\ell < k \leq \ell'$.  The same argument shows that $s'_k(\varphi_E) = s'_k [h+1]$ for all $\ell \leq k < \ell'$.  Moreover, the three functions $\varphi_B, \varphi_D,$ and $\varphi_E$ in $\Sigma_2$ are tropically dependent, and the dependence is illustrated schematically in Figure~\ref{Fig:BisTLC}.  A priori, one might expect there to be a region to the right of $\gamma_{\ell'}$ where $\varphi_D$ and $\varphi_E$ agree in this dependence, but our assumptions on edge lengths preclude this.  Specifically, since $\varphi_D$ has higher slope than $\varphi_B$ and $\varphi_E$ along the bridge $\beta_{\ell+1}$, it cannot obtain the minimum to the right of this bridge.  It follows that $s_k (\varphi_E) = s_k[h+2]$ for all $k > \ell'$.
\end{proof}

\begin{figure}[H]
\begin{tikzpicture}

\draw (0,-1)--(6,-1);
\draw [ball color=white] (2,-1) circle (0.55mm);
\draw [ball color=white] (4,-1) circle (0.55mm);
\draw [ball color=black] (1,-1) circle (0.55mm);

\draw (2,-1.45) node {{\tiny $\gamma_{\ell}$}};
\draw (4,-1.45) node {{\tiny $\gamma_{\ell'}$}};

\draw (0.5,-0.7) node {{\tiny $BD$}};
\draw (3,-0.7) node {{\tiny $BE$}};

\end{tikzpicture}
\caption{The function $\varphi_B$ is a tropical linear combination of $\varphi_D$ and $\varphi_E$.}
\label{Fig:BisTLC}
\end{figure}

\begin{lemma}
\label{Lem:Case3Inequivalent}
The functions $\varphi_C$ and $\varphi_D$ are not equivalent on any loop.  Moreover, for any pair $k' \leq k$, with $k \neq \ell$ and $k' \neq \ell'$, either $\varphi_A$ or $\varphi_E$ is not equivalent to $\varphi_C$ on $\gamma_{k}$, and is not equivalent to $\varphi_D$ on $\gamma_{k'}$.
\end{lemma}

\begin{proof}
Note that the white dots in Figure~\ref{Fig:Case3} representing $\gamma_{\ell}$ and $\gamma_{\ell'}$ divide the graph into 3 regions.  Identify the regions containing $\gamma_{k'}$ and $\gamma_k$.  For each of the 6 possibilities, one of the functions $\varphi_A$ or $\varphi_E$ is not equivalent to $\varphi_D$ on the region containing $\gamma_{k'}$ and not equivalent to $\varphi_C$ on the region containing $\gamma_k$.  For example, if $k' \leq k \leq \ell$, then $s'_{k'} (\varphi_A) \neq s'_{k'} (\varphi_D)$, so $\varphi_A$ is not equivalent to $\varphi_D$ on $\gamma_{k'}$, and $s'_k (\varphi_A) \neq s'_k (\varphi_C)$, so $\varphi_A$ is not equivalent to $\varphi_C$ on $\gamma_k$.  The other 5 cases are similar.
\end{proof}

Note that the three functions $\varphi_A, \varphi_C,$ and $\varphi_E$ in $\Sigma_1$ are tropically dependent; the dependence is illustrated schematically in Figure~\ref{Fig:Case4cDependence}.  Let $y$ and $y'$ be points, in $\Gamma_L$ to the right and left of $\gamma_{\ell}$, respectively, where all three functions simultaneously achieve the minimum in this dependence.

\begin{figure}[H]
\begin{tikzpicture}

\draw (0,0)--(6,0);
\draw [ball color=white] (2,0) circle (0.55mm);
\draw [ball color=white] (4,0) circle (0.55mm);
\draw [ball color=black] (1,0) circle (0.55mm);
\draw [ball color=black] (5,0) circle (0.55mm);

\draw (2,-0.45) node {{\tiny $\gamma_{\ell}$}};
\draw (4,-0.45) node {{\tiny $\gamma_{\ell'}$}};
\draw [<->] (1.05,-0.15) -- (1.95,-0.15);
\draw [<->] (2.05,-0.15) -- (4.95,-0.15);
\draw (1.5,-0.4) node {{\tiny $t'$}};
\draw (3.5,-0.4) node {{\tiny $t$}};
\draw (1,-0.4) node {{\tiny $y'$}};
\draw (5,-0.4) node {{\tiny $y$}};

\draw (0.5,0.3) node {{\tiny $CE$}};
\draw (3,0.3) node {{\tiny $AE$}};
\draw (5.5,0.3) node {{\tiny $AC$}};

\end{tikzpicture}
\caption{The dependence satisfied by $\varphi_A$, $\varphi_C$, and $\varphi_E$ in Case 4c.}
\label{Fig:Case4cDependence}
\end{figure}

\begin{definition}
Let $t$ be the distance, measured along the bridges and bottom edges, from $w_{\ell}$ to $y$, and similarly let $t'$ be the distance from $v_{\ell}$ to $y'$.
\end{definition}

Just as in Case 2c, $t'$ is an increasing piecewise affine function of $t$.  We now describe how to choose the set $\cB$, depending on the parameter $t$, in a manner similar to Definition~\ref{Def:BBOnBridge}.

\begin{definition}
Let $\widetilde \cB$ be the set of pairwise sums of elements of $\cA$.  Suppose that there are two indices $j$ and $j'$ such that
$\sigma_{\ell} = s'_{\ell} [h] + s'_{\ell} [j] = s'_{\ell} [h] + s'_{\ell} [j'] +1$.  Then
\begin{enumerate}
\item  if $t_1 < m_{\ell}$, let $\widehat{\cB} = \widetilde\cB \smallsetminus \{ \varphi + \varphi_{j'} \mid s_{\ell+1} (\varphi) = s_{\ell} (\varphi) + 1 = s_{\ell+1} [h+1] \}$;
\item   if $t_1 \geq m_{\ell}$, let $\widehat{\cB} = \widetilde \cB \smallsetminus \{ \varphi + \varphi_j \mid s_{\ell+1} (\varphi) = s_{\ell} (\varphi) = s_{\ell+1} [h] \}$.
\end{enumerate}
Otherwise, let $\widehat{\cB} = \widetilde \cB$.

Now, if there are indices $i, i'$ such that
$\sigma_{\ell'} = s'_{\ell'} [h+1] + s'_{\ell'} [i] = s'_{\ell'} [h+1] + s'_{\ell'} [i'] +1$, let $\cB = \widehat{\cB} \smallsetminus \{ \varphi + \varphi_i \mid s_{\ell'+1} (\varphi) = s_{\ell'} (\varphi) = s_{\ell'+1} [h+1] \}$.
Otherwise, let $\cB = \widehat{\cB}$.
\end{definition}

Note that the point where $\varphi_B, \varphi_D$, and $\varphi_E$ simultaneously achieve the minimum in Figure~\ref{Fig:BisTLC} is to the left of $\gamma_{\ell}$.  The distance from $v_{\ell'}$ to this point is therefore larger than $m_{\ell'}$, and the construction of $\cB$ from $\widehat{\cB}$ is analogous to the construction of $\widehat{\cB}$ from $\widetilde{\cB}$ in case (2).

The set $\cB$ satisfies properties $(\mathbf B)$ and $(\mathbf B')$ just as in Lemma~\ref{Lem:BBOnBridge}.  Therefore, there exists a template $\theta$ and an assignment function $\alpha$ satisfying \ref{T1}-\ref{T5}.  Our choice of $\cT$ is very similar to Case 2c.  Specifically, if $i,j \notin \{ h,h+1,h+2 \}$, then we put $\varphi_{ij}, \varphi_C + \varphi_j,$ and $\varphi_D + \varphi_j$ in $\cT$.  We then put one of $\varphi_A + \varphi_j$ or $\varphi_E + \varphi_j$ in $\cT$, depending on where the best approximation of $\theta$ by $\varphi_C + \varphi_j$ and $\varphi_D + \varphi_j$ achieves equality.  We use Lemma~\ref{Lem:Case3Inequivalent} to choose this function, as follows.  By Lemma~\ref{Lem:Replace}, the function $\varphi_C + \varphi_j$ achieves equality on the region where some pairwise sum of building blocks $\psi \in \cB$ achieves the minimum, and $\varphi_D + \varphi_j$ achieves equality on the region where some pairwise sum of building blocks $\psi' \in \cB$ achieves the minimum.  By Lemma~\ref{Lem:Case3Inequivalent}, $\psi$ and $\psi'$ are not assigned to the same loop or bridge, and one of the functions $\varphi_A + \varphi_j$ or $\varphi_E + \varphi_j$ is not equivalent to $\varphi_C + \varphi_j$ on the loop or bridge where $\psi$ is assigned, and is not equivalent to $\varphi_D + \varphi_j$ on the loop or bridge where $\psi'$ is assigned.  We put this function in $\cT$.

Similarly, we include six pairwise sums of elements of $\{ \varphi_A , \varphi_C , \varphi_D , \varphi_E \}$ in $\cT$.  In all cases, we put $\varphi_C + \varphi_C , \varphi_C + \varphi_D$, and $\varphi_D + \varphi_D$ in $\cT$.  Then, depending on where the best approximation of $\theta$ by these functions achieves equality, we put one of $\varphi_A + \varphi_C$ or $\varphi_C + \varphi_E$ in $\cT$, and one of $\varphi_A + \varphi_D$ or $\varphi_D + \varphi_E$ in $\cT$.  Finally, depending on where the best approximation of $\theta$ by these two functions achieves equality, we put one of $\varphi_A + \varphi_A , \varphi_A + \varphi_E$, or $\varphi_E + \varphi_E$ in $\cT$.

\begin{proof}[Proof of Theorem~\ref{thm:independence}, Case 4c]
The proof of this subcase is very similar to that of Case 2c.  The construction of $\cB$ guarantees that, in the best approximation, each function in $\cT$ achieves equality on a region where some function in $\cB$ achieves the minimum in $\theta$.  Lemma~\ref{Lem:Case3Inequivalent} then shows that no two of these functions achieve the minimum on the same loop or bridge.
\end{proof}

\medskip

\subsubsection{Case 4d:  $h'=h-1$}  \label{sec:4d}

In the previous three cases, our analysis reduced to the study of certain rank 1 tropical linear subseries.  In this last case, we instead reduce to a rank 2 linear series.  Nevertheless, the arguments are of a similar flavor, with just a few more combinatorial possibilities.  As in the previous cases, we begin by describing the subset $\cS \subset \Sigma$.  It consists of the functions $\varphi_i$, for $i \notin \{h-1, h, h+1 \}$, together with functions $\varphi_A, \ldots, \varphi_H$ illustrated in Figures~\ref{Fig:Case4} and~\ref{Fig:Case4More}.

\begin{figure}[h]
\begin{tikzpicture}

\draw (-0.2,4) node {{\tiny $\varphi_A$}};
\draw (0,4)--(6,4);
\draw [ball color=white] (2,4) circle (0.55mm);
\draw [ball color=white] (4,4) circle (0.55mm);
\draw [ball color=black] (5,4) circle (0.55mm);
\draw (1,4.2) node {{\tiny $h-1$}};
\draw (3,4.2) node {{\tiny $h-1$}};
\draw (4.5,4.2) node {{\tiny $h$}};
\draw (5.5,4.2) node {{\tiny $h-1$}};

\draw (-0.2,3) node {{\tiny $\varphi_B$}};
\draw (0,3)--(6,3);
\draw [ball color=white] (2,3) circle (0.55mm);
\draw [ball color=white] (4,3) circle (0.55mm);
\draw [ball color=black] (1,3) circle (0.55mm);
\draw (0.5,3.2) node {{\tiny $h+1$}};
\draw (1.5,3.2) node {{\tiny $h$}};
\draw (3,3.2) node {{\tiny $h+1$}};
\draw (5,3.2) node {{\tiny $h+1$}};

\draw (-0.2,2) node {{\tiny $\varphi_C$}};
\draw (0,2)--(6,2);
\draw [ball color=white] (2,2) circle (0.55mm);
\draw [ball color=white] (4,2) circle (0.55mm);
\draw [ball color=black] (5,2) circle (0.55mm);
\draw (1,2.2) node {{\tiny $h+1$}};
\draw (3,2.2) node {{\tiny $h-1$}};
\draw (4.5,2.2) node {{\tiny $h$}};
\draw (5.5,2.2) node {{\tiny $h-1$}};

\draw (-0.2,1) node {{\tiny $\varphi_D$}};
\draw (0,1)--(6,1);
\draw [ball color=white] (2,1) circle (0.55mm);
\draw [ball color=white] (4,1) circle (0.55mm);
\draw [ball color=black] (1,1) circle (0.55mm);
\draw (0.5,1.2) node {{\tiny $h+1$}};
\draw (1.5,1.2) node {{\tiny $h$}};
\draw (3,1.2) node {{\tiny $h+1$}};
\draw (5,1.2) node {{\tiny $h-1$}};

\draw (-0.2,0) node {{\tiny $\varphi_E$}};
\draw (0,0)--(6,0);
\draw [ball color=white] (2,0) circle (0.55mm);
\draw [ball color=white] (4,0) circle (0.55mm);
\draw [ball color=black] (1,0) circle (0.55mm);
\draw [ball color=black] (5,0) circle (0.55mm);
\draw (0.5,0.2) node {{\tiny $h+1$}};
\draw (1.5,0.2) node {{\tiny $h$}};
\draw (3,0.2) node {{\tiny $h$}};
\draw (4.5,0.2) node {{\tiny $h$}};
\draw (5.5,0.2) node {{\tiny $h-1$}};

\draw (2,-0.25) node {{\tiny $\gamma_{\ell}$}};
\draw (4,-0.25) node {{\tiny $\gamma_{\ell'}$}};

\end{tikzpicture}
\caption{A schematic depiction of the five functions of Proposition~\ref{Prop:Case4Fns}.}
\label{Fig:Case4}
\end{figure}

\begin{proposition}
\label{Prop:Case4Fns}
There is a rank $2$ tropical linear series $\Sigma' \subset \Sigma$ that contains functions $\varphi_A , \varphi_B , \varphi_C , \varphi_D , \varphi_E$ with the following properties:
\begin{enumerate}
\item  $s'_k (\varphi_A) = s'_k [h-1]$ for all $k <\ell'$;
\item  $s_k (\varphi_B) = s_k [h+1]$ for all $k > \ell$;
\item  $s_k (\varphi_C) = s_k [h+1]$ for all $k \leq \ell$ and $s'_k (\varphi_C) = s'_k [h-1]$ for all $\ell \leq k < \ell'$;
\item  $s_k (\varphi_D) = s_k [h+1]$ for all $\ell < k \leq \ell'$, and $s'_k (\varphi_D) = s'_k [h-1]$ for all all $k \geq \ell'$;
\item  $s'_{k-1} (\varphi_E) = s_k (\varphi_E) = s_k [h]$ for all $\ell < k \leq \ell'$;
\item  $s_k (\varphi_\bullet) \in \{ s_k [h-1], s_k [h], s_k [h+1]\}$ and $s'_k (\varphi_\bullet) \in \{ s'_k [h-1], s'_k [h], s'_k [h+1]\}$, for all $k$.
\end{enumerate}
\end{proposition}

\begin{proof}
Apply Lemma~\ref{Lemma:InductiveExistence} twice to obtain a rank 2 tropical linear series
\[
\Sigma' \subset \{ \varphi \in \Sigma \mid s'_0(\varphi) \leq s'_0[h+1] \mbox{ and } s_{g+1}(\varphi) \geq s_{g+1}[h-1] \} .
\]
Choose $\varphi_A$ and $\varphi_B$ as in Proposition~\ref{Prop:BridgeFns}.  The functions $\varphi_C$ and $\varphi_D$ are constructed as in Proposition~\ref{Prop:Case2Fns}.  Finally, let $\varphi_E$ be a function in $\Sigma'$ with $s_{\ell+1} (\varphi_E) \leq s_{\ell+1} [h]$ and $s_{\ell'} (\varphi_E) \geq s_{\ell'} [h]$.
\end{proof}

\begin{lemma}
\label{Lem:Case4Dependence1}
Either $s'_k (\varphi_A) = s'_k [h-1]$ for all $k \geq \ell'$, or $s'_k (\varphi_E) = s'_k [h-1]$ for all $k \geq \ell'$.
\end{lemma}

\begin{proof}
Because $\Sigma'$ has rank $2$ the functions $\varphi_A, \varphi_B, \varphi_D, $ and $\varphi_E$ are tropically dependent.  Consider a dependence among them.  Only $\varphi_B$ and $\varphi_D$ have the same slope along $\beta_{\ell'}$, thus these two achieve the minimum at $v_{\ell'}$.  Because of this, $\varphi_D$ must also achieve the minimum at $w_{\ell'}$.  Since it has slope $s_{\ell'}[h-1]$ along $\beta_{\ell'+1}$, there must be a second function among these four with this same slope along $\beta_{\ell'+1}$.  This function can only be $\varphi_A$ or $\varphi_E$.
\end{proof}

\begin{lemma}
\label{Lem:Case4Dependence2}
Either $s_k (\varphi_B) = s_k [h+1]$ for all $k \leq \ell$, or $s_k (\varphi_E) = s_k [h+1]$ for all $k \leq \ell$.
\end{lemma}

\begin{proof}
This is similar to the proof of Lemma~\ref{Lem:Case4Dependence1}, using the functions $\varphi_A, \varphi_B, \varphi_C,$ and $\varphi_E$.
\end{proof}

Lemmas~\ref{Lem:Case4Dependence1} and~\ref{Lem:Case4Dependence2} together produce 4 possible cases.  In all but one of these cases, $\Sigma'$ has only one switching loop.

\begin{lemma}
\label{Lem:Case4ReduceToOne}
If $s'_k (\varphi_A) = s'_k [h-1]$ for all $k \geq \ell'$, then $\gamma_{\ell'}$ is not a switching loop for $\Sigma'$.  Similarly, if $s_k (\varphi_B) = s_k [h+1]$ for all $k \leq \ell$, then $\gamma_{\ell}$ is not a switching loop for $\Sigma'$.
\end{lemma}

\begin{proof}
We consider the case where $s'_k (\varphi_A) = s'_k [h-1]$ for all $k \geq \ell'$.  The other case is similar.  Let $\varphi \in \Sigma'$ be a function with $s_{\ell'+1} (\varphi) = s_{\ell'+1} [h]$.  Because $\Sigma'$ has rank $2$, the functions $\varphi, \varphi_A, \varphi_B$ and $\varphi_E$ are tropically dependent.  Because only $\varphi$ and $\varphi_E$ have the same slope on $\beta_{\ell'+1}$, in this dependence they must achieve the minimum at $w_{\ell'}$.  Because of this, $\varphi_E$ achieves the minimum at $v_{\ell'}$ as well, hence the minimum has slope at least $s_{\ell'} [h]$ along $\beta_{\ell'}$.  Because this slope must be obtained twice, and the three functions $\varphi_A , \varphi_B,$ and $\varphi_E$ have distinct slopes there, we see that $s_{\ell'} (\varphi) \geq s_{\ell'} [h]$.
\end{proof}

If $\Sigma'$ has only one switching loop, then the argument is essentially identical to Case 3.

\medskip

For the remainder of this section, we assume that there exists $k' \geq \ell'$ such that $s'_{k'} (\varphi_A) = s'_{k'} [h]$, and there exists $k \leq \ell$ such that $s_k (\varphi_B) = s_k [h+1]$.  By Lemmas~\ref{Lem:Case4Dependence1} and~\ref{Lem:Case4Dependence2}, this implies that the slopes of $\varphi_E$ are as pictured in Figure~\ref{Fig:LastCase}.

\begin{figure}[H]
\begin{tikzpicture}

\draw (0,4)--(6,4);
\draw [ball color=white] (2,4) circle (0.55mm);
\draw [ball color=white] (4,4) circle (0.55mm);
\draw (1,4.2) node {{\tiny $h+1$}};
\draw (3,4.2) node {{\tiny $h$}};
\draw (5,4.2) node {{\tiny $h-1$}};

\end{tikzpicture}
\caption{A schematic depiction of the function $\varphi_E$ when $\Sigma'$ has two switching loops.}
\label{Fig:LastCase}
\end{figure}

We now describe additional functions in $\Sigma'$.  These functions are illustrated in Figure~\ref{Fig:Case4More}.

\begin{proposition}
\label{Prop:MoreFns}
There exist functions $\varphi_F, \varphi_G, \varphi_H \in \Sigma'$ with the following properties:
\begin{enumerate}
\item  $s'_k (\varphi_F) = s'_k [h]$ for all $k \leq \ell$, and $s_{\ell} (\varphi_F) = s_{\ell}[h+1]$;
\item  $s'_{\ell'} (\varphi_G) = s'_{\ell'}[h-1]$ and $s_k (\varphi_G) = s_k [h]$ for all $k > \ell'$;
\item  either
\begin{enumerate}
\item  $s'_k (\varphi_H) = s'_k[h]$ for all $k \leq \ell$ and $s_k (\varphi_H) = s_k[h]$ for all $k > \ell'$, or
\item  $s'_k (\varphi_H) = s'_k[h]$ for all $k \leq \ell$ and $s_k (\varphi_H) = s_k[h+1]$ for all $\ell < k \leq \ell'$, or
\item  $s'_k (\varphi_H) = s'_k[h-1]$ for all $\ell \leq k < \ell'$, and $s_k (\varphi_H) = s_k[h]$ for all $k > \ell'$.
\end{enumerate}
\end{enumerate}
\end{proposition}

\begin{proof}
By Definition~\ref{Def:TLS}\eqref{DefItem:Recursive}, there are rank 1 tropical linear subseries $\Sigma_1$ and $\Sigma_2$ of $\Sigma'$ that contain $\{\varphi_B, \varphi_E \}$ and $\{ \varphi_A, \varphi_E \}$, respectively. Let $\varphi_F$ be a function in $\Sigma_1$ with $s'_0 (\varphi_F) \neq s'_0 [h+1]$.  Similarly, let $\varphi_G$ be a function in $\Sigma_2$ with $s_{g'+1} (\varphi_G) \neq s_{g'+1} [h-1]$.  We let $\varphi_H$ be a function in $\trop (W)$ such that $s'_0 (\varphi_H) \leq s'_0[h]$ and $s_{g'+1} (\varphi_H) \geq s_{g'+1}[h]$.

To see that the functions have the required slopes, we make use of various dependences between them and the functions $\varphi_A , \varphi_B, \varphi_E$.  Specifically, because the functions $\varphi_B, \varphi_E$, and $\varphi_F$ are contained in a tropical linear subseries of rank 1, they are tropically dependent.  The dependence between them is very similar to the dependence between $\varphi_A, \varphi_B$, and $\varphi_C$ in Case 2, and is depicted in the top line of Figure~\ref{Fig:LastCaseDependence}.  In this dependence, $\varphi_B$ and $\varphi_F$ agree in a neighborhood of $\gamma_{\ell}$, which determines the slopes of $\varphi_F$ on the bridges to either side of this loop.

Similarly, the functions $\varphi_A, \varphi_E$, and $\varphi_G$ are tropically dependent, and the dependence between them is illustrated in the bottom line of Figure~\ref{Fig:LastCaseDependence}.  The functions $\varphi_A, \varphi_B, \varphi_E$, and $\varphi_H$ also satisfy a dependence.  There are three possibilities for this dependence, as shown in Figure~\ref{Fig:HDependence}.
\end{proof}

\begin{figure}[h]
\begin{tikzpicture}

\draw (-0.4,.7) node {{\footnotesize $\varphi_F$}};
\draw (0,.7)--(6,.7);
\draw [ball color=white] (2,.7) circle (0.55mm);
\draw [ball color=white] (4,.7) circle (0.55mm);
\draw [ball color=black] (3,.7) circle (0.55mm);
\draw (1,.9) node {{\tiny $h$}};
\draw (2.5,.9) node {{\tiny $h+1$}};
\draw (3.5,.9) node {{\tiny $h$}};
\draw (5,.9) node {{\tiny $h-1$}};

\draw (-0.2,-.3) node {{\tiny or}};
\draw (0,-.3)--(6,-.3);
\draw [ball color=white] (2,-.3) circle (0.55mm);
\draw [ball color=white] (4,-.3) circle (0.55mm);
\draw [ball color=black] (5,-.3) circle (0.55mm);
\draw (1,-.1) node {{\tiny $h$}};
\draw (3,-.1) node {{\tiny $h+1$}};
\draw (4.5,-.1) node {{\tiny $h+1$}};
\draw (5.5,-.1) node {{\tiny $h-1$}};

\draw (-0.4,-1.6) node {{\footnotesize $\varphi_G$}};
\draw (0,-1.6)--(6,-1.6);
\draw [ball color=white] (2,-1.6) circle (0.55mm);
\draw [ball color=white] (4,-1.6) circle (0.55mm);
\draw [ball color=black] (3,-1.6) circle (0.55mm);
\draw (1,-1.4) node {{\tiny $h$}};
\draw (2.5,-1.4) node {{\tiny $h$}};
\draw (3.5,-1.4) node {{\tiny $h-1$}};
\draw (5,-1.4) node {{\tiny $h$}};

\draw (-0.2,-2.6) node {{\tiny or}};
\draw (0,-2.6)--(6,-2.6);
\draw [ball color=white] (2,-2.6) circle (0.55mm);
\draw [ball color=white] (4,-2.6) circle (0.55mm);
\draw [ball color=black] (1,-2.6) circle (0.55mm);
\draw (0.5,-2.4) node {{\tiny $h+1$}};
\draw (1.5,-2.4) node {{\tiny $h-1$}};
\draw (3,-2.4) node {{\tiny $h-1$}};
\draw (5,-2.4) node {{\tiny $h$}};

\draw (-0.4,-4) node {{\footnotesize $\varphi_H$}};
\draw (0,-4)--(6,-4);
\draw [ball color=white] (2,-4) circle (0.55mm);
\draw [ball color=white] (4,-4) circle (0.55mm);
\draw [ball color=black] (2.66,-4) circle (0.55mm);
\draw [ball color=black] (3.33,-4) circle (0.55mm);
\draw (1,-3.8) node {{\tiny $h$}};
\draw (2.5,-3.8) node {{\tiny $h+1$}};
\draw (3,-3.8) node {{\tiny $h$}};
\draw (3.5,-3.8) node {{\tiny $h-1$}};
\draw (5,-3.8) node {{\tiny $h$}};

\draw (-0.2,-5) node {{\tiny or}};
\draw (0,-5)--(6,-5);
\draw [ball color=white] (2,-5) circle (0.55mm);
\draw [ball color=white] (4,-5) circle (0.55mm);
\draw [ball color=black] (5,-5) circle (0.55mm);
\draw (1,-4.8) node {{\tiny $h$}};
\draw (3,-4.8) node {{\tiny $h+1$}};
\draw (4.5,-4.8) node {{\tiny $h+1$}};
\draw (5.5,-4.8) node {{\tiny $h$}};

\draw (-0.2,-6) node {{\tiny or}};
\draw (0,-6)--(6,-6);
\draw [ball color=white] (2,-6) circle (0.55mm);
\draw [ball color=white] (4,-6) circle (0.55mm);
\draw [ball color=black] (1,-6) circle (0.55mm);
\draw (0.5,-5.8) node {{\tiny $h$}};
\draw (1.5,-5.8) node {{\tiny $h-1$}};
\draw (3,-5.8) node {{\tiny $h-1$}};
\draw (5,-5.8) node {{\tiny $h$}};

\end{tikzpicture}
\caption{A schematic depiction of the three functions of Proposition~\ref{Prop:MoreFns}.}
\label{Fig:Case4More}
\end{figure}

\begin{figure}[h]
\begin{tikzpicture}

\draw (0,6.3)--(6,6.3);
\draw [ball color=white] (2,6.3) circle (0.55mm);
\draw [ball color=white] (4,6.3) circle (0.55mm);
\draw [ball color=black] (1,6.3) circle (0.55mm);
\draw [ball color=black] (3,6.3) circle (0.55mm);
\draw (0.5,6.5) node {{\tiny $BE$}};
\draw (2,6.5) node {{\tiny $BF$}};
\draw (4,6.5) node {{\tiny $EF$}};

\draw (0,5.3)--(6,5.3);
\draw [ball color=white] (2,5.3) circle (0.55mm);
\draw [ball color=white] (4,5.3) circle (0.55mm);
\draw [ball color=black] (3,5.3) circle (0.55mm);
\draw [ball color=black] (5,5.3) circle (0.55mm);
\draw (2,5.5) node {{\tiny $EG$}};
\draw (4,5.5) node {{\tiny $AG$}};
\draw (5.5,5.5) node {{\tiny $AE$}};

\end{tikzpicture}
\caption{Dependences between the functions $\varphi_A , \varphi_B , \varphi_E, \varphi_F$, and $\varphi_G$.}
\label{Fig:LastCaseDependence}
\end{figure}

\begin{figure}[h]
\begin{tikzpicture}

\draw (0,4)--(6,4);
\draw [ball color=white] (2,4) circle (0.55mm);
\draw [ball color=white] (4,4) circle (0.55mm);
\draw [ball color=black] (1,4) circle (0.55mm);
\draw [ball color=black] (2.75,4) circle (0.55mm);
\draw [ball color=black] (3.25,4) circle (0.55mm);
\draw [ball color=black] (5,4) circle (0.55mm);
\draw (0.5,4.2) node {{\tiny $BE$}};
\draw (2,4.2) node {{\tiny $BH$}};
\draw (3,4.2) node {{\tiny $EH$}};
\draw (4,4.2) node {{\tiny $AH$}};
\draw (5.5,4.2) node {{\tiny $AE$}};

\draw (-0.2,3) node {{\tiny or}};
\draw (0,3)--(6,3);
\draw [ball color=white] (2,3) circle (0.55mm);
\draw [ball color=white] (4,3) circle (0.55mm);
\draw [ball color=black] (0.75,3) circle (0.55mm);
\draw [ball color=black] (1.25,3) circle (0.55mm);
\draw [ball color=black] (5,3) circle (0.55mm);
\draw (0.35,3.2) node {{\tiny $BE$}};
\draw (1,3.2) node {{\tiny $BH$}};
\draw (3,3.2) node {{\tiny $AH$}};
\draw (5.5,3.2) node {{\tiny $AE$}};

\draw (-0.2,2) node {{\tiny or}};
\draw (0,2)--(6,2);
\draw [ball color=white] (2,2) circle (0.55mm);
\draw [ball color=white] (4,2) circle (0.55mm);
\draw [ball color=black] (1,2) circle (0.55mm);
\draw [ball color=black] (4.75,2) circle (0.55mm);
\draw [ball color=black] (5.25,2) circle (0.55mm);
\draw (0.5,2.2) node {{\tiny $BE$}};
\draw (3,2.2) node {{\tiny $BH$}};
\draw (5,2.2) node {{\tiny $AH$}};
\draw (5.5,2.2) node {{\tiny $AE$}};

\end{tikzpicture}
\caption{Possibilities for the dependence between $\varphi_A , \varphi_B , \varphi_E$, and $\varphi_H$.}
\label{Fig:HDependence}
\end{figure}

\begin{lemma}
\label{Lem:Case4Inequivalent}
The functions $\varphi_A$ and $\varphi_B$ are not equivalent on any loop.  Moreover, for any pair $k' \leq k$ with $k' \neq \ell'$ and $k \neq \ell$, one of the four functions $\varphi_E, \varphi_F, \varphi_G, \varphi_H$ is not equivalent to $\varphi_B$ on $\gamma_{k'}$, and is not equivalent to $\varphi_A$ on $\gamma_k$.
\end{lemma}

\begin{proof}
The proof is similar to that of Lemma~\ref{Lem:Case3Inequivalent}.
\end{proof}

We choose the set $\cA$ and the set $\cB$ satisfying properties $(\mathbf B)$ and $(\mathbf B')$ exactly as in 4c.  Then there exists a template $\theta$ and an assigment function $\alpha$ satisfying \ref{T1}-\ref{T5}.  The choice of $\cT$ is also similar to Case 4c.  First, if $i,j \notin \{ h-1,h,h+1 \}$, then we put $\varphi_{ij}, \varphi_A + \varphi_j$, and $\varphi_B + \varphi_j$ in $\cT$.  Then $\varphi_A + \varphi_j$ achieves equality on the region where some pairwise sum of building blocks $\psi \in \cB$ achieves the minimum, and $\varphi_B + \varphi_j$ achieves equality on the region where some pairwise sum of building blocks $\psi'$ achieves the minimum.  By Lemma~\ref{Lem:Case4Inequivalent}, $\psi$ and $\psi'$ are not assigned to the same loop or bridge, and one of the four functions $\varphi_E + \varphi_j , \varphi_F + \varphi_j , \varphi_G + \varphi_j$, or $\varphi_H + \varphi_j$ is not equivalent to $\varphi_A + \varphi_j$ on the loop or bridge where $\psi$ is assigned, and not equivalent to $\varphi_B + \varphi_j$ on the loop or bridge where $\psi'$ is assigned.  We put this function in $\cT$.

\begin{proof}[Proof of Theorem~\ref{thm:independence}, Case 4d]
The proof is identical to that of Case 4c, using Lemma~\ref{Lem:Case4Inequivalent} in place of Lemma~\ref{Lem:Case3Inequivalent}.
\end{proof}

\section{Effectivity of the virtual classes} \label{sec:genfinite}

We fix $g = 22$ or $23$, $d = g+3$, and study linear series of rank $r = 6$.  In \S\ref{virtualis_divizorok} we defined an open substack $\widetilde{\mathfrak{M}}_g$ of the moduli stack of stable curves, a stack $\Grd$ of generalized limit linear series of rank $r$ and degree $d$ over $\widetilde{\mathfrak{M}}_g$, and a morphism of vector bundles $\phi:\mbox{Sym}^2(\E)\rightarrow \F$ over $\Grd$, whose degeneracy locus is denoted by $\fU$.

In \S\ref{Sec:Generic} we used the method of tropical independence to prove Theorems~\ref{thm:independence} and \ref{Thm:MainThm}, establishing the Strong Maximal Rank Conjecture for $g$, $r$, and $d$.  As a consequence, we know that the push forward $\sigma_*[\fU]^\vir$ under the proper forgetful map $\sigma:\Grd\rightarrow \widetilde{\mathcal{M}}_g$ is a divisor, not just a divisor class.  We now proceed to prove Theorem~\ref{thm:genfinite}, which says that $\fU$ is generically finite over each component of this divisor.  This implies that $\sigma_*[\fU]^\vir$ is effective.  By Theorem~\ref{thm:slopes}, the slope of this effective divisor is less than $\frac{13}{2}$, and it follows that $\overline{\mathcal{M}}_{22}$ and $\overline{\mathcal{M}}_{23}$ are of general type.

\medskip

\subsection{Multiplication maps with ramification} To study the fibers of $\sigma_{\vert\fU}$ over singular curves, we consider linear series $\ell = (L,V)$ of degree $d$ and rank $r$ on a pointed curve $(X, p)$ of genus $g' \leq g$ that satisfy a ramification condition at $p$.  More precisely, we consider the cases where
\begin{enumerate}
\item $g' = g$,
\item $g' = g-1$ and $a_1^\ell(p) \geq 2$, or
\item $g' = g-2$ and either $a_1^\ell(p) \geq 3$ or $a_0^\ell(p) + a_2^\ell(p) \geq 5$.
\end{enumerate}
We deduced Theorem~\ref{Thm:MainThm} from the case of Theorem~\ref{thm:independence} where $g' = g$.  The cases where $g'$ is equal to $g-1$ or $g-2$ have the following analogous consequences involving multiplication maps for linear series with ramification on a general pointed curve of genus $g'$.

\begin{theorem}
\label{thm:MRCg-1}
Let $X$ be a general curve of genus $g' = 20+\rho$ and let $p \in X$ be a general point.  Then the multiplication map
\[
\phi_{\ell} : \Sym^2 V \to H^0(X, L^{\otimes 2})
\]
is injective for all linear series $\ell=(L,V)\in G^6_{24+\rho}(X)$ satisfying the vanishing condition
$
a_1^{\ell}(p) \geq 2$.
\end{theorem}

\begin{theorem}
\label{thm:MRCg-2}
Let $X$ be a general curve of genus $g' = 19+ \rho$ and let $p \in X$ be a general point.  Then the multiplication map
\[
\phi_{\ell} : \Sym^2 V \to H^0(X,L^{\otimes 2})
\]
is injective for all linear series $\ell=(L,V)\in G_{24+\rho}^6(X)$ satisfying either of the vanishing conditions:
\[
a_1^{\ell}(p) \geq 3 \ \mbox{ or } \ a_0^{\ell}(p) + a_2^{\ell}(p) \geq 5.
\]
\end{theorem}

\subsection{Effectivity via numerical vanishing} For the remainder of the section, suppose $Z \subseteq \mm_g$ is an irreducible divisor and that $\sigma_{\vert \fU}$ has positive dimensional fibers over the generic point of $Z$.  Our strategy for proving Theorem~\ref{thm:genfinite} is to show, using the vanishing criterion from \S\ref{sec:vanishingcondition}, that $[Z] = 0$ in $CH^1(\mm_g)$.  This is impossible, since $\mm_g$ is projective, and hence no such $Z$ exists.  To apply the vanishing criterion, we must show:
\begin{enumerate} [label=(V\arabic*)]
\item \label{it:closure} $Z$ is the closure of a divisor in $\cM_g$,
\item \label{it:j2} $\jmath_2^*(Z) = 0$,
\item \label{it:codim2} $Z$ does not contain any codimension $2$ stratum $\Delta_{2,j}$, and
\item \label{it:j3} if $g$ is even then $\jmath_3^*(Z)$ is a nonnegative combination of the classes $\big [\overline{\mathcal{W}}_3 \big]$ and $\big [\overline{\mathcal{H}}_3 \big]$ on $\mm_{3,1}$.
\end{enumerate}

The only irreducible divisors on $\widetilde{\mathcal M}_g$ in the complement of $\mathcal{M}_g$ are $\Delta_0^\circ$ and $\Delta_1^\circ$.  Therefore, \ref{it:closure} is a consequence of the following proposition.

\begin{proposition}
\label{prop:noboundary}
The image of the degeneracy locus $\fU$ does not contain $\Delta_0^\circ$ or $\Delta_1^\circ$.
\end{proposition}

\begin{proof}
Let $[X,p]\in \cM_{g-1,1}$ be a general pointed curve and consider the curve $Y$ obtained by gluing a nodal rational curve $E_{\infty}$ to $X$ at the point $p$.  Note that $[Y] \in \Delta_0^{\circ}\cap \Delta_1^{\circ}$.  The $X$-aspect of a generalized limit linear series of type $\mathfrak g^r_d$ on $Y$ is a linear series $\ell \in G^r_d(X)$ satisfying the condition $a_1^{\ell}(p) \geq 2$.  Then  Theorem~\ref{thm:MRCg-1} implies that $[Y] \not \in \sigma(\fU)$.
\end{proof}

In the proofs of \ref{it:j2}-\ref{it:j3}, we use the following lemma.

\begin{lemma}
\label{lem:ramify}
If $[X] \in Z$ and $p \in X$ then there is a linear series  $\ell \in G^r_d (X)$ that is ramified at $p$ such that $\phi_\ell$ is not injective.
\end{lemma}

\begin{proof}
If $[X]\in Z$, then there are infinitely many linear series $\ell \in G^r_d (X)$ for which $\phi_\ell$ fails to be injective. By \cite[Lemma~2.a]{Schubert91}, at least one such linear series is ramified at the point $p$.
\end{proof}

\subsection{Pulling back to $\mm_{2,1}$} In order to verify \ref{it:j2}, we now consider the preimage of $Z$ under the map $\jmath_2\colon \mm_{2,1} \rightarrow \mm_g$ obtained by attaching an arbitrary pointed curve of genus $2$ to a fixed general pointed curve $(X, p)$ of genus $g -2$.

\begin{lemma}
\label{lem:Weierstrass}
The preimage $\jmath_2^{-1}(Z)$ is contained in the Weierstrass divisor $\overline{\mathcal{W}}_2$ in $\mm_{2,1}$.
\end{lemma}

\begin{proof}
Let $C$ be an arbitrary curve of genus 2 and (abusing notation slightly) let $p \in C$ be a non-Weierstrass point.  If $[Y] := [X \cup_p C]$ is in $Z$, then it is in the closure of the generic point $[Y_t]$ of a one-parameter family in $\sigma(\fU)$.  Since $[Y_t]$ is in $\sigma (\fU)$, there is a linear series $\ell_t$ on $Y_t$ for which the multiplication map $\phi_{\ell_t}$ is not injective.  Hence there is a limit linear series $\ell$ on $Y$ such that the multiplication map on each aspect of $\ell$ is not injective.

We claim that the $X$-aspect $\ell_X$ of any limit linear series $\ell$ on $X \cup_p C$ satisfies one of the ramification conditions
$
a_1^{\ell_X}(p) \geq 3 \ \mbox{ or } a_0^{\ell_X}(p) + a_2^{\ell_X}(p) \geq 5$.
Suppose both inequalities fail.  By failure of the first inequality and the definition of a limit linear series, we have $a_{r-1}^{\ell_C}(p) \geq d-2$.  Since $p \in C$ is not a Weierstrass point, this forces $a_r^{\ell_C}(p) = d-1$.  By the definition of a limit linear series, this gives $a_0^{\ell_X}(p) \geq 1$ and hence $a_2^{\ell_X}(p) \geq 3$.  By failure of the second inequality, we have $a_2^{\ell_X}(p) = 3$, and hence $a_{r-2}^{\ell_C}(p) \geq d-3$.  Then $\mbox{dim } \bigl| \ell_C(-(d-3)p) \bigr| \geq 2$, which contradicts Riemann-Roch.  This proves the claim, and the result then follows from Theorem~\ref{thm:MRCg-2}.
\end{proof}

In the proof of the next proposition, and for the remainder of the paper, our arguments use tropical and nonarchimedean analytic geometry.  All of the curves and maps that appear are defined over our fixed nonarchimedean field $K$.

\begin{proposition}
\label{prop:j2}
We have $\jmath_2^*(Z) = 0$.
\end{proposition}

\begin{proof}
Since the Weierstrass divisor $\overline{\mathcal{W}}_2$ is irreducible, we only need to show that $\jmath_2^{-1}(Z)$ does not contain $\overline{\mathcal{W}}_2$.  To do this, we exhibit a point in the Weierstrass divisor that does not lie in $\jmath_2^{-1}(Z)$, as follows.  Let $\Gamma$ be a chain of $g-2$ loops with bridges whose edge lengths are admissible in the sense of Definition~\ref{Def:Admissible}, and let $Y$ be a smooth curve of genus $g-2$ over $K$ whose skeleton is $\Gamma$.  Let $p \in Y$ be a point specializing to the left endpoint of $\Gamma$.  We consider the map $\jmath_2 \colon \mm_{2,1} \to \mm_g$ obtained by attaching the pointed curve $(Y,p)$ to an arbitrary stable pointed curve of genus $2$.

Let $Y'$ be a smooth curve of genus $2$ over $K$ whose skeleton $\Gamma'$ is a chain of $2$ loops connected by a bridge.  The tropicalization of the Weierstrass points on $Y'$ are known, and do not depend on the choice of curve with this skeleton.  See, e.g., \cite{Amini14} or \cite[Theorem~1.1]{JensenLen18}.  In particular, there is a Weierstrass point $p \in Y'(K)$ whose specialization is a $2$-valent point on the right loop.  Let $Y'' :=  Y \cup_p Y'$.  The skeleton of $Y''$ is obtained from $\Gamma$ and $\Gamma'$ by attaching infinitely long bridges at the specializations of $p$, and then gluing the infinitely far endpoints to each other, as in Figure~\ref{Fig:LongBridge}.\footnote{We recall that the topological space $Y^\an$ is obtained from its skeleton $\Gamma$ by attaching an $\RR$-tree rooted at each point.  The $K$-points of $Y$ naturally correspond to the leaves of these $\RR$-trees, and each leaf is infinitely far from the skeleton $\Gamma$, in the natural metric on $Y^\an \smallsetminus Y(K)$.  Hence, the analytification of the nodal curve $Y \cup Y'$ contains a skeleton which is the union of $\Gamma$, $\Gamma'$, and the infinite length paths from $\Gamma$ and $\Gamma'$, respectively, to the node $p$. See, for instance, \cite[\S8.3]{acp}.}  Note that $[Y''] \in \jmath_2 (\overline{\mathcal{W}}_2)$.  We will show that $[Y''] \not \in Z$.

\begin{figure}[H]
\begin{tikzpicture}[thick, scale=0.8]

\begin{scope}[grow=right, baseline]
\draw (-5,0)--(-4,0);
\draw (-3.5,0) circle (0.5);
\draw (-3,0)--(-2,0);
\draw (-2.5,-1) node {$\Gamma'$};
\draw (-1.5,0) circle (0.5);
\draw (-1,0)--(0.6,0);
\draw (2.4,0)--(4,0);
\draw (2.1,0) node {$\cdots$};
\draw (1,0) node {$\cdots$};

\filldraw (1.5,0) circle (0.05);
\draw (1.58,.4) node {{\small $p$}};

\draw (4.5,0) circle (0.5);
\draw (5,0)--(6,0);
\draw (6.5,0) circle (0.5);
\draw (7.5,-1) node {$\Gamma$};
\draw (7,0)--(7.5,0);

\draw (8,0) node {$\cdots$};
\draw (8.5,0)--(9,0);
\draw (9.5,0) circle (0.5);
\draw (10,0)--(11,0);

\draw (-5,0.25) node {\small $w_0$};
\draw (-4.2,0.25) node {\small $v_1$};
\draw (-2.8,0.25) node {\small $w_1$};
\draw (-2.2,0.25) node {\small $v_2$};
\draw (-0.8,0.25) node {\small $w_2$};

\end{scope}
\end{tikzpicture}

\caption{The skeleton of $Y''$ is the union of the skeletons $\Gamma'$ and $\Gamma$ of $Y'$ and $Y$, respectively, and the unique embedded paths from these skeletons to $p$.}
\label{Fig:LongBridge}
\end{figure}

If $[Y''] \in Z$, then $Z$ contains smooth curves $X$ whose skeletons are arbitrarily close to the skeleton of $Y''$.
Here, the topology is as in \cite[\S4]{acp}. We topologize $\RR_{>0} \cup \{\infty\}$ as an open subspace of the one point compactification of $\RR_{\geq 0}$.  Then the space of skeletons with underlying graph $G$ is identified with $(\RR_{>0} \cup \{\infty\})^{E(G)} / \mathrm{Aut}(G)$ by specifying the positive (and possibly infinite) length of each edge.
In particular, for each integer $N > 0$, there is an $[X] \in Z$ with skeleton a chain of loops $\widetilde \Gamma_X$ with edge lengths as follows.

Label the vertices and edges of $\widetilde \Gamma_X$, as in Figure~\ref{Fig:TheGraph}.  Then the bridge $\beta_3$ has length greater than $N$, and each other edge has length within $\frac{1}{N}$ of the corresponding edge in $\Gamma$ and $\Gamma'$. The metric graph $\widetilde \Gamma_X$ is similar to the skeleton pictured in Figure~\ref{Fig:LongBridge}, except that the doubly infinite bridge containing $p$ is replaced by an ordinary finite bridge that is much longer than all other edges.  We divide $\widetilde \Gamma_X$ into two subgraphs $\widetilde \Gamma'$ and $\widetilde \Gamma$, to the left and right, respectively, of the midpoint of the long bridge $\beta_3$.  (These subgraphs are arbitrarily close to $\Gamma'$ and $\Gamma$, respectively.)  Let $q \in X$ be a point specializing to $v_{g+1}$.  Since $[X] \in Z$, by Lemma~\ref{lem:ramify} there is a linear series in the degeneracy locus over $X$ that is ramified at $q$.  We now show that this impossible.

Let $V \subset H^0(X,L)$ be a linear series of degree $d$ and rank $6$ ramified at $q$.  We may assume that $L = \mathcal{O}(D_X)$, where $D = \Trop (D_{X})$ is a break divisor, and consider $\Sigma = \trop (V)$.  We will show that there are 28 tropically independent pairwise sums of functions in $\Sigma$ using a variant of the arguments in \S\S\ref{Sec:Construction}-\ref{Sec:Generic}.  It follows that the multiplication map $\Sym^2 V \to H^0(X, L^{\otimes 2})$ is injective, and hence $[X]$ is not  in $Z$.

To produce 28 tropically independent pairwise sums of functions in $\Sigma$, following the methods of \S\S\ref{Sec:Construction}-\ref{Sec:Generic}, we first consider the slope sequence along the long bridge $\beta_3$.  First, suppose that either $s_3 [5] \leq 2$ or $s_3 [4] + s_3 [6] \leq 5$.  The restriction of $\Sigma$ to $\widetilde \Gamma$ is a tropical linear series of rank $6$ with ramification at the left endpoint.  The proof of Theorem~\ref{thm:independence} then goes through verbatim, yielding a tropical linear combination of 28 functions in $\Sigma$ such that each function achieves the minimum uniquely at some point of $\widetilde \Gamma \subset \widetilde \Gamma_X$.

For the remainder of the proof, we therefore assume that $s_3[5] \geq 3$ and $s_3 [4] + s_3 [6] \geq 6$.  Since $\deg D_{\vert {\widetilde \Gamma'}} = 5$, we see that $s_3[6] \leq 5$.  Moreover, since the divisor $D_{\vert {\widetilde \Gamma'}} - s_3[5]w_2$ has positive rank on $\widetilde \Gamma'$, and no divisor of degree 1 on $\widetilde \Gamma'$ has positive rank, $s_3[5]$ must be exactly 3.  Since the canonical class is the only divisor class of degree 2 and rank 1 on $\widetilde \Gamma'$, we see that $D_{\vert{\widetilde \Gamma'}} \sim K_{\widetilde \Gamma'} + 3w_2$.  This yields an upper bound on each of the slopes $s_3[i]$, and these bounds determine the slopes for $i \geq 3$:
\[
s_3[6] = 5, \ s_3[5] = 3, \ s_3 [4] = 1, \ s_3 [3] = 0.
\]
Moreover, we must have $s'_2 [i] = s_3 [i]$ for $3 \leq i \leq 6$.   Since the linear series $V$ is ramified at $q$, we also have $s_{g+1}[6] \geq 7$. By Proposition~\ref{Thm:BNThm}, these conditions together imply that the sum of the multiplicities of all loops and bridges on $\widetilde \Gamma$ is at most 2.

To construct a certificate of independence on $\widetilde \Gamma_X$, we first construct a certificate of independence for 5 functions on $\widetilde \Gamma'$.  The construction is analogous to that in \S\ref{Sec:Switch}, with the second loop of $\widetilde \Gamma'$ playing the role of a switching loop.  The details are as follows.

For $i = 5,6$, there is a function $\varphi_i \in \Sigma$ such that
\[
s_k (\varphi_i) = s_k[i] \mbox{ for all } k \leq 3 \mbox{ and } s'_k (\varphi_i) = s'_k[i] \mbox{ for all } k \leq 2.
\]
We also have $\varphi_B, \varphi_C$ in $\Sigma$ (analogous to the similarly labeled functions in \S\ref{Sec:Switch}) satisfying:
\[
s'_0 (\varphi_C) = s_1 (\varphi_C) = s'_1 (\varphi_C) = s_2 (\varphi_C) = s_1[4] = 1, s'_2 (\varphi_C) = s_3 (\varphi_C) = s_3[3] = 0,
\]
\[
s'_2 (\varphi_B) = s_3 (\varphi_B) = s_3[4] = 1.
\]
Moreover, the slope of $\varphi_B$ at any point along the first 3 bridges is either 0 or 1.  Note in particular that all of the functions $\psi$ in the set $\{ \varphi_{66}, \varphi_{56}, \varphi_{55}, \varphi_B + \varphi_6 \}$ satisfy $s_3 (\psi) \geq 6$, and $s_3 (\varphi_C + \varphi_6) = 5$.

On the first bridge and first loop, we build a certificate of independence for the functions $\varphi_{66}, \varphi_{56}, \varphi_{55},$ and $\varphi_B + \varphi_6$ as in Figure~\ref{Fig:NoSwitch}.  Since all 4 of these functions have slope at least 6 along the very long bridge $\beta_3$, and $\varphi_C + \varphi_6$ has slope 5, we may set the coefficient of $\varphi_C + \varphi_6$ so that it obtains the minimum at some point of the very long bridge, but not at any point of the first two loops or bridges.

\begin{figure}[H]
\begin{tikzpicture}[thick, scale=0.8]

\begin{scope}[grow=right, baseline]
\draw (-5,0)--(-3,0);
\draw (-2,0) circle (1);
\draw (-1,0)--(1,0);
\draw (3,0)--(7,0);
\draw (2,0) circle (1);
\draw [ball color=black] (-1.29,0.71) circle (0.55mm);
\draw [ball color=black] (-2,1) circle (0.55mm);
\draw [ball color=black] (-4.5,0) circle (0.55mm);
\draw [ball color=black] (-3.5,0) circle (0.55mm);
\draw [ball color=black] (5.5,0) circle (0.55mm);
\draw (-5,0.3) node {$\varphi_{66}$};
\draw (-4,0.3) node {$\varphi_{56}$};
\draw (-.9,1.25) node {$\varphi_{B} + \varphi_6$};
\draw (0,0.3) node {$\varphi_{55}$};
\draw (6.5,0.3) node {$\varphi_{C} + \varphi_6$};

\end{scope}
\end{tikzpicture}

\caption{a certificate of independence on $\widetilde \Gamma'$}
\label{Fig:NoSwitch}
\end{figure}

We now construct a certificate of independence for 23 pairwise sums of functions in $\Sigma$ restricted to $\widetilde \Gamma$.  By Theorem~\ref{Thm:BNThm}, our computation of the slopes $s'_2 (\Sigma)$, together with the fact that the linear series $V$ is ramified at $q$, imply that the sum of the multiplicities of all loops and bridges on $\widetilde \Gamma$ is at most 2.  Just as in \S\ref{sec:choosingslopes}, but restricting to $\widetilde \Gamma$, we associate a sequence of partitions to $\Sigma$, use these partitions to characterize integers $z$ and $z'$, and thereby define a non-decreasing integer sequence $\sigma = (\sigma_3, \ldots, \sigma_{g+1})$, given by
\begin{equation} 
\sigma_k = \left\{ \begin{array}{ll}
4 & \textrm{if $3 \leq k \leq z$} \\
3 & \textrm{if $z+1 \leq k \leq z'-2$} \\
2 & \textrm{if $z'-1 \leq k \leq g'+1$.}
\end{array} \right.
\end{equation}
We then follow \S\ref{Sec:Generic} to identify a set $\cA$ of building blocks on $\widetilde \Gamma$ and a set $\cB \subset 2\cA$ satisfying properties $(\mathbf{B})$ and $(\mathbf{B'})$.  We proceed to construct a template $\theta$ exactly as in \S\ref{Sec:Construction}, except that we skip the step named ``Start at the First Bridge''.  Instead, we initialize the coefficients of the permissible functions on $\gamma_3$ in $\cB$ so that they agree with $\varphi_C + \varphi_6$ at the midpoint of $\beta_3$.  We then apply the loop subroutine on $\gamma_3$ and follow the algorithm until it terminates.

The arguments in \S\ref{Sec:Construction} go through without change, except for Lemma~\ref{lem:firstbridge}.  Specifically, since $2s'_2[5] = s'_2[6]+s'_2[4] = 6$, it is possible that two functions in $\cB$ have identical slopes greater than or equal to 5 along the bridge $\beta_3$.  In \S\ref{Sec:Construction}, Lemma~\ref{lem:firstbridge} is used only to guarantee that no two functions assigned to the first bridge of $\widetilde \Gamma$ agree on that bridge, and to count the number of cohorts on the first loop.  Here, we have not assigned any functions to the bridge $\beta_3$.  By arguments identical to those in \S\ref{Sec:Construction}, there are at most 3 cohorts on $\gamma_3$, and at most 2 if $\gamma_3$ is skippable.  We define the sets $\cS$ and $\cT$ exactly as in \S\ref{Sec:Generic}, and let $\cT' = \{ \psi \in \cT \mid s'_2 (\psi) \leq 4 \}$.  In each of the cases in \S\ref{Sec:Generic}, the number of functions in $\cT \smallsetminus \cT'$ is equal to the number of pairs $(i,j)$ such that $s'_2 [i] + s'_2 [j] \geq 5$.  Since there are precisely 5 such pairs, we see that $\vert \cT' \vert = 23$.  Then we show that the best approximation of the master template on $\widetilde \Gamma$ by $\cT'$ is a certificate of independence on $\widetilde \Gamma$, exactly as in \S\ref{Sec:Generic}.

Finally, note that any function $\psi$ that obtains the minimum on $\widetilde \Gamma$ satisfies $s'_2 (\psi) \leq 4$.  Similarly, each of the functions $\psi$ that obtains the minimum on $\widetilde \Gamma'$ satisfies $s_3 (\psi) \geq 5$.  Since the bridge $\beta_3$ is very long, it follows that no function that obtains the minimum on one of the two subgraphs can obtain the minimum on the other.  Thus, we have constructed a constructed tropical linear combination of 28 pairwise sums of functions in $\Sigma$ in which 5 achieve the minimum uniquely at some point of $\widetilde \Gamma'$ and 23 achieve the minimum uniquely at some point of $\widetilde \Gamma$.  In particular, this is a certificate of independence, as required.
\end{proof}

\subsection{Higher codimension boundary strata}

In order to verify \ref{it:codim2}, we now consider the intersection of $Z$ with the boundary strata $\Delta_{2,j}$, each of which has codimension 2 in $\mm_g$.

\begin{proposition}
\label{prop:nocodim2}
The component $Z$ does not contain any codimension 2 stratum $\Delta_{2,j}$.
\end{proposition}

\begin{proof}
The proof is again a variation on the independence constructions from the proof of Theorem~\ref{thm:independence}.  We fix $\ell = g-j-2$.  Let $Y_1$ be a smooth curve of genus 2 over $K$ whose skeleton $\Gamma_1$ is a chain of 2 loops with bridges, and let $p \in Y_1$ be a point specializing to the right endpoint of $\Gamma_1$.  Similarly, let $Y_2$ and $Y_3$ be smooth curves of genus $\ell$ and $j$, respectively, whose skeletons $\Gamma_2$ and $\Gamma_3$, are chains of $\ell$ loops and $j$ loops with admissible edge lengths.  Suppose further that the edges in the final loop of $\Gamma_2$ are much longer than those in the first loop of $\Gamma_3$.  Let $p, q \in Y_2$ be points specializing to the left and right endpoints of $\Gamma_2$, respectively, and let $q \in Y_3$ be a point specializing to the left endpoint of $\Gamma_3$.  We show that  $[Y'] = [Y_1 \cup_{p} Y_2 \cup_{q }Y_3] \in \Delta_{2,j}$ is not contained in $Z$.

As in the proof of Proposition~\ref{prop:j2}, if $[Y'] \in Z$, then $Z$ contains points $[X]$ corresponding to smooth curves whose skeletons are arbitrarily close to the skeleton of $Y'$ in the natural topology on $\M_g^{\trop}$.  In particular, there is an $X \in Z$ whose skeleton is a chain of loops $\Gamma_X$ whose edge lengths satisfy all of the conditions in Definition~\ref{Def:Admissible}, except that the bridges $\beta_3$ and $\beta_{\ell}$ are exceedingly long in comparison to the other edges.

Let $\Gamma$ be the subgraph of $\Gamma_X$ to the right of the midpoint of the bridge $\beta_3$. Note that $\Gamma$ is a chain of $g-2$ loops, labeled $\gamma_3, \ldots, \gamma_g$, with bridges labeled $\beta_3, \ldots, \beta_{g+1}$.

By Lemma~\ref{lem:ramify}, there is a linear series $V \subset H^0(X, L)$ of degree $g+3$ and rank $6$ on $X$ that is ramified at a point $x$ specializing to the righthand endpoint $v_{g+1}$, and such that the multiplication map $\Sym^2 V \to H^0(X, L^{\otimes 2})$ is not injective.  We will show that this is not possible, by adapting the tropical independence constructions from \S\S\ref{Sec:Construction}-\ref{Sec:Generic}.  We define building blocks as PL functions on $\Gamma$ exactly as in \S\ref{Sec:Construction}, and then, to account for the length of $\beta_\ell$, we use the following variant on the definition of permissible functions (Definition~\ref{def:permissible}).

\begin{definition}
\label{def:ellpermiss}
Let $\sigma = (\sigma_3, \ldots, \sigma_{g+1})$ be a non-increasing integer sequence and let $\psi \in \PL (\Gamma)$ be a function with constant slope along each bridge.  We say that $\psi$ is $(\sigma, \ell)$-\emph{permissible} on $\gamma_k$ if
\begin{enumerate} [label=(\roman*)]
\item  either $s_j (\psi) \leq \sigma_k$ for all $j \leq k$; or $s_{\ell} (\psi) < \sigma_\ell$ and $s_j (\psi) \leq \sigma_j$ for all $\ell < j \leq k$,
\item  $s_{k+1} (\psi) \geq \sigma_k$, and
\item  if $s_j (\psi) < \sigma_k$ for some $j>k$, then $j \neq \ell$ and $s_{k'}(\psi) > \sigma_{k'}$ for some $k'$ such that $k < k' < j$.
\end{enumerate}
\end{definition}

\noindent This notion of $(\sigma,\ell)$-permissibility is the natural analog of $(\sigma)$-permissibility when the bridge $\beta_\ell$ is much longer than all of the other edges. In particular, if $\theta$ is a PL function whose average slope on $\beta_j$ is very close to $\sigma_j$, for all $j$, and if the best approximation from above of $\theta$ by $\psi$ achieves the minimum on $\gamma_k$, then $\psi$ must be $(\sigma, \ell)$-permissible on $\gamma_k$, c.f. Remark~\ref{rem:permissible}.  Also, if $\psi$ has constant slope along each bridge then either $s_3 (\psi) > s_3 (\theta)$, $s_{g+1} (\psi) < s_{g+1} (\theta)$, or $\psi$ is $\ell$-permissible on some loop $\gamma_k$, c.f. Lemma~\ref{Lem:VAEverythingIsPermissible}.

Let $\Sigma = \trop(V)$. Since Proposition~\ref{Thm:BNThm} does not depend on the lengths of the bridges, we have that either $s'_2 [5] \leq 2$ or $s'_2 [4] + s'_2 [6] \leq 5$.  Also, since $V$ is ramified at $x$, we have $s_{g+1}[6] > 6$. The restriction of $\Sigma$ to $\Gamma$ is a tropical linear series of rank $6$, and we proceed to apply the arguments from \S\S\ref{Sec:Construction}-\ref{Sec:Generic}.

We construct the master template exactly as in \S\ref{Sec:Construction}, using $(\sigma, \ell)$-permissible functions in place of permissible functions.  Definition~\ref{def:ellpermiss} ensures that we only assign a function $\psi$ to the left of $\beta_{\ell}$ if $s_{\ell} (\psi) \geq \sigma_{\ell}$, and we only assign it to the right of $\beta_{\ell}$ if $s_{\ell} (\psi) \leq \sigma_{\ell}$.  Specifically, if $\psi$ is permissible on $\gamma_k$ for $k < \ell$, then by Definition~\ref{def:ellpermiss}(iii), $s_{\ell} (\psi) \geq \ell$, and if $\psi$ is permissible on $\gamma_k$ for $k \geq \ell$, then by Definition~\ref{def:ellpermiss}(i), $s_{\ell} (\psi) \leq \sigma_{\ell}$.

Next, with the template fixed, we specify a set $\cS$ of elements of $\Sigma$ and a set $\cT$ of pairwise sums of elements of $\cS$, using precisely the same algorithm as in \S\ref{Sec:Generic}.  In order to prove that the best approximation $\Upsilon$ of $\theta$ by $\cT$ is a certificate of independence, some care must be taken to account for the length of $\beta_\ell$, and we explain the details as follows.

The ramification conditions imply that the sum of the multiplicities of all the bridges and loops is at most 1.  Hence, there are no switching bridges, and at most one switching loop.  Moreover, if there is a switching loop, it has multiplicity 1, and there are no decreasing loops or bridges.  Hence the choice of $\cS$ and $\cT$ follows either Case 1 or Case 3, from \S\ref{sec:noswitching} or \S\ref{Sec:Switch}, respectively.

Among these cases, there is only one situation where the proof that $\vartheta_\cT$ is a certificate of independence uses the assumption that the bridges decrease in length from left to right: when there is a loop $\gamma_{\ell'}$ that switches slope $h$, the interval $I$ has length at least $m_{\ell'}$, and there are indices $j$ and $j'$ such that
\[
\sigma_{\ell'} = s'_{\ell'}[h] + s'_{\ell'}[j] = s'_{\ell'}[h] + s'_{\ell'} [j'] + 1.
\]
In this situation, we must show that the best approximation of $\theta$ by $\varphi_C + \varphi_j$ achieves equality on the region where $\varphi^0_{h+1} + \varphi_j$ achieves the minimum, and the best approximation of $\theta$ by $\varphi_A + \varphi_j$ achieves equality on the region where $\varphi^{\infty}_h + \varphi_j$ achieves the minimum.  In Lemma~\ref{Lem:TwoRegions} and related arguments, this is done by noting that both functions have slope larger than that of $\theta$ on intervals of length $t \geq m_{\ell'}$.  In the present case, this is insufficient, because we may have $\ell' < \ell$, and the bridge $\beta_{\ell}$ is very long.

However, since there are no decreasing loops or bridges, we have
\[
s_k [h] + s_k[j] \geq s'_{\ell'}[h] + s'_{\ell'}[j] \mbox{ for all } k \geq \ell'.
\]
It follows that $s_k (\varphi_A + \varphi_j) \geq \sigma_k$ and $s_k (\varphi_C + \varphi_j) \geq \sigma_k$ for all $k \geq \ell'$, and the result follows.  Therefore, the construction yields a certificate of independence for 28 pairwise sums of functions in $\Sigma$, and the proposition follows.
\end{proof}

Propositions~\ref{prop:noboundary}, \ref{prop:j2}, and~\ref{prop:nocodim2} show that $Z$ satisfies conditions \ref{it:closure}-\ref{it:codim2}.  For $g = 23$, we conclude that $\fU\subseteq \widetilde{\mathcal{G}}^6_{26}$ is generically finite over each codimension one component of its image in $\mm_{23}$, and hence $\mm_{23}$ is of general type.

\medskip

For $g = 22$, we proceed to verify \ref{it:j3} by studying the pull back of $Z$ to $\mm_{3,1}$.

\subsection{Pulling back to $\mm_{3,1}$}

Recall that $\jmath_3 : \mm_{3,1} \rightarrow \mm_g$ is the map obtained by attaching a fixed general pointed curve of genus $g-3$ to an arbitrary stable pointed curve of genus $3$.

\begin{proposition} \label{prop:j3}
The preimage $\jmath_3^{-1}(Z)$ is contained in the union of the Weierstrass locus $\overline{\mathcal{W}}_{3}$ and the hyperelliptic locus $\overline{\mathcal{H}}_3$ in $\mm_{3,1}$.
\end{proposition}

We prove this proposition using a variation of the arguments from the vertex avoiding case in \S\ref{Sec:VertexAvoiding}, as follows.  Let $X$ be a curve of genus $19$ over $K$ whose skeleton $\Gamma$ is a chain of $19$ loops with bridges, with admissible edge lengths. Let $q \in X$ be a point specializing to the left endpoint $w_0$ of $\Gamma$, and let $\jmath_3:\mm_{3,1}\rightarrow \mm_g$ be the map obtained by attaching an arbitrary stable pointed curve $(X',q)$ of genus $3$ to $(X,q)$.  We now show that the curve $[Y] = [X \cup_q X']$ is not in $Z$ when $X'$ is not hyperelliptic and $q \in X'$ is not a Weierstrass point.

As in Lemma~\ref{lem:Weierstrass}, if $[Y] \in Z$, then there is a limit linear series $\ell$ of degree 25 and rank $6$ on $Y$ such that the multiplication map on each aspect of $\ell$ is not injective.  Let $\ell_X \subseteq H^0(X, \mathcal{O}(D_X))$ be the $X$-aspect of such a limit linear series. As in Lemma~\ref{lem:ramify}, we may assume that $\ell_X$ is ramified at a point $p$ specializing to the right endpoint $v_{20}$ of $\Gamma$.  To complete the proof of the proposition, we use a variation on the arguments from \S\ref{Sec:VertexAvoiding} to show that there are 28 tropically independent pairwise sums of functions in $\Sigma := \trop(\ell_X)$.

We may assume that $D = \Trop(D_X)$ is a break divisor.  We claim that
\[
\sum_{i=0}^3 a_i^{\ell_X} (q) \geq 13.
\]
Since $X'$ is not hyperelliptic, we have $a_{r-1}^{\ell_{X'}} (q) \leq d-3$.  Furthermore, if equality holds, then since $q' \in X'$ is not a Weierstrass point, we have $a_r^{\ell_{X'}} (q) \leq d-1$.  The claim then follows from the definition of a limit linear series.

It follows that  $\sum_{i=3}^6 s'_0[i] \leq 9$ (see Example~\ref{ex:ramified}).  The ramification condition at the point specializing to $v_{20}$ implies $s_{20}[6] \geq 7$.  By Proposition~\ref{Thm:BNThm}, it follows that all of the bridges and loops have multiplicity zero, and the inequalities on slopes must in fact be equalities:
\[
\sum_{i=3}^6 s'_0[i] = 9, \mbox{ and } s_{20}[6] = 7.
\]
Because of this, we treat this case in a similar manner to the vertex avoiding case of \S\ref{Sec:VertexAvoiding}.  There are finitely many such classes; each corresponds to a standard Young tableaux on one of the three shapes depicted in Figure~\ref{Fig:ThreeShapes}.  The particular shape is determined by the sequence of slopes along the first bridge $\beta_1$.  More precisely, the three missing boxes from the upper left corner form the partition $\lambda'_0$.  That this partition consists of precisely 3 boxes corresponds to the fact that $\sum_{i=3}^6 s'_0[i] = 9$.  We refer to these three shapes as Case A, Case B, and Case C, respectively.

\begin{figure}[h]
\begin{tikzpicture}[thick, scale=0.8]

\begin{scope}[grow=right, baseline]
\draw (1,5)--(3,5);
\draw (-0.5,4.5)--(3,4.5);
\draw (-0.5,4)--(3,4);
\draw (-0.5,3.5)--(3,3.5);
\draw (-0.5,3)--(0,3);
\draw (-0.5,4.5)--(-0.5,3);
\draw (0,4.5)--(0,3);
\draw (.5,4.5)--(.5,3.5);
\draw (1,5)--(1,3.5);
\draw (1.5,5)--(1.5,3.5);
\draw (2,5)--(2,3.5);
\draw (2.5,5)--(2.5,3.5);
\draw (3,5)--(3,3.5);
\draw (1.5,2.85) node {Case A};

\draw (5,5)--(7.5,5);
\draw (4.5,4.5)--(7.5,4.5);
\draw (4,4)--(7.5,4);
\draw (4,3.5)--(7.5,3.5);
\draw (4,3)--(4.5,3);
\draw (4,4)--(4,3);
\draw (4.5,4.5)--(4.5,3);
\draw (5,5)--(5,3.5);
\draw (5.5,5)--(5.5,3.5);
\draw (6,5)--(6,3.5);
\draw (6.5,5)--(6.5,3.5);
\draw (7,5)--(7,3.5);
\draw (7.5,5)--(7.5,3.5);
\draw (6,2.85) node {Case B};

\draw (9,5)--(12,5);
\draw (9,4.5)--(12,4.5);
\draw (9,4)--(12,4);
\draw (8.5,3.5)--(12,3.5);
\draw (8.5,3)--(9,3);
\draw (8.5,3.5)--(8.5,3);
\draw (9,5)--(9,3);
\draw (9.5,5)--(9.5,3.5);
\draw (10,5)--(10,3.5);
\draw (10.5,5)--(10.5,3.5);
\draw (11,5)--(11,3.5);
\draw (11.5,5)--(11.5,3.5);
\draw (12, 5)--(12,3.5);
\draw (10.5,2.85) node {Case C};

\end{scope}
\end{tikzpicture}

\caption{Skew Partitions Corresponding to the Given Slope Conditions}
\label{Fig:ThreeShapes}
\end{figure}

In Case C, when $s'_0 [6] = 6$ and $s'_0 [i] = i-3$ for all $i<6$,  the functions $\varphi_{66}, \varphi_{56}, \varphi_{46}, \varphi_{36},$ and $\varphi_{26}$ all have distinct slopes greater than or equal to $5$ on $\beta_1$.   Let $z$ be the 4th smallest entry appearing in the first two rows of $T$, and let $z'$ be the 8th smallest entry appearing in the union of the second and third row.  Define $\sigma$ as in Definition~\ref{def:z}.  Applying the algorithm from \S\ref{Sec:VertexAvoiding} then produces a certificate of independence, with all 5 of these functions assigned to the first bridge.

In the other two cases, however, we see that
\[
s_1 [4] + s_1[6] = 2s_1[5] = 6 \geq 5.
\]
The algorithm from \S\ref{Sec:VertexAvoiding} still works to produce a certificate of independence, but we must choose the input $\sigma = (\sigma_1, \ldots, \sigma_{19})$ differently in this case.  Let $z_1$ be the smallest symbol appearing in the first two rows of the tableau.  (Note that, because this is a skew tableau, $z_1$ is not necessarily equal to 1.)  Similarly, let $z_2$ be the second smallest symbol in the first two rows of the tableau.  In Case A, let $z_3$ be the 4th smallest symbol appearing in the union of the first and third row, and in Case B, let $z_3$ be the 5th smallest symbol appearing in the union of the first and third row.  Finally, in Case A, let $z_4$ be the 9th smallest symbol appearing in the union of the second and third row, and in Case B, let $z_4$ be the 8th smallest symbol appearing the union of the second and third row.

We then define
\begin{displaymath}
\sigma_k = \left\{ \begin{array}{ll}
6 & \textrm{if $k \leq z_1$,} \\
5 & \textrm{if $z_1 < k \leq z_2$,}\\
4 & \textrm{if $z_2 < k \leq z_3$,} \\
3 & \textrm{if $z_3 < k \leq z_4$,}\\
2 & \textrm{if $z_4 < k \leq 19$.}
\end{array} \right.
\end{displaymath}

We now count the number of $\sigma$-permissible functions on each region of the graph where $\sigma$ is constant, as in \S\ref{Sec:VertexAvoiding}.  We first show the following.

\begin{lemma}
\label{lem:AtMostThreeGenus19}
For any loop $\gamma_k$, there are at most 3 non-departing $\sigma$-permissible functions on $\gamma_k$.  Moreover, there are at most 3 $\sigma$-permissible functions on the loops $\gamma_1 , \gamma_{z_1+1},$ and $\gamma_{z_4+1}$.
\end{lemma}

\begin{proof}
The proof of Lemma~\ref{Lem:VAAtMostThree} holds in all cases, except when $z_2 < k \leq z_3$. This last case is handled as follows.  Suppose there are 4 non-departing permissible functions on $\gamma_k$. Then, as in the proof of Lemma~\ref{Lem:AtMostThree}, we must have
\[
s_{k+1} (\varphi_i) + s_{k+1} (\varphi_{6-i}) = 4 \mbox{ for all } i.
\]
In other words, if we consider the skew tableau consisting of symbols less than or equal to $k$, we see that the sum of heights of the $i$th column and the $(6-i)$ith column must be equal to 4.  We therefore see that $k > z_3$, and the lemma follows.
\end{proof}

We now define 3 more loops.  Each will be to the right of $\gamma_{z_2}$.  Let $b_3$ be the third smallest symbol in the first two rows of the tableau.  In Case A, let $b_4$ be the 5th smallest symbol appearing in the union of the first and third row, and in Case B, let $b_4$ be the 6th smallest symbol appearing in the union of the first and third row.  Finally, in Case A, let $b_5$ be the 10th smallest symbol appearing the union of the second and third row, and in Case B, let $b_5$ be the 9th smallest symbol appearing in the union of the second and third row.  Note the following inequalities:
\[
z_1 < z_2 < b_3 < z_3 < b_4 < z_4 < b_5 .
\]

\begin{lemma}
\label{Lem:CountingGenus19}
If $\ell \in \{ z_i , b_i \}$ and $\ell \notin \{ 1, z_i+1 \mid 1 \leq i \leq 4 \}$, then there are no new permissible functions on $\gamma_\ell$.  If either $b_3 = z_2 + 1$, then there are only 3 permissible functions on $\gamma_{b_3}$, and if $b_4 = z_4 +1$, then there are only 3 permissible functions on $\gamma_{b_4}$.
Similarly, if $z_1 = 1$, then there are only 2 permissible functions on $\gamma_{z_1}$, if $z_2 = z_1 + 1$, then there are only 2 permissible functions on $\gamma_{z_2}$, and if $b_5 = z_4 + 1$, then there are only 2 permissible functions on $\gamma_{b_5}$.  If $k \neq z_i$ or $b_i$ for any $i$, then there is a new permissible function on $\gamma_k$.
\end{lemma}

\begin{proof}
The proof is identical to that of Proposition~\ref{prop:nonew}.  For each of these loops $\gamma_\ell$, first enumerate the possible sequences of slopes $s_{\ell+1}[i]$.  Then note that, for any value $i$ that could satisfy $s_{\ell+1}[i] > s_\ell[i]$, there is no value $j$ such that $s_{\ell+1}[i] + s_{\ell+1}[j] = \sigma_\ell$.  Such values of $i$ must necessarily satisfy $s_{\ell+1}[i] > s_{\ell+1}[i-1]+1$, but the converse is not true.  For example, we consider the case $\ell = z_1$, and leave the remaining cases to the interested reader.  The possible sequences of slopes $s_{z_1+1}[i]$ are:
\[
(-3, -2, -1, 1, 2, 3, 4)
\]
\[
(-3, -2, -1, 0, 2, 3, 5)
\]
\[
(-3, -2, -1, 0, 1, 4, 5).
\]
\[
(-3, -2, -1, 0, 1, 4, 6).
\]
\[
(-3, -2, -1, 0, 1, 4, 7).
\]
By the definition of $z_1$, in the last two cases we have $s_{z_1+1}[6] = s_{z_1}[6]$.  In each of the cases, we see that for any of the remaining values of $i$ satisfying $s_{\ell+1}[i] > s_{\ell+1}[i-1]+1$, there is no value $j$ such that $s_{\ell+1}[i] + s_{\ell+1}[j] = 6$.

We will prove the last statement in the case where $z_1 \neq 1$, $z_2 \neq z_1 + 1$, $b_3 \neq z_2 + 1$, $b_4 \neq z_3 + 1$, and $b_5 \neq z_4 +1$.  The other cases are similar.  Note that there are 2 functions $\varphi_{ij}$ satisfying $s_1 (\varphi_{ij}) > 6$, and two more functions $\varphi_{ij}$ satisfying $s_{20} (\varphi_{ij}) < 2$.  Each of the remaining 24 functions is permissible on some loop.  Of these, the number of functions that are new on $\gamma_{\ell}$ for $\ell \in \{ 1, z_i+1 \mid 1 \leq i \leq 4 \}$ is at most $3+3+4+4+3=17$, leaving at least 7 functions.  There are no new new functions on $\gamma_{\ell}$ for $\ell \in \{ z_i , b_i \}$.  There are $19-7-5=7$ values of $\ell$ that are in neither of these two sets.  Since the number of functions remaining is greater than or equal to the number of values of $\ell$ in neither set, we see that we must in fact have equality, and there must be a new function on $\gamma_{\ell}$ for $\ell \notin \{ z_i , b_i \}$.
\end{proof}

To complete the proof it suffices to show that, for our choice of $\sigma$, every function $\varphi_{i,j}$ is assigned to some bridge or loop.  By construction, $\alpha(\varphi_{6,6}) = \alpha(\varphi_{5,6}) = \beta_1$, and $\alpha(\varphi_{0,1}) = \alpha(\varphi_{0,0}) = \beta_{g+1}$.  If $z_1 \neq 1$, there are 3 permissible functions on $\gamma_1$ by Lemma~\ref{Lem:CountingGenus19}, and if $z_1 = 1$, there are 2.  On each loop $\gamma_k$ for $k < z_1$, there is one new permissible function.  To each such loop, we assign a function $\varphi \in \cB$.  Moreover, if there is an unassigned departing permissible function on $\gamma_k$, we assign it to $\gamma_k$.  It follows that there are 3 unassigned permissible functions on each loop $\gamma_k$ with $k < z_1$, and on $\gamma_{z_1}$, there are 2.  These two functions are assigned to the loop $\gamma_{z_1}$ and the bridge $\beta_{z_1+1}$.

A similar analysis, using Lemma~\ref{Lem:CountingGenus19}, shows that we assign a function to every loop $\gamma_k$ with $k > z_1$, we assign one function to $\beta_{z_i+1}$ for all $i$, and one additional function (aside from $\varphi_{0,1}$ and $\varphi_{0,0}$) to $\beta_{23}$.  In particular, since there are precisely 19 loops, the total number of functions assigned to a bridge or loop is 28.  Hence, the output of the algorithm is a certificate of independence among all 28 functions in $\cB = \{ \varphi_{i,j} \mid 0 \leq i \leq j \leq 6\}$.

Propositions~\ref{prop:noboundary}, \ref{prop:j2}, \ref{prop:nocodim2}, and \ref{prop:j3} together show that $Z$ satisfies the vanishing conditions from Proposition~\ref{prop:vanishing}.  We conclude that the degeneracy locus $Z$ is generically finite over each codimension 1 component of its image.  This proves Theorem~\ref{thm:genfinite} and completes the proof of Theorem~\ref{thm:generaltype}. \qed

\bibliography{math(4)}

\end{document}